\documentclass[10pt,aps,prb, preprint, reqno]{amsart}
\usepackage{amsmath, amssymb, amsthm}
\usepackage{mathrsfs}
\usepackage{mathtools}
\usepackage{MnSymbol} 
\usepackage{comment}
\usepackage{yfonts}
\usepackage{graphicx}
\usepackage {latexsym}
\usepackage[linktocpage=true]{hyperref}
\usepackage{cancel}
\usepackage{setspace}
\usepackage{esint}
\usepackage{enumitem}
\usepackage{ushort}
\usepackage{cite}

\newif\ifshow
\showtrue

\ifshow

\else
\excludecomment{details}
\fi

\newtheorem{theorem}{Theorem}[section]
\newtheorem{lemma}[theorem]{Lemma}
\newtheorem{proposition}[theorem]{Proposition}

\newtheorem{define}{Definition}[section]

\newtheorem{remark}{Remark}[section]

\newcommand{\joinR}{\hspace{-.1em}}
\newcommand{\RomanI}{I}
\newcommand{\RomanII}{\mbox{\RomanI\joinR\RomanI}}
\newcommand{\RomanIII}{\mbox{\RomanI\joinR\RomanII}}
\newcommand{\RomanIV}{\mbox{\RomanI\joinR\RomanV}}
\newcommand{\RomanV}{V}

\DeclareMathOperator*{\Tr}{Tr}
\DeclareMathOperator*{\Id}{Id}
\DeclareMathOperator*{\esssup}{esssup}
\DeclareMathOperator*{\supp}{supp}
\DeclareMathOperator*{\divergence}{div}

\DeclarePairedDelimiter\ceil{\lceil}{\rceil}
\DeclarePairedDelimiter\floor{\lfloor}{\rfloor}

\title[SQG equations forced by random noise]{Surface quasi-geostrophic equations forced by random noise: prescribed energy and \\non-unique Markov selections}
\author[Elliott Walker]{Elliott Walker}  
\address{Texas Tech University, Department of Mathematics and Statistics, Lubbock, TX, 79409-1042, U.S.A.; Phone: 806-834-6112; Fax: 806-742-1112}
\email{ellwalke@ttu.edu} 
\author[Kazuo Yamazaki]{Kazuo Yamazaki}  
\address{University of Nebraska, Lincoln, Department of Mathematics, 243 Avery Hall, PO Box 880130, Lincoln, NE 68588-0130, U.S.A.; Phone: 402-473-3731; Fax: 402-472-8466}
\email{kyamazaki2@unl.edu}
\date{}

\allowdisplaybreaks

\begin{document}
\maketitle

\begin{abstract} 
We consider the momentum formulation of the two-dimensional surface quasi-geostrophic equations forced by random noise, of both additive and linear multiplicative types. For any prescribed deterministic function under some conditions, we construct solutions to each system whose energy is the fixed function. Consequently, we prove non-uniqueness of almost sure Markov selections of suitable class of weak solutions associated to the momentum surface quasi-geostrophic equations in both cases of noise. 
\end{abstract}

\footnote{The second author is supported by the Simons Foundation (962572, KY)}
\footnote{Kazuo Yamazaki is the corresponding author.}

\setcounter{tocdepth}{3}
\makeatletter
\def\l@section{\@tocline{2}{2pt}{2pc}{}{}}
\def\l@subsection{\@tocline{2}{2pt}{4pc}{}{}}
\def\l@subsubsection{\@tocline{2}{2pt}{6pc}{}{}}
\makeatother
\small
\tableofcontents
\normalsize

\section{Introduction}
\subsection{Physical Motivation} 
Mathematical study of properties such as the well-posedness and uniqueness of solutions for a given system of partial differential equations (PDEs) in hydrodynamics provides insights as to the consistency of a model with the physical reality of fluid flows, and justifies their usage in predictive simulations in a scientific and engineering context. The two-dimensional (2D) surface quasi-geostrophic (SQG) equations \eqref{est 115} model the motion of buoyancy or potential temperature $\theta$ on a horizontal fluid boundary and have been studied since the pioneering work of Charney \cite{C48} in 1948. This system has applications in oceanography, meteorology, and related fields in applied sciences (e.g. \cite{G82,HPGS95, MW06, MGNIP22}). Additionally, it has attracted the attention of mathematicians, largely due to analytical similarities with other important systems such as the Euler and the Navier-Stokes (NS) equations \eqref{est 114}, as shown by Constantin, Majda, and Tabak \cite{CMT94}. 

In 1956, Landau and Lifshitz \cite{LL57} introduced the idea of ``fluctuating hydrodynamics,'' where by including a forcing term with a random component in a hydrodynamic system of PDEs, one can better model the role of both internal and external noise on the physical system (also \cite{N65}). Many works such as \cite{K11} have by now demonstrated, potentially counterintuitively from physics point of view (e.g. \cite{SSO07}), that certain noisy fluctuations can have a regularizing effect on the solution leading an initial value problem to become probabilistically well-posed.

The purpose of this manuscript is to employ convex integration on the momentum SQG equations forced by noise of additive and linear multiplicative types and construct global-in-time probabilistically strong analytically weak solutions with prescribed energy that are non-unique in law. Additionally, we prove non-uniqueness of Markov selections in both cases. 

\subsection{Model Equations and Previous Works}\label{Section 1.2}
We denote $\mathbb{N} \triangleq \{1, 2, \hdots, \}$ while $\mathbb{N}_{0} \triangleq \{0\} \cup \mathbb{N}$. We denote the $j$-th component of any $d$-dimensional vector $x$ by $x_{j}$ for $j \in \{1, \hdots, d\}$. For any square matrix $A$, the $(m,l)$ entry is denoted by superscript $A^{ml}$, and its trace-free part by $\mathring{A}$. The spatial domain of our main interest will be $\mathbb{T}^2\triangleq[-\pi,\pi]^2$, although some of our discussions can be extended to $\mathbb{T}^{d}$ for general $d \in \mathbb{N}$, or even $\mathbb{R}^{d}$. We denote $\partial_{t} \triangleq \frac{\partial}{\partial t}$ and $\partial_{j} \triangleq \frac{\partial}{\partial x_{j}}$ for $j \in \{1, \hdots, d\}$. We recall a fractional Laplacian $\Lambda^{r} \triangleq (-\Delta)^{\frac{r}{2}}$ for any $r \in \mathbb{R}$ as the Fourier operator with its Fourier symbol $\lvert m \rvert^{r}$ so that $\Lambda^{r} f(x) \triangleq \sum_{m\in\mathbb{Z}^{2}} \lvert m \rvert^{r} \hat{f}(m) e^{im\cdot x}$ where $\hat{f}$ denotes the Fourier transform of $f$;  $\check{f}$ will denote the Fourier inverse of $f$. We write $a\wedge b\triangleq\min\{a,b\}$. Finally, alongside the standard function spaces of the H$\ddot{\mathrm{o}}$lder spaces $C^\gamma$ for $\gamma \geq 0$ (see \cite{BCD11} for generalizations via Besov spaces), Lebesgue spaces $L^p$ for $p \in [1,\infty]$, and the homogeneous Sobolev spaces $\dot{H}^s$ for $s \in \mathbb{R}$ with their typical norms, we additionally define $\dot{H}_{\sigma}^{\frac{1}{2}} \triangleq \{ f \in \dot{H}^{\frac{1}{2}} (\mathbb{T}^{2}): \nabla\cdot f = 0, \fint_{\mathbb{T}^{2}} f dx = 0 \}$ where $\fint_{\mathbb{T}^{2}} f dx \triangleq \lvert \mathbb{T}^{2} \rvert^{-1} \int_{\mathbb{T}^{2}} f(x) dx$, as well as $\dot{H}_{\sigma}^{s}$ and $H_{\sigma}^{s}$ for $s \geq 0$ similarly. We also define for each $q\in\mathbb{N}_{0}$ the norms
\begin{subequations}
\begin{align}
&\lVert f \rVert_{C_{t,x}} \triangleq \sup_{s \in [0,t], x \in \mathbb{T}^{2}} \lvert f(s,x) \rvert, \hspace{7mm}  \lVert f \rVert_{C_{t,x}^{n}} \triangleq \sum_{0 \leq j + \lvert \beta \rvert \leq n} \lVert \partial_{t}^{j} D^{\beta} f \rVert_{C_{t,x}}, \\
&||f||_{C_{t,x,q}}\triangleq \sup_{s \in[t_q,t], x\in\mathbb{T}^2}|f(s,x)|, \hspace{3mm} \lVert f \rVert_{C_{t,x,q}^{n}} \triangleq \sum_{0 \leq j + \lvert \beta \rvert \leq n} \lVert \partial_{t}^{j} D^{\beta} f \rVert_{C_{t,x,q}}, 
\end{align}
\end{subequations} 
for $j \in \mathbb{N}_{0}, \beta \in \mathbb{N}_{0}^{2}$, and $t_{q}$ to be defined subsequently in \eqref{define tq}.

For comparison purposes, we first introduce an initial value problem of  
\begin{subequations}\label{est 114} 
\begin{align}
& \partial_{t} u + \divergence (u\otimes u) + \nabla \pi = \eta \Delta u,  \hspace{3mm} \nabla\cdot u = 0,  \hspace{3mm} \text{ for } t > 0,\label{est 114a}  \\
& u(0,x) = u^{\text{in}}(x), \label{est 114b}
\end{align}
\end{subequations} 
where $\eta \geq 0$ represents viscosity coefficient, $u: \mathbb{R}_{+} \times \mathbb{T}^{d} \mapsto \mathbb{R}^{d}$ and $\pi: \mathbb{R}_{+} \times \mathbb{T}^{d} \mapsto \mathbb{R}$ represent the fluid velocity and pressure, respectively. This system \eqref{est 114} represents the Euler equations when $\eta = 0$ whereas the NS equations when $\eta > 0$. 

The initial value problem of the SQG equations is, for $\gamma \in (0,2]$,  
\begin{subequations}\label{est 115} 
\begin{align}
&\partial_{t}\theta + \divergence (u\theta)  + \nu \Lambda^{\gamma }\theta = 0, \hspace{3mm}  u =  \mathcal{R}^{\perp}\theta, \hspace{3mm} \text{ for } t > 0,  \label{est 115a}\\
& \theta(0,x) = \theta^{\text{in}}(x), \label{est 115b} 
\end{align}
\end{subequations} 
where $\mathcal{R}$ is the Riesz transform vector, $(x_{1},x_{2})^\perp \triangleq (-x_{2},x_{1})$, $\nu \geq 0$ the thermal diffusivity coefficient, and $\theta: \mathbb{R}_{+} \times \mathbb{T}^{2} \mapsto \mathbb{R}$ represents potential temperature. The generalized NS equations with $\Delta u$ in \eqref{est 114a} replaced by $- \Lambda^{\gamma} u$ has also been extensively investigated since the pioneering work \cite{L59} by Lions in 1959. The 2D SQG equations \eqref{est 115} with $\gamma = 1$ is also a physically important model for the Ekmann pumping effect in the boundary layer near the surface (e.g. \cite{C02}). We define $v\triangleq \Lambda^{-1}u$ as the potential velocity and transform the SQG equations \eqref{est 115} into a momentum formulation 
\begin{equation}\label{est 116}
\partial_t v + (u\cdot \nabla) v - (\nabla v)^T\cdot u + \nabla p + \Lambda^{\gamma} v = 0, \hspace{3mm} u = \Lambda v, \hspace{3mm}  \nabla\cdot v = 0, 
\end{equation} 
where we have taken $\nu=1$ for simplicity. 

Next, to model the effects of random noise as previously discussed, we consider a stochastic force added to the system. The general form of the momentum SQG equations forced by random noise of our interest is 
\begin{equation}\label{general msqg}
dv+[(u \cdot \nabla)v-(\nabla v)^{T}\cdot u + \nabla p + \Lambda^{\gamma}v]dt=G(v)dB,  \hspace{3mm} u = \Lambda v, \hspace{3mm} \nabla\cdot v = 0.
\end{equation} 
There are various types of noise which can be incorporated; the two which are of interest to us here are additive and linear multiplicative types. The momentum SQG equations with additive noise so that $G(v) dB$ in \eqref{general msqg} is replaced by $dB = \sum_{i=1}^{\infty} G_{i} dB_{i}$ where $B$ is a $GG^{\ast}$-Wiener process to be described subsequently (see \eqref{regularity of noise}) is 
\begin{equation}\label{additive msqg}
dv+[(u \cdot \nabla)v-(\nabla v)^{T}\cdot u + \nabla p + \Lambda^{\gamma}v]dt=dB, \hspace{3mm} u = \Lambda v, \hspace{3mm} \nabla\cdot v = 0,
\end{equation} 
while that with linear multiplicative noise so that $G(v) dB$ in \eqref{general msqg} is replaced by $v dB$ where $B$ is a $\mathbb{R}$-valued Wiener process is 
\begin{equation}\label{mult msqg}
dv + [(u \cdot\nabla) v - (\nabla v)^{T} \cdot u + \nabla p + \Lambda^{\gamma} v] dt = v dB, \hspace{3mm} u = \Lambda v, \hspace{3mm}\nabla \cdot v = 0. 
\end{equation} 
 
For the purpose of convenience in the subsequent sections, we recall the following definitions. First, from \cite[p. 198]{CK20}, given a set $\Phi$, we let $\Upsilon \triangleq \{\gamma: [0,\infty) \mapsto \Phi \}$ and then $S \triangleq \{S(\phi)\}_{\phi \in \Phi}$ where $S(\phi) \triangleq \{ \gamma \in \Upsilon: \gamma (0) = \phi \}$. A selection of the family $S$ is a map $u: [0,\infty) \times \Phi \mapsto \Phi$ such that $u(t, \phi) \in S(\phi)$ for every $\phi \in \Phi$; i.e., a selection picks a solution for every initial data $\phi$. We defer the definitions of almost sure (a.s.) Markov families and selections respectively to Definition \ref{Definition 5} and Lemma \ref{Lemma 7.3}. For clarity, we consider the system \eqref{additive msqg} in the following Definition \ref{Definition 1}
\begin{define}\label{Definition 1} 
\begin{enumerate}[wide=0pt]
\indent 
\item Fix any probability space $(\Omega, \mathcal{F}, \mathbf{P})$. A solution $v$ to \eqref{additive msqg} is said to be probabilistically strong if it is adapted to the canonical right-continuous filtration generated by the given Wiener process and augmented by all the negligible sets. It is probabilistically weak otherwise.
\item Uniqueness in law holds if any solution $(v, B)$ and $(\tilde{v}, \tilde{B})$ potentially defined on different filtered probability spaces and starting from the same initial distributions have the same laws $\mathcal{L}(v) = \mathcal{L}(\tilde{v})$. Path-wise uniqueness holds if two solutions $(v, B)$ and $(\tilde{v}, B)$ on same probability space $(\Omega, \mathcal{F}, \mathbf{P})$ and  starting from the same initial data satisfy $v(t) = \tilde{v}(t)$ for all $t$ $\mathbf{P}$-a.s. 
\end{enumerate} 
\end{define} 
\noindent Path-wise uniqueness implies uniqueness in law due to the classical Yamada-Watanabe theorem while the converse implication is generally false (\hspace{1sp}\cite[Example 2.2]{C03}). 

We now describe previous works. First, we focus on the results prior to the recent development of the breakthrough technique of convex integration on the deterministic SQG equations \eqref{est 115}-\eqref{est 116}. The global existence and uniqueness of solutions starting from smooth initial data for \eqref{est 115} has caught much attention. The $L^{\infty} (\mathbb{T}^{2})$-norm of the solution seems to be the best bounded quantity for the solutions $\theta$ to \eqref{est 115} starting from smooth initial data and \eqref{est 115} has the scaling property of $\lambda^{\gamma - 1} \theta (\lambda^{\gamma} t, \lambda x)$ for any $\lambda \in \mathbb{R}_{+}$; consequently, the cases $\gamma \in (0,1), \gamma =1,$ and $\gamma \in (1, 2]$ are respectively considered as the $L^{\infty}(\mathbb{T}^{2})$-supercritical, critical, and subcritical. The subcritical case was treated by Constantin and Wu \cite{CW99}, the critical case required new breakthrough approaches from Kiselev, Nazarov, and Volberg \cite{KNV07} and Caffarelli and Vasseur \cite{CV10}; we also refer to \cite{CV12, K10, KN10}.  In the research direction of weak solutions, Resnick in \cite[Theorem 2.1]{R95} constructed a global weak solution to the SQG equations \eqref{est 114}, even in case $\nu = 0$, starting from initial data $\theta^{\text{in}} \in L^{2}(\mathbb{T}^{2})$. This result was subsequently extended by Marchand \cite[Theorem 1.2]{M08} to $\theta^{\text{in}} \in L^{p} (\mathbb{T}^{2})$ for $p \geq \frac{4}{3}$ or $\theta^{\text{in}} \in \dot{H}^{-\frac{1}{2}}(\mathbb{T}^{2})$. 

Second, we focus on the results on the stochastic SQG equations, which refers to \eqref{est 115} forced by random noise similarly to \eqref{additive msqg}-\eqref{mult msqg}, as well as the stochastic momentum SQG equations \eqref{additive msqg}-\eqref{mult msqg} themselves. In fact, we first recall the important work \cite{FR08} by Flandoli and Romito, in which the global existence of a Leray-Hopf weak solution to the NS equations forced by additive noise was proven; we call their solution Leray-Hopf because it requires a regularity of $L_{T}^{2}\dot{H}_{x}^{1}$ and satisfies the appropriate energy inequality (see \cite[Definition 3.3]{FR08}). Besides, Flandoli and Romito \cite[Theorem 4.1]{FR08} proved the existence of an a.s. Markov selection for the 3D NS equations forced by additive noise. Subsequently, Goldys, R$\ddot{\mathrm{o}}$ckner, and Zhang \cite{GRZ09} extended \cite{FR08} to the general case that particularly includes the linear multiplicative noise (see also \cite{DO06} in such a case of state-dependent noise).

Third, concerning the SQG equations forced by random noise, Zhu in \cite[Definition 4.2.1 and Theorem 4.2.4]{Z12b} in the additive case and \cite[Definition 4.3.1 and Theorem 4.3.2]{Z12b} in a more general multiplicative case defined its solution to require the regularity of $L_{T}^{\infty} L_{x}^{2} \cap L_{T}^{2} \dot{H}_{x}^{\gamma}$ and proved the global existence of an analytically weak and probabilistically weak solution in case $\nu > 0, \gamma > 0$.  Although \cite[Definitions 4.2.1 and 4.3.1]{Z12b} did not include energy inequality, it could have been readily included (see e.g. 4th line from top on \cite[p. 80]{Z12b}) and hence such solutions can be called the Leray-Hopf solution to the stochastic SQG equations. Path-wise uniqueness of such a martingale solution in the subcritical regime $\gamma > 1$ with a multiplicative noise was also proven in \cite[Theorem 4.4.4]{Z12b} under the hypothesis that $\theta^{\text{in}} \in L^{p}(\mathbb{T}^{2})$ for $p > \frac{2}{\gamma -1}$ but remains open in a more general setting. Moreover, the existence of an a.s. Markov selection in case $\nu > 0, \gamma > 0$ was proven for the stochastic SQG equations in \cite[Theorem 4.2.5]{Z12a}; its proof consisted of an application of \cite[Theorem 4.7]{GRZ09} and hence applies to both cases of additive and multiplicative noise. We also refer to \cite[Theorems 4.3.3, 4.3.8, 4.3.10, and 4.4.5]{Z12a} for results respectively concerning strong Feller property, support theorem, the existence of a unique invariant measure, and exponential convergence (see also \cite{RZZ15}); we also refer to \cite{Y22a} on ergodicity and \cite{BNSW20} in which certain multiplicative noise was demonstrated to possess a regularizing effect. Concerning the momentum SQG equations forced by random noise, the global existence of a martingale solution in case the noise is additive was proven therein in the \cite[Section 5.1]{Y23a} by applying \cite[Theorem 4.6]{GRZ09}; it was proven for a class of weak solutions with low spatial regularity (see \cite[Definition 4.1]{Y23a} for the specific definition) but the proof can be extended to the Leray-Hopf level immediately. In fact, the hypotheses of \cite[Theorems 4.6 and 4.7]{GRZ09} are identical and thus the computations in \cite[Section 5.1]{Y23a} also assures the existence of an a.s. Markov family for \eqref{additive msqg} with $\gamma > 0$. Lastly, as we already discussed, the setting of \cite{GRZ09} allows linear multiplicative noise. In conclusion, both \eqref{additive msqg}-\eqref{mult msqg} admit global-in-time analytically and probabilistically weak solutions, as well as a.s. Markov families. 

Next, we discuss the recent developments on the ill-posedness results via the convex integration technique. Onsager's conjecture \cite{O49} in 1949 suggested the threshold regularity of $C_{x}^{\frac{1}{3}}$ that determines the conservation of the kinetic energy by the solution to the Euler equations. Constantin, E, and Titi \cite{CET94}, as well as Eyink \cite{E94} although in a slightly stronger norm than H$\ddot{\mathrm{o}}$lder continuity, established the positive direction that if $u \in C_{x}^{\alpha}$ for $\alpha > \frac{1}{3}$, then the solution conserves the energy as predicted.  Proving the negative direction of the Onsager's conjecture that there exists a solution in $C_{x}^{\alpha}$ for $\alpha < \frac{1}{3}$ that does not conserve energy required a new approach. De Lellis and Sz$\acute{\mathrm{e}}$kelyhidi Jr. initiated a series of important works \cite{DS09, DS10, DS13} adapting the convex integration technique from geometry to the Euler equations. After further developments (e.g. \cite{BDIS15}), Isett \cite{I18} solved the negative direction of the Onsager's conjecture in all spatial dimensions $d \geq 3$, and the case $d = 2$ was recently established by Giri and Radu \cite{GR23}. Such breakthrough results of ill-posedness did not stop with the Euler equations. Introducing intermittency, Buckmaster and Vicol \cite{BV19a} proved non-uniqueness of weak solutions to the 3D NS equations. 

In relevance to our manuscript, we highlight that C$\acute{\mathrm{o}}$rdoba, Faraco, and Gancedo \cite{CFG11} proved non-uniqueness of weak solutions to the incompressible porous media equation; subsequently, Isett and Vicol \cite{IV15} extended such a result to a more general class of active scalars. Interestingly, \cite{IV15} was forced to exclude the SQG equations due to the technical reason arising from the odd Fourier multiplier of its nonlinear term, specifically the Riesz transform $\mathcal{R}^{\bot}$ in \eqref{est 115a}. This was overcome by Buckmaster, Shkoller, and Vicol \cite{BSV19} by introducing the momentum SQG equations \eqref{est 116} (see also \cite{IM21}). Subsequently, Cheng, Kwon, and Li \cite{CKL21} constructed non-trivial stationary weak solutions to the SQG equations \eqref{est 115} by considering the equation solved by the potential temperature $f \triangleq \Lambda^{-1} \theta$, in contrast to the potential velocity $v = \Lambda^{-1} u$. We refer to the following works and references therein on convex integration adapted to other equations: \cite{BCV18, LQ20, LT20} on the NS equations with fractional Laplacian fully within the $L^{2}(\mathbb{T}^{d})$-supercritical regime, \cite{BBV20} on the magnetohydrodynamics (MHD) system, \cite{MS20} on transport equation, and \cite{BV19b} for a survey. 

Next, before we review the recent developments on the convex integration technique applied on stochastic PDEs, we recall that upon any probabilistic Galerkin approximation (e.g. \cite[Appendix A]{FR08}), if one takes mathematical expectation to obtain uniform estimates, then the resulting solution ends up being probabilistically weak. In case the noise is additive, one can consider a linear stochastic PDE with the same fixed noise, e.g.
\begin{equation}\label{Stokes}
dz + [\nabla p_{1} + \Lambda^{\gamma} z] dt = dB, \hspace{3mm} \nabla\cdot z = 0, \hspace{3mm} \text{ for } t > 0, 
\end{equation} 
in case of \eqref{additive msqg}, and work on a random PDE satisfied by the difference between these two equations; however, the solution obtained via this approach as a converging subsequence will depend on the fixed realization and hence not be probabilistically strong. The classical Yamada-Watanabe theorem would imply that path-wise uniqueness, together with the existence of a probabilistically weak solution implies the existence of a  probabilistically strong solution. However, the consensus was that proving path-wise uniqueness for the 3D NS equations forced by random noise seems to be no easier than the deterministic case and hence a significant amount of effort has been devoted to prove uniqueness in law instead (e.g. \cite[pp. 878--879]{DD03}). In conclusion, the existence of probabilistically strong, even if analytically weak, solution for the stochastic PDEs of fluid mechanics in general were open except in case of low spatial dimensions or hyperdiffusion (see \cite[p. 84]{F08}). 

Remarkably, any solution obtained via convex integration technique for stochastic PDEs turns out to be probabilistically strong, as seen already in the first works \cite{BFH20} by Breit, Feireisel, and Hofmanov\'a, as well as \cite{CFF19} by Chiodaroli, Feireisl, and Flandoli, which proved path-wise non-uniqueness of certain stochastic Euler equations. In short, the reason is because the solution obtained via convex integration is a limit of a Cauchy sequence constructed inductively. Subsequently, \cite{HZZ19} proved non-uniqueness in law of the 3D NS equations forced by noise of additive, linear multiplicative, and nonlinear types. Cherny's theorem \cite{C03} implies that uniqueness in law and the existence of probabilistically strong solution lead to path-wise uniqueness; therefore, proving non-uniqueness in law among probabilistically strong solutions, as done in \cite{HZZ19}, has an important consequence that the strategy of proving the global well-posedness of the corresponding system via  Cherny's theorem fails. In relevance to our manuscript, we highlight that following the approach of \cite[Theorems A and B]{BMS21} by Burczak, Modena, and Sz$\acute{\mathrm{e}}$kelyhidi Jr., Hofmanov\'a, Zhu, and Zhu \cite[Theorem 1.5]{HZZ21a} prescribed energy $e(t)$ under some hypothesis and constructed a probabilistically strong solution to the 3D NS equations forced by additive noise with its energy identically equal to $e(t)$ up to a stopping time which implies non-uniqueness in law; we mention that Berkemeier \cite{B23} extended this such a result to the case of linear multiplicative noise. Additionally, Hofmanov\'a, Zhu, and Zhu \cite[Theorem 1.3]{HZZ21a} extended the convex integration solution globally in time and showed the non-uniqueness of the corresponding Markov selection. 

Concerning the application of convex integration technique on the SQG equations and momentum SQG equations forced by random noise, first, Hofmanov\'a, Zhu, and Zhu \cite{HZZ22a} proved non-uniqueness in law of the SQG equations forced by additive white-in-space noise via the convex integration scheme of \cite{CKL21}. Independently, Yamazaki \cite{Y23a} proved the non-uniqueness in law of the momentum SQG equations \eqref{additive msqg} forced by additive white-in-time noise via the convex integration scheme of \cite{BSV19}. Extending \cite{Y23a} to \eqref{mult msqg} with linear multiplicative noise proved to be difficult due to technical difficulties that are elaborated in \cite[Remark 1.3]{Y23b} and summarized here. 
\begin{enumerate}
\item The main nonlinear term of the momentum SQG equations \eqref{est 116} is 
\begin{equation}\label{est 120} 
(u \cdot\nabla) v - (\nabla v)^{T} \cdot u  = u^{\bot} ( \nabla^{\bot} \cdot v), \hspace{3mm} u =\Lambda v 
\end{equation}  
(see \cite[Equation (5)]{Y23a}), which is not of divergence form. First, this implies that applying the inverse divergence operator from Definition \ref{Definition 2} can potentially lead to a complete waste of a derivative in estimates, unless \eqref{est 120} is compactly supported away from the origin in Fourier frequencies. Furthermore, even if $v$ were a shear flow, the second term $(\nabla v)^{T} \cdot \Lambda v$ would not vanish, which created difficulties in the initial inductive step at $q = 0$. 
\item The analysis on  \eqref{mult msqg} typically starts by defining 
\begin{equation}\label{define upsilon}
\Upsilon(t) \triangleq e^{B(t)} \text{ and } y(t) \triangleq \Upsilon^{-1}(t)v(t), 
\end{equation} 
which will lead to a stochastic commutator error of the form 
\begin{equation}\label{est 55e} 
\divergence R_{\text{Com2}} = (\Upsilon - \Upsilon_{l}) [(\Lambda y_{q+1} \cdot \nabla) y_{q+1} - (\nabla y_{q+1})^{T} \cdot \Lambda y_{q+1} ], 
\end{equation} 
where we limit ourselves at the moment to only explain that $\Upsilon_{l}$ is $\Upsilon$ mollified in time (see \eqref{define l}). A standard mollifier estimate can be employed on the difference $\Upsilon - \Upsilon_{l}$ in \eqref{est 55e}; however, the rest of the products are too singular, specifically one derivative worse than the NS equations, and not of divergence form in contrast to $\divergence (u\otimes u)$ in \eqref{est 114a}.  
\end{enumerate} 
Very recently, this was overcome in \cite{Y23b} with multiple new ideas, most of which we incorporate in this manuscript. On the other hand, the building blocks discovered in \cite{CKL21} seem to have special advantages and led to further advances in \cite{BLW23, HLZZ24} both on the SQG equations forced by additive noise. We also refer to the following references for further developments concerning convex integration technique applied to stochastic PDE: \cite{HZZ20} on the Euler equations, \cite{CDZ22, HZZ21b, HZZ22b, RS21, Y20a, Y20c, Y21c} on the NS equations, \cite{Y21a} on the Boussinesq system, \cite{Y21d} on the MHD system, \cite{LZ22} on the power-law model, \cite{KY22} on the transport equation, as well as more recently \cite{BJLZ23, FR24, HLP23, HPZZ23a, HPZZ23b, HZZ23a, LZ23a, LZ24, LZ24b, MS23, P23}.  

\subsection[Main Results]{\for{toc}{Main Results}\except{toc}{Statement of Main Results}}\label{main results}
Our main results are the constructions of global-in-time solutions to the momentum SQG equations forced by random noise of both additive and linear multiplicative types with prescribed energy, and consequently their non-uniqueness in law, as well as non-uniqueness of their Markov selections. Key distinctions in the inductive hypothesis of convex integration from the previous works \cite{Y23a, Y23b} and the current work are \eqref{inductive 3} and \eqref{mult inductive 3}, namely the control on the difference between the $\dot{H}^{\frac{1}{2}}(\mathbb{T}^{2})$-norm of the inductive solution and the prescribed energy. The convex integration scheme employed in \cite{HZZ21a} utilizes intermittency and is $L^{p}$-based while those of \cite{Y23a, Y23b} utilized Beltrami plane waves from \cite{CDS12a} and were $C^{k}$-based; while the latter sometimes has the advantage of obtaining higher regularity of the solution, it also implies less flexibility in estimates at $L^{p}$-level. Because the natural energy for the solution to the momentum SQG equations \eqref{additive msqg} and \eqref{mult msqg} are $\lVert \Lambda^{\frac{1}{2}} \cdot \rVert_{L^{2} (\mathbb{T}^{2})}^{2}$, it was not clear to us at first if we can adapt the convex integration schemes of \cite{Y23a, Y23b} to prescribe energy; this was part of our initial motivation for this work. We now present our first result in the case of additive noise. 

\begin{theorem}\label{Theorem 1.1}	
Suppose that two parameters $\gamma$ and $\sigma$ satisfy the following: 
\begin{equation}\label{gamma hypo}
\gamma \in \left(0, \frac{3}{2} \right)
\end{equation} 
and 
\begin{equation}\label{new}
\sigma \in \left(0, \frac{1}{2} [ \frac{3}{2} - \gamma ] \right]. 
\end{equation}
Fix any probability space $(\Omega,\mathcal{F},\mathbf{P})$. Then for every $GG^*$-Wiener process $B$ on $(\Omega, \mathcal{F}, \mathbf{P})$ such that
\begin{equation}\label{regularity of noise}
\Tr ( (-\Delta)^{\frac{7}{4} + 4 \sigma} GG^{\ast}) < \infty,
\end{equation}
there exist $\iota >0$ and a $\mathbf{P}$-a.s. strictly positive stopping time $\mathfrak{t}$ such that the following holds. Given any energy $e\in C^1_b(\mathbb{R};[\ushort{e},\infty))$ that satisfies
\begin{equation}\label{est 118}
\lVert e\rVert_{C([-2,1])} \leq \bar{e} \hspace{1mm}  \text{ and } \hspace{1mm} \lVert e'\rVert_{C([-2,1])}\leq \tilde{e} \hspace{1mm}  \text{ for some } \hspace{1mm} \bar{e}\geq \ushort{e}>4, \hspace{1mm} \text{ and } \hspace{1mm} \tilde{e}>0, 
\end{equation}
there exists a probabilistically strong, analytically weak solution $v\in C([0,\mathfrak{t}];C^{\frac{1}{2}+\iota}(\mathbb{T}^2)) \cap C^{\frac{1}{3} + \iota} ([0, \mathfrak{t} ]; C(\mathbb{T}^{2}))$ to \eqref{additive msqg} which starts from a deterministic initial data $v^{\text{in}}$, obeys
\begin{equation}\label{est 119}
\esssup_{\omega\in\Omega} \lVert v(\omega)\rVert_{C_{\mathfrak{t}}C_x^{\frac{1}{2}+\iota}} + \esssup_{\omega \in \Omega} \lVert v(\omega) \rVert_{C_{\mathfrak{t}}^{\frac{1}{3} + \iota} C_{x}}<\infty,
\end{equation}
and has its energy satisfy 
\begin{equation}\label{energy} 
\lVert v(t)\rVert^2_{\dot{H}_{x}^{\frac{1}{2}}}=e(t) \hspace{2mm} \forall \hspace{1mm} t \in [0, \mathfrak{t}].
\end{equation}
Additionally, if two such energies $e_1, e_2$ obeying the same bounds $\ushort{e},\bar{e},$ and $\tilde{e}$ in \eqref{est 118} coincide on $[0,t]$, then the corresponding solutions $v_1,v_2$ also coincide on $[0,t\wedge \mathfrak{t}]$, implying non-uniqueness in law for \eqref{additive msqg}.
\end{theorem}
Theorem \ref{Theorem 1.1} already constructed a solution on $[0,\mathfrak{t}]$ and this implies non-uniqueness in law. Theorem \ref{Theorem 1.2} (1) extends the solution from Theorem \ref{Theorem 1.1} globally in time. We recall from Section \ref{Section 1.2} that \eqref{additive msqg} admits a global-in-time analytically weak and probabilistically weak solutions starting from initial data in $\dot{H}_{\sigma}^{\frac{1}{2}}$, and an a.s. Markov selection.  

\begin{theorem}\label{Theorem 1.2} 
\indent
\begin{enumerate}
\item Non-uniqueness in law holds for \eqref{additive msqg} on $[0,\infty)$. 
\item There exists $\mathcal{C} (\xi^{\text{in}}, C_{q})$ for $\xi^{\text{in}} \in \dot{H}_{\sigma}^{\frac{1}{2}}$ and $q \in \mathbb{N}$, a suitable class of weak solutions to \eqref{additive msqg} specifically defined in Definition \ref{Definition 6}, such that Markov selections of $\mathcal{C} (\xi^{\text{in}}, C_{q})$ are not unique. 
\end{enumerate} 
\end{theorem}

Next, taking advantage of the new ideas from \cite{Y23b}, we extend both Theorems \ref{Theorem 1.1}-\ref{Theorem 1.2} to \eqref{mult msqg} with linear multiplicative noise, which is of interest due to its known regularizing effect (recall \cite{BNSW20, K11}); in fact, in contrast to Theorem \ref{Theorem 1.1}, the stopping time can be made arbitrarily large with high probability.  

\begin{theorem}\label{Theorem 1.3}
Suppose that \eqref{gamma hypo} holds, that $B$ is a $\mathbb{R}$-valued Wiener process on a fixed probability space $(\Omega,\mathcal{F},\mathbf{P})$.  Let $T>0$ and $\kappa \in (0,1)$ be given. Then there exists $\iota>0$, $L > 1$ sufficiently large,  and a $\mathbf{P}$-a.s. strictly positive stopping time $\mathfrak{t} = \mathfrak{t}(L)$ satisfying 
\begin{equation}\label{est 50} 
\mathbf{P}(\mathfrak{t}\geq T) > \kappa 
\end{equation}
such that the following holds. Given any energy $e\in C_b^1(\mathbb{R};[\ushort{e},\infty))$ that satisfies
\begin{equation}\label{est 118 mult}
\lVert e\rVert_{C([-2,L])} \leq \bar{e} \hspace{1mm}  \text{ and } \hspace{1mm} \lVert e'\rVert_{C([-2,L])}\leq \tilde{e} \hspace{1mm}  \text{ for some } \hspace{1mm} \bar{e}\geq \ushort{e}>4, \hspace{1mm} \text{ and } \hspace{1mm} \tilde{e}>0, 
\end{equation}
there exists a probabilistically strong, analytically weak solution $v\in C([0,\mathfrak{t}];C^{\frac{1}{2}+\iota}(\mathbb{T}^2))$ to \eqref{mult msqg} which starts from a deterministic initial data $v^{\text{in}}$, obeys \eqref{est 119}, and has its energy satisfy \eqref{energy}. Additionally, if two such energies $e_1, e_2$ obeying the same bounds $\ushort{e},\bar{e},$ and $\tilde{e}$ in \eqref{est 118} coincide on $[0,t]$, then the corresponding solutions $v_1,v_2$ also coincide on $[0,t\wedge \mathfrak{t}]$, implying non-uniqueness in law for \eqref{mult msqg}.
\end{theorem}
Theorem \ref{Theorem 1.3} already constructed a solution on $[0, \mathfrak{t}]$ where $\mathfrak{t}$ can be made large with high probability, and this implies non-uniqueness in law. Theorem \ref{Theorem 1.4} (1) extends the solution from Theorem \ref{Theorem 1.3} globally in time. Our approach here can be applied to obtain analogue of Theorem \ref{Theorem 1.4} for the 3D NS equations forced by linear multiplicative noise; we choose to not pursue this direction here. Again, we recall from Section \ref{Section 1.2} that \eqref{mult msqg} admits a global-in-time analytically and probabilistically weak solutions starting from initial data in $\dot{H}_{\sigma}^{\frac{1}{2}}$, as well as an a.s. Markov selection.  

\begin{theorem}\label{Theorem 1.4}
\indent
\begin{enumerate}
\item Non-uniqueness in law holds for \eqref{mult msqg} on $[0,\infty)$. Furthermore, for every given $T > 0$, non-uniqueness in law for \eqref{mult msqg} holds on $[0,T]$. 
\item There exists $\mathcal{W} (\xi^{\text{in}}, \theta^{\text{in}}, C_{q})$ for $(\xi^{\text{in}}, \theta^{\text{in}}) \in \dot{H}_{\sigma}^{\frac{1}{2}} \times \mathbb{R}$ and $q \in \mathbb{N}$, a suitable class of weak solutions to \eqref{mult msqg} specifically defined in Definition \ref{Definition 8}, such that Markov selections of $\mathcal{W} (\xi^{\text{in}}, \theta^{\text{in}}, C_{q})$ are not unique. 
\end{enumerate} 
\end{theorem} 
 Let us make serval remarks on Theorems \ref{Theorem 1.1}-\ref{Theorem 1.4}. 
\begin{remark}\label{Remark 1.1}
To the best of our knowledge, this is the first construction of the solutions to the SQG or the momentum SQG equations forced by random noise with prescribed energy. In \cite{Y23a, Y23b}, solutions to the momentum SQG equations forced by additive and linear multiplicative noise such that their  energies grow twice faster than the energies of the classical solutions constructed by Galerkin approximation. To prove Theorems \ref{Theorem 1.2} (2) and \ref{Theorem 1.4} (4) concerning non-uniqueness of the corresponding a.s. Markov families, we will need to rely on Theorems \ref{Theorem 1.1} and \ref{Theorem 1.3} that construct solutions with prescribed energy. 
\end{remark} 
\begin{remark}\label{Remark 1.2}
Concerning Theorems \ref{Theorem 1.1}-\ref{Theorem 1.2}, we need $\Tr ( (-\Delta)^{\frac{7}{4} + 4 \sigma} GG^{\ast}) < \infty$ according to \eqref{regularity of noise}. In comparison, the only assumption in \cite{HZZ21a} for the 3D NS equations forced by additive noise was $\Tr (GG^{\ast}) < \infty$. In short, the worst terms come from the stochastic commutator error, specifically $R_{\text{Com2,2}}$ and $R_{\text{Com2,3}}$ in \eqref{est 35} and \eqref{est 34} which will require $z \in C_{T} H_{x}^{\frac{5}{2}+ \sigma}$.

On the other hand, the hypothesis on the additive noise in \cite[Theorem 2.1]{Y23a} was $\Tr ((-\Delta)^{4- \frac{\gamma}{2} + 2 \sigma} GG^{\ast}) < \infty$ with $\gamma \in (0, \frac{3}{2})$. In comparison, in Theorem \ref{Theorem 1.1} we fixed an arbitrary small $\sigma > 0$ and assumed the hypothesis of $\Tr ((-\Delta)^{\frac{7}{4} + 4 \sigma }GG^{\ast}) < \infty$, which is significantly better than $\Tr ((-\Delta)^{4- \frac{\gamma}{2} + 2 \sigma} GG^{\ast}) < \infty$ with $\gamma \in (0, \frac{3}{2})$. Our strategy on this improvement is as follows. First, let us emphasize that in terms of results of non-uniqueness, the higher $\gamma > 0$ is more desirable; second, as we can see from $\Tr ((-\Delta)^{4- \frac{\gamma}{2} + 2 \sigma} GG^{\ast}) < \infty$, the high exponent $\gamma$ allows us to reduce the hypothesis on the noise and still attain higher regularity for the solution $z$ to \eqref{Stokes}. With this in mind, our strategy is to include the diffusive term $\Lambda^{\frac{3}{2} - 2\sigma}z$ rather than the traditional approach of $\Lambda^{\gamma} z$ (cf. \eqref{Stokes} and \eqref{est 42a}); this allows $z$ to gain higher regularity with a relatively weaker hypothesis on the noise (see \eqref{est 41}). Of course, this will create extra terms $\Lambda^{\gamma} z -\Lambda^{\frac{3}{2} - 2 \sigma} z$ in \eqref{randomPDEa}, but it turns out that we are able to handle them via careful estimates. This strategy was inspired by the recent work  \cite{HLZZ24}.  
\end{remark}

\begin{remark}\label{Remark 1.3} 
Theorems \ref{Theorem 1.3}-\ref{Theorem 1.4} in the linear multiplicative case allows one to take any $T > 0$ arbitrarily large and prove non-uniqueness in law for \eqref{mult msqg} over $[0,T]$, providing a second proof of such a result (cf. \cite[Theorem 2.2]{Y23b} for the first proof). On the other hand, Theorems \ref{Theorem 1.1}-\ref{Theorem 1.2} in the additive case, similarly to \cite[Theorem 1.5]{HZZ21a}, shows that upon prescribing energy, such a result seems difficult, although non-uniqueness in law over $[0,T]$ for an arbitrary $T > 0$ in the additive case is actually known due to \cite[Theorem 2.2]{Y23a}; the additive noise therein was smoother than the noise in Theorems \ref{Theorem 1.1}-\ref{Theorem 1.2}, but one should be able to overcome this difference by the same strategy of this manuscript. The technical reason for this difference will be explained in Remark \ref{Remark 5.1} and the proof of Proposition \ref{Proposition 6.5} (1).
\end{remark} 

\begin{remark}\label{Remark 1.4}
The results and proofs in this manuscript have the potential toward the following extensions and generalizations. 
\begin{enumerate}
\item The previous works \cite{Y23a, Y23b} partially considered the generalized momentum SQG equations where $\Lambda^{\gamma} v$ and $\Lambda v$ in \eqref{est 116} were replaced respectively by $\Lambda^{\gamma_{1}} v$ and $\Lambda^{2-\gamma_{2}} v$. Such a generalization is possible. 
\item \cite[Theorem 1.1 and Corollary 1.2]{HZZ21a} had prescribed initial data rather than prescribed energy; we leave this direction for future works. 
\item Via short sketch of the proof, \cite[Corollary 1.4]{HZZ21a} also proved that the deterministic 3D NS equations generates non-unique semiflow solutions as a consequence of \cite[Theorem 1.3]{HZZ21a}. The proof consists of the following steps: 1) observing that convex integration goes through with zero noise, the case in which the stopping time defined in \eqref{additive stopping time} reduces to $\mathfrak{t} = 1$,  2) extending the solution globally in time by the classical Leray-Hopf weak solution, 3) modifying the relaxed energy inequality in the Definition \ref{Definition 6} according to Remark \ref{Remark 4.1}, 4) applying \cite[Theorem 2.5]{CK20} to deduce the existence of a semiflow that holds for almost all time $t$, 5) extending it to hold for all time $t$ by including the energy as an additional variable, and 6) verifying its non-uniqueness almost identically to the proof of Theorem \ref{Theorem 1.4}. We leave this direction for future works. 
\end{enumerate} 
\end{remark} 

We describe other technical difficulties subsequently as they require further notations (see Remarks \ref{Remark 2.1}, \ref{Remark 3.1}, \ref{Remark 3.2}, \ref{Remark 4.1}, \ref{Remark 5.1}, \ref{Remark 5.2},  and \ref{Remark 5.3}). In Section \ref{Preliminaries}, we set up additional parameters and definitions. In Sections \ref{Section 3}-\ref{Section 6} we prove respectively Theorems \ref{Theorem 1.1}-\ref{Theorem 1.4}. In Section \ref{Appendix} Appendix, we include additional useful lemmas and detailed computations for completeness. 

\section{Preliminaries}\label{Preliminaries} 
We write $A \overset{(\cdot)}{\lesssim}_{a,b}B$ to imply the existence of $C(a,b) \geq 0$ such that $A \leq CB$ due to an equation $(\cdot)$ and $A \approx B$ if $A \lesssim B$ and $B \lesssim A$. We define $\sum_{k\leq i\leq l}\triangleq 0$ when $l < k$.

\subsection{Choice of Parameters}\label{choice of parameters}
Certain parameters will be chosen identically in both cases of additive and linear multiplicative noise; we introduce them and their respective bound for convenience and to avoid unnecessary repetition. We take  
\begin{equation}\label{bound beta}
\beta \in \left(\frac{1}{2}, \frac{3}{4}\right)
\end{equation}
and 
\begin{equation}\label{bound a}
a \in 5 \mathbb{N}, \hspace{3mm} a\geq a_0 \geq e^{16}
\end{equation}
where $\beta > \frac{1}{2}$ and $a_{0} \geq e^{16}$ will be subsequently taken sufficiently close to $\frac{1}{2}$ and sufficiently large as needed, respectively. Besides, we fix 
\begin{equation}\label{bound alpha}
\alpha \in \left(1,\frac{3}{2} \right).
\end{equation}
We will later fix a value $b\in\mathbb{N}$ dependent on the type of noise, \eqref{additive bound b} for the additive case and \eqref{mult bound b} for the linear multiplicative case. Based on this we define for $q\in \mathbb{N}$
\begin{equation}\label{define lambda delta}
\lambda_{q}\triangleq a^{b^{q}}, \hspace{3mm}  \delta_{q}\triangleq \lambda_{1}^{2\beta-1}\lambda_{q}^{-2\beta}.
\end{equation} 
The fact that $\lambda_{1}\delta_{1}=1$ will be crucial to our proofs of Propositions \ref{additive step q0} and \ref{mult step q0}. We define
\begin{equation}\label{define l}
l_{q+1}\triangleq\lambda_{q+1}^{-\alpha}.
\end{equation}
We define a start time $t_{q}$ at each step $q\in\mathbb{N}_0$ to be
\begin{equation}\label{define tq}
t_{q}\triangleq -2+\sum_{1\leq \iota \leq q} \delta_{\iota}^{\frac{1}{2}} \overset{\eqref{define lambda delta}}{\in} [-2, 0), 
\end{equation}
where the upper bound can be readily verified considering \eqref{bound beta}, \eqref{bound a}, and that $b \in \mathbb{N}$.  For each step $q$, we work on $t\in[t_q,\mathfrak{t}]$. At step $q+1$ we split $[t_{q+1},\mathfrak{t}]$ into finitely many sub-intervals of size $\tau_{q+1}$, where
\begin{equation}\label{tau}
\tau_{q+1}^{-1} \triangleq l_{q+1}^{-\frac{1}{2}} \lambda_{q+2}^{\frac{1}{2}} \delta_{q+2}^{\frac{1}{2}} \lambda_{q+1}^{\frac{1}{2}} \delta_{q+1}^{-\frac{1}{4}}. 
\end{equation}

\begin{remark}\label{Remark 2.1}
Our choice of $\tau_{q+1}$ is notably different from
\begin{equation*}
\tau_{q+1} =  l_{q+1}^{\frac{1}{2}}\lambda_{q+1}^{-1}\delta_{q+1}^{-\frac{1}{4}}
\end{equation*}
in \cite[Equation (53)]{Y23a} and \cite[Equation (56)]{Y23b}, due to a key difference in our inductive hypotheses (cf. \eqref{inductive 2b} and \cite[Equation (90)]{Y23a}). This parameter is chosen via optimization over the most difficult error terms $R_{T}$ in \eqref{est 121} and $R_{\text{O,approx}}$ in \eqref{additive Oapprox}, which depend on the bounds imposed by these hypotheses.
\end{remark}

Given this, we then define $\mathbf{J}\triangleq \{\floor{t_{q+1}\tau_{q+1}^{-1}},\floor{t_{q+1}\tau_{q+1}^{-1}}+1,\dots,\ceil{\mathfrak{t}\tau_{q+1}^{-1}}\}$, dependent on our choice of $\mathfrak{t}$: for additive noise see \eqref{additive stopping time}, and for linear multiplicative see \eqref{mult stopping time}. We consider $\Gamma_j$ from Lemma \ref{geometric lemma}, and denote by 
\begin{equation}\label{est 122} 
\Gamma_j\triangleq 
\begin{cases}
\Gamma_1 \text{ whenever } j \text{ is odd,} \\
\Gamma_2 \text{ whenever } j \text{ is even}. 
\end{cases}
\end{equation} 
With that in mind, when no confusion arises, we will denote
\begin{align*}
\sum_j\triangleq \sum_{j\in\mathbf{J}}=\sum_{j=\floor{t_{q+1}\tau_{q+1}^{-1}}}^{\ceil{\mathfrak{t}\tau_{q+1}^{-1}}},  \hspace{2mm} \sum_k&\triangleq \sum_{k\in \Gamma_j}~~\text{for \emph{a priori} fixed} ~j\in\mathbf{J}, \hspace{2mm} \mathrm{and~} \hspace{2mm} \sum_{j,k}\triangleq \sum_j\sum_k.
\end{align*}
We define $\chi:\mathbb{R}\to [0,1]$ to be a smooth cutoff function with $\supp \chi\subset (-1,1)$ such that $\chi\equiv 1$ on the interval $(-\frac{1}{4},\frac{1}{4})$ and $\sum_{j\in\mathbb{Z}} \chi^2(t-j)=1$ for all $t\in\mathbb{R}$. Then we define, for $j\in\mathbf{J}$, 
\begin{equation}\label{est 18} 
\chi_j(t)\triangleq \chi(\tau_{q+1}^{-1}t-j) \text{ so that }\sum_{j} \chi_{j}^{2}(t) = \sum_{j} \chi^{2} (\tau_{q+1}^{-1} t - j) = 1. 
\end{equation} 
We will frequently rely on the fact that for any fixed $t$, there are at most 2 nontrivial cutoffs $\chi_j$. 

\subsection{Operators and Probabilistic Elements}\label{Section 2.2}
We denote the Leray projection of mean-zero functions onto the space of divergence-free vector fields by $\mathbb{P}$ and mean-zero projection by $\mathbb{P}_{\neq 0}$. 

\begin{define}\rm{(\hspace{1sp}\cite[Definition 4.1]{BSV19})\label{Definition 2}
Let $f: \mathbb{T}^{2} \mapsto \mathbb{R}^{2}$ be divergence-free and mean-zero. Then component-wise, we define the inverse divergence operator $\mathcal{B}$ by $(\mathcal{B}f)^{ij}\triangleq -\partial_j \Lambda^{-2}f_i - \partial_i \Lambda^{-2} f_j$ for $i, j \in \{1,2\}$, so that we have $f=\divergence{(\mathcal{B}f)}$. For $f$ which is not necessarily mean-zero or divergence-free, we further define $\mathcal{B}f\triangleq \mathcal{B}\mathbb{P}\mathbb{P}_{\neq 0}f$. It follows from this definition that $\mathcal{B}f$ is a symmetric trace-free matrix.}
\end{define}
Lastly, for $k\in \mathbb{S}^{1}$ we define
\begin{equation}\label{est 66}
b_k(\xi)\triangleq ik^{\perp}e^{ik\cdot\xi}, \hspace{3mm} c_k(\xi)\triangleq e^{ik\cdot\xi}, \hspace{3mm} \mathbb{P}_{q+1,k}\triangleq \mathbb{P} P_{\approx k\lambda_{q+1}},
\end{equation}
where $P_{\approx k\lambda_{q+1}}$ is a Fourier operator with a Fourier symbol $\hat{K}_{\approx k\lambda_{q+1}}(\xi)\triangleq \hat{K}_{\approx 1}\left(\frac{\xi}{\lambda_{q+1}}-k\right)$, and $\hat{K}_{\approx 1}$ is a bump function such that $supp(\hat{K}_{\approx 1})\subset B(0,\frac{1}{8})$ and $\hat{K}_{\approx 1}\equiv 1$ on $B(0,\frac{1}{16})$. We point out that $\supp \widehat{ \mathbb{P}_{q+1, k} f} \subset \{\xi: \frac{7}{8} \lambda_{q+1} \leq \lvert \xi \rvert \leq \frac{9}{8} \lambda_{q+} \}$. From \cite[Equations (50)-(51)]{Y23a} we know that we can write $\mathbb{P}_{q+1, k} f = K_{q+1, k} \ast f$ where the kernel $K_{q+1, k}$ satisfies for all $a, b \geq 0$, 
\begin{equation}\label{new 8}
\left\lVert \lvert x \rvert^{b} \nabla_{x}^{a} K_{q+1, k} (x) \right\rVert_{L^{1} (\mathbb{R}^{2})} \leq C(a,b) \lambda_{q+1}^{a-b}, 
\end{equation} 
and we can find a universal constant $C_{1}$ to satisfy 
\begin{equation}\label{est 8}
\lVert\mathbb{P}_{q+1,k}\rVert_{C_x\to C_x}\leq C_{1};
\end{equation}
we emphasize that $C_{1}$ is independent of $q$. Additionally, we define $\tilde{P}_{\approx \lambda_{q+1}}$ to be the Fourier operator with a symbol supported in $\{\xi: \frac{\lambda_{q+1}}{4} \leq \lvert \xi \rvert \leq 4 \lambda_{q+1} \}$ and is identically one on $\{\xi: \frac{3 \lambda_{q+1}}{8} \leq \lvert \xi \rvert \leq 3 \lambda_{q+1}\}$. Finally, we define a large parameter $M_{0}$ that requires (see \eqref{additive M0} and \eqref{mult M0})  
\begin{equation}\label{M0}
M_{0}^{\frac{1}{2}} \geq \frac{4C_1}{\pi}\sup_{k \in \Gamma_{1} \cup \Gamma_{2}}\lVert\gamma_{k}\rVert_{C(B(\Id, \varepsilon_\gamma ))},
\end{equation}
where $C_{1}$ is from \eqref{est 8}, $\Gamma_{1}, \Gamma_{2}$ are from \eqref{est 122}, and $\gamma_{k}, \varepsilon_\gamma$ are from Lemma \ref{geometric lemma}. While this seems similar to the time-varying parameter $M_{0}(t)$ of previous works (e.g. \cite[Equation (78)]{Y23a}), we take it instead as a constant, and indeed, its purpose more closely mirrors the value $C_{0}$ of \cite[Equations (88), (89), and (91)]{Y23a}. 

For any Polish space $H$, we define $\mathcal{B}(H)$ to be the $\sigma$-algebra of Borel sets in $H$. We acknowledge that we denoted the inverse divergence operator by $\mathcal{B}$ in Definition \ref{Definition 2}, but we do not believe that this will cause any confusion. Given any probability measure $P$, $\mathbb{E}^{P}$ denotes a mathematical expectation with respect to (w.r.t.) $P$. We denote $L^{2}(\mathbb{T}^{2})$-inner products by $\langle \cdot, \cdot \rangle$, a duality pairing of $\dot{H}^{-\frac{1}{2}}(\mathbb{T}^{2})-\dot{H}^{\frac{1}{2}}(\mathbb{T}^{2})$ by $\langle \cdot, \cdot \rangle_{\dot{H}_{x}^{-\frac{1}{2}}-\dot{H}_{x}^{\frac{1}{2}}}$, a quadratic variation of $A$ and $B$ by $\langle\langle A, B \rangle \rangle$, as well as $\langle \langle A \rangle \rangle \triangleq \langle \langle A, A \rangle \rangle$. Next, for a general set $X$, we denote the set of all compact subsets of $X$ by $\text{Comp}(X)$. We let  
\begin{equation}\label{est 209}
\Omega_{t} \triangleq C( [t,\infty); (H_{\sigma}^{4})^{\ast} ) \cap L_{\text{loc}}^{\infty} ([t,\infty); \dot{H}_{\sigma}^{\frac{1}{2}}), \hspace{3mm} t \geq 0,  
\end{equation} 
where $(H_{\sigma}^{4})^{\ast}$ denotes the dual of $H_{\sigma}^{4}$. We also denote by $\mathcal{P} (\Omega_{0})$ the set of all probability measures on $(\Omega_{0}, \mathcal{B})$ where $\mathcal{B}$ is the Borel $\sigma$-field of $\Omega_{0}$ from the topology of locally uniform convergence on $\Omega_{0}$. We define the canonical process $\xi: \Omega_{0} \mapsto (H_{\sigma}^{4})^{\ast}$ by $\xi_{t}(\omega) \triangleq \omega(t)$. For general $t \geq 0$ we equip $\Omega_{t}$ with Borel $\sigma$-field 
\begin{equation}\label{est 183} 
\mathcal{B}^{t} \triangleq \sigma \{\xi(s): s \geq t \}, 
\end{equation} 
and additionally define 
\begin{equation}\label{est 184} 
\mathcal{B}_{t}^{0} \triangleq \sigma \{\xi(s): s \leq t \}, \text{ and } \mathcal{B}_{t} \triangleq \cap_{s > t} \mathcal{B}_{s}^{0}, \hspace{3mm} t \geq 0.
\end{equation}  
For any Hilbert space $U$, we denote by $L_{2}(U, \dot{H}_{\sigma}^{s})$ with $s \geq 0$ the space of all Hilbert-Schmidt operators from $U$ to $\dot{H}_{\sigma}^{s}$ with norm $\lVert \cdot \rVert_{L_{2}(U, \dot{H}_{\sigma}^{s})}$ so that in the additive case $G$ is a Hilbert-Schmidt operator from $U$ to $\dot{H}_{\sigma}^{\frac{1}{2}}$. 

Additionally, we assume the existence of another Hilbert space $U_{1}$ such that the embedding $U \subset U_{1}$ is Hilbert-Schmidt and define 
\begin{equation*} 
\bar{\Omega}_{t} \triangleq C([t,\infty); (H_{\sigma}^{4})^{\ast} \times U_{1}) \cap L_{\text{loc}}^{\infty} ([t,\infty); \dot{H}_{\sigma}^{\frac{1}{2}} \times U_{1}). 
\end{equation*}  
We let $\mathcal{P} (\bar{\Omega}_{0})$ represent the set of all probability measures on $(\bar{\Omega}_{0}, \bar{\mathcal{B}})$ where $\bar{\mathcal{B}}$ is the Borel $\sigma$-algebra on $\bar{\Omega}_{0}$. Furthermore, we define the canonical process on $\bar{\Omega}_{0}$ by $(\xi, \theta): \bar{\Omega}_{0} \mapsto (H_{\sigma}^{4})^{\ast} \times U_{1}$ to satisfy $(\xi, \theta)_{t}(\omega) \triangleq \omega(t)$. Finally, for $t \geq 0$ we define 
\begin{equation}\label{est 241}
\bar{\mathcal{B}}^{t} \triangleq \sigma \{ (\xi, \theta)(s): s \geq t \}, \hspace{3mm} \bar{\mathcal{B}}_{t}^{0} \triangleq \sigma \{(\xi, \theta)(s): s \leq t \}, \hspace{3mm} \bar{\mathcal{B}}_{t} \triangleq \cap_{s > t} \bar{\mathcal{B}}_{s}^{0}. 
\end{equation} 
\begin{define}\rm{(\hspace{1sp}\cite[Definition 3.2]{FR08}, \cite[Definition 4.1]{HZZ21a})}\label{Definition 3} 
Let $\vartheta$ be an $(\mathcal{B}_{t})_{t\geq 0}$-adapted (respectively (resp.) $(\bar{\mathcal{B}}_{t})_{t\geq 0}$-adapted) process on $\Omega_{0}$ (resp. $\bar{\Omega}_{0})$. Then we call $\vartheta$ an a.s. $(\mathcal{B}_{t})_{t \geq 0}$-supermartingale ($(\bar{\mathcal{B}}_{t})_{t \geq 0}$-supermartingale) under $P$ provided $\vartheta$ is $P$-integrable and 
\begin{align}
&\mathbb{E}^{P} [ \vartheta_{t} 1_{A} ] \leq \mathbb{E}^{P} [ \vartheta_{s} 1_{A} ] \label{est 123} \\
& \text{ for almost every (a.e.) } s \geq 0, \hspace{1mm} \forall \hspace{1mm} t \geq s, \forall \hspace{1mm} A \in \mathcal{B}_{s} \hspace{1mm} (A \in \bar{\mathcal{B}}_{s}). \nonumber
\end{align}
We define $(\mathcal{B}_{t}^{0})_{t \geq 0}$-supermartingale ($(\bar{\mathcal{B}}_{t}^{0})_{t \geq 0}$-supermartingale) similarly with ``$\mathcal{B}_{t}$'' and ``$\mathcal{B}_{s}$'' replaced resp. by $\mathcal{B}_{t}^{0}$ and $\mathcal{B}_{s}^{0}$ (resp. ``$\bar{\mathcal{B}}_{t}$'' and ``$\bar{\mathcal{B}}_{s}$'' replaced resp. by $\bar{\mathcal{B}}_{t}^{0}$ and $\bar{\mathcal{B}}_{s}^{0}$). The $s$ that satisfies \eqref{est 123} is called a regular time of $\vartheta$ and an exceptional time of $\vartheta$ otherwise. 
\end{define} 

\begin{lemma}\rm{(\hspace{1sp}\cite[Proposition 4.2]{HZZ21a})}\label{Lemma 2.1}
Suppose that $\vartheta$ is an $(\mathcal{B}_{t})_{t \geq 0}$-supermartingale that is $(\mathcal{B}_{t}^{0})_{t\geq 0}$-adapted (resp. $(\bar{\mathcal{B}}_{t})_{t \geq 0}$-supermartingale that is $(\bar{\mathcal{B}}_{t}^{0})_{t\geq 0}$-adapted).  Then $\vartheta$ is an a.s. $(\mathcal{B}_{t}^{0})_{t\geq 0}$-supermartingale (resp. a.s. $(\bar{\mathcal{B}}_{t}^{0})_{t\geq 0}$-supermartingale). 
\end{lemma}

\begin{lemma}\rm{(\hspace{1sp}\cite[Theorem 4.15]{HZZ21a}, \cite[pp. 1727--1728]{GRZ09}, \cite[p. 413]{FR08})}\label{Lemma 2.2}
Given $P \in \mathcal{P} (\Omega_{0})$ (resp. $P\in \mathcal{P} (\bar{\Omega}_{0})$), there exists a regular conditional probability distribution (r.c.p.d.) $P ( \cdot \lvert \mathcal{B}_{T}^{0})(\omega)$ for $\omega \in \Omega_{0}$  w.r.t. $\mathcal{B}_{T}^{0}$ (resp. $P ( \cdot \lvert \bar{\mathcal{B}}_{T}^{0})(\omega)$ for $\omega \in \bar{\Omega}_{0}$  w.r.t. $\bar{\mathcal{B}}_{T}^{0}$) such that the following holds. 
\begin{enumerate}
\item For every $\omega \in \Omega_{0}, P( \cdot \lvert \mathcal{B}_{T}^{0})(\omega) \in \mathcal{P}(\Omega_{0})$ (resp. $\omega \in \bar{\Omega}_{0}, P( \cdot \lvert \bar{\mathcal{B}}_{T}^{0})(\omega) \in \mathcal{P} (\bar{\Omega}_{0})$). 
\item For every $A \in \mathcal{B}$, the mapping $\omega \mapsto P ( A \lvert \mathcal{B}_{T}^{0})(\omega)$ is $\mathcal{B}_{T}^{0}$-measurable (resp. $A \in \bar{\mathcal{B}}$, the mapping $\omega \mapsto P ( A \lvert \bar{\mathcal{B}}_{T}^{0})(\omega)$ is $\bar{\mathcal{B}}_{T}^{0}$-measurable).
\item There exists a $P$-null set $N \in \mathcal{B}_{T}^{0}$ (resp. $N \in \bar{\mathcal{B}}_{T}^{0}$) such that for any $\omega \notin N$, 
\begin{align*}
&P ( \{ \tilde{\omega}: \xi(s,\tilde{\omega}) = \xi(s,\omega) \hspace{1mm} \forall \hspace{1mm} s \in [0,T] \} \lvert \mathcal{B}_{T}^{0}) (\omega) = 1\\
& \text{(resp. }P ( \{ \tilde{\omega}: (\xi,\theta)(s,\tilde{\omega}) = (\xi,\theta)(s,\omega) \hspace{1mm} \forall \hspace{1mm} s \in [0,T] \} \lvert \bar{\mathcal{B}}_{T}^{0}) (\omega) = 1).
\end{align*}
\item For any $A \in \mathcal{B}_{T}^{0}$ and any $E \in \mathcal{B}$ (resp. $A \in \bar{\mathcal{B}}_{T}^{0}$ and any $E \in \bar{\mathcal{B}}$)
\begin{align*}
P \left(\{\xi \rvert_{[0,T]} \in A, \xi\rvert_{[T,\infty)} \in E \} \right) = \int_{\tilde{\omega} \rvert_{[0,T]} \in A} P (\{ \xi \rvert_{[T,\infty)} \in E \lvert \mathcal{B}_{T}^{0} \}) (\tilde{\omega} ) d P(\tilde{\omega})
\end{align*}
(resp.
\begin{align*}
P \left(\{(\xi,\theta) \rvert_{[0,T]} \in A, (\xi,\theta)\rvert_{[T,\infty)} \in E \} \right) = \int_{\tilde{\omega} \rvert_{[0,T]} \in A} P (\{ (\xi,\theta) \rvert_{[T,\infty)} \in E \lvert \bar{\mathcal{B}}_{T}^{0} \}) (\tilde{\omega}) d P(\tilde{\omega}) ).
\end{align*} 
\end{enumerate} 
\end{lemma}

The following reconstruction theorem follows immediately from \cite[Theorem 6.1.2]{SV97}. 
\begin{lemma}\rm{(\hspace{1sp}\cite[Theorem 4.16]{HZZ21a}, \cite[Theorem 2.2]{GRZ09}, \cite[Definition 2.3]{FR08})}\label{Lemma 2.3}
Let $\omega \mapsto Q_{\omega}$ be a mapping from $\Omega_{0}$ to $\mathcal{P}(\Omega_{0})$ (resp. $\bar{\Omega}_{0}$ to $\mathcal{P}(\bar{\Omega}_{0})$) such that for any $A \in \mathcal{B}$ (resp. $A \in \bar{\mathcal{B}}$), the mapping $\omega \mapsto Q_{\omega}(A)$ is $\mathcal{B}_{T}^{0}$-measurable (resp. $\bar{\mathcal{B}}_{T}^{0}$-measurable) and for all $\omega \in \Omega_{0}$ (resp. $\omega \in \bar{\Omega}_{0}$) 
\begin{align*}
&Q_{\omega} ( \{ \tilde{\omega} \in \Omega_{0}: \xi(T, \tilde{\omega}) = \xi(T, \omega) \}) = 1 \\
&(\text{resp. } Q_{\omega} ( \{ \tilde{\omega} \in \Omega_{0}: (\xi,\theta)(T, \tilde{\omega}) = (\xi,\theta) (T, \omega) \}) = 1). 
\end{align*}
Then for any $P \in \mathcal{P} (\Omega_{0})$ (resp. $P \in \mathcal{P} (\bar{\Omega}_{0})$), there exists a unique $P \otimes_{T} Q \in \mathcal{P} (\Omega_{0})$ (resp. $P \otimes_{T} Q \in \mathcal{P} (\bar{\Omega}_{0})$) such that 
\begin{align*}
(P \otimes_{T} Q) (A) = P(A) \text{ for all } A \in \mathcal{B}_{T}^{0} \hspace{3mm} \text{(resp. } A \in \bar{\mathcal{B}}_{T}^{0})
\end{align*}
and for $P \otimes_{T}Q$-a.e. $\omega \in \Omega_{0}$ (resp. $\omega \in \bar{\Omega}_{0}$), 
\begin{align*}
\delta_{\omega} \otimes_{T} Q_{\omega} = (P\otimes_{T} Q) (\cdot \lvert \mathcal{B}_{T}^{0})(\omega) \hspace{3mm} (\text{resp. } \delta_{\omega} \otimes_{T} Q_{\omega} = (P\otimes_{T} Q) (\cdot \lvert \bar{\mathcal{B}}_{T}^{0})(\omega)).
\end{align*} 
\end{lemma}

\begin{define}\rm{(\hspace{1sp}\cite[p. 557]{HZZ21a}, \cite[Definition 2.3]{GRZ09})}\label{Definition 4}
We say $P \in \mathcal{P} (\Omega_{0})$ (resp. $P \in \mathcal{P} (\bar{\Omega}_{0})$) is concentrated on the paths with values in $\dot{H}_{\sigma}^{\frac{1}{2}}$ (resp. $\dot{H}_{\sigma}^{\frac{1}{2}} \times U_{1}$) if there exists $A \in \mathcal{B}$ (resp. $A \in \bar{\mathcal{B}}$) such that 
\begin{align*}
P(A) = 1,  & \hspace{1mm}A \subset \{ \omega \in \Omega_{0}: \omega(t) \in \dot{H}_{\sigma}^{\frac{1}{2}} \text{ for all } t \geq 0 \} \\
&  (\text{resp. } A \subset \{ \omega \in \bar{\Omega}_{0}: \omega(t) \in \dot{H}_{\sigma}^{\frac{1}{2}} \times U_{1} \text{ for all } t \geq 0 \} )
\end{align*}  
and denote the set of all such $P$ by $\mathcal{P}_{\dot{H}_{\sigma}^{\frac{1}{2}}}(\Omega_{0})$ (resp. $\mathcal{P}_{\dot{H}_{\sigma}^{\frac{1}{2}} \times U_{1}}(\bar{\Omega}_{0})$). 
\end{define} 
We see that $\mathcal{B} ( \mathcal{P}_{\dot{H}_{\sigma}^{\frac{1}{2}}}(\Omega_{0}) )= \mathcal{B} ( \mathcal{P} ( \Omega_{0})) \cap \mathcal{P}_{\dot{H}_{\sigma}^{\frac{1}{2}}}(\Omega_{0})$ (resp. $\mathcal{B} ( \mathcal{P}_{\dot{H}_{\sigma}^{\frac{1}{2}} \times U_{1}}(\bar{\Omega}_{0}) )= \mathcal{B} ( \mathcal{P} ( \bar{\Omega}_{0})) \cap \mathcal{P}_{\dot{H}_{\sigma}^{\frac{1}{2}} \times U_{1} }(\bar{\Omega}_{0})$).

\begin{define}\rm{(\hspace{1sp}\cite[Equation (2.1) on p. 1727 and Definition 2.5]{GRZ09}, \cite[Definition 4.17]{HZZ21a})}\label{Definition 5}
We denote the shift operator $\Phi_{t}: \Omega_{0} \mapsto \Omega_{t}$ by $\Phi_{t} (\omega)(s) \triangleq \omega(s-t)$ for $s \geq t$  (resp. $\Phi_{t}: \bar{\Omega}_{0} \mapsto \bar{\Omega}_{t}$ by $\Phi_{t} (\omega)(s) \triangleq \omega(s-t)$ for $s \geq t$).  

A family $\{ P_{\xi^{\text{in}}} \}_{\xi^{\text{in}} \in \dot{H}_{\sigma}^{\frac{1}{2}}}$  (resp.  $\{ P_{\xi^{\text{in}}, \theta^{\text{in}}} \}_{(\xi^{\text{in}}, \theta^{\text{in}}) \in \dot{H}_{\sigma}^{\frac{1}{2}} \times U_{1}}$) of probability measures in $\mathcal{P}_{\dot{H}_{\sigma}^{\frac{1}{2}}} (\Omega_{0})$ (resp. $\mathcal{P}_{\dot{H}_{\sigma}^{\frac{1}{2}} \times U_{1}} (\bar{\Omega}_{0})$) is called an a.s. Markov family provided 
\begin{enumerate}
\item for any $A \in \mathcal{B}$ (resp. $A \in \bar{\mathcal{B}}$), the mapping $\xi^{\text{in}}  \mapsto P_{\xi^{\text{in}} } (A)$ is $\mathcal{B} (\dot{H}_{\sigma}^{\frac{1}{2}})/ \mathcal{B} ([0,1])$-measurable  (resp. $(\xi^{\text{in}}, \theta^{\text{in}}) \mapsto P_{\xi^{\text{in}}, \theta^{\text{in}} } (A)$ is $\mathcal{B} (\dot{H}_{\sigma}^{\frac{1}{2}} \times U_{1})/ \mathcal{B} ([0,1])$-measurable), 
\item for every $\xi^{\text{in}} \in \dot{H}_{\sigma}^{\frac{1}{2}}$ (resp. $(\xi^{\text{in}}, \theta^{\text{in}}) \in \dot{H}_{\sigma}^{\frac{1}{2}} \times U_{1}$), there exists a Lebesgue null set $\mathcal{T} \subset (0,\infty)$ such that for all $T \notin \mathcal{T}$, 
\begin{align*}
&P_{\xi^{\text{in}}} ( \cdot \lvert \mathcal{B}_{T}^{0}) (\omega) = P_{\omega(T)} \circ \Phi_{T}^{-1} \hspace{3mm} \text{ for } P_{\xi^{\text{in}}} \text{-a.e. } \omega \in \Omega_{0} \\
&(\text{resp. } P_{\xi^{\text{in}}, \theta^{\text{in}}} ( \cdot \lvert \bar{\mathcal{B}}_{T}^{0}) (\omega) = P_{\omega(T)} \circ \Phi_{T}^{-1} \hspace{3mm} \text{ for } P_{\xi^{\text{in}}, \theta^{\text{in}}} \text{-a.e. } \omega \in \bar{\Omega}_{0} ). 
\end{align*}
\end{enumerate} 
\end{define}  
We postpone additional preliminaries to Section \ref{Section 7.1}.   

\section{Proof of Theorem \ref{Theorem 1.1}}\label{Section 3}
\subsection{Setup}\label{additive setup}
We define a function $f:\mathbb{N}_0\to\mathbb{R}$ by
\begin{equation}\label{define f}
f(q)\triangleq \frac{\lambda_{q}}{4},
\end{equation}
which will be used in defining $z_q$ in \eqref{define zq}. E.g., \eqref{new1} and \eqref{new 39} require that $f(q) < \frac{\lambda_{q}}{2}$ for which $f(q) = \frac{\lambda_{q}}{4}$ suffices. In contrast, \cite[Equation (81)]{Y23a} required only $\supp \hat{z}_{q} \subset B \left(0, \frac{\lambda_{q}}{4}\right)$, implying that $f(q) \in (0, \frac{\lambda_{q}}{4})$ sufficed therein. The difference in our current manuscript is the additional inductive estimate \eqref{inductive 3} to control the energy. E.g. at \eqref{new2} we want $f(q)^{-1} \ll \lambda_{q+2}^{1-2\beta} \lambda_{1}^{2\beta -1}$ for $\beta > \frac{1}{2}$ sufficiently close to $\frac{1}{2}$; thus, we chose $f(q) = \frac{\lambda_{q}}{4}$ in \eqref{define f} for convenience. 

We now fix $\gamma$ and $\sigma$ according to \eqref{gamma hypo}-\eqref{new}. Then we also fix  $\bar{e}\geq \ushort{e}> 4$, $\tilde{e}>0$, and $e: \mathbb{R} \to [\ushort{e},\infty)$ such that for all $t\in [0,1]$, we have $\lVert e\rVert_{C([-2,1])} \leq \bar{e}$ and $\lVert e'\rVert_{C([-2,1])}\leq \tilde{e}$ according to \eqref{est 118}. We fix $b\in\mathbb{N}$ to be large enough to satisfy
\begin{equation}\label{additive bound b}
b>\frac{6}{\alpha-\frac{1}{2}}.
\end{equation}
Concerning $\beta$ and $a$ in \eqref{bound beta}-\eqref{bound a}, we will need
\begin{equation}\label{est 1} 
a^{b^{q+1} (2\beta -1) (b-1)} \geq \frac{4}{3}
\end{equation}
in \eqref{est 4}; while we will take $\beta > \frac{1}{2}$ sufficiently close to $\frac{1}{2}$, we can increase $a_{0}$ as needed and achieve \eqref{est 1}. Finally, we fix a $GG^*$-Wiener process $B$ satisfying \eqref{regularity of noise} on $(\Omega, \mathcal{F}, \mathbf{P})$ and $(\mathcal{F}_{t})_{t\geq 0}$ as the canonical filtration of $B$ augmented by all the $\mathbf{P}$-negligible sets.  

\subsubsection{\except{toc}{Formation of }Random PDEs and Stopping Time}
\begin{proposition}\label{prop solution z}
Under the hypothesis \eqref{regularity of noise}, the solution $z$ to 
\begin{subequations}\label{est 42} 
\begin{align}
&dz+[\nabla p_1 +  \Lambda^{\frac{3}{2} - 2 \sigma} z]dt=dB, \hspace{3mm} \nabla\cdot z=0, \hspace{3mm} \text{ for } t > 0, \label{est 42a}\\
&z(0)=0,\label{est 42b} 
\end{align}
\end{subequations} 
satisfies, for all $T> 0$, $\delta \in (0, \frac{1}{2})$, and $l \in \mathbb{N}$,  
\begin{equation}\label{est 41} 
\mathbb{E}^{\mathbf{P}} \left[ \lVert z \rVert_{C_{T} \dot{H}_{x}^{\frac{5}{2} + 2 \sigma}}^{l} + \lVert z \rVert_{C_{T}^{\frac{1}{2} - \delta} \dot{H}_{x}^{\frac{7}{4} + 4 \sigma}}^{l} \right] < \infty. 
\end{equation}
\end{proposition}

\begin{proof}[Proof of Proposition \ref{prop solution z}]
Due to similarities to previous works (e.g. \cite[Proposition 34]{D13} and \cite[Proposition 3.6]{HZZ19}), we sketch it in Section \ref{Section 7.2} for completeness.
\end{proof}

Now, in the system of interest \eqref{additive msqg}, if we let 
\begin{equation}\label{est 135}
y\triangleq v-z 
\end{equation} 
and $p_{2} \triangleq p-p_1$ where $(z,p_{1})$ solves \eqref{est 42}, we obtain the random PDE
\begin{subequations}\label{additive random PDE}
\begin{align}
&\partial_t y + (\Lambda (y+z)\cdot \nabla)(y+z)-(\nabla (y+z))^{T}\cdot \Lambda (y+z) \nonumber  \\
&\hspace{20mm} + \nabla p_2 + \Lambda^{\gamma} y + \Lambda^{\gamma} z - \Lambda^{\frac{3}{2} - 2 \sigma} z = 0, \hspace{3mm}  \nabla \cdot y = 0, \hspace{3mm} \text{ for } t > 0, \label{randomPDEa}\\
& y(0)=v(0). \label{randomPDEc}
\end{align}
\end{subequations}
So that we may use it on the time interval of consideration, we extend $z$ by $z(t)=z(0)=0$ on $[-4,0]$. For all $q\in\mathbb{N}_0$, with $f$ defined in \eqref{define f}, we define $z_q$ by 
\begin{equation}\label{define zq}
\hat{z}_q(\xi)\triangleq 1_{\{|\xi| \leq f(q)\}}(\xi)\hat{z}(\xi).
\end{equation}
We add an error term $\divergence \mathring{R}_q $, where $\mathring{R}_q$ is a trace-free symmetric matrix called the Reynolds stress, to the right hand side of \eqref{additive random PDE} to obtain for all $q \in \mathbb{N}_{0}$, 
\begin{align}
&\partial_t y_q + (\Lambda (y_q+z_q)\cdot \nabla)(y_q+z_q)-(\nabla (y_q+z_q))^{T}\cdot \Lambda (y_q+z_q) \nonumber \\
& \hspace{25mm} + \nabla p_q + \Lambda^{\gamma} y_q  + \Lambda^{\gamma} z_{q} -\Lambda^{\frac{3}{2} - 2 \sigma} z_{q} = \divergence \mathring{R}_{q}, \hspace{5mm} \nabla \cdot y_q = 0. \label{additive convPDE}
\end{align}

Next, for any $\delta\in (0,\frac{1}{4})$, we take our stopping time to be
\begin{align}\label{additive stopping time}
\mathfrak{t}\triangleq 1 \wedge& \inf{\{t\geq 0 : C_S\lVert z(t)\rVert_{\dot{H}_{x}^{\frac{5}{2} + 2 \sigma}}\geq 1\}}\\
\wedge & \inf{\{t\geq 0:C_S\lVert z\rVert_{C_t^{\frac{1}{2}-2\delta}\dot{H}_{x}^{\frac{7}{4} + 4 \sigma}}\geq 1\}} \nonumber \\
\wedge & \inf \left\{t\geq 0 : C_{0} \left( C_{S} \lVert z (t) \rVert_{\dot{H}_{x}^{\frac{3}{2}+ \frac{\sigma}{2}}}^{2} + 2 \lVert z (t) \rVert_{\dot{H}_{x}^{\frac{3}{2} - \sigma}}^{2} \right) \geq \frac{\varepsilon_\gamma}{C_{S} 32 (2\pi)^{2}} a^{(1-2\beta) b(b-1)} \ushort{e}  \right\},\nonumber
\end{align}
where $C_S \geq 1$ and $C_0$ are defined to satisfy
\begin{subequations}\label{est 92} 
\begin{align} 
&\lVert f\rVert_{L_{x}^\infty}\leq C_S\lVert f\rVert_{\dot{H}_{x}^{1+\sigma}},  \hspace{8mm}  \lVert fg \rVert_{\dot{H}_{x}^{\sigma}} \leq C_{S} \lVert f \rVert_{\dot{H}_{x}^{\frac{1+\sigma}{2}}} \lVert  g\rVert_{\dot{H}_{x}^{\frac{1+\sigma}{2}}},   \label{est 92a}\\
& \lVert \mathcal{B}f\rVert_{C_{x}}\leq C_0\lVert f\rVert_{C_{x}}, \hspace{10mm}  \lVert \mathcal{B} f \rVert_{\dot{H}_{x}^{1+ \sigma}} \leq C_{0} \lVert f \rVert_{\dot{H}_{x}^{\sigma}}, \label{est 92b}
\end{align} 
\end{subequations} 
for every mean-zero function $f,g\in \dot{H}^{1+\sigma}(\mathbb{T}^{2})$; we refer to \cite[Lemma 2.5]{Y14} concerning the second inequality of \eqref{est 92a}. Due to Proposition \ref{prop solution z} we observe that $\mathfrak{t}\in (0,\infty)$ $\mathbf{P}$-a.s.

\subsubsection{Mollification}\label{additive mollification}
We let $\{ \phi_{l} \}_{l > 0}$ and $\{\varphi_{l}\}_{l > 0}$ be families of standard mollifiers with mass one and compact support respectively on $\mathbb{R}^{2}$ and in $(\tau_{q+1}, 2\tau_{q+1}] \subset \mathbb{R}_{+}$. At each step $q$, we take $l\triangleq l_{q+1}$ from \eqref{define l} and mollify as follows: 
\begin{equation}\label{additive mollified} 
y_l \triangleq y_q\ast_x\phi_l\ast_t\varphi_l,  \hspace{3mm} \mathring{R}_l \triangleq \mathring{R}_q\ast_x\phi_l\ast_t\varphi_l,  \hspace{3mm} z_l\triangleq z_q\ast_x\phi_l\ast_t\varphi_l. 
\end{equation} 
Now that we have mollified, we may introduce a special material derivative operator 
\begin{equation}\label{est 56} 
D_{t,q} \triangleq \partial_{t} + ( \Lambda y_{l} + \Lambda z_{l}) \cdot \nabla.
\end{equation}

\subsection{Inductive Hypothesis}
At each step $q$ we require that for all $t\in [t_q,\mathfrak{t}]$, the solution $(y_{q}, \mathring{R}_{q})$ to \eqref{additive convPDE} satisfies 
\begin{subequations}\label{inductive 1}
\begin{align}
&\supp \hat{y}_q \subset B(0,2\lambda_q),\label{inductive 1a}\\
&||y_q||_{C_{t,x,q}}\leq M_{0}^{\frac{1}{2}}\left(1+\sum_{1\leq j\leq q}\delta_j^{\frac{1}{2}} \right)\bar{e}^{\frac{1}{2}}\leq 3M_{0}^{\frac{1}{2}}\bar{e}^{\frac{1}{2}},\label{inductive 1b}\\
&||y_q||_{C_{t,q}C^1_x}+||\Lambda y_q||_{C_{t,x,q}}\leq M_{0}^{\frac{1}{2}}\lambda_q\delta_q^{\frac{1}{2}}\bar{e}^{\frac{1}{2}},\label{inductive 1c}\\
&|| D_{t,q} y_q||_{C_{t,x,q}}\leq M_{0} \lambda_q^2\delta_q\bar{e}\label{inductive 1d}, 
\end{align}
\end{subequations}
\begin{subequations}\label{inductive 2}
\begin{align}
&\supp{\widehat{\mathring{R}}_q}\subset B(0,4\lambda_q),\label{inductive 2a}\\
&||\mathring{R}_q||_{C_{t,x,q}}\leq \frac{\varepsilon_\gamma}{32(2\pi)^2}\lambda_{q+2}\delta_{q+2} e(t); \label{inductive 2b}
\end{align}
\end{subequations}
finally, to control the energy, we impose 
\begin{equation}\label{inductive 3}
\frac{3}{4}\lambda_{q+1}\delta_{q+1}e(t)\leq e(t)-||(y_q+z_q)(t)||^2_{\dot{H}_{x}^{\frac{1}{2}}}\leq \frac{5}{4}\lambda_{q+1}\delta_{q+1}e(t), 
\end{equation}
where $D_{t,q}$ was defined in \eqref{est 56}, and $\varepsilon_\gamma$  in \eqref{inductive 2b} is that defined in Lemma \ref{geometric lemma}.

\subsection{Base Step $q=0$}
\begin{proposition}\label{additive step q0}
Suppose that $b$ satisfies \eqref{additive bound b}.  Let $y_{0} \equiv 0$. Then together with 
\begin{equation}\label{form r0}
\mathring{R}_0=\mathcal{B}\left(\Lambda z_0^\perp (\nabla^\perp\cdot z_0) + \Lambda^{\gamma} z_{0} - \Lambda^{\frac{3}{2} - 2 \sigma} z_{0}\right), 
\end{equation}
the pair $(y_{0}, \mathring{R}_{0})$ solves \eqref{additive convPDE} and satisfies the hypotheses \eqref{inductive 1}-\eqref{inductive 3} at level $q = 0$ on $[t_{0}, \mathfrak{t}]$. Moreover, $y_{0}(t,x)$ and $\mathring{R}_{0}(t,x)$ are both deterministic over $[t_{0}, 0]$. 
\end{proposition}

\begin{proof}[Proof of Proposition \ref{additive step q0}]
Using the identity \eqref{est 120} makes it clear that $(y_{0},\mathring{R}_{0})$ solves \eqref{additive convPDE} with $p_{0} \equiv 0$. The hypotheses \eqref{inductive 1} are trivially satisfied while \eqref{inductive 2a} follows from \eqref{define zq} and \eqref{define f}. Concerning \eqref{inductive 2b}, we estimate by
\begin{align*} 
||\mathring{R}_0||_{C_{t,x,0}}& \overset{\eqref{form r0} \eqref{est 92} }{\leq}  C_{S} [ \lVert \mathcal{B} ( \Lambda z_{0}^{\bot} ( \nabla^{\bot} \cdot z_{0} ) ) \rVert_{C_{t,0} \dot{H}_{x}^{1+ \sigma}} + C_{S} \lVert \mathcal{B} ( \Lambda^{\gamma} z_{0} -\Lambda^{\frac{3}{2} - 2 \sigma} z_{0} ) \rVert_{C_{t,0} \dot{H}_{x}^{1+ \sigma}} ] \\
& \overset{\eqref{est 92} \eqref{new} \eqref{additive stopping time}}{\leq}  \frac{\varepsilon_\gamma}{32(2\pi)^2} \lambda_{2} \delta_{2} e(t);
\end{align*}
therefore, \eqref{inductive 2b} holds. Now, because $z(t)$ is mean-zero, we can readily verify $\lVert z_0(t)\rVert^2_{\dot{H}^\frac{1}{2}} \leq \frac{1}{4} e(t)$ using \eqref{additive stopping time} and that $e(t) \geq \ushort{e} \geq 4$ due to \eqref{est 118}; consequently, as $\lambda_1\delta_1=1$, \eqref{inductive 3} is satisfied. Finally, because $z \equiv 0$ on $[-2, 0] \overset{\eqref{define tq}}{=} [t_{0}, 0]$, it is clear that $z_{0} \equiv 0$ on $[-2, 0]$ and therefore $\mathring{R}_{0}$ defined by \eqref{form r0} and $y_{0}$ are both deterministic over $[t_{0}, 0]$. 
\end{proof}

\subsection[Main Iteration]{\for{toc}{Main Iteration}\except{toc}{Main Iteration: Step $q\to q+1$}}
\begin{proposition}\label{additive q to qplus1}
Suppose that $b$ satisfies \eqref{additive bound b}. Then there exists a choice of $a_{0}$ in \eqref{bound a} and $\beta$ in \eqref{bound beta} such that the following holds: If $(y_q,\mathring{R}_q)$ is an $(\mathcal{F}_t)_{t\geq 0}$-adapted solution to \eqref{additive convPDE} on $[t_q,\mathfrak{t}]$ satisfying the  inductive hypotheses \eqref{inductive 1}-\eqref{inductive 3} for some $q\in\mathbb{N}_0$, then there exists a solution $(y_{q+1},\mathring{R}_{q+1})$ to \eqref{additive convPDE} which is also $(\mathcal{F}_t)_{t\geq 0}$-adapted satisfying the inductive hypotheses \eqref{inductive 1}-\eqref{inductive 3} and the Cauchy difference
\begin{equation}\label{cauchy difference}
\lVert y_{q+1}(t)-y_q(t)\rVert_{C_x}\leq M_{0}^{\frac{1}{2}} \bar{e}^{\frac{1}{2}}\delta_{q+1}^\frac{1}{2}
\end{equation}
on $[t_{q+1},\mathfrak{t}]$. Moreover, the perturbation defined by 
\begin{equation}\label{est 44}
w_{q+1} \triangleq y_{q+1}-y_l 
\end{equation}
satisfies
\begin{equation}\label{support wqplus1}
\supp \hat{w}_{q+1} \subset \left\{\xi :\frac{\lambda_{q+1}}{2}\leq |\xi|\leq 2\lambda_{q+1}\right\}.
\end{equation}
Finally, if $(y_q,\mathring{R}_q)(t)$ is deterministic on $[t_q,0]$, then so is $(y_{q+1},\mathring{R}_{q+1})$ on $[t_{q+1},0]$.
\end{proposition}
The proof of Proposition \ref{additive q to qplus1} is split to Propositions \ref{prop additive amplitudes}-\ref{prop additive control energy} that consist of estimates on the perturbation, the decomposed stress, and the control of the energy. 

\subsubsection[Construction of Solution]{Construction of Solution at Step $q+1$}\label{additive construction}
For all $j \in \mathbf{J}=\{ \lfloor t_{q+1} \tau_{q+1}^{-1} \rfloor, \hdots, \lceil \mathfrak{t} \tau_{q+1}^{-1} \rceil \}$, similarly to \cite[Equation (129)]{Y23a}, we define $\Phi_{j}(t,x)$ as a solution to 
\begin{equation}\label{def Phi j}
D_{t,q} \Phi_{j}  = 0, \hspace{10mm} \Phi_{j} ( \tau_{q+1} j, x) = x. 
\end{equation} 
Further, we define $\mathring{R}_{q,j}$ to be the solution to 
\begin{equation}\label{est 16}
D_{t,q} \mathring{R}_{q,j} = 0, \hspace{10mm} \mathring{R}_{q,j} (\tau_{q+1} j, x) = \mathring{R}_{l} (\tau_{q+1} j, x). 
\end{equation} 
We define
\begin{subequations}\label{est 0}
\begin{align}
& \gamma_{q} (t) \triangleq \frac{1}{4(2\pi)^{2}} [ e(t) (1 - \lambda_{q+2} \delta_{q+2} ) - \lVert (y_{q} + z_{q}) (t) \rVert_{\dot{H}_{x}^{\frac{1}{2}}}^{2}],   \hspace{2mm} \gamma_{l} \triangleq \gamma_{q} \ast_{t} \varphi_{l},  \label{est 0b} \\
& \rho_{j} \triangleq \varepsilon_\gamma^{-1} \sqrt{ l^{2} + \lVert \mathring{R}_{q,j} (\tau_{q+1} j) \rVert_{C_{x}}^{2}} + \gamma_{l} (\tau_{q+1} j),  \label{est 0a}
\end{align}
\end{subequations}
and thereby define amplitude functions for all $k\in\Gamma_j$
\begin{equation}\label{est 3}
a_{k,j} (t,x) \triangleq \lambda_{q+1}^{-\frac{1}{2}} \rho_{j}^{\frac{1}{2}} \gamma_{k} \left( \Id - \frac{ \mathring{R}_{q,j} (t,x)}{\rho_{j}} \right).
\end{equation}

\begin{remark}\label{Remark 3.1}
The analogue of $\rho_{j}$ in \cite{HZZ21a} was a function of $x$ and $t$. We needed to make it a constant in both $x$ and $t$ due to technicalities. First, we needed to make it a constant in $x$ in order to be able to compute 
\begin{align*}
\sum_{j,k} \int_{\mathbb{T}^{2}} \lambda_{q+1} \tilde{w}_{q+1, j, k}  \cdot \tilde{w}_{q+1, j, -k} dx = \sum_{j} 4 \chi_{j}^{2} \rho_{j} (2\pi)^{2}  
\end{align*}
in \eqref{est 43}. Moreover, our inductive hypothesis in \eqref{inductive 1c} is at the level of  $C^{1}$ only in $x$, not $C_{t,x}^{1}$, in contrast to \cite{HZZ21a}. This relaxes the need to compute $\lVert a_{k,j} \rVert_{C_{t,x}^{N}}$ in contrast to \cite{HZZ21a};  we only need to estimate $\lVert D^{N} a_{k,j} \rVert_{C_{t,x,q+1}}$ in \eqref{est 19b}. This is preferable for us because such estimates for us will rely on transport-equation estimates from Lemma \ref{transport lemma} which work well for $D^{N}$, not $\partial_{t}^{N}$.  
\end{remark}

We observe that \eqref{est 1} guarantees that 
\begin{equation}\label{est 4}
\frac{3}{4} \lambda_{q+1} \delta_{q+1} \geq \lambda_{q+2} \delta_{q+2} 
\end{equation}
and consequently due to \eqref{inductive 3},  
\begin{equation}\label{est 21}
\gamma_{q}(t) \geq 0. 
\end{equation}
Because $\gamma_{q} \geq 0$, we see that 
\begin{align}\label{est 27}
\gamma_{l} (t) \overset{\eqref{est 0b}}{\geq}  0, 
\end{align}
which has an implication that 
\begin{align}\label{est 9} 
\left\lvert \Id - \left( \Id - \frac{ \mathring{R}_{q,j} (t,x)}{\rho_{j}} \right) \right\rvert = \left\lvert \frac{ \mathring{R}_{q,j} (t,x)}{\rho_{j}} \right\rvert \overset{\eqref{est 2a} \eqref{est 16} \eqref{est 0a}}{\leq} \frac{ \lVert \mathring{R}_{l} (\tau_{q+1} j) \rVert_{C_{x}}}{\varepsilon_\gamma^{-1} \lVert \mathring{R}_{q,j} (\tau_{q+1} j) \rVert_{C_{x}}}  = \varepsilon_\gamma,
\end{align} 
so that $\gamma_{k}\left( \Id - \frac{ \mathring{R}_{q,j} (t,x)}{\rho_{j}} \right)$ within the amplitude functions \eqref{est 3} is well-defined. Next, our perturbation is defined by
\begin{subequations}\label{est 14}
\begin{align}
& w_{q+1} (t,x) \triangleq \sum_{j,k} \chi_{j}(t) \mathbb{P}_{q+1, k} \left( a_{k,j} (t,x) b_{k} (\lambda_{q+1} \Phi_{j} (t,x)) \right),   \label{est 14a}   \\
& \tilde{w}_{q+1, j,k} (t,x) \triangleq \chi_{j} (t)a_{k,j} (t,x) b_{k} ( \lambda_{q+1} \Phi_{j} (t,x)), \label{est 14b} 
\end{align}
\end{subequations}
after which it is useful to also define
\begin{equation}\label{est 14c}
\psi_{q+1, j, k} (t,x) \triangleq \frac{ c_{k }(\lambda_{q+1} \Phi_{j}(t,x))}{c_{k} (\lambda_{q+1} x)} = e^{i\lambda_{q+1} (\Phi_{j} (t,x) - x) \cdot k},
\end{equation}
so that overall we may see that
\begin{subequations}\label{est 67} 
\begin{align}
& w_{q+1} = \sum_{j,k} \mathbb{P}_{q+1, k} \tilde{w}_{q+1, j,k}, \label{est 67a}\\
& b_{k} (\lambda_{q+1} \Phi_{j} (x)) = ik^{\bot} e^{ik \cdot \lambda_{q+1} \Phi_{j} (x)} = b_{k} (\lambda_{q+1} x) \psi_{q+1, j, k} (x). \label{est 67b}
\end{align}
\end{subequations} 
Then we define $y_{q+1}$ by \eqref{est 44}. Now, we wish to decompose the Reynolds stress error term. First, we may define the first commutator error
\begin{subequations}
\begin{equation}\label{Rcom1}
R_{\text{Com1}} \triangleq \mathcal{B} \left[ \Lambda ( y_{l} + z_{l})^{\bot} \nabla^{\bot} \cdot (y_{l} + z_{l}) - [ \Lambda (y_{q} + z_{q})^{\bot} \nabla^{\bot} \cdot (y_{q} + z_{q})] \ast_{x} \phi_{l} \ast_{t} \varphi_{l}  \right] \\
\end{equation}
and mollify the pressure in space and time
\begin{equation}\label{pl}
p_{l} \triangleq p_{q} \ast_{x} \phi_{l} \ast_{t} \varphi_{l}
\end{equation}
\end{subequations}
so that they satisfy the following relationship according to \eqref{additive convPDE}: 
\begin{align}
& \partial_{t} y_{l}+ (\Lambda (y_{l} + z_{l}) \cdot \nabla) (y_{l} + z_{l}) - ( \nabla (y_{l} + z_{l} ))^{T} \cdot \Lambda ( y_{l} + z_{l}) \nonumber \\
& \hspace{20mm} + \nabla p_{l} + \Lambda^{\gamma} y_{l} + \Lambda^{\gamma} z_{l} - \Lambda^{\frac{3}{2} - 2 \sigma} z_{l} = \divergence ( \mathring{R}_{l} + R_{\text{Com1}}). \label{new4}
\end{align}
If we define the pressure $p_{q+1}$ by 
\begin{equation*}
p_{q+1}\triangleq p_{l}+w_{q+1}\cdot\Lambda(y_{l}+z_{q+1}), 
\end{equation*}
then, alongside the first commutator error $R_{\text{Com1}}$ from \eqref{Rcom1}, we can define the transport error, the Nash error, the linear error, the oscillation error, and finally the second commutator error by
\begin{subequations}\label{est 22} 
\begin{align}
\divergence R_{T} \triangleq& \partial_{t} w_{q+1} + (\Lambda y_{l} + \Lambda z_{l}) \cdot \nabla w_{q+1}, \label{est 22a} \\
\divergence R_{N} \triangleq& ( \nabla  \Lambda (y_{l} + z_{q} ))^{T} \cdot w_{q+1} + (\Lambda w_{q+1} \cdot \nabla) y_{l} - (\nabla y_{l})^{T} \cdot \Lambda w_{q+1}, \label{est 22b}\\
\divergence R_{L} \triangleq& \Lambda^{\gamma} w_{q+1}  - \Lambda^{\gamma} z_{l} + \Lambda^{\frac{3}{2} - 2 \sigma} z_{l} + \Lambda^{\gamma} z_{q+1} - \Lambda^{\frac{3}{2} - 2 \sigma} z_{q+1} + \Lambda w_{q+1}^{\bot} (\nabla^{\bot} \cdot z_{l}),\label{est 22c}  \\
\divergence R_{O} \triangleq& \divergence \mathring{R}_{l} + (\Lambda w_{q+1} \cdot \nabla) w_{q+1} - (\nabla w_{q+1})^{T} \cdot \Lambda w_{q+1},  \label{est 22d}\\
\divergence R_{\text{Com2}} \triangleq&  \Lambda y_{q+1} \cdot \nabla (z_{q+1} - z_{q}) + \Lambda (z_{q+1} - z_{q}) \cdot \nabla y_{q+1}+ (\nabla y_{l})^{T} \cdot \Lambda (z_{q} - z_{q+1}) \nonumber \\
& \hspace{4mm} + (\nabla (z_{q} - z_{q+1}))^{T} \cdot \Lambda y_{l} + (\nabla \Lambda (z_{q+1} - z_{q}))^{T} \cdot w_{q+1} \nonumber \\
&\hspace{4mm} -\Lambda w_{q+1}^{\bot} \nabla^{\bot} \cdot (z_{l} - z_{q}) + (\nabla (z_{q} - z_{q+1}))^{T} \cdot \Lambda w_{q+1}   \nonumber \\
&\hspace{4mm} + \Lambda y_{l}^{\bot} \nabla^{\bot} \cdot (z_{q} - z_{l}) + \Lambda (z_{q} - z_{l}) \cdot \nabla w_{q+1} + \Lambda (z_{q} - z_{l})^{\bot} \nabla^{\bot} \cdot y_{l} \nonumber \\
&\hspace{4mm} + \Lambda (z_{q+1} - z_{q})^{\bot} \nabla^{\bot} \cdot z_{q+1} + \Lambda (z_{q} - z_{l})^{\bot}\nabla^{\bot} \cdot z_{q+1} \nonumber \\
&\hspace{4mm} + \Lambda z_{l}^{\bot} \nabla^{\bot} \cdot (z_{q+1} - z_{q}) + \Lambda z_{l}^{\bot} \nabla^{\bot} \cdot (z_{q} - z_{l}), \label{est 22e}
\end{align}
\end{subequations}
and verify by \eqref{additive convPDE} that we successfully decomposed $\mathring{R}_{q+1}$ by 
\begin{equation}\label{additive decomposition}
\divergence \mathring{R}_{q+1} = \divergence (R_{T} + R_{N} + R_{L} + R_{O} + R_{\text{Com1}} + R_{\text{Com2}}).
\end{equation} 

\begin{remark}\label{Remark 3.2}
This is the same decomposition of $\mathring{R}_{q+1}$ from \cite[Equations (124)-(125)]{Y23a}, aside from the $R_{L}$ added by 
\begin{equation*}
- \Lambda^{\gamma} z_{l} + \Lambda^{\frac{3}{2} - 2 \sigma} z_{l} + \Lambda^{\gamma} z_{q+1} - \Lambda^{\frac{3}{2} - 2 \sigma} z_{q+1}.
\end{equation*}
\end{remark}

\subsubsection{Perturbative Estimates}
First, for subsequent convenience, we estimate by \eqref{inductive 1c}, \eqref{inductive 1d}, and \eqref{additive stopping time}, 
\begin{align}\label{temp} 
\lVert \partial_{t} y_{q} (t)  \rVert_{C_{x}}  \lesssim   M_{0} \lambda_{q}^{2} \delta_{q} \bar{e} + (M_{0}^{\frac{1}{2}} \lambda_{q} \delta_{q}^{\frac{1}{2}} \bar{e}^{\frac{1}{2}} + \lVert \Lambda z_{l} \rVert_{C_{t,x,q+1}} )M_{0}^{\frac{1}{2}} \lambda_{q} \delta_{q}^{\frac{1}{2}} \bar{e}^{\frac{1}{2}} \lesssim M_{0} \lambda_{q}^{2} \delta_{q} \bar{e} 
\end{align}
and consequently,   
\begin{equation}\label{est 12}
\lVert y_{l} - y_{q} \rVert_{C_{t, x, q+1}} \lesssim l (\lVert y_{q} \rVert_{C_{t, q+1}^{1} C_{x}} +  \lVert y_{q} \rVert_{C_{t, q+1} C_{x}^{1}} )\overset{\eqref{temp} \eqref{inductive 1c}}{\lesssim} l M_{0} \lambda_{q}^{2} \delta_{q} \bar{e}.
\end{equation}
Next, we estimate from \eqref{est 0b}-\eqref{est 0a} 
\begin{align}
\lvert \rho_{j} \rvert \leq& \varepsilon_\gamma^{-1} ( l + \lVert \mathring{R}_{q,j} (\tau_{q+1} j) \rVert_{C_{x}}) \nonumber \\
&+ \left\lVert \frac{1}{4(2\pi)^{2}} [ e(t) (1 - \lambda_{q+2} \delta_{q+2} ) - \lVert (y_{q} + z_{q}) (t) \rVert_{\dot{H}_{x}^{\frac{1}{2}}}^{2}]\right\rVert_{C_{\tau_{q+1} j, q+1}}.  \label{est 5}
\end{align}
From the inductive hypothesis \eqref{inductive 3} at step $q$ and \eqref{est 4}, we have for all $t \in [t_{q}, \mathfrak{t}]$, 
\begin{equation}\label{est 112}
\left\lvert e(t) (1- \lambda_{q+2} \delta_{q+2}) - \lVert (y_{q} + z_{q}) (t) \rVert_{\dot{H}_{x}^{\frac{1}{2}}}^{2} \right\rvert  \leq \frac{5}{4} \lambda_{q+1} \delta_{q+1} \bar{e}. 
\end{equation}
Applying this estimate to \eqref{est 5} now gives us 
\begin{equation}\label{est 28}
\lvert \rho_{j} \rvert \leq \varepsilon_\gamma^{-1} ( l + \lVert \mathring{R}_{q,j} (\tau_{q+1} j) \rVert_{C_{x}}) + \frac{1}{4(2\pi)^{2}} \frac{5}{4} \lambda_{q+1} \delta_{q+1} \bar{e} \overset{\eqref{est 16} \eqref{inductive 2b}}{\leq} \frac{1}{8 \pi^{2}} \lambda_{q+1} \delta_{q+1} \bar{e}.
\end{equation}
For subsequent purposes, it is convenient to compute  for $\beta > \frac{1}{2}$ sufficiently close to $\frac{1}{2}$ and $a_{0}$ sufficiently large, 
\begin{align} 
\tau_{q+1} \lVert D ( \Lambda y_{l} + \Lambda z_{l}) \rVert_{C_{t, x, q+1}} &\overset{\eqref{inductive 1} \eqref{additive stopping time}}{\lesssim}  \tau_{q+1} \lambda_{q}^{2} \delta_{q}^{\frac{1}{2}}   \nonumber \\
\lesssim&  a^{b[ \frac{\beta}{2} - \frac{1}{4} ]} a^{b^{q} [ 2 - \beta + b(-\frac{\alpha}{2} - \frac{1}{2} - \frac{\beta}{2} ) + b^{2} (\beta - \frac{1}{2} )]}\overset{\eqref{additive bound b}}{\ll} 1. \label{est 17} 
\end{align}

\begin{proposition}\label{prop additive amplitudes}
Define $a_{k,j}$ by \eqref{est 3}. Then it satisfies the following bounds: for all $t \in \supp \chi_{j}$ and all $N \in \mathbb{N}$, 
\begin{subequations}\label{est 19} 
\begin{align}
& \lVert a_{k,j}(t) \rVert_{C_{x}} \leq \frac{1}{\sqrt{8} \pi} \bar{e}^{\frac{1}{2}} \delta_{q+1}^{\frac{1}{2}} \sup_{k \in \Gamma_{1} \cup \Gamma_{2}} \lVert \gamma_{k} \rVert_{C(B(\Id, \varepsilon_\gamma ))},  \label{est 10}  \\
& \lVert D^{N} a_{k,j}(t) \rVert_{C_{x}}  \lesssim \delta_{q+1}^{\frac{1}{2}} \bar{e}^{\frac{1}{2}} \lambda_{q}^{N}.  \label{est 19b} 
\end{align}
\end{subequations}
Furthermore, $y_{q+1}$, $w_{q+1}$, and $\mathring{R}_{q+1}$ from \eqref{est 44}, \eqref{est 14a}, and \eqref{additive decomposition} satisfy the hypotheses \eqref{inductive 1a} and \eqref{inductive 2a} at level $q+1$, as well as the support condition \eqref{support wqplus1}.
\end{proposition}

\begin{proof}[Proof of Proposition \ref{prop additive amplitudes}]
We leave this proof in the Appendix Section \ref{Section 7.2} for completeness. 
\end{proof}

\begin{proposition}\label{prop additive perturbation}
Define $D_{t,q}$ by \eqref{est 56}. Then $w_{q+1}$ defined by \eqref{est 14a} satisfies the following bounds over $[t_{q+1}, \mathfrak{t}]$: for $C_{1}$ from \eqref{est 8},  
\begin{subequations}\label{est 11} 
\begin{align}
& \lVert w_{q+1}(t) \rVert_{C_{x}} \leq \frac{C_{1}}{\sqrt{2} \pi} \bar{e}^{\frac{1}{2}} \delta_{q+1}^{\frac{1}{2}} \sup_{k \in \Gamma_{1} \cup \Gamma_{2}} \lVert \gamma_{k} \rVert_{C(B(\Id, \varepsilon_\gamma))},  \label{est 11a}\\
& \lVert D_{t,q} w_{q+1} (t) \rVert_{C_{x}} \lesssim \delta_{q+1}^{\frac{1}{2}} \bar{e}^{\frac{1}{2}} \tau_{q+1}^{-1}\label{est 11b}. 
\end{align}
\end{subequations}
Consequently, the hypotheses \eqref{inductive 1b}, \eqref{inductive 1c}, and \eqref{inductive 1d} at level $q+1$, as well as the Cauchy difference \eqref{cauchy difference} all hold.
\end{proposition}

\begin{proof}[Proof of Proposition \ref{prop additive perturbation}]
First, we verify \eqref{est 11a} starting from  \eqref{est 14a} using the fact that for all $t \in [t_{q+1}, \mathfrak{t}]$ fixed, there exist at most two non-trivial cutoffs, 
\begin{align*}
\lVert w_{q+1} (t) \rVert_{C_{x}} \overset{\eqref{est 8}}{\leq}& C_{1} \sum_{j,k} 1_{\supp \chi_{j}}(t) \lVert a_{k,j} (t) \rVert_{C_{x}}   \overset{\eqref{est 10}}{\leq}  \frac{C_{1}}{\sqrt{2} \pi} \bar{e}^{\frac{1}{2}} \delta_{q+1}^{\frac{1}{2}} \sup_{k \in \Gamma_{1} \cup \Gamma_{2}} \lVert \gamma_{k} \rVert_{C(B(\Id, \varepsilon_\gamma))}. 
\end{align*}
It follows that for $\beta > \frac{1}{2}$ sufficiently close to $\frac{1}{2}$ and $a_{0}$ sufficiently large
\begin{align}
\lVert y_{q+1}(t) - y_{q}(t) \rVert_{C_{x}} \overset{\eqref{est 44}\eqref{est 11a} \eqref{est 12} \eqref{M0}}{\leq} \frac{1}{2\sqrt{2}} M_{0}^{\frac{1}{2}} \bar{e}^{\frac{1}{2}} \delta_{q+1}^{\frac{1}{2}} + C l \lambda_{q}^{2} \delta_{q} \bar{e}  \overset{\eqref{additive bound b}}{\leq} M_{0}^{\frac{1}{2}} \bar{e}^{\frac{1}{2}} \delta_{q+1}^{\frac{1}{2}};  \label{weaker additive M0}
\end{align}
thus, we have verified \eqref{cauchy difference}. Next, using $M_{0}^{\frac{1}{2}} \geq \frac{4C_{1}}{\pi} \sup_{k \in \Gamma_{1} \cup \Gamma_{2}} \lVert \gamma_{k} \rVert_{C(B(\Id, \varepsilon_\gamma ))}$ from \eqref{M0}, we estimate 
\begin{equation*}
\lVert y_{q+1} \rVert_{C_{t,x, q+1}} \overset{\eqref{est 44}}{\leq}  \lVert y_{l} \rVert_{C_{t,x,q+1}} + \lVert w_{q+1} \rVert_{C_{t,x,q+1}}  \overset{\eqref{inductive 1b} \eqref{est 11a} \eqref{M0}}{\leq}  M_{0}^{\frac{1}{2}} \left( 1+ \sum_{1 \leq j \leq q+1} \delta_{j}^{\frac{1}{2}} \right) \bar{e}^{\frac{1}{2}} 
\end{equation*} 
verifying \eqref{inductive 1b} at level $q+1$. Then we deduce \eqref{inductive 1c} at level $q+1$ as follows: for $\beta > \frac{1}{2}$ sufficiently close to $\frac{1}{2}$ and $a_{0}$ sufficiently large 
\begin{align}
&\lVert y_{q+1} \rVert_{C_{t,q+1} C_{x}^{1}} + \lVert \Lambda y_{q+1} \rVert_{C_{t,x,q+1}} 
\overset{\eqref{est 44} \eqref{support wqplus1} \eqref{inductive 1c}}{\leq} 2(2 \lambda_{q+1} ) \lVert w_{q+1} \rVert_{C_{t,x,q+1}}  + CM_{0}^{\frac{1}{2}} \lambda_{q} \delta_{q}^{\frac{1}{2}} \bar{e}^{\frac{1}{2}} \nonumber \\
& \hspace{10mm}  \overset{\eqref{est 11a} \eqref{M0}}{\leq}  \frac{M_{0}^{\frac{1}{2}}}{\sqrt{2}} \lambda_{q+1} \delta_{q+1}^{\frac{1}{2}} \bar{e}^{\frac{1}{2}} + C M_{0}^{\frac{1}{2}} \lambda_{q} \delta_{q}^{\frac{1}{2}} \bar{e}^{\frac{1}{2}} \overset{\eqref{bound beta}}{\leq}  M_{0}^{\frac{1}{2}} \lambda_{q+1} \delta_{q+1}^{\frac{1}{2}} \bar{e}^{\frac{1}{2}}. 
\end{align}
Next, we can find directly from \eqref{est 56}, \eqref{est 3}, \eqref{est 66}, and \eqref{def Phi j}-\eqref{est 16} that
\begin{equation}\label{est 23} 
D_{t,q}a_{k,j} = 0 \hspace{1mm} \text{ and } \hspace{1mm} D_{t,q}b_k(\lambda_{q+1}\Phi_{j}(t,x)) = 0.
\end{equation} 
Thus, using \eqref{est 67a}, \eqref{est 14b}, \eqref{est 56}, and \eqref{est 23}, we can compute simplify $D_{t,q} w_{q+1}$ as 
\begin{equation}\label{new 9}
D_{t,q} w_{q+1} = \sum_{j,k} [ D_{t,q}, \mathbb{P}_{q+1, k} ] \tilde{w}_{q+1, j, k} + \mathbb{P}_{q+1, k} [ (\partial_{t} \chi_{j}) a_{k,j} b_{k} (\lambda_{q+1} \Phi_{j} ) ]. 
\end{equation}
Therefore, we now estimate, with the use of \eqref{new 8} and \cite[Corollary A.8]{BSV19},
\begin{align}
& \lVert D_{t,q} w_{q+1} (t) \rVert_{C_{x}}\label{est 13} \\
\lesssim& \sum_{j,k} ( \lambda_{q} \lVert \Lambda y_{l} (t) \rVert_{C_{x}} + \lVert \nabla \Lambda z_{l} (t) \rVert_{C_{x}}) 1_{\supp \chi_{j}} (t) \lVert a_{k,j} (t) \rVert_{C_{x}} + 1_{\supp \chi_{j}} (t) \tau_{q+1}^{-1} \lVert a_{k,j} (t) \rVert_{C_{x}} \nonumber\\
\lesssim&\bar{e}^{\frac{1}{2}}  \sum_{j,k} 1_{\supp \chi_{j}} (t)  [ M_{0} \lambda_{q}^{2} \delta_{q}^{\frac{1}{2}}  + \lambda_{q} \lVert z_{q} \rVert_{C_{t,q+1} H_{x}^{2+ \sigma}} ]\delta_{q+1}^{\frac{1}{2}} \bar{e}^{\frac{1}{2}} + \tau_{q+1}^{-1} \delta_{q+1}^{\frac{1}{2}} \bar{e}^{\frac{1}{2}} \overset{ \eqref{est 17}}{\lesssim} \delta_{q+1}^{\frac{1}{2}} \bar{e}^{\frac{1}{2}} \tau_{q+1}^{-1}, \nonumber
\end{align}
showing \eqref{est 11b}.
Now identically to \cite[p. 24]{Y23a} we have 
\begin{align*}
D_{t, q+1} &y_{q+1} 
= D_{t,q} w_{q+1} + (\Lambda y_{l} \ast_{t} \phi_{\lambda_{q+2}^{-\alpha}} \ast_{t} \varphi_{\lambda_{q+2}^{-\alpha}} - \Lambda y_{l}) \cdot \nabla w_{q+1} \\
&+ (\Lambda z_{q+1} \ast_{x} \phi_{\lambda_{q+2}^{-\alpha}} \ast_{t} \varphi_{\lambda_{q+2}^{-\alpha}} - \Lambda z_{l}) \cdot \nabla w_{q+1}\nonumber  \\
&+ \Lambda  w_{q+1} \ast_{x} \phi_{\lambda_{q+2}^{-\alpha}} \ast_{t} \varphi_{\lambda_{q+2}^{-\alpha}} \cdot \nabla w_{q+1} + \Lambda w_{q+1} \ast_{x} \phi_{\lambda_{q+2}^{-\alpha}} \ast_{t} \varphi_{\lambda_{q+2}^{-\alpha}} \cdot \nabla y_{l}\nonumber  \\
&+ [D_{t,q} y_{q}]\ast_{x} \phi_{\lambda_{q+1}^{-\alpha}} \ast_{t} \varphi_{\lambda_{q+1}^{-\alpha}} + \Lambda y_{l} \ast_{x} \phi_{\lambda_{q+2}^{-\alpha}} \ast_{t} \varphi_{\lambda_{q+2}^{-\alpha}} \cdot \nabla y_{l}\nonumber \\
& - [\Lambda y_{q} \ast_{x} \phi_{\lambda_{q+1}^{-\alpha}} \ast_{t} \varphi_{\lambda_{q+1}^{-\alpha}} \cdot \nabla y_{q} ] \ast_{x} \phi_{\lambda_{q+1}^{-\alpha}} \ast_{t} \varphi_{\lambda_{q+1}^{-\alpha}}\nonumber  \\
&+ \Lambda z_{q+1} \ast_{x} \phi_{\lambda_{q+2}^{-\alpha}} \ast_{t} \varphi_{\lambda_{q+2}^{-\alpha}} \cdot \nabla y_{l} - [\Lambda z_{q} \ast_{x} \phi_{\lambda_{q+1}^{-\alpha}} \ast_{t} \varphi_{\lambda_{q+1}^{-\alpha}} \cdot \nabla y_{q} ] \ast_{x} \phi_{\lambda_{q+1}^{-\alpha}} \ast_{t} \varphi_{\lambda_{q+1}^{-\alpha}} \nonumber 
\end{align*}
which we can estimate 
\begin{align}\label{additive M0}
&\lVert D_{t, q+1} y_{q+1} \rVert_{C_{t,x,q+1}} \nonumber\\ 
&\leq \lVert D_{t,q} w_{q+1} \rVert_{C_{t,x,q+1}} + 2\lVert \Lambda  y_{l} \rVert_{C_{t,x,q+1}} \lVert \nabla w_{q+1} \rVert_{C_{t,x,q+1}} + 2\lVert \Lambda  z \rVert_{C_{t,x,q+1}} \lVert \nabla w_{q+1} \rVert_{C_{t,x,q+1}} \nonumber\\
&+ \lVert\Lambda  w_{q+1} \rVert_{C_{t,x,q+1}} \lVert \nabla w_{q+1} \rVert_{C_{t,x,q+1}} + \lVert \Lambda  w_{q+1} \rVert_{C_{t,x,q+1}} \lVert \nabla y_{q} \rVert_{C_{t,x,q+1}} + \lVert D_{t,q} y_{q} \rVert_{C_{t,x,q+1}}  \nonumber\\
&+ 2\lVert \Lambda  y_{q} \rVert_{C_{t,x,q+1}} \lVert \nabla y_{q} \rVert_{C_{t,x,q+1}} + 2\lVert \Lambda  z \rVert_{C_{t,x,q+1}} \lVert \nabla y_{q} \rVert_{C_{t,x,q+1}} \nonumber \\
&\overset{\eqref{additive stopping time} \eqref{support wqplus1} }{\leq} 4 \lambda_{q+1}^{2} \lVert w_{q+1} \rVert_{C_{t,x,q+1}}^{2}  \nonumber \\
&+ C (\lVert D_{t,q} w_{q+1} \rVert_{C_{t,x,q+1}} + \lVert \Lambda y_{q} \rVert_{C_{t,x,q+1}} \lambda_{q+1} \lVert w_{q+1} \rVert_{C_{t,x,q+1}} + \lambda_{q+1} \lVert w_{q+1} \rVert_{C_{t,x,q+1}} \nonumber\\
& \hspace{5mm} + \lambda_{q+1} \lVert w_{q+1} \rVert_{C_{t,x,q+1}} \lVert \nabla y_{q} \rVert_{C_{t,x,q+1}} + \lVert D_{t,q} y_{q} \rVert_{C_{t,x,q+1}}  \nonumber \\
& \hspace{5mm} + \lVert \Lambda y_{q} \rVert_{C_{t,x,q+1}} \lVert \nabla y_{q} \rVert_{C_{t,x,q+1}} + \lVert \nabla y_{q} \rVert_{C_{t,x,q+1}}) \overset{\eqref{est 13} \eqref{M0}}{\leq} M_{0} \lambda_{q+1}^{2} \delta_{q+1} \bar{e}.
\end{align}
Hence, \eqref{inductive 1d} is satisfied at level $q+1$, concluding the proof of Proposition \ref{prop additive perturbation}.
\end{proof}

\subsubsection{Stress Decomposition Estimates}
It will first be useful to determine some bounds on the functions $\psi_{q+1, j, k}$ from \eqref{est 14c}.
\begin{proposition}\label{prop additive psi}
Define $\psi_{q+1, j,k}$ by \eqref{est 14c}. Then it satisfies the following bounds: for all $t \in \supp \chi_{j}$ and all $N\in\mathbb{N}$, 
\begin{subequations}\label{psi}
\begin{align}
&  \lVert \psi_{q+1, j, k} (t) \rVert_{C_{x}} = 1, \hspace{18mm}  \lVert D^{N} \psi_{q+1, j,k} (t) \rVert_{C_{x}} \lesssim \lambda_{q+1} \tau_{q+1} \lambda_{q}^{N+1} \delta_{q}^{\frac{1}{2}} \bar{e}^{\frac{1}{2}} M_{0}^{\frac{1}{2}}, \label{psi 1} \\
&  \lVert D (a_{k,j} \psi_{q+1, j,k} ) (t) \rVert_{C_{x}} \lesssim  \lambda_{q} \delta_{q+1}^{\frac{1}{2}} \bar{e}^{\frac{1}{2}}  \hspace{7mm}  \lVert D^{2} (a_{k,j} \psi_{q+1, j, k} ) (t) \rVert_{C_{x}}  \lesssim \lambda_{q}^{2} \delta_{q+1}^{\frac{1}{2}} \bar{e}^{\frac{1}{2}}.\label{psi 2} 
\end{align}
\end{subequations} 
\end{proposition}

\begin{proof}[Proof of Proposition \ref{prop additive psi}]
We leave this proof in the Appendix Section \ref{Section 7.2} for completeness. 
\end{proof}
Now, we may individually estimate the six components of the Reynolds stress $\mathring{R}_{q+1}$ in \eqref{additive decomposition}. First, we bound the transport error component.
 
\begin{proposition}\label{prop RT}
$R_{T}$ from \eqref{est 22a} satisfies for $\beta > \frac{1}{2}$ sufficiently close to $\frac{1}{2}$ and $a_{0}$ sufficiently large 
\begin{equation*}
\lVert R_{T} \rVert_{C_{t,x,q+1}} \ll \lambda_{q+3} \delta_{q+3}. 
\end{equation*}
\end{proposition}

\begin{proof}[Proof of Proposition \ref{prop RT}]
We can write via \eqref{est 22a}, \eqref{est 67a}, and \eqref{est 14b} 
\begin{align}
R_{T} =& \sum_{j,k} 1_{\supp \chi_{j}} \mathcal{B} ( [ ( \Lambda y_{l} + \Lambda z_{l}) \cdot \nabla, \mathbb{P}_{q+1, k} ] \tilde{w}_{q+1, j, k} ) \nonumber \\
& \hspace{5mm} + \mathbb{P}_{q+1, k} ( [ \partial_{t} + (\Lambda y_{l} + \Lambda z_{l}) \cdot \nabla] \left( \chi_{j} a_{k,j}  b_{k} (\lambda_{q+1} \Phi_{j} ) \right) \label{new 11}
\end{align}
where \eqref{est 56} and \eqref{est 23} reveal that 
\begin{align}
[ \partial_{t} + (\Lambda y_{l} + \Lambda z_{l} ) \cdot \nabla] \left(\chi_{j} a_{k,j}  b_{k} (\lambda_{q+1} \Phi_{j}) \right)  = \partial_{t} \chi_{j}  a_{k,j}  b_{k} (\lambda_{q+1} \Phi_{j }). \label{new 10}
\end{align}
It is here that the fact $D_{t,q} a_{k,j} =0$ from \eqref{est 23} is essential. Applying \eqref{new 10} to \eqref{new 11} and relying on \eqref{inductive 1a} and \eqref{define f} inform us that $\supp ( \Lambda y_{l} + \Lambda z_{l})\hspace{0.5mm}\hat{}\hspace{1mm} \subset B(0,2 \lambda_{q})$ and 
\begin{align*}
\supp  (\mathbb{P}_{q+1, k} f)\hspace{0.5mm}\hat{}\hspace{1mm} \subset \left\{ \xi: \frac{7}{8} \lambda_{q+1} \leq \lvert \xi \rvert \leq \frac{9}{8} \lambda_{q+1} \right\}, \hspace{1mm} \supp ( \tilde{P}_{\approx \lambda_{q+1}} f)\hspace{0.5mm}\hat{}\hspace{1mm} \subset \left\{ \xi: \frac{\lambda_{q+1}}{4} \leq \lvert \xi \rvert \leq 4 \lambda_{q+1}\right\}
\end{align*}
from Section \ref{Section 2.2}, we can write similarly to \cite[Equation (154)]{Y23a}
\begin{align}
R_{T} =& \sum_{j,k} 1_{\supp \chi_{j}} (t) \mathcal{B} \tilde{P}_{\approx \lambda_{q+1}}  \nonumber \\
& \times \left( [ ( \Lambda y_{l} + \Lambda z_{l}) \cdot \nabla, \mathbb{P}_{q+1, k} ] \tilde{w}_{q+1, j, k} + \mathbb{P}_{q+1, k} ( \partial_{t} \chi_{j} a_{k,j} b_{k} (\lambda_{q+1} \Phi_{j} ))\right). \label{new 12}
\end{align}
This allows us to estimate by \cite[Lemma A.6]{BSV19} for $\beta > \frac{1}{2}$ sufficiently close to $\frac{1}{2}$ and $a_{0}$ sufficiently large 
\begin{align}
& \lVert R_{T} \rVert_{C_{t,x,q+1}} \lesssim \lambda_{q+1}^{-1} [ ( \lVert \nabla \Lambda y_{l} \rVert_{C_{t,x,q+1}} + \lVert \nabla \Lambda z_{l} \rVert_{C_{t,x,q+1}} ) \lVert \tilde{w}_{q+1, j, k} \rVert_{C_{t,x,q+1}} + \tau_{q+1}^{-1} \delta_{q+1}^{\frac{1}{2}} \bar{e}^{\frac{1}{2}} ] \nonumber\\
&\overset{ \eqref{inductive 1a} \eqref{inductive 1c} \eqref{tau} \eqref{additive stopping time} \eqref{define l}}{\lesssim}\lambda_{q+1}^{-1-\beta} \lambda_{q}^{2-\beta} \lambda_{1}^{2\beta -1} \bar{e}  + \lambda_{q+1}^{-\frac{1}{2} + \frac{\alpha}{2} -\frac{\beta}{2}} \lambda_{q+2}^{\frac{1}{2} - \beta} \lambda_{1}^{(2\beta -1) \frac{3}{4}}  \bar{e}^{\frac{1}{2}} \ll  \lambda_{q+3} \delta_{q+3}. \label{est 121} 
\end{align}
\end{proof}

Next, we treat the Nash error.
\begin{proposition}\label{prop RN}
$R_{N}$ from \eqref{est 22b} satisfies for $\beta > \frac{1}{2}$ sufficiently close to $\frac{1}{2}$ and $a_{0}$ sufficiently large 
\begin{equation}\label{est 246a}
\lVert R_{N} \rVert_{C_{t,x,q+1}} \ll \lambda_{q+3} \delta_{q+3}. 
\end{equation}
\end{proposition}

\begin{proof}[Proof of Proposition \ref{prop RN}]
Similarly to \cite[Equations (156)-(157)]{Y23a}, we can further break this component into
\begin{equation}\label{break N}
R_{N} = N_{1} + N_{2},\text{ where }  N_{1} \triangleq \mathcal{B} ( ( \nabla \Lambda (y_{l} + z_{q} ))^{T} \cdot w_{q+1}), \hspace{1mm} N_{2} \triangleq \mathcal{B} ( \Lambda w_{q+1}^{\bot} ( \nabla^{\bot} \cdot y_{l} ) ). 
\end{equation}
Similar reasoning to \eqref{new 12} gives us 
\begin{equation}\label{additive frequency support}
\supp \left( ( \nabla \Lambda (y_{l} + z_{q}))^{T} \cdot w_{q+1} \right)\hspace{1mm}\widehat{}\hspace{1mm} \subset \left\{ \frac{\lambda_{q+1}}{4} \leq \lvert \xi \rvert \leq 4 \lambda_{q+1} \right\}. 
\end{equation}
With these in mind, we are able to compute from \eqref{break N}, for $\beta > \frac{1}{2}$ sufficiently close to $\frac{1}{2}$ and $a_{0}$ sufficiently large, 
\begin{align}
\lVert N_{1} \rVert_{C_{t, x, q+1}} \overset{ \eqref{additive frequency support} \eqref{est 14a}}{\lesssim}& \lambda_{q+1}^{-1} \sum_{j,k} \left( \lVert \nabla \Lambda y_{l} \rVert_{C_{t,x,q+1}} + \lVert \nabla \Lambda z_{q} \rVert_{C_{t,x,q+1}} \right) \lVert 1_{\supp \chi_{j}} a_{k,j} \rVert_{C_{t,x,q+1}}  \nonumber \\
\overset{\eqref{inductive 1a} \eqref{inductive 1c} \eqref{est 10}\eqref{additive stopping time}}{\lesssim}&  \lambda_{q+1}^{-1} \lambda_{q}^{2-\beta} \lambda_{q+1}^{-\beta} \lambda_{1}^{2\beta -1} \bar{e}  \overset{\eqref{additive bound b}}{\ll}  \lambda_{q+3} \delta_{q+3}.  \label{est 243}
\end{align}
Similarly to \cite[Equation (161)]{Y23a}, we can write 
\begin{align*}
N_{2}&(t,x) = \lambda_{q+1}^{-1} \mathcal{B}  \left( \nabla ( \nabla^{\bot} \cdot y_{l} )(t,x) \sum_{j,k} \Lambda \mathbb{P}_{q+1, k} (\chi_{j} (t) a_{k,j} (t,x) c_{k} (\lambda_{q+1} \Phi_{j} (t,x))) \right) \\
&+ \lambda_{q+1}^{-1} \mathcal{B}  \left( ( \nabla^{\bot} \cdot y_{l}) (t,x) \sum_{j,k} \Lambda \mathbb{P}_{q+1, k} ( \chi_{j} (t) \nabla ( a_{k,j} (t,x) \psi_{q+1, j, k} (t,x) ) c_{k} (\lambda_{q+1} x) ) \right)
\end{align*}
and thus we can estimate by \cite[Equation (A.11)]{BSV19} 
\begin{align}
\lVert N_{2} \rVert_{C_{t,x,q+1}} \lesssim& \lambda_{q+1}^{-2} \sum_{j,k} \lVert \nabla ( \nabla^{\bot} \cdot y_{l} ) \rVert_{C_{t,x,q+1}} \lVert \Lambda \mathbb{P}_{q+1, k} \left( \chi_{j} a_{k,j} c_{k} (\lambda_{q+1} \Phi_{j} )  \right) \rVert_{C_{t,x,q+1}} \nonumber \\
&+ \lVert \nabla^{\bot} \cdot y_{l} \rVert_{C_{t,x,q+1}} \lVert \Lambda \mathbb{P}_{q+1, k} ( \chi_{j} \nabla (a_{k,j} \psi_{q+1, j, k}) c_{k} ) \rVert_{C_{t,x,q+1}} \nonumber  \\
\overset{\eqref{inductive 1a} \eqref{inductive 1c} \eqref{est 19} \eqref{psi 2} }{\lesssim}& \lambda_{q+1}^{-1} \lambda_{q}^{2} \delta_{q}^{\frac{1}{2}} \delta_{q+1}^{\frac{1}{2}} \approx   \lambda_{q+1}^{-1-\beta} \lambda_{q}^{2-\beta} \lambda_{1}^{2\beta -1}  \overset{\eqref{additive bound b}}{\ll} \lambda_{q+3} \delta_{q+3} \label{est 244}
\end{align}
for $\beta > \frac{1}{2}$ sufficiently close to $\frac{1}{2}$ and $a_{0}$ sufficiently large. Hence, by applying \eqref{est 243} and \eqref{est 244} to \eqref{break N}, we conclude \eqref{est 246a}.
\end{proof}

Next, we treat the linear error. 
\begin{proposition}\label{prop RL}
$R_{L}$ from \eqref{est 22c} satisfies for $\beta > \frac{1}{2}$ sufficiently close to $\frac{1}{2}$ and $a_{0}$ sufficiently large 
\begin{equation}\label{est 248}
\lVert R_{L} \rVert_{C_{t,x,q+1}} \ll \lambda_{q+3} \delta_{q+3}. 
\end{equation}
\end{proposition}

\begin{proof}[Proof of Proposition \ref{prop RL}]
We can break the linear error into three as
\begin{equation}\label{new 14}
R_{L} = L_{1} + L_{2} + L_{3}
\end{equation}
where 
\begin{subequations}\label{new 15}
\begin{align}
L_{1} &\triangleq \mathcal{B} \Lambda^{\gamma} w_{q+1},\label{new 15a}\\
L_{2} &\triangleq - \mathcal{B} \left( \Lambda^{\gamma} (z_{l} - z_{q+1}) - \Lambda^{\frac{3}{2} - 2 \sigma} (z_{l} - z_{q+1} ) \right), \label{new 15b}\\
L_{3} &\triangleq \mathcal{B} [ \Lambda w_{q+1}^{\bot} (\nabla^{\bot}  \cdot z_{l} ) ]. \label{new 15c}
\end{align}
\end{subequations}
First, we estimate for $\beta >\frac{1}{2}$ sufficiently close to $\frac{1}{2}$ and $a_{0}$ sufficiently large, 
\begin{align}
\lVert L_{1} \rVert_{C_{t,x,q+1}} \lesssim \lambda_{q+1}^{\gamma - 1} \lVert w_{q+1} \rVert_{C_{t,x,q+1}} \overset{\eqref{est 11a}}{\lesssim} \lambda_{q+1}^{\gamma -1 - \beta}  \lambda_{1}^{\beta - \frac{1}{2}} \bar{e}^{\frac{1}{2}}  \ll \lambda_{q+3} \delta_{q+3}.\label{est 245}
\end{align} 
Second, we rewrite
\begin{align*}
& - \Lambda^{\gamma} ( z_{l} - z_{q+1}) + \Lambda^{\frac{3}{2} - 2 \sigma} (z_{l} - z_{q+1}) \\
=& -\Lambda^{\gamma} \left( (z_{q} - z_{q+1}) \ast_{x} \phi_{l} \ast_{t} \varphi_{l} + z_{q+1} \ast_{x} \phi_{l} \ast_{t} \varphi_{l} - z_{q+1} \right) \\
&+ \Lambda^{\frac{3}{2} - 2 \sigma} \left( ( z_{q} - z_{q+1} ) \ast_{x} \phi_{l} \ast_{t} \varphi_{l} + z_{q+1} \ast_{x} \phi_{l} \ast_{t} \varphi_{l} - z_{q+1} \right) 
\end{align*}
and estimate using $\supp (z_{q} - z_{q+1})\hspace{0.5mm}\hat{}\hspace{1mm} \subset \{ \xi: 4 \lambda_{q} < \lvert \xi \rvert \leq 4 \lambda_{q+1} \}$, for $\beta > \frac{1}{2}$ sufficiently close to $\frac{1}{2}$, $a_{0}$ sufficiently large, and any $\delta \in (0, \frac{1}{4})$, 
\begin{align}
\lVert L_{2} \rVert_{C_{t,x,q+1}} \overset{\eqref{new}}{\lesssim}& \lVert \Lambda^{\frac{1}{2} - 2 \sigma} \left( ( z_{q} - z_{q+1}) \ast_{x} \phi_{l} \ast_{t} \varphi_{l} + z_{q+1}\ast_{x} \phi_{l} \ast_{t} \varphi_{l} - z_{q+1} \right) \rVert_{C_{t,x,q+1}} \nonumber \\
\overset{\eqref{additive stopping time}}{\lesssim}& \lambda_{q}^{-1} + l^{\frac{1}{2} - \delta} + l  \overset{\eqref{define l}}{\lesssim} \lambda_{q}^{-1} + \lambda_{q+1}^{-\alpha(\frac{1}{2} - \delta)}  \ll \lambda_{q+3} \delta_{q+3}. \label{est 246}
\end{align}
Finally, we make use of the fact that $\supp w_{q+1}^{\bot} \subset \{ \xi: \frac{\lambda_{q+1}}{2} \leq \lvert \xi \rvert \leq 2 \lambda_{q+1} \}$ due to \eqref{support wqplus1} while $\supp \hat{z}_{l} \subset \{ \xi: \lvert \xi  \rvert \leq \frac{\lambda_{q}}{4}\}$ due to \eqref{define f} and compute via \cite[Equation (A.11)]{BSV19}, for $\beta > \frac{1}{2}$ sufficiently close to $\frac{1}{2}$ and $a_{0}$ sufficiently large, 
\begin{equation}\label{est 247}
\lVert L_{3} \rVert_{C_{t,x,q+1}} \overset{\eqref{new 15c}}{\lesssim} \lambda_{q+1}^{-1} \lVert \Lambda w_{q+1} \rVert_{C_{t,x,q+1}} \lVert z_{l} \rVert_{C_{t,q+1} C_{x}^{1}}  \overset{\eqref{est 11a}\eqref{additive stopping time}}{\lesssim} \delta_{q+1}^{\frac{1}{2}} \bar{e}^{\frac{1}{2}} \ll  \lambda_{q+3} \delta_{q+3}; 
\end{equation} 
hence, by applying \eqref{est 245}, \eqref{est 246}, and \eqref{est 247} to \eqref{new 15}, we obtain \eqref{est 248}.
\end{proof}

Next, we treat the oscillatory error.
\begin{proposition}\label{prop RO}
$R_{O}$ from \eqref{est 22d} satisfies for $\beta > \frac{1}{2}$ sufficiently close to $\frac{1}{2}$ and $a_{0}$ sufficiently large 
\begin{equation}\label{est 195}
\lVert R_{O} \rVert_{C_{t,x,q+1}} \ll \lambda_{q+3} \delta_{q+3}. 
\end{equation}
\end{proposition}

\begin{proof}[Proof of Proposition \ref{prop RO}]
Much of this proof is similar to \cite[Section 4.1.4]{Y23a}, but the intermediate steps of some estimates differ (e.g. \eqref{new 17} and \eqref{est 31}), so we include these details. To do this, we first define 
\begin{subequations}\label{est 85}
\begin{align}
& \vartheta_{j,k} \triangleq \nabla^{\bot} \cdot \mathbb{P}_{q+1, k} \chi_{j} a_{k,j} b_{k} (\lambda_{q+1} \Phi_{j}), \label{est 85a}\\
& s^{m} (\lambda, \eta) \triangleq \int_{0}^{1} \frac{ i( ( 1-r)^{\eta} - r \lambda )_{m}}{\lvert (1-r) \eta - r \lambda \rvert} dr, \label{est 85b}\\
& ( \mathcal{T}^{m} (f, g))\hspace{0.5mm}\hat{}\hspace{1mm} (\xi) \triangleq \int_{\mathbb{R}^{2}} s^{m} (\xi - \eta, \eta) \hat{f} (\xi - \eta) \hat{g}(\eta) d \eta, \label{est 85c}\\
& \mathcal{L}_{j,k}^{ml} \triangleq \frac{1}{2} \mathcal{T}^{m} ( \Lambda^{-1} \vartheta_{j,k}, \mathcal{R}_{l} \vartheta_{j, -k} ), \label{est 85d}
\end{align}
\end{subequations}
similarly to \cite[Equations (179), (183), (188) and (192)]{Y23a}. This leads to 
\begin{equation}\label{new 33}
R_{O} = R_{O, \text{approx}} + R_{O, \text{low}} + R_{O, \text{high}}
\end{equation}
where  
\begin{subequations}\label{new 16}
\begin{align}
R_{O, \text{approx}} \triangleq& \sum_{j} \chi_{j}^{2} (\mathring{R}_{l} - \mathring{R}_{q,j}), \label{new 16a}\\
R_{O, \text{low}} \triangleq& \sum_{j} \chi_{j}^{2} \mathring{R}_{q,j} + \sum_{j,k} \mathring{\mathcal{L}}_{j,k}, \label{new 16b}\\
R_{O, \text{high}} \triangleq& \mathcal{B} \tilde{P}_{\approx \lambda_{q+1}} \big[ \sum_{j, j', k, k': k + k' \neq 0} ( \Lambda \mathbb{P}_{q+1, k} \tilde{w}_{q+1, j, k}) \cdot \nabla \mathbb{P}_{q+1, k'} \tilde{w}_{q+1, j', k'} \nonumber\\
& \hspace{30mm} - (\nabla \mathbb{P}_{q+1, k} \tilde{w}_{q+1, j, k})^{T} \cdot \Lambda \mathbb{P}_{q+1, k'} \tilde{w}_{q+1, j', k'} \big];\label{new 16c}
\end{align}
\end{subequations}
here, $\mathring{\mathcal{L}}_{j,k}$ in \eqref{new 16b} is the trace-free part of $\mathcal{L}_{j,k}$ defined via \eqref{est 85d} (see \cite[Equations (170), (173), and (174)]{Y23a}). Because $D_{t,q} (\mathring{R}_{l} - \mathring{R}_{q,j}) = D_{t,q} \mathring{R}_{l}$ and $(\mathring{R}_{l} - \mathring{R}_{q,j}) (\tau_{q+1} j) = 0$ due to \eqref{est 16}, \eqref{est 2a} gives us for all $t \in \supp \chi_{j}$, 
\begin{align}
&\lVert ( \mathring{R}_{l} - \mathring{R}_{q,j} )(t) \rVert_{C_{x}} \overset{ \eqref{inductive 2b}}{\lesssim} \tau_{q+1} \left[ l^{-1} \lambda_{q+2} \delta_{q+2} \bar{e} + \left(\lambda_{q} \delta_{q}^{\frac{1}{2}} \bar{e}^{\frac{1}{2}} + \lVert z_{q} \rVert_{C_{t,q+1} H_{x}^{2+ \sigma}} \right) \lambda_{q} \lambda_{q+2} \delta_{q+2} \bar{e} \right]  \nonumber \\
& \hspace{15mm} \overset{\eqref{additive stopping time}}{\lesssim} \tau_{q+1} \left[ l^{-1} \lambda_{q+2} \delta_{q+2} \bar{e} + \lambda_{q}^{2} \lambda_{q+2} \delta_{q}^{\frac{1}{2}} \delta_{q+2} \bar{e}^{\frac{3}{2}} \right] \lesssim \tau_{q+1} l^{-1} \lambda_{q+2} \delta_{q+2} \bar{e}. \label{new 17}
\end{align}
Therefore, we obtain the following estimate of $R_{O, \text{approx}}$ from \eqref{new 16a}:  
\begin{equation}\label{additive Oapprox}
\lVert R_{O,\text{approx}} \rVert_{C_{t,x,q+1}}  \leq \sum_{j} \lVert 1_{\supp \chi_{j}} (\mathring{R}_{l} - \mathring{R}_{q,j} ) \rVert_{C_{t,x,q+1}} \lesssim \tau_{q+1} l^{-1} \lambda_{q+2} \delta_{q+2} \bar{e}. 
\end{equation}
Now, we can rewrite $\mathcal{L}_{j,k}^{ml}$ from \eqref{est 85d}  as 
\begin{equation}\label{est 86}
\mathcal{L}_{j,k}^{ml} (t,x) = \frac{1}{2} \chi_{j}^{2} (t) \lambda_{q+1} (k^{\bot} \otimes k^{\bot} - \Id)^{ml} a_{k,j}^{2} (t,x) + \tilde{\mathcal{L}}_{j,k}^{ml} (t,x) 
\end{equation}
where 
\begin{equation*}
\tilde{\mathcal{L}}_{j,k}^{ml}(t,x) \triangleq (\tilde{\mathcal{L}}_{j,k}^{(1)})^{ml} (t,x) + (\tilde{\mathcal{L}}_{j,k}^{(2)})^{ml} (t,x) 
\end{equation*}
and
\begin{subequations}\label{est 267}
\begin{align}
& ( \tilde{\mathcal{L}}_{j,k}^{(1)})^{ml} (t,x) \triangleq \chi_{j}^{2}(t) \frac{1}{2} \int_{0}^{1} \int_{0}^{1} \int_{\mathbb{R}^{2} \times \mathbb{R}^{2}} ( \mathcal{K}_{k,r,\bar{r}}^{(1)})^{ml} (x-z, z - \tilde{z})\cdot \nabla (a_{k,j} \psi_{q+1, j,k}) (z)  \nonumber \\
& \hspace{22mm}  \times (a_{k,j} \psi_{q+1, j, -k}) (\tilde{z}) dz d \tilde{z} dr d \bar{r}, \label{est 267a}\\
& (\tilde{\mathcal{L}}_{j,k}^{(2)})^{ml} (t,x) \triangleq \chi_{j}^{2}(t) \frac{1}{2} \int_{0}^{1} \int_{0}^{1}  \int_{\mathbb{R}^{2} \times \mathbb{R}^{2}} (\mathcal{K}_{k,r,\bar{r}}^{(2)})^{ml} (x-z, x- \tilde{z}) \cdot (a_{k,j} \psi_{q+1,j,k}) (z)  \nonumber \\
& \hspace{22mm} \times \nabla (a_{k,j} \psi_{q+1,j,-k}) (\tilde{z}) dz d \tilde{z} dr d \bar{r} \label{est 267b}
\end{align}
\end{subequations}
(see \cite[Equations (204), (206), and (208)-(209)]{Y23a}). The specific definition of $( \mathcal{K}_{k,r, \bar{r}}^{(j)})^{ml}, j \in \{1,2\}$, can be found in \cite[Equation (207)]{Y23a}; here, we only need to rely on the result that for all $j \in \{1,2 \}, (z, \tilde{z}) \in \mathbb{R}^{4}$, and $0 \leq \lvert a \rvert, \lvert b \rvert \leq 1$, 
\begin{equation}\label{new 19}
\lVert (z, \tilde{z})^{a} \nabla_{(z, \tilde{z})}^{b} ( \mathcal{K}_{k,r,\bar{r}}^{(j)})^{ml} \rVert_{L_{z, \tilde{z}}^{1} (\mathbb{R}^{2} \times \mathbb{R}^{2})} \leq C_{a,b} \left( \frac{ \lambda_{q+1}}{\bar{r}} \right)^{\lvert b \rvert - \lvert a \rvert},
\end{equation}
uniformly in $r \in (0,1)$ from \cite[Equation (210)]{Y23a}. Thus, defining 
\begin{subequations}\label{est 87}
\begin{align}
& O_{1} \triangleq \sum_{j} \chi_{j}^{2} \mathring{R}_{q,j} + \frac{\lambda_{q+1}}{2} \sum_{j,k} ( k^{\bot} \otimes k^{\bot} - \Id) \chi_{j}^{2} a_{k,j}^{2}, \label{est 87a}\\
& O_{2} \triangleq O_{21} + O_{22} \hspace{1mm} \text{ where } \hspace{1mm} O_{21} \triangleq \sum_{j,k} \tilde{\mathcal{L}}_{j,k}^{(1)} \hspace{1mm} \text{ and } \hspace{1mm}  O_{22} \triangleq \sum_{j,k} \tilde{\mathcal{L}}_{j,k}^{(2)}, \label{est 87b}
\end{align}
\end{subequations}
gives us 
\begin{equation}\label{new 18}
R_{O, \text{low}} = \mathring{O}_{1} + \mathring{O}_{2} 
\end{equation}
(see \cite[Equations (211)-(212)]{Y23a}). Now we make the important observation that due to Lemma \ref{geometric lemma} that  
\begin{align}\label{est 31}
O_{1} \overset{\eqref{est 87a} \eqref{est 3}}{=}& \sum_{j} \chi_{j}^{2} \lambda_{q+1} \left[ \frac{ \mathring{R}_{q,j}}{\lambda_{q+1}} + \frac{1}{2} \sum_{k} (k^{\bot} \otimes k^{\bot}) \lambda_{q+1}^{-1} \rho_{j} \gamma_{k}^{2} \left( \Id - \frac{\mathring{R}_{q,j}}{\rho_{j}} \right) - \frac{1}{2} \sum_{k} \Id a_{k,j}^{2} \right] \nonumber \\
=& \sum_{j}\chi_{j}^{2} \lambda_{q+1} \left[ \frac{ \mathring{R}_{q,j}}{\lambda_{q+1}} + \frac{\rho_{j}}{\lambda_{q+1}} \left( \Id - \frac{\mathring{R}_{q,j}}{\rho_{j}} \right) - \frac{1}{2} \sum_{k} \Id a_{k,j}^{2} \right] \nonumber \\
=& \sum_{j} \chi_{j}^{2} \lambda_{q+1} \Id \left[ \frac{\rho_{j}}{\lambda_{q+1}} - \frac{1}{2} \sum_{k} a_{k,j}^{2} \right], 
\end{align}
so that we can simplify \eqref{new 18}  to 
\begin{equation}\label{est 89}
R_{O, \text{low}} \overset{\eqref{new 18}}{=} \mathring{O}_{1} + \mathring{O}_{2} = \mathring{O}_{2}  \overset{\eqref{est 87b}}{=} \mathring{O}_{21} + \mathring{O}_{22}.
\end{equation}
Next, we estimate from \eqref{est 87b} 
\begin{align}
& \lVert O_{21} \rVert_{C_{t,x,q+1}} \lesssim \sup_{s \in [t_{q+1}, \mathfrak{t}]}  \sum_{j,k}  \chi_{j}^{2} (s)\lVert \nabla (a_{k,j} \psi_{q+1, j, k}) \rVert_{C_{t,x,q+1}} \lVert a_{k,j} \psi_{q+1, j, -k} \rVert_{C_{t,x,q+1} } \nonumber \\
& \hspace{5mm}  \times \sup_{r, \bar{r} \in [0,1]} \lVert \mathcal{K}_{k,r,\bar{r}}^{(1)} \rVert_{L_{z, \bar{z}}^{1} (\mathbb{R}^{2} \times \mathbb{R}^{2})}  \overset{\eqref{est 10} \eqref{psi 2} \eqref{new 19} }{\lesssim} (\lambda_{q} \delta_{q+1}^{\frac{1}{2}} \bar{e}^{\frac{1}{2}}) (\delta_{q+1}^{\frac{1}{2}} \bar{e}^{\frac{1}{2}}) \approx \lambda_{q} \delta_{q+1} \bar{e}. \label{new 20}
\end{align}
An identical bound holds for $O_{22}$; thus, 
\begin{equation}\label{additive Olow}
\lVert R_{O,\text{low}} \rVert_{C_{t,x,q+1}} \overset{\eqref{est 89}}{\leq} \sum_{l=1}^{2} \lVert O_{2l} \rVert_{C_{t,x,q+1}} \overset{\eqref{new 20}}{\lesssim} \lambda_{q} \delta_{q+1} \bar{e}. 
\end{equation}
Next, similarly to\cite[Equations (219)-(220)]{Y23a} we may define
\begin{align*}
&  O_{3} \triangleq \mathcal{B} \tilde{P}_{\approx \lambda_{q+1}} \left( \sum_{j, j', k, k': k + k' \neq 0} (\Lambda \mathbb{P}_{q+1, k} \tilde{w}_{q+1, j, k}) \cdot \nabla \mathbb{P}_{q+1, k'} \tilde{w}_{q+1, j', k'} \right), \\
& O_{4} \triangleq \mathcal{B} \tilde{P}_{\approx \lambda_{q+1}} \left( \sum_{j, j', k, k': k + k' \neq 0}  (\nabla \mathbb{P}_{q+1, k} \tilde{w}_{q+1, j, k})^{T} \cdot \Lambda \mathbb{P}_{q+1, k'} \tilde{w}_{q+1, j', k'} \right) 
\end{align*}
which allow us to write 
\begin{equation}\label{new 30}
R_{O,\text{high}} = O_{3} - O_{4}.
\end{equation}
To estimate $O_{3}$ first, we proceed to define 
\begin{subequations}\label{new 22} 
\begin{align} 
& O_{311}(x) \triangleq  \mathcal{B} \tilde{P}_{\approx \lambda_{q+1}} \frac{\lambda_{q+1}}{2} \sum_{j, j', k, k': k + k' \neq 0} \chi_{j} \chi_{j'} \nabla (a_{k,j} (x) a_{k', j'} (x) \psi_{q+1, j', k'}(x) \psi_{q+1, j, k}(x))  \nonumber \\
& \hspace{35mm} \times (b_{k}(\lambda_{q+1} x) \cdot b_{k'} (\lambda_{q+1} x) - e^{i(k+k') \cdot \lambda_{q+1} x}), \label{new 22a} \\
& O_{312}(x) \triangleq  \mathcal{B} \tilde{P}_{\approx \lambda_{q+1}} \lambda_{q+1} \sum_{j,j',k,k': k + k' \neq 0} \chi_{j} \chi_{j'}  \nonumber\\
& \hspace{10mm} \times b_{k'}(\lambda_{q+1} x) \otimes b_{k} (\lambda_{q+1} x) \nabla (a_{k,j} (x) a_{k', j'} (x) \psi_{q+1, j', k'} (x) \psi_{q+1, j, k} (x)), \label{new 22b}\\
& O_{32}(x) \triangleq \mathcal{B} \tilde{P}_{\approx \lambda_{q+1}} \lambda_{q+1} \sum_{j, j',k, k': k + k'\neq 0}  \nonumber\\
& \hspace{15mm} \times \divergence ( \tilde{w}_{q+1, j, k}(x) \otimes \chi_{j'} [\mathbb{P}_{q+1, k'}, a_{k', j'}(x) \psi_{q+1, j', k'}(x)] b_{k'} (\lambda_{q+1} x) ), \label{new 22c}\\
& O_{33}(x) \triangleq \mathcal{B} \tilde{P}_{\approx \lambda_{q+1}}  \sum_{j, j', k, k': k + k'\neq 0} \nonumber\\
& \hspace{8mm} \times \divergence (\chi_{j} [\mathbb{P}_{q+1,k} \Lambda, a_{k, j}(x) \psi_{q+1, j, k}(x)] b_{k} (\lambda_{q+1} x) \otimes \mathbb{P}_{q+1, k'} \tilde{w}_{q+1, j', k'}(x)) \label{new 22d}
\end{align} 
\end{subequations}
similarly to \cite[Equations (221)-(227)]{Y23a} and split
\begin{equation}\label{new 21}
O_{3} = \sum_{k=1}^{3} O_{3k}, \text{ with } O_{31} \triangleq O_{311} + O_{312}.
\end{equation}
We first estimate from \eqref{new 22a}, \eqref{new 22b} by using \eqref{psi 2}, and \eqref{est 10}, 
\begin{align}
&\lVert O_{31} \rVert_{C_{t,x,q+1}} \label{new 23} \\
\overset{\eqref{new 21}}{\lesssim}& \sum_{j,j', k, k': k + k' \neq 0} \lVert 1_{\supp \chi_{j}} (t) \nabla (a_{k,j} \psi_{q+1, j, k} ) \rVert_{C_{t,x,q+1}} \lVert 1_{\supp \chi_{j'}}(t) a_{k', j'} \psi_{q+1, j', k'} \rVert_{C_{t,x,q+1}}  \nonumber \\
& + \lVert 1_{\supp \chi_{j'}} (t) \nabla (a_{k', j'} \psi_{q+1, j', k'} ) \rVert_{C_{t,x,q+1}} \lVert 1_{\supp \chi_{j}} (t) a_{k,j} \psi_{q+1, j,k} \rVert_{C_{t,x,q+}} \lesssim  \lambda_{q} \delta_{q+1} \bar{e}.  \nonumber 
\end{align}
Next, we estimate from \eqref{new 22c}-\eqref{new 22d} by relying on \cite[Equation (A.17)]{BSV19} 
\begin{align}
& \lVert O_{32} \rVert_{C_{t,x,q+1}} + \lVert O_{33} \rVert_{C_{t,x,q+1}} \label{new 24} \\
\lesssim&  \sum_{j, j', k, k': k + k' \neq 0}  \lVert 1_{\supp \chi_{j}} a_{k,j}  \rVert_{C_{t,x,q+1}}   \lVert 1_{\supp \chi_{j'}} \nabla (a_{k', j'} \psi_{q+1, j', k'} ) \rVert_{C_{t,x,q+1}}   \nonumber \\
&+ \lVert 1_{\supp \chi_{j}} \nabla (a_{k,j} \psi_{q+1, j, k} ) \rVert_{C_{t,x,q+1}}     \lVert 1_{\supp \chi_{j}'}  a_{k',j'} b_{k'} (\lambda_{q+1} \Phi_{j'} ) \rVert_{C_{t,x,q+1}} \overset{\eqref{psi} \eqref{est 10} }{\lesssim}  \lambda_{q} \delta_{q+1} \bar{e}.  \nonumber 
\end{align}
Therefore, we conclude that 
\begin{equation}\label{new 31}
\lVert O_{3} \rVert_{C_{t,x,q+1}}  \overset{\eqref{new 21}}{\leq} \sum_{k=1}^{3} \lVert O_{3k} \rVert_{C_{t,x,q+1}} \overset{\eqref{new 23} \eqref{new 24}}{\lesssim} \lambda_{q} \delta_{q+1} \bar{e}. 
\end{equation}
Similarly to \cite[Equations (231)-(232)]{Y23a}, we can define 
\begin{subequations}\label{new 26}
\begin{align}
&O_{41}(x) \triangleq \sum_{j, j', k, k': k + k' \neq 0}  \mathcal{B} \tilde{P}_{\approx \lambda_{q+1}}  \nonumber \\
&  \times (\chi_{j} \nabla ( [ \mathbb{P}_{q+1, k}, a_{k,j} (x) \psi_{q+1, j, k} (x) ] b_{k} (\lambda_{q+1} x) )^{T} \cdot \lambda_{q+1} \tilde{w}_{q+1, j', k'}(x)), \label{new 26a} \\
&O_{42}(x) \triangleq \sum_{j, j', k, k': k + k' \neq 0}  \mathcal{B} \tilde{P}_{\approx \lambda_{q+1}} \nonumber \\
&\times((\nabla \mathbb{P}_{q+1, k} \tilde{w}_{q+1, j, k} (x))^{T} \cdot \chi_{j'} [\mathbb{P}_{q+1, k'} \Lambda, a_{k', j'}(x) \psi_{q+1, j', k'} (x)] b_{k'} (\lambda_{q+1} x) ) \label{new 26b}
\end{align}
\end{subequations} 
to write 
\begin{equation}\label{new 28}
O_{4} = O_{41} + O_{42}.     
\end{equation}
To estimate $O_{41}$, we rewrite similarly to \cite[Equation (234)]{Y23a}
\begin{align}
& \nabla ( [ \mathbb{P}_{q+1, k}, a_{k,j} (x) \psi_{q+1, j, k} (x) ] b_{k} (\lambda_{q+1} x) ) \label{new 25} \\
=& [ \mathbb{P}_{q+1, k}, \nabla (a_{k,j} \psi_{q+1, j, k})(x) ] b_{k} (\lambda_{q+1} x) + [\mathbb{P}_{q+1, k}, a_{k,j}(x) \psi_{q+1, j, k}(x) ] \nabla b_{k} (\lambda_{q+1} x)  \nonumber 
\end{align}
and compute from \eqref{new 26a} using \cite[Equation (A.17)]{BSV19}, 
\begin{align}
& \lVert O_{41} \rVert_{C_{t,x,q+1}} \label{new 27}\\
&\overset{\eqref{new 25} \eqref{est 14b}}{\lesssim} \sum_{j, j', k, k': k+ k' \neq 0} [ \lambda_{q+1}^{-1} \lVert 1_{\supp \chi_{j}} \nabla^{2} (a_{k,j} \psi_{q+1, j, k}) \rVert_{C_{t,x,q+1}} \lVert b_{k} (\lambda_{q+1} x) \rVert_{C_{t,x,q+1}}  \nonumber \\
& \hspace{5mm} + \lambda_{q+1}^{-1} \lVert 1_{\supp \chi_{j}} \nabla (a_{k,j} \psi_{q+1, j,k}) \rVert_{C_{t,x,q+1}} \lVert \nabla b_{k} (\lambda_{q+1} x) \rVert_{C_{t,x,q+1}} ] \lVert a_{k', j'} \rVert_{C_{t,x,q+1}}  \nonumber \\
& \overset{\eqref{psi 2} \eqref{est 10} }{\lesssim} \sum_{j, j', k, k': k + k' \neq 0} [ \lambda_{q+1}^{-1} (\lambda_{q}^{2} \delta_{q+1}^{\frac{1}{2}} \bar{e}^{\frac{1}{2}} ) + \lambda_{q+1}^{-1} (\lambda_{q} \delta_{q+1}^{\frac{1}{2}} \bar{e}^{\frac{1}{2}} ) \lambda_{q+1} ] \delta_{q+1}^{\frac{1}{2}}\bar{e}^{\frac{1}{2}} \lesssim  \lambda_{q} \delta_{q+1} \bar{e}.  \nonumber 
\end{align}
We can estimate $O_{42}$ from \eqref{new 26b} via \cite[Equation (A.17)]{BSV19} 
\begin{align}
& \lVert O_{42} \rVert_{C_{t,x,q+1}} \lesssim \lambda_{q+1}^{-1} \sum_{j, j', k, k': k + k' \neq 0} \lambda_{q+1} \lVert 1_{\supp \chi_{j}}  \tilde{w}_{q+1, j, k} \rVert_{C_{t,x,q+1}}  \nonumber \\
& \hspace{20mm} \times \lVert 1_{\supp\chi_{j'}}  \nabla (a_{k', j'} \psi_{q+1, j', k'}) \rVert_{C_{t,x,q+1}} \lVert b_{k'} (\lambda_{q+1} x) \rVert_{C_{t,x,q+1}} \nonumber \\
\overset{\eqref{est 14b} \eqref{psi 2} }{\lesssim}& \sum_{j, j', k, k': k + k' \neq 0} \lVert 1_{\supp \chi_{j}}  a_{k,j} \rVert_{C_{t,x,q+1}}  \lambda_{q} \delta_{q+1}^{\frac{1}{2}} \bar{e}^{\frac{1}{2}} \overset{\eqref{est 10}}{\lesssim} \lambda_{q} \delta_{q+1} \bar{e}.  \label{new 29}
\end{align}
Hence, we conclude that 
\begin{equation}\label{new 32}
\lVert O_{4} \rVert_{C_{t,x,q+1}} \overset{\eqref{new 28}}{\leq} \sum_{k=1}^{2} \lVert O_{4k} \rVert_{C_{t,x,q+1}} \overset{\eqref{new 27}\eqref{new 29}}{\lesssim} \lambda_{q}\delta_{q+1} \bar{e}. 
\end{equation}
Therefore, the high-frequency oscillation has bound
\begin{equation}\label{additive Ohigh}
\lVert R_{O,\text{high}} \rVert_{C_{t,x,q+1}} \overset{\eqref{new 30}}{\leq} \lVert O_{3} \rVert_{C_{t,x,q+1}} + \lVert O_{4} \rVert_{C_{t,x,q+1}} \overset{\eqref{new 31} \eqref{new 32}}{\lesssim} \lambda_{q} \delta_{q+1} \bar{e}. 
\end{equation}
Finally, for $\beta > \frac{1}{2}$ sufficiently close to $\frac{1}{2}$ and $a_{0}$ sufficiently large, we arrive at 
\begin{align*}
\lVert R_{O} \rVert_{C_{t,x,q+1}} \overset{\eqref{new 33}}{\leq}& \lVert R_{O,\text{approx}} \rVert_{C_{t,x,q+1}} + \lVert R_{O,\text{low}} \rVert_{C_{t,x,q+1}} + \lVert R_{O,\text{high}} \rVert_{C_{t,x,q+1}}  \nonumber \\
\overset{\eqref{additive Oapprox} \eqref{additive Olow} \eqref{additive Ohigh}}{\lesssim}& \bar{e} \left( \lambda_{q+1}^{\frac{\alpha}{2} - \frac{1}{2} - \frac{\beta}{2}} \lambda_{q+2}^{\frac{1}{2} - \beta} \lambda_{1}^{(2\beta -1) (\frac{3}{4})}  + \lambda_{q} \lambda_{1}^{2\beta -1} \lambda_{q+1}^{-2\beta}  \right) \overset{\eqref{bound alpha} \eqref{additive bound b}}{\ll} \lambda_{q+3} \delta_{q+3}.  
\end{align*}
\end{proof}

Finally, we bound the two commutator errors. 
\begin{proposition}\label{prop Rcom1}
$R_{\text{Com1}}$ from \eqref{Rcom1} satisfies for $\beta > \frac{1}{2}$ sufficiently close to $\frac{1}{2}$ and $a_{0}$ sufficiently large 
\begin{equation*}
\lVert R_{\text{Com1}} \rVert_{C_{t,x,q+1}} \ll \lambda_{q+3} \delta_{q+3}. 
\end{equation*}
\end{proposition}

\begin{proof}[Proof of Proposition \ref{prop Rcom1}]
We split 
\begin{equation}\label{est 264}
\lVert R_{\text{Com1}} \rVert_{C_{t,x,q+1}} \leq \sum_{k=1}^{4} \RomanI_{k}
\end{equation} 
where 
\begin{subequations}\label{est 182} 
\begin{align}
\RomanI_{1} \triangleq& \lVert \mathcal{B} [( \Lambda y_{l})^{\bot} \nabla^{\bot} \cdot y_{l} - [ (\Lambda y_{q})^{\bot} \nabla^{\bot} \cdot y_{q} ] \ast_{x} \phi_{l} \ast_{t} \varphi_{l} ] \rVert_{C_{t,x,q+1}}, \label{est 182a}\\
\RomanI_{2} \triangleq& \lVert \mathcal{B} [( \Lambda y_{l})^{\bot} \nabla^{\bot} \cdot z_{l} - [ (\Lambda y_{q})^{\bot} \nabla^{\bot} \cdot z_{q} ] \ast_{x} \phi_{l} \ast_{t} \varphi_{l} ] \rVert_{C_{t,x,q+1}}, \label{est 182b}\\
\RomanI_{3} \triangleq& \lVert \mathcal{B} [( \Lambda z_{l})^{\bot} \nabla^{\bot} \cdot y_{l} - [ (\Lambda z_{q})^{\bot} \nabla^{\bot} \cdot y_{q} ] \ast_{x} \phi_{l} \ast_{t} \varphi_{l} ] \rVert_{C_{t,x,q+1}}, \label{est 182c}\\
\RomanI_{4} \triangleq& \lVert \mathcal{B} [( \Lambda z_{l})^{\bot} \nabla^{\bot} \cdot z_{l} - [ (\Lambda z_{q})^{\bot} \nabla^{\bot} \cdot z_{q} ] \ast_{x} \phi_{l} \ast_{t} \varphi_{l} ] \rVert_{C_{t,x,q+1}}, \label{est 182d}
\end{align}
\end{subequations} 
similarly to  \cite[Equations (240)-(241)]{Y23a}. We use \eqref{est 8} and the standard mollifier estimates (e.g. \cite[Lemma 1]{CDS12b}) to estimate separately as follows:  for $\delta \in (0, \frac{1}{4})$ 
\begin{subequations}\label{est 249}
\begin{align} 
& \RomanI_{1} \overset{\eqref{inductive 1a}}{\lesssim} l^{2} (\lambda_{q}^{2} \lVert y_{q} \rVert_{C_{t,q+1}^{1} C_{x}}^{2} + \lambda_{q}^{2}  \lVert \Lambda y_{q} \rVert_{C_{t,x,q+1}}^{2}) \overset{\eqref{temp} \eqref{inductive 1c}}{\lesssim} l^{2} \lambda_{q}^{6} \delta_{q}^{2}, \\
& \RomanI_{2} \lesssim l^{1- 4 \delta} \lVert \Lambda y_{q} \rVert_{C_{t,x,q+1}}^{\frac{1}{2} + 2 \delta} \lVert \Lambda y_{q} \rVert_{C_{t,q+1}^{1} C_{x}}^{\frac{1}{2} - 2 \delta} \lambda_{q}^{\frac{1}{4}}  \lVert \nabla^{\bot} z_{q} \rVert_{C_{t,q+1}^{\frac{1}{2} - 2 \delta} H_{x}^{\frac{3}{4}+ 4\sigma}}  + l^{2} (\lambda_{q} \delta_{q}^{\frac{1}{2}} ) \lambda_{q}\lVert z_{q} \rVert_{C_{t,q+1} H_{x}^{2+ \sigma}}  \nonumber \\
& \hspace{25mm} \overset{\eqref{additive stopping time} \eqref{inductive 1c} \eqref{temp} \eqref{inductive 1a}}{\lesssim}  l^{1- 4 \delta} \lambda_{q}^{\frac{9}{4}-4\delta} \delta_{q}^{\frac{3}{4} - \delta} + l^{2} \lambda_{q}^{2} \delta_{q}^{\frac{1}{2}}, \\
&\RomanI_{3} \overset{\eqref{inductive 1}}{\lesssim}  l^{1- 4 \delta} \lambda_{q}^{\frac{1}{4}} \lVert z_{q} \rVert_{C_{t,q+1}^{\frac{1}{2} - 2 \delta} \dot{H}_{x}^{\frac{7}{4}+\sigma}}  \lVert \nabla^{\bot} y_{q} \rVert_{C_{t,q+1} C_{x}}^{\frac{1}{2} + 2 \delta} \lVert \nabla^{\bot} \cdot y_{q} \rVert_{C_{t,q+1}^{1} C_{x}}^{\frac{1}{2} - 2 \delta} + l^{2} \lVert z_{q} \rVert_{C_{t,q+1} H_{x}^{2+ \sigma}}  \lambda_{q}^{2} \delta_{q}^{\frac{1}{2}} \bar{e}^{\frac{1}{2}}  \nonumber\\
& \hspace{25mm} \overset{\eqref{additive stopping time} \eqref{inductive 1c} \eqref{temp} \eqref{inductive 1a}}{\lesssim}  l^{1- 4 \delta} \lambda_{q}^{\frac{9}{4}- 4 \delta} \delta_{q}^{\frac{3}{4} - \delta} + l^{2} \lambda_{q}^{2} \delta_{q}^{\frac{1}{2}},  \\
& \RomanI_{4} \lesssim l^{1- 4\delta} (\lambda_{q}^{\frac{1}{4}} \lVert z_{q} \rVert_{C_{t,q+1}^{\frac{1}{2} - 2 \delta} \dot{H}_{x}^{\frac{7}{4} + \sigma}}) + l^{2}  (\lambda_{q}^{\frac{1}{2}} \lVert z_{q} \rVert_{C_{t,q+1} \dot{H}_{x}^{\frac{5}{2} + 2 \sigma}})^{2}  \overset{\eqref{additive stopping time}}{\lesssim} l^{1- 4 \delta} \lambda_{q}^{\frac{1}{2}}.
\end{align} 
\end{subequations} 
Therefore, we conclude by applying \eqref{est 249} to \eqref{est 264} that for $\delta \in (0, \frac{1}{4})$ sufficiently small, $\beta > \frac{1}{2}$ sufficiently close to $\frac{1}{2}$, and $a_{0}$ sufficiently large, 
\begin{align*}
\lVert R_{\text{Com1}} \rVert_{C_{t,x,q+1}} \overset{\eqref{est 264}}{\leq}& \sum_{k=1}^{4} \RomanI_{k} \lesssim l^{2} \lambda_{q}^{6} \delta_{q}^{2} + l^{1- 4 \delta} \lambda_{q}^{\frac{9}{4}-4\delta} \delta_{q}^{\frac{3}{4} - \delta} + l^{2} \lambda_{q}^{2} \delta_{q}^{\frac{1}{2}} + l^{1-4 \delta} \lambda_{q}^{\frac{1}{2}} \nonumber\\
\lesssim&  l^{2} \lambda_{q}^{6} \delta_{q}^{2} + l^{1- 4 \delta} \lambda_{q}^{\frac{9}{4}-4\delta} \delta_{q}^{\frac{3}{4} - \delta} \ll \lambda_{q+3} \delta_{q+3}.
\end{align*}
\end{proof}

\begin{proposition}\label{prop Rcom2}
$R_{\text{Com2}}$ from \eqref{est 22e} satisfies for $\beta > \frac{1}{2}$ sufficiently close to $\frac{1}{2}$ and $a_{0}$ sufficiently large 
\begin{equation*}
\lVert R_{\text{Com2}} \rVert_{C_{t,x,q+1}} \ll \lambda_{q+3} \delta_{q+3}. 
\end{equation*}
\end{proposition}

\begin{proof}[Proof of Proposition \ref{prop Rcom2}]
For convenience of estimates we compose 
\begin{equation}\label{new 38}
R_{\text{Com2}} = \sum_{k=1}^{13} R_{\text{Com2,k}}
\end{equation} 
where 
\begin{subequations}\label{new 34} 
\begin{align} 
R_{\text{Com2,1}} \triangleq& \mathcal{B} ( \Lambda y_{q+1} \cdot \nabla (z_{q+1} - z_{q}) + \Lambda (z_{q+1} - z_{q}) \cdot \nabla y_{q+1} ),  \label{new 34a} \\
R_{\text{Com2,2}} \triangleq& \mathcal{B} ( (\nabla y_{l})^{T} \cdot \Lambda (z_{q} - z_{q+1}) ),   \label{new 34b} \\
R_{\text{Com2,3}} \triangleq& \mathcal{B} ( (\nabla (z_{q} - z_{q+1}))^{T} \cdot \Lambda y_{l} ),   \label{new 34c}\\
R_{\text{Com2,4}} \triangleq& \mathcal{B} ( (\nabla \Lambda (z_{q+1} - z_{q}))^{T} \cdot w_{q+1} ),  \label{new 34d}\\
R_{\text{Com2,5}} \triangleq& -\mathcal{B} ( \Lambda w_{q+1}^{\bot} \nabla^{\bot} \cdot (z_{l} - z_{q}) ),  \label{new 34e}\\
R_{\text{Com2,6}} \triangleq& \mathcal{B} ( (\nabla (z_{q} - z_{q+1}))^{T} \cdot \Lambda w_{q+1}),   \label{new 34f}\\
R_{\text{Com2,7}} \triangleq& \mathcal{B} ( \Lambda y_{l}^{\bot} \nabla^{\bot} \cdot (z_{q} - z_{l})),   \label{new 34g}\\
R_{\text{Com2,8}} \triangleq& \mathcal{B} (  \Lambda (z_{q} - z_{l}) \cdot \nabla w_{q+1}),  \label{new 34h}\\
R_{\text{Com2,9}} \triangleq& \mathcal{B} ( \Lambda (z_{q} - z_{l})^{\bot} \nabla^{\bot} \cdot y_{l}),   \label{new 34i}\\
R_{\text{Com2,10}} \triangleq& \mathcal{B} ( \Lambda (z_{q+1} - z_{q})^{\bot} \nabla^{\bot} \cdot z_{q+1}),   \label{new 34j}\\
R_{\text{Com2,11}} \triangleq& \mathcal{B} ( \Lambda (z_{q} - z_{l})^{\bot}\nabla^{\bot} \cdot z_{q+1} ),  \label{new 34k}\\
R_{\text{Com2,12}} \triangleq& \mathcal{B} ( \Lambda z_{l}^{\bot} \nabla^{\bot} \cdot (z_{q+1} - z_{q}) ),   \label{new 34l}\\
R_{\text{Com2,13}} \triangleq& \mathcal{B} ( \Lambda z_{l}^{\bot} \nabla^{\bot} \cdot (z_{q} - z_{l})).  \label{new 34m}
\end{align}
\end{subequations} 
We further split $R_{\text{Com2,1}}$ using \eqref{est 44} as 
\begin{equation*}
R_{\text{Com2,1}} = \sum_{k=1}^{3} R_{\text{Com2,1,k}}
\end{equation*} 
where 
\begin{subequations}\label{new 35} 
\begin{align}
& R_{\text{Com2,1,1}} \triangleq \mathcal{B} \left( \Lambda w_{q+1} \cdot \nabla ( z_{q+1} - z_{q}) \right), \label{new 35a}\\
& R_{\text{Com2,1,2}} \triangleq \mathcal{B} \divergence \left((z_{q+1} - z_{q}) \otimes \Lambda y_{l} \right), \label{new 35b}\\
& R_{\text{Com2,1,3}} \triangleq \mathcal{B} \divergence \left( y_{q+1} \otimes \Lambda (z_{q+1} - z_{q})  \right).  \label{new 35c} 
\end{align}
\end{subequations} 
We will frequently make use of $\supp (z_{q+1} - z_{q})\hspace{0.5mm}\hat{}\hspace{1mm}  \subset \{ \xi: \frac{\lambda_{q}}{4} \leq \lvert \xi \rvert \leq \frac{\lambda_{q+1}}{4} \}$ due to \eqref{define f} and \eqref{define zq}. E.g., this, together with \eqref{support wqplus1}, implies 
\begin{equation}\label{new1}
\supp \left( \Lambda w_{q+1} \cdot \nabla (z_{q+1} -z_{q}) \right)\hspace{0.5mm}\hat{}\hspace{1mm} \subset \left\{ \xi : \frac{\lambda_{q+1}}{4} \leq \lvert \xi \rvert \leq 2 \lambda_{q+1}\right\}; 
\end{equation}
we note that this is one part that required $f(q) <\frac{\lambda_{q}}{2}$. With this in mind, we estimate for $\beta > \frac{1}{2}$ sufficiently close to $\frac{1}{2}$ and $a_{0}$ sufficiently large 
\begin{subequations}
\begin{align}
& \lVert R_{\text{Com2,1,1}} \rVert_{C_{t,x,q+1}} + \lVert R_{\text{Com2,6}} \rVert_{C_{t,x,q+1}} \overset{\eqref{new1}}{\lesssim}  \lambda_{q+1}^{-1} \lVert \Lambda w_{q+1} \rVert_{C_{t,x,q+1}} \lVert \nabla ( z_{q+1} - z_{q}) \rVert_{C_{t,x,q+1}}  \nonumber \\
& \hspace{45mm}  \overset{ \eqref{support wqplus1} \eqref{est 11a}}{\lesssim}  \delta_{q+1}^{\frac{1}{2}} \lVert z \rVert_{C_{t,q+1} H_{x}^{2+ \sigma}} \overset{\eqref{additive stopping time}}{\ll}  \lambda_{q+3} \delta_{q+3}, \\
& \lVert R_{\text{Com2,1,2}} \rVert_{C_{t,x,q+1}} \lesssim \lVert \Lambda y_{l} \rVert_{C_{t,x,q+1}} \lVert z_{q+1} - z_{q} \rVert_{C_{t,x,q+1}}  \overset{\eqref{inductive 1c} \eqref{additive stopping time}}{\lesssim} \delta_{q}^{\frac{1}{2}}  \ll \lambda_{q+3} \delta_{q+3}, 
\\
& \lVert R_{\text{Com2,1,3}} \rVert_{C_{t,x,q+1}}  \lesssim \lVert \Lambda (z_{q+1} - z_{q}) \rVert_{C_{t,x,q+1}} \lVert y_{q+1} \rVert_{C_{t,x,q+1}}  \overset{\eqref{additive stopping time}}{\lesssim} \lambda_{q}^{-\sigma} \ll \lambda_{q+3} \delta_{q+3}, \\
& \lVert R_{\text{Com2,2}} \rVert_{C_{t,x,q+1}} + \lVert R_{\text{Com2,9}}  \rVert_{C_{t,x,q+1}} \nonumber \\
&\overset{\eqref{inductive 1c} \eqref{define f}}{\lesssim} (\lambda_{q} \delta_{q}^{\frac{1}{2}}) [ \lambda_{q}^{-\frac{1}{2} - \sigma} \lVert z_{q} - z_{q+1} \rVert_{C_{t,q+1} C_{x}^{\frac{3}{2} + \sigma}} +  l^{\frac{1}{2}} \lVert z_{q} \rVert_{C_{t,q+1} C_{x}^{\frac{3}{2}}} \nonumber \\
& \hspace{35mm}  + l^{\frac{1}{2} - 2 \delta} \lambda_{q} \lVert z_{q} \rVert_{C_{t,q+1}^{\frac{1}{2} - 2 \delta} C_{x}}] \overset{\eqref{additive stopping time}}{\lesssim}  \lambda_{q}^{\frac{1}{2} - \sigma} \delta_{q}^{\frac{1}{2}}  \ll \lambda_{q+3} \delta_{q+3}, \label{est 35}  \\ 
&  \lVert R_{\text{Com2,3}} \rVert_{C_{t,x,q+1}} + \lVert R_{\text{Com2,7}} \rVert_{C_{t,x,q+1}}  \nonumber \\
&\overset{\eqref{inductive 1c} \eqref{define f}}{\lesssim} \lambda_{q}^{-\frac{1}{2} - \sigma} \lVert z_{q} - z_{q+1} \rVert_{C_{t,q+1} C_{x}^{\frac{3}{2} + \sigma}} (\lambda_{q} \delta_{q}^{\frac{1}{2}}) + \lambda_{q} \delta_{q}^{\frac{1}{2}} [ l^{\frac{1}{2}} \lVert z_{q} \rVert_{C_{t,q+1} C_{x}^{\frac{3}{2}}}  \nonumber \\
& \hspace{35mm} + l^{\frac{1}{2} - 2 \delta} \lambda_{q} \lVert z_{q} \rVert_{C_{t,q+1}^{\frac{1}{2} - 2 \delta} C_{x}} ] \overset{\eqref{define l} \eqref{additive stopping time}}{\lesssim} \delta_{q}^{\frac{1}{2}} \lambda_{q}^{\frac{1}{2} - \sigma} \ll \lambda_{q+3} \delta_{q+3},  \label{est 34}  \\
& \lVert R_{\text{Com2,4}} \rVert_{C_{t,x,q+1}} \overset{\eqref{new1}}{\lesssim}  \lambda_{q+1}^{-1}\lVert \nabla \Lambda (z_{q+1} - z_{q}) \rVert_{C_{t,x,q+1}} \lVert w_{q+1} \rVert_{C_{t,x,q+1}} \nonumber \\
& \hspace{60mm} \overset{\eqref{est 11a}}{\lesssim}  \lambda_{q+1}^{-\frac{1}{2} - \beta} \lambda_{1}^{\beta - \frac{1}{2}} \ll \lambda_{q+3} \delta_{q+3}, \\ 
& \lVert R_{\text{Com2,5}} \rVert_{C_{t,x,q+1}}  \lesssim \lambda_{q+1}^{-1} \lVert \Lambda w_{q+1} \rVert_{C_{t,x,q+1}} \lVert z \rVert_{C_{t,q+1} \dot{H}_{x}^{2+ \sigma}} \overset{\eqref{est 11a}}{\lesssim}  \delta_{q+1}^{\frac{1}{2}} \ll \lambda_{q+3} \delta_{q+3}, \\  
& \lVert R_{\text{Com2,8}} \rVert_{C_{t,x,q+1}}   \lesssim\lambda_{q+1}^{-1} \lVert \Lambda (z_{q} - z_{l}) \rVert_{C_{t,x,q+1}} \lVert \nabla w_{q+1} \rVert_{C_{t,x,q+1}} \nonumber \\
& \hspace{60mm}  \overset{\eqref{support wqplus1} \eqref{additive stopping time} \eqref{est 11a}}{\lesssim} \delta_{q+1}^{\frac{1}{2}} \ll \lambda_{q+3} \delta_{q+3}, \\
&  \sum_{k=10}^{13} \lVert R_{\text{Com2,k}} \rVert_{C_{t,x,q+1}}  \overset{\eqref{define f} \eqref{additive stopping time}}{\lesssim} \lambda_{q}^{-\sigma} \lVert z \rVert_{C_{t,q+1} \dot{H}_{x}^{2+ 2 \sigma}}  \nonumber \\
&  \hspace{20mm} + l^{\frac{1}{2} - 2 \delta} (\lVert z \rVert_{C_{t,q+1} C_{x}^{\frac{3}{2} - 2 \delta}} + \lambda_{q} \lVert z \rVert_{C_{t,q+1}^{\frac{1}{2} - 2 \delta} C_{x}}) \overset{\eqref{additive stopping time}}{\lesssim}  \lambda_{q}^{-\sigma} \ll \lambda_{q+3} \delta_{q+3}. 
\end{align}
\end{subequations} 
Applying these estimates to \eqref{new 38} completes the proof of Proposition \ref{prop Rcom2}. 
\end{proof}

\subsubsection{Control of the Energy}

The last step before we can conclude the proof of Proposition \ref{additive q to qplus1} is to verify that the energy control hypothesis is satisfied.
\begin{proposition}\label{prop additive control energy}
Define 
\begin{equation}\label{est 268}
\delta E(t) \triangleq \lvert e(t) (1- \lambda_{q+2} \delta_{q+2}) - \lVert (y_{q+1} + z_{q+1} ) (t) \rVert_{\dot{H}_{x}^{\frac{1}{2}}}^{2} \rvert. 
\end{equation}
Then, for all $t \in [t_{q+1}, \mathfrak{t}]$ 
\begin{equation}\label{est 20}
\delta E(t) \leq \frac{1}{4} \lambda_{q+2} \delta_{q+2} e(t) 
\end{equation}
so that \eqref{inductive 3} holds at level $q+1$.
\end{proposition}

\begin{proof}[Proof of Proposition \ref{prop additive control energy}]
We first write using \eqref{est 0b} and \eqref{est 44} 
\begin{align}
\delta E(t) =& \Big\lvert \gamma_{q}(t) 4 (2\pi)^{2}  \label{new 41} \\
& + \int_{\mathbb{T}^{2}} \Lambda^{\frac{1}{2}} (y_{q} + z_{q}) (t) \cdot \Lambda^{\frac{1}{2}} (y_{q} + z_{q}) (t) - \lvert \Lambda^{\frac{1}{2}} (y_{l} + w_{q+1} + z_{q+1} ) (t) \rvert^{2} dx \Big\rvert. \nonumber 
\end{align}
Similarly to \eqref{new1} we make the key observation that 
\begin{equation}\label{new 39} 
\supp \hat{y}_{l} \cap \supp \hat{w}_{q+1} = \emptyset \text{ and } \supp \hat{w}_{q+1} \cap \supp\hat{z}_{q+1} =\emptyset 
\end{equation}
due to \eqref{inductive 1a}, \eqref{support wqplus1}, and \eqref{define f}; here, we used the fact that our choice of $f$ in \eqref{define f} permits $f(q+1) = \frac{\lambda_{q+1}}{4} < \frac{\lambda_{q+1}}{2}$ where $\frac{\lambda_{q+1}}{2}$ is the lower bound of the frequency support of $w_{q+1}$, and this is what led to $\supp \hat{w}_{q+1} \cap \supp\hat{z}_{q+1} =\emptyset$. Consequently, we can simplify 
\begin{align}
& \int_{\mathbb{T}^{2}} \lvert \Lambda^{\frac{1}{2}} (y_{l} + w_{q+1} + z_{q+1}) \rvert^{2} dx  \nonumber \\
=& \int_{\mathbb{T}^2} \lvert \Lambda^{\frac{1}{2}} y_{l} \rvert^{2} + \lvert \Lambda^{\frac{1}{2}} w_{q+1} \rvert^{2} + \lvert \Lambda^{\frac{1}{2}} z_{q+1} \rvert^{2} dx +  2 \sum_{m \in \mathbb{Z}^{2} \setminus \{0\}} \lvert m \rvert [ \hat{y}_{l} \cdot \hat{z}_{q+1}](m). \label{new 42} 
\end{align}
The next two identities \eqref{est 95} and\eqref{est 96} follow the reasoning of \cite[p. 1861]{BSV19}. We can write from \eqref{est 67a} and write 
\begin{align}
\int_{\mathbb{T}^{2}} \lvert \Lambda^{\frac{1}{2}} w_{q+1} \rvert^{2} &dx \overset{\eqref{est 67a}}{=} \sum_{j,k} \int_{\mathbb{T}^{2}} \lvert \Lambda^{\frac{1}{2}} \mathbb{P}_{q+1, k} \tilde{w}_{q+1, j,k} \rvert^{2} dx  \nonumber \\
=& \sum_{j,k} \int_{\mathbb{T}^{2}} \mathbb{P}_{q+1, k} ( \chi_{j} a_{k,j} b_{k} (\lambda_{q+1} \Phi_{j} )) \cdot \Lambda \mathbb{P}_{q+1, -k} (\chi_{j} a_{-k,j} b_{-k} (\lambda_{q+1} \Phi_{j} ))dx  \nonumber \\
\overset{\eqref{est 14b}}{=}& \sum_{j,k} \int_{\mathbb{T}^{2}} \mathbb{P}_{q+1, k} \tilde{w}_{q+1, j, k}  \cdot \Lambda \mathbb{P}_{q+1, -k} \tilde{w}_{q+1, j, -k}dx. \label{est 95}
\end{align}
We recall two identities from \cite[Equations (175) and (177)]{Y23a}: 
\begin{subequations}\label{new 40}
\begin{align}
& \mathbb{P}_{q+1, k} \tilde{w}_{q+1, j, k}(x) =\tilde{w}_{q+1, j, k} (x) +\chi_{j} [\mathbb{P}_{q+1, k}, a_{k,j} (x) \psi_{q+1, j, k} (x) ] b_{k} (\lambda_{q+1} x), \label{new 40a}\\
& \Lambda  \mathbb{P}_{q+1, k} \tilde{w}_{q+1, j, k}(x) \nonumber\\
& \hspace{20mm} = \lambda_{q+1} \tilde{w}_{q+1, j, k}(x) + \chi_{j} [\mathbb{P}_{q+1, k} \Lambda, a_{k,j}(x) \psi_{q+1, j, k}(x) ] b_{k} (\lambda_{q+1} x). \label{new 40b} 
\end{align}
\end{subequations} 
We apply them to \eqref{est 95} to obtain 
\begin{align}
\int_{\mathbb{T}^{2}} \lvert \Lambda^{\frac{1}{2}} w_{q+1} \rvert^{2} dx =& \sum_{j,k} \int_{\mathbb{T}^{2}} \lambda_{q+1} \tilde{w}_{q+1,j, k} \cdot  \tilde{w}_{q+1,j, -k} \nonumber \\
&+ \tilde{w}_{q+1, j, k} \chi_{j} \cdot [ \mathbb{P}_{q+1, -k} \Lambda, a_{-k, j} \psi_{q+1, j, -k} ] b_{-k} (\lambda_{q+1} x) \nonumber \\
&+ \chi_{j} [ \mathbb{P}_{q+1, k}, a_{k, j} \psi_{q+1, j, k} ] b_{k} (\lambda_{q+1} x) \cdot \Lambda w_{q+1}dx.   \label{est 96}
\end{align}
Next, we focus on the first term $ \sum_{j,k} \int_{\mathbb{T}^{2}} \lambda_{q+1} \tilde{w}_{q+1,j, k} \cdot  \tilde{w}_{q+1,j, -k}$ of \eqref{est 96}. First, we use Lemma \ref{geometric lemma} to write 
\begin{align}
& \sum_{j,k} \int_{\mathbb{T}^{2}} \lambda_{q+1} \tilde{w}_{q+1, j,k} \cdot \tilde{w}_{q+1, j, -k}dx  \nonumber\\
\overset{\eqref{est 14b} \eqref{est 66}\eqref{est 3} }{=}& \sum_{j,k} \int_{\mathbb{T}^{2}} \lambda_{q+1} \chi_{j}^{2}  \left[ \lambda_{q+1}^{-\frac{1}{2}} \rho_{j}^{\frac{1}{2}} \gamma_{k} \left( \Id - \frac{ \mathring{R}_{q,j}}{\rho_{j}} \right) \right]^{2} \Tr (k^{\bot} \otimes k^{\bot} ) dx\nonumber\\
=& \sum_{j} \int_{\mathbb{T}^{2}} \chi_{j}^{2}  \rho_{j} 2 \Tr \left( \Id - \frac{\mathring{R}_{q,j}}{\rho_{j}} \right)dx.  \label{est 197}
\end{align}
Then we use the fact that $\mathring{R}_{q,j}$ is trace-free and that $\rho_{j}$ is a constant to conclude using \eqref{est 0a} that 
\begin{align}\label{est 43} 
& \sum_{j,k} \int_{\mathbb{T}^{2}} \lambda_{q+1} \tilde{w}_{q+1, j,k} \cdot \tilde{w}_{q+1, j, -k}dx = \sum_{j} 4 \chi_{j}^{2} \rho_{j} (2\pi)^{2} \nonumber \\
=& (2\pi)^{2} 4 \varepsilon_\gamma^{-1} \sum_{j} \chi_{j}^{2} \sqrt{ l^{2} + \lVert \mathring{R}_{q,j} (\tau_{q+1} j) \rVert_{C_{x}}^{2}} + 4 (2\pi)^{2} \sum_{j} \chi_{j}^{2} \gamma_{l} (\tau_{q+1} j).
\end{align}
Therefore, we have
\begin{align}
\delta E(t) &\overset{\eqref{new 41} \eqref{new 42}}{=} \Big\lvert \gamma_{q} (t) 4(2\pi)^{2} + \int_{\mathbb{T}^{2}} \Lambda^{\frac{1}{2}} (y_{q} + z_{q}) (t) \cdot \Lambda^{\frac{1}{2}} (y_{q} + z_{q}) (t)dx \nonumber\\
& - \int_{\mathbb{T}^{2}} \lvert \Lambda^{\frac{1}{2}} y_{l} (t) \rvert^{2} + \lvert \Lambda^{\frac{1}{2}} w_{q+1}(t) \rvert^{2} + \lvert \Lambda^{\frac{1}{2}} z_{q+1} (t) \rvert^{2} dx - 2 \sum_{m \in \mathbb{Z}^{2} \setminus \{0\}} \lvert m \rvert [ \hat{y}_{l} \cdot \hat{z}_{q+1} ](t,m)\Big\rvert \nonumber \\
&\overset{\eqref{est 96} \eqref{est 43}}{=} \lvert \sum_{k=1}^{5}\text{(E-k)} \rvert , 
\end{align}
where 
\begin{subequations}
\begin{align}
& \text{(E-1)} \triangleq   \gamma_{q}(t) 4(2\pi)^{2} - 4(2\pi)^{2} \sum_{j} \chi_{j}^{2}(t) \gamma_{l} (\tau_{q+1} j) \label{term E1} \tag{E-1}, \\
& \text{(E-2)} \triangleq - \int_{\mathbb{T}^2}\lvert \Lambda^{\frac{1}{2}} y_{l} (t) \rvert^{2} + \lvert \Lambda^{\frac{1}{2}} z_{q+1} (t) \rvert^{2}dx - 2 \sum_{m \in \mathbb{Z}^{2} \setminus \{0\}} \lvert m \rvert [ \hat{y}_{l} \cdot \hat{z}_{q+1} ](t,m) \label{term E2} \tag{E-2}, \\
& \text{(E-3)} \triangleq - (2\pi)^{2} 4 \varepsilon_\gamma^{-1} \sum_{j} \chi_{j}^{2}(t) \sqrt{ l^{2} + \lVert \mathring{R}_{q,j} (\tau_{q+1} j) \rVert_{C_{x}}^{2}}  \nonumber \\
& \hspace{20mm} + \int_{\mathbb{T}^{2}} \Lambda^{\frac{1}{2}} (y_{q} + z_{q}) \cdot \Lambda^{\frac{1}{2}} (y_{q}+  z_{q})(t)dx \label{term E3} \tag{E-3}, \\
& \text{(E-4)} \triangleq - \sum_{j,k} \int_{\mathbb{T}^2} \tilde{w}_{q+1, j,k}(t) \chi_{j}(t) [ \mathbb{P}_{q+1, -k} \Lambda, a_{-k, j} \psi_{q+1, j, -k} ] b_{-k} (\lambda_{q+1} x) dx \label{term E4} \tag{E-4}, \\
& \text{(E-5)} \triangleq - \sum_{j,k} \int_{\mathbb{T}^2} \chi_{j}(t) [ \mathbb{P}_{q+1, k}, a_{k,j} \psi_{q+1, j, k} ] b_{k} (\lambda_{q+1} x) \cdot \Lambda w_{q+1}(t) dx.\tag{E-5}\label{term E5}
\end{align}
\end{subequations}

We treat the terms within \eqref{term E1} carefully. First, we use \eqref{est 18} to write 
\begin{align}\label{est 300}
\lvert \gamma_{q} (t) 4(2\pi)^{2} - 4(2\pi)^{2} \sum_{j} \chi_{j}^{2} (t) \gamma_{l} (\tau_{q+1} j) \rvert \leq  4(2\pi)^{2} [ \RomanIV_{1}+  \RomanIV_{2}], 
\end{align}
where we define 
\begin{equation}
\RomanIV_{1} \triangleq  \lvert \gamma_{q}(t) - \gamma_{l} (t) \rvert \hspace{1mm} \text{ and } \hspace{1mm} \RomanIV_{2} \triangleq \sum_{j} \chi_{j}^{2}(t) \lvert \gamma_{l}(t) -\gamma_{l} (\tau_{q+1} j) \rvert. 
\end{equation} 
We further decompose the first term within \eqref{est 300} by using \eqref{est 0b} as 
\begin{equation}\label{new 45}
4( 2\pi)^{2} \RomanIV_{1}  \leq \sum_{k=1}^{4} \RomanII_{k}
\end{equation} 
where 
\begin{subequations}\label{new 43} 
\begin{align}
\RomanII_{1} \triangleq&  (1- \lambda_{q+2} \delta_{q+2}) \lvert e(t) \ast_{t} \varphi_{l} - e(t) \rvert, \label{new 43a} \\
\RomanII_{2} \triangleq&  \left\lvert \int_{\mathbb{T}^{2}} \lvert \Lambda^{\frac{1}{2}} y_{q} \rvert^{2} \ast_{t} \varphi_{l} - \lvert \Lambda^{\frac{1}{2}} y_{q} \rvert^{2} dx \right\rvert, \label{new 43b} \\
\RomanII_{3} \triangleq& \left\lvert \int_{\mathbb{T}^{2}} 2( \Lambda^{\frac{1}{2}} y_{q} \cdot \Lambda^{\frac{1}{2}} z_{q}) \ast_{t} \varphi_{l} - 2 \Lambda^{\frac{1}{2}} y_{q} \cdot \Lambda^{\frac{1}{2}} z_{q}  dx \right\rvert, \label{new 43c} \\ 
\RomanII_{4} \triangleq& \left\lvert \int_{\mathbb{T}^{2}} \lvert \Lambda^{\frac{1}{2}} z_{q} \rvert^{2} \ast_{t} \varphi_{l} - \lvert \Lambda^{\frac{1}{2}} z_{q} \rvert^{2} dx \right\rvert. \label{new 43d} 
\end{align}
\end{subequations}
Bounding $\lvert 1- \lambda_{q+2} \delta_{q+2} \rvert \leq 1$ and relying on standard mollifier estimates give us 
\begin{equation}\label{est 102}
 \RomanII_{1}  \lesssim l \lVert e' \rVert_{C_{t}} \overset{\eqref{est 118}}{\lesssim} l \tilde{e},\hspace{8mm} \RomanII_{2} \lesssim l \lVert \Lambda y_{q} \rVert_{C_{t,x,q+1}} \lVert \partial_{t} y_{q} \rVert_{C_{t,x,q+1}}  \overset{\eqref{inductive 1c} \eqref{temp}}{\lesssim}  l  \lambda_{q}^{3} \delta_{q}^{\frac{3}{2}} \bar{e}^{\frac{3}{2}}. 
\end{equation} 
Next, we rewrite \eqref{new 43c} to a convenient form of 
\begin{align}
\RomanII_{3}  =& 2 \Big\lvert \int_{\mathbb{T}^{2}} (y_{q} \cdot \Lambda z_{q}) \ast_{t} \varphi_{l}  - \left( [ y_{q} \ast_{t} \varphi_{l} ] \cdot \Lambda z_{q} \right) \ast_{t} \varphi_{l} \nonumber \\
& + \left( [ y_{q} \ast_{t} \varphi_{l} ] \cdot  \Lambda z_{q}\right) \ast_{t} \varphi_{l} - [ ( y_{q} \ast_{t} \varphi_{l}) \cdot ( \Lambda z_{q} \ast_{t} \varphi_{l} ) ] \ast_{t} \varphi_{l}  \nonumber \\
&+ [ (y_{q} \ast_{t} \varphi_{l} ) \cdot (\Lambda z_{q} \ast_{t} \varphi_{l} ) ] \ast_{t} \varphi_{l} - (y_{q} \ast_{t} \varphi_{l}) \cdot (\Lambda z_{q} \ast_{t} \varphi_{l})  \nonumber \\
&+ (y_{q} \ast_{t} \varphi_{l}) \cdot (\Lambda z_{q} \ast_{t} \varphi_{l}) - (y_{q} \ast_{t} \varphi_{l}) \cdot \Lambda z_{q} + (y_{q} \ast_{t} \varphi_{l}) \cdot \Lambda z_{q} -  y_{q} \cdot \Lambda z_{q} dx \Big\rvert. \label{new 44}
\end{align}
Standard mollifier estimates now give us for any $\delta \in (0, \frac{1}{4})$ 
\begin{align}
\RomanII_{3} \lesssim& l \lVert y_{q} \rVert_{C_{t,q+1}^{1} C_{x}} \lVert z_{q} \rVert_{C_{t,q+1} \dot{H}_{x}^{1}} + l^{\frac{1}{2} - 2 \delta} \lVert y_{q} \rVert_{C_{t,x,q+1}}  \lVert z_{q} \rVert_{C_{t,q+1}^{\frac{1}{2} - 2 \delta} \dot{H}_{x}^{1}} \nonumber \\
& \overset{\eqref{inductive 1b} \eqref{temp} \eqref{additive stopping time}}{\lesssim} l \lambda_{q}^{2} \delta_{q} \bar{e}  + l^{\frac{1}{2} - 2 \delta}\bar{e}^{\frac{1}{2}}. \label{est 204}
\end{align}
Yet, another mollifier estimate allows us to estimate from \eqref{new 43d} 
\begin{equation}\label{est 205a} 
\RomanII_{4} \lesssim l^{\frac{1}{2} - 2 \delta} \lVert z_{q} \rVert_{C_{t,q+1}^{\frac{1}{2} - 2 \delta} L_{x}^{2}} \lVert \Lambda z_{q} \rVert_{C_{t,q+1}L_{x}^{2}} \overset{\eqref{additive stopping time}}{\lesssim} l^{\frac{1}{2} - 2 \delta}. 
\end{equation} 
Hence, applying these estimates \eqref{est 102}, \eqref{est 204}, and \eqref{est 205a} to \eqref{new 45} gives us now  using the hypothesis that $e(t) \geq \ushort{e}$ due to \eqref{est 118}, for $\beta > \frac{1}{2}$ sufficiently close to $\frac{1}{2}$ and $a_{0}$ sufficiently large, 
\begin{align}
&4 (2\pi)^{2} \RomanIV_{1} \leq \sum_{k=1}^{4} \RomanII_{k}  \label{new 46} \\
\lesssim& l \left(\tilde{e} + \lambda_{q}^{3} \delta_{q}^{\frac{3}{2}} \bar{e}^{\frac{3}{2}} + \lambda_{q}^{2} \delta_{q} \bar{e} \right) + l^{\frac{1}{2} - 2 \delta}  \lesssim \lambda_{q+1}^{-\alpha} \lambda_{q}^{3- 3\beta} \lambda_{1}^{(2\beta -1) \frac{3}{2}} \bar{e}^{\frac{3}{2}} + \lambda_{q+1}^{-\alpha( \frac{1}{2} - 2 \delta)} \ll \lambda_{q+2} \delta_{q+2}  e(t). \nonumber 
\end{align}

Next, to estimate $4(2\pi)^{2} \RomanIV_{2}$ in \eqref{est 300} we observe that
\begin{equation}\label{est 305} 
\lvert \gamma_{l}' (t) \rvert \lesssim l^{-1} \lVert \gamma_{q} \rVert_{C_{t,q+1}}  \overset{\eqref{est 21}}{\approx} l^{-1}\sup_{t}\gamma_{q}(t) \overset{\eqref{est 0b}}{\lesssim} l^{-1} \bar{e} 
\end{equation}
so that we may estimate using the hypothesis that $e(t) \geq \ushort{e}$ due to \eqref{est 118}, for $\beta > \frac{1}{2}$ sufficiently close to $\frac{1}{2}$ and $a_{0}$ sufficiently large 
\begin{align}
4(2\pi)^{2}\RomanIV_{2} \lesssim&  \sum_{j} \chi_{j}^{2} (t) \lvert \tau_{q+1} j - t \rvert \lVert \gamma_{l}' \rVert_{C_{t,q+1}}   \overset{\eqref{est 18} \eqref{est 305}}{\lesssim} \tau_{q+1} l^{-1}  \nonumber \\
& \overset{\eqref{tau} \eqref{define lambda delta} \eqref{define l}}{\approx} a^{b[-\frac{\beta}{2} + \frac{1}{4} ]} a^{b^{q+1} [ \frac{\alpha}{2} - \frac{\beta}{2} - \frac{1}{2} + b[\beta - \frac{1}{2} ]]} \overset{\eqref{bound alpha}}{\ll} \lambda_{q+2} \delta_{q+2} e(t).  \label{est 104} 
\end{align}
Thus, we now apply \eqref{new 46} and \eqref{est 104} to $\eqref{est 300}$ to conclude that \eqref{term E1} satisfies for $\beta > \frac{1}{2}$ sufficiently close to $\frac{1}{2}$ and $a_{0}$ sufficiently large 
\begin{equation}\label{new 49}
4 (2\pi)^{2} \lvert \gamma_{q}(t) - \sum_{j}\chi_{j}^{2}(t) \gamma_{l} (\tau_{q+1} j) \rvert \overset{\eqref{est 300}}{\leq} 4(2\pi)^{2} [ \RomanIV_{1}+  \RomanIV_{2}]  \ll \lambda_{q+2} \delta_{q+2} e(t). 
\end{equation}

Next, for the first term of \eqref{term E3}, we estimate for $\beta > \frac{1}{2}$ sufficiently close to $\frac{1}{2}$ and $a_{0}$ sufficiently large, 
\begin{align}\label{est 101}
& (2\pi)^{2} 4 \varepsilon_\gamma^{-1} \sum_{j} \chi_{j}^{2}(t) \sqrt{ l^{2} + \lVert \mathring{R}_{q,j} (\tau_{q+1} j) \rVert_{C_{x}}^{2}} \nonumber \\
\overset{\eqref{est 16}}{\leq}&  (2\pi)^{2} 4 \varepsilon_\gamma^{-1} \sum_{j} \chi_{j}^{2} (t) ( l + \lVert \mathring{R}_{q} \ast_{t} \varphi_{l} (\tau_{q+1} j) \rVert_{C_{x}} \nonumber \\
\overset{\eqref{est 18} \eqref{inductive 2b}}{\leq}& (2\pi)^{2} 4 \varepsilon_\gamma^{-1} \lambda_{q+1}^{-\alpha} + \frac{1}{8} \lambda_{q+2} \delta_{q+2} e(t) \overset{\eqref{additive bound b}}{\leq} \frac{1}{6} \lambda_{q+2} \delta_{q+2} e(t).
\end{align}

Next, we estimate \eqref{term E4} and \eqref{term E5} for $\beta > \frac{1}{2}$ sufficiently close to $\frac{1}{2}$ and $a_{0}$ sufficiently large, by relying on \cite[Equation (A.17)]{BSV19} and the hypothesis that $e(t) \geq \ushort{e}$ due to \eqref{est 118},  
\begin{align}
& \sum_{j,k} \Big\lvert \int_{\mathbb{T}^2} \tilde{w}_{q+1, j, k}(t)  \chi_{j}(t)  [ \mathbb{P}_{q+1, -k} \Lambda, a_{-k, j} \psi_{q+1, j, -k} ] b_{-k} (\lambda_{q+1} x) dx \Big\rvert \nonumber \\
&\hspace{5mm}+ \Big\lvert \int_{\mathbb{T}^2} \chi_{j}(t)  [ \mathbb{P}_{q+1, k}, a_{k,j} \psi_{q+1,j, k} ] b_{k} (\lambda_{q+1} x) \cdot \Lambda w_{q+1}(t)  dx \Big\rvert  \nonumber \\
\overset{\eqref{support wqplus1}}{\lesssim}& \sum_{j,k} 1_{\supp \chi_{j}} (t) [ \lVert \tilde{w}_{q+1, j, k} (t) \rVert_{C_{x}} \lVert \nabla (a_{-k,j} \psi_{q+1, j, -k} ) \rVert_{C_{x}} \lVert b_{-k} (\lambda_{q+1} x) \rVert_{C_{x}} \nonumber \\
&+ \lambda_{q+1}^{-1} \lVert \nabla (a_{k,j} \psi_{q+1, j, k} ) \rVert_{C_{x}} \lVert b_{k} (\lambda_{q+1} x) \rVert_{C_{x}} \lambda_{q+1} \lVert w_{q+1} (t)  \rVert_{C_{x}}] \nonumber \\
\overset{\eqref{psi 2} \eqref{est 11a} \eqref{est 10}}{\lesssim}&\lambda_{q} \delta_{q+1} \bar{e} \overset{\eqref{additive bound b}}{\ll} \lambda_{q+2} \delta_{q+2} e(t).\label{est 100}
\end{align}
 
We write the rest of the terms, namely \eqref{term E2} and the latter of \eqref{term E3}, as 
\begin{align}
& \int_{\mathbb{T}^{2}} \Lambda^{\frac{1}{2}} (y_{q} + z_{q}) \cdot \Lambda^{\frac{1}{2}} (y_{q} + z_{q}) dx- \int_{\mathbb{T}^{2}} \lvert \Lambda^{\frac{1}{2}} y_{l} \rvert^{2} + \lvert \Lambda^{\frac{1}{2}} z_{q+1} \rvert^{2} dx \nonumber \\
& \hspace{30mm}- 2 \sum_{m \in \mathbb{Z}^{2} \setminus \{0\}} \lvert m \rvert [ \hat{y}_{l} \cdot \hat{z}_{q+1} ](m)  \nonumber \\
=& \sum_{m \in \mathbb{Z}^{2} \setminus \{0\}} \lvert m \rvert \big[ \lvert \hat{y}_{q}  (m) \rvert^{2} - \lvert\hat{y}_{l} (m) \rvert^{2} + \lvert \hat{z}_{q} (m) \rvert^{2} - \lvert \hat{z}_{q+1}(m) \rvert^{2} \nonumber \\
& \hspace{30mm} + 2 [ (\hat{y}_{q} - \hat{y}_{l}) \cdot \hat{z}_{q} + \hat{y}_{l} \cdot ( \hat{z}_{q} - \hat{z}_{q+1})] (m)  \big]; \label{new 48}
\end{align}
here, it is crucial to deduce the last bracket $2 [ (\hat{y}_{q} - \hat{y}_{l}) \cdot \hat{z}_{q} + \hat{y}_{l} \cdot ( \hat{z}_{q} - \hat{z}_{q+1})] (m)$ for the subsequent mollifier estimates and it was possible only due to coupling the \eqref{term E2} and the latter part of  \eqref{term E3}. Having written all in favorable forms, standard mollifier estimates lead us to  
\begin{subequations}\label{new 47}
\begin{align}
&  \Big\lvert \sum_{m \in \mathbb{Z}^{2} \setminus \{0\}} \lvert m \rvert [ \lvert \hat{y}_{q} (t,m) \rvert^{2} - \lvert \hat{y}_{l} (t,m) \rvert^{2} ]  \Big\rvert  \\
& \hspace{10mm} \lesssim l ( \lVert y_{q} \rVert_{C_{t,q+1}C_{x}^{1}} +  \lVert \partial_{t} y_{q} \rVert_{C_{t,x,q+1}} ) \lVert \Lambda y_{q} \rVert_{C_{t,x,q+1}} \overset{\eqref{inductive 1c} \eqref{temp}}{\lesssim}l \lambda_{q}^{3} \delta_{q}^{\frac{3}{2}}\bar{e}^{\frac{3}{2}},  \nonumber \\
&  \Big\lvert  \sum_{m \in \mathbb{Z}^{2} \setminus \{0\}} \lvert m \rvert [ \lvert \hat{z}_{q}(t,m) \rvert^{2} - \lvert \hat{z}_{q+1}(t,m) \rvert^{2} ] \Big\rvert   \nonumber \\
& \hspace{20mm} \lesssim \lVert z_{q} - z_{q+1} \rVert_{C_{t,q+1}L_{x}^{2}} \lVert z\rVert_{C_{t,q+1}\dot{H}_{x}^{1}}  \overset{\eqref{additive stopping time}}{\lesssim} \lambda_{q}^{-1}, \\
& \Big\lvert  2 \sum_{m \in \mathbb{Z}^{2} \setminus \{0\}} \lvert m \rvert \left[ ( \hat{y}_{q} - \hat{y}_{l}) \cdot \hat{z}_{q} + \hat{y}_{l} \cdot ( \hat{z}_{q} - \hat{z}_{q+1} ) \right] (t,m)  \Big\rvert \\
\lesssim& l \lVert y_{q} \rVert_{C_{t,q+1}C_{x}^{1}}  + l \lVert \partial_{t} y_{q} \rVert_{C_{t,q+1}L_{x}^{2}} \lVert z_{q} \rVert_{C_{t,q+1}\dot{H}_{x}^{1}} +  \lambda_{q}^{-1} \overset{\eqref{inductive 1c} \eqref{additive stopping time} \eqref{temp}}{\lesssim} l \lambda_{q} \delta_{q}^{2} \bar{e} + \lambda_{q}^{-1}. \nonumber 
\end{align}
\end{subequations}
Applying \eqref{new 47} to \eqref{new 48}, we obtain by relying on the hypothesis that $e(t) \geq \ushort{e}$ due to \eqref{est 118}, for all $\beta > \frac{1}{2}$ sufficiently close to $\frac{1}{2}$ and $a_{0}$ sufficiently large, that \eqref{term E2} and the latter of \eqref{term E3} satisfy
\begin{align}
& \Big\lvert   \int_{\mathbb{T}^{2}} \Lambda^{\frac{1}{2}} (y_{q} + z_{q})(t) \cdot \Lambda^{\frac{1}{2}} (y_{q} + z_{q})(t) dx- \int_{\mathbb{T}^{2}} \lvert \Lambda^{\frac{1}{2}} y_{l}(t) \rvert^{2} + \lvert \Lambda^{\frac{1}{2}} z_{q+1}(t) \rvert^{2}dx   \label{new2}\\
& \hspace{5mm}  -2 \sum_{m \in \mathbb{Z}^{2} \setminus\{0\}} \lvert m \rvert [\hat{y}_{l} \cdot \hat{z}_{q+1} ](t,m) \Big\rvert   \overset{\eqref{define f} }{\lesssim} \lambda_{q+1}^{-\alpha} \lambda_{q}^{3-3\beta} \lambda_{1}^{(2\beta -1) \frac{3}{2}} \bar{e}^{\frac{1}{2}} + \lambda_{q}^{-1}  \overset{\eqref{additive bound b}}{\ll} \lambda_{q+2} \delta_{q+2} e(t). \nonumber 
\end{align}

Applying \eqref{new 49}, \eqref{est 101}, \eqref{est 100}, and \eqref{new2} to \eqref{term E1}-\eqref{term E5} shows that for any $\iota > 0$, we can bound 
\begin{equation}
\delta E(t) \leq \frac{1}{6} \lambda_{q+2} \delta_{q+2} e(t) + \iota \lambda_{q+2} \delta_{q+2} e(t)
\end{equation} 
so that taking $\iota \leq \frac{1}{12}$ completes the proof of \eqref{est 20}. 
\end{proof}

\subsubsection{Proof of Proposition \ref{additive q to qplus1}}\label{proof additive q to qplus1}
With our given choice of parameters, we have constructed $y_{q+1}$ and $\mathring{R}_{q+1}$ as laid out in Section \ref{additive construction}. From the statement of Proposition \ref{prop additive amplitudes}, this construction satisfies hypotheses \eqref{inductive 1a} and \eqref{inductive 2a} and the support condition \eqref{support wqplus1}; similarly, by Proposition \ref{prop additive perturbation}, hypotheses \eqref{inductive 1b}-\eqref{inductive 1d} and the Cauchy difference \eqref{cauchy difference} are all satisfied. Additionally, Propositions \ref{prop RT}-\ref{prop Rcom2} show that \eqref{inductive 2b} holds at level $q+1$. Lastly, the control of the energy \eqref{inductive 3} has been verified by Proposition \ref{prop additive control energy}; hence, all the necessary bounds have been satisfied.

Finally, it must be argued that the solutions $y_{q+1},\mathring{R}_{q+1}$ are $(\mathcal{F}_t)_{t\geq 0}$-adapted and deterministic on $[t_{q+1},0]$. The proof of these facts is essentially identical to that in previous works (e.g. \cite{HZZ19, HZZ21a, Y23a}) and thus omitted. 

We now conclude the proof of Theorem \ref{Theorem 1.1}. 
\begin{proof}[Proof of Theorem \ref{Theorem 1.1}]
Starting from $(y_0,\mathring{R}_0)=\left(0,\mathcal{B}(z_0^\perp (\nabla^\perp\cdot z_0) + \Lambda^{\gamma} z_{0} - \Lambda^{\frac{3}{2} - 2\sigma} z_{0}\right)$ in Proposition \ref{additive step q0}, by induction based on Proposition \ref{additive q to qplus1} we can obtain a family of solutions $\{(y_q,\mathring{R}_q)\}_{q\in\mathbb{N}_0}$ to \eqref{additive convPDE} that are $(\mathcal{F}_{t})_{t\geq 0}$-adapted and satisfy the inductive hypotheses \eqref{inductive 1}-\eqref{inductive 3} and the Cauchy difference \eqref{cauchy difference}. Then for all $\beta' \in (\frac{1}{2}, \beta)$, 
\begin{align*}
\sum_{q\geq 0} \lVert y_{q+1} (t) - y_{q} (t) \rVert_{C_{x}^{\beta'}} \lesssim& \sum_{q\geq 0} \lVert y_{q+1} (t) - y_{q} (t) \rVert_{C_{x}}^{1- \beta'} \lVert y_{q+1} (t) - y_{q} (t) \rVert_{C_{x}^{1}}^{\beta'}  \\
\overset{\eqref{cauchy difference} \eqref{inductive 1c}}{\lesssim}&  \bar{e}^{\frac{1}{2}} M_{0}^{\frac{1}{2}} a^{b(\beta - \frac{1}{2})} \sum_{q\geq 0} a^{b(q+1) (\beta ' - \beta)} < \infty. 
\end{align*}
We also compute for all $\eta \in [0, \frac{\beta}{2-\beta})$ and all $q \in \mathbb{N}$, using \eqref{temp} and \eqref{cauchy difference}, 
\begin{align*}
\lVert y_{q} - y_{q-1} \rVert_{C_{t}^{\eta} C_{x}} \lesssim \lVert y_{q} - y_{q-1} \rVert_{C_{t,x}}^{1-\eta} \lVert y_{q} - y_{q-1} \rVert_{C_{t}^{1}C_{x}^{0}}^{\eta} \lesssim M_{0}^{\frac{1+\eta}{2}} \bar{e}^{\frac{1+ \eta}{2}} \lambda_{1}^{(\beta - \frac{1}{2})(1+ \eta)} \lambda_{q}^{-\beta (1+ \eta) + 2 \eta},
\end{align*}
which vanishes as $q\nearrow \infty$. Therefore, $\{y_{q}\}_{q\in \mathbb{N}_{0}}$ is Cauchy in $C( [0, \mathfrak{t}]; C_{x}^{\beta'}) \cap C^{\eta} ([0, \mathfrak{t} ]; C_{x})$. Due to \eqref{inductive 2b} we have $\lVert \mathring{R}_{q} \rVert_{C_{t,x,q}} \leq \frac{\varepsilon_{\gamma}}{32 (2\pi)^{2}} \lambda_{q+2} \delta_{q+2} \bar{e} \searrow 0$ as $q \nearrow + \infty$. Moreover, because $z \in C_{T} \dot{H}_{x}^{\frac{5}{2} + 2 \sigma}$ $\mathbf{P}$-a.s. due to Proposition \ref{prop solution z}, we have $z_{q} \to z$ in $L_{T}^{2} \dot{H}_{x}^{1}$ as $q \nearrow + \infty$ which suffices for the purpose of verifying that $(y,z)$ with $y$ defined by 
\begin{equation*}
y \triangleq \lim_{q\to \infty} y_{q}  \in C([0, \mathfrak{t} ]; C_{x}^{\beta'} ) \cap C^{\eta} ( [0, \mathfrak{t} ]; C_{x}) 
\end{equation*}
solves \eqref{additive random PDE} analytically weakly and that there exists a constant $C_{2}<\infty$ such that 
\begin{equation}\label{est 136}
\lVert y \rVert_{C_{\mathfrak{t}} C_{x}^{\beta'}} + \lVert y \rVert_{C_{\mathfrak{t}}^{\eta} C_{x}} \leq C_{2} \hspace{3mm} \forall \hspace{1mm} \beta' \in \left(\frac{1}{2}, \beta \right), \hspace{1mm} \forall \hspace{1mm}\eta \in \left[0, \frac{\beta}{2-\beta}\right). 
\end{equation}
We mention that upon taking limit to verify the weak formulation, it becomes crucial to rely on the Calder$\acute{\mathrm{o}}$n commutator estimate (e.g.\cite[Lemma A.5]{BSV19}, \cite[Lemma 2.2]{M08}, and \cite[Theorem 10.3]{L02}) and refer to \cite{Y23a} for details. Together with Proposition \ref{prop solution z} and \eqref{est 135}, \eqref{est 136} gives us \eqref{est 119}. Finally, taking $q \nearrow \infty$ in \eqref{inductive 3} shows that $e(t) = \lVert (y+z)(t) \rVert_{\dot{H}_{x}^{\frac{1}{2}}}^{2} = \lVert v(t) \rVert_{\dot{H}_{x}^{\frac{1}{2}}}^{2}$ due to \eqref{est 135}.  

The argument concerning non-uniqueness is as follows. If we start with two different energies $e_{1}$ and $e_{2}$ that satisfies the hypothesis and $e_{1} \equiv e_{2}$ on $[0,t \wedge \mathfrak{t}]$ for some $t > 0$, then the perturbation and therefore $\{y_{q}^{1}\}_{q\in\mathbb{N}_{0}}$ and $\{y_{q}^{2}\}_{q\in\ \mathbb{N}_{0}}$ and consequently $\{v_{q}^{1}\}_{q \in \mathbb{N}_{0}}$ and $\{ v_{q}^{2}\}_{q \in \mathbb{N}_{0}}$ corresponding respectively to $e_{1}$ and $e_{2}$ are identical on $[0, t \wedge \mathfrak{t}]$ (recall \eqref{define tq}); as a result, $v_{1}$ and $v_{2}$ corresponding respectively to $e_{1}$ and $e_{2}$ are identical on $[0, t \wedge \mathfrak{t}$]. This completes the proof of Theorem \ref{Theorem 1.1}. 
\end{proof}

\section{Proof of Theorem \ref{Theorem 1.2}}\label{Section 4}
\subsection{Martingale Solution}\label{Section 4.1}
To prove Theorem \ref{Theorem 1.2}, we start with a definition of a martingale solution to \eqref{general msqg} in the additive case $G(v) dB = G dB$, that relies on the definitions of an $(\mathcal{B}_{t}^{0})_{t\geq s}$-supermartingale, as well as regular time and exceptional time from Definition \ref{Definition 3}. 

\begin{define}\label{Definition 6} 
Fix $\iota \in (0,1)$ from Theorem \ref{Theorem 1.1}. Let $s \geq 0$ and $\xi^{\text{in}} \in \dot{H}_{\sigma}^{\frac{1}{2}}$. Then $P \in \mathcal{P} (\Omega_{0})$ is a martingale solution to \eqref{general msqg} where $G(v) dB = dB$, $B$ being a $GG^{\ast}$-Wiener process, with initial data $\xi^{\text{in}}$ at initial time $s$ if 
\begin{enumerate}[wide=0pt]
\item [](M1) $P (\{ \xi(t) = \xi^{\text{in}} \hspace{1mm} \forall \hspace{1mm} t \in [0,s] \}) = 1$ and 
\begin{equation}\label{est 211}
P \left( \{ \xi \in L_{\text{loc}}^{\infty} ([0,\infty); \dot{H}_{\sigma}^{\frac{1}{2}} ) \cap L_{\text{loc}}^{2} ([0,\infty); \dot{H}_{\sigma}^{\frac{1}{2} + \iota} ) \} \right) = 1, 
\end{equation} 
\item [](M2) for all $\psi^{k} \in C^{\infty} (\mathbb{T}^{2}) \cap \dot{H}_{\sigma}^{\frac{1}{2}}$ and $t\geq s$, the process 
\begin{align}
M_{t,s}^{k} \triangleq& \langle \xi(t) - \xi^{\text{in}}, \psi^{k} \rangle \label{est 179}\\
-& \int_{s}^{t} \sum_{i,j=1}^{2} \langle \Lambda \xi_{i}, \partial_{i} \psi_{j}^{k} \xi_{j}\rangle_{\dot{H}_{x}^{-\frac{1}{2}} - \dot{H}_{x}^{\frac{1}{2}}} - \frac{1}{2} \langle \partial_{i} \xi_{j}, [ \Lambda, \psi_{i}^{k} ] \xi_{j} \rangle_{\dot{H}_{x}^{-\frac{1}{2}} -\dot{H}_{x}^{\frac{1}{2}}} - \langle \xi, \Lambda^{\gamma} \psi^{k} \rangle dr \nonumber 
\end{align} 
is a continuous, square-integrable $(\mathcal{B}_{t}^{0})_{t\geq s}$-martingale under $P$ such that $\langle \langle M_{t,s}^{k} \rangle \rangle =  \lVert G^{\ast} \psi^{k} \rVert_{U}^{2} (t-s)$,
\item [](M3) for any $q \in \mathbb{N}$, there exists $C_{q} \geq 1$ such that the process $\{E^{q} (t) - E^{q} (s) \}_{t \geq s}$ with $E^{q}$ defined by 
\begin{align}
E^{q} (t) \triangleq& \lVert \xi(t) \rVert_{\dot{H}_{x}^{\frac{1}{2}}}^{2q} \label{est 124}  \\
& + 2q \int_{0}^{t} \lVert \xi(r) \rVert_{\dot{H}_{x}^{\frac{1}{2}}}^{2q-2} \lVert \xi(r) \rVert_{\dot{H}_{x}^{\frac{1}{2} + \iota}}^{2} dr - C_{q} \lVert G \rVert_{L_{2}(U, \dot{H}_{\sigma}^{\frac{1}{2}})}^{2} \int_{0}^{t} \lVert \xi(r) \rVert_{\dot{H}_{x}^{\frac{1}{2}}}^{2q-2} dr  \nonumber 
\end{align} 
is an a.s. $(\mathcal{B}_{t}^{0})_{t\geq s}$-supermartingale under $P$ and $s$ is a regular time of $E^{q}$. 
\end{enumerate} 
The set of all such martingale solutions with the same constant $C_{q}$ in \eqref{est 124} will be denoted by $\mathcal{C} ( s, \xi^{\text{in}}, C_{q} )$; for brevity we denote $\mathcal{C} (\xi^{\text{in}}, C_{q}) \triangleq \mathcal{C} (s, \xi^{\text{in}}, C_{q}) \rvert_{s=0}$. Given any martingale solution $P$, we define  the exceptional times of $P$ as $T_{P}  \triangleq \cup_{q \in \mathbb{N}} \{\text{exceptional times of } E^{q}  \}$.
\end{define} 

\begin{remark}\label{Remark 4.1} 
We will see in \eqref{est 138} and \eqref{est 140} lower bounds for $C_{q}$, $q \in \mathbb{N}$, that depend on $\lVert G \rVert_{L_{2} (U, \dot{H}_{\sigma}^{\frac{1}{2}})}$ and thereby forbid the case $G = 0$, i.e., the deterministic case. Nevertheless, the non-uniqueness of the deterministic solutions with semiflows was proven for the 3D NS equations in \cite[Corollary 1.4]{HZZ21a} and we can follow the same approach (specifically \cite[Remark 4.6]{HZZ21a}) to see that it is possible to modify the relaxed energy inequality \eqref{est 124} to e.g. for some $\tilde{C}_{q} > 0$
\begin{align}
E^{q} (t) \triangleq& \lVert \xi(t) \rVert_{\dot{H}_{x}^{\frac{1}{2}}}^{2q} \label{est 139}  \\
& + 2q \int_{0}^{t} \lVert \xi(r) \rVert_{\dot{H}_{x}^{\frac{1}{2}}}^{2q-2} \lVert \xi(r) \rVert_{\dot{H}_{x}^{\frac{1}{2} + \iota}}^{2} dr - \left( \tilde{C}_{q} + C_{q} \lVert  G \rVert_{L_{2}(U, \dot{H}_{\sigma}^{\frac{1}{2}})}^{2} \right) \int_{0}^{t} \lVert \xi(r) \rVert_{\dot{H}_{x}^{\frac{1}{2}}}^{2q-2} dr.  \nonumber 
\end{align}   
\end{remark} 

Next, we define a martingale solution up to a stopping time. 
\begin{define}\label{Definition 7}
Fix $\iota \in (0,1)$ from Theorem \ref{Theorem 1.1}. Let $s \geq 0$, $\xi^{\text{in}} \in \dot{H}_{\sigma}^{\frac{1}{2}}$, and $\tau: \Omega_{0} \mapsto [0, \infty]$ be a $(\mathcal{B}_{t})_{t\geq s}$-stopping time. Define the space of trajectories stopped at $\tau$ as 
\begin{equation*}
\Omega_{0,\tau} \triangleq \{\omega( \cdot \wedge \tau(\omega)): \omega \in \Omega_{0}  \} = \{ \omega \in \Omega_{0}: \xi(t,\omega) = \xi(t\wedge \tau(\omega), \omega) \hspace{1mm} \forall \hspace{1mm} t \geq 0 \}, 
\end{equation*}  
which is a Borel subset of $\Omega_{0}$ due to Borel measurability of $\tau$ so that $\mathcal{P} (\Omega_{0,\tau}) \subset \mathcal{P} ( \Omega_{0})$.  Then $P \in \mathcal{P} (\Omega_{0,\tau})$ is a martingale solution to \eqref{general msqg} where $G(v) dB = dB$, $B$ being a $GG^{\ast}$-Wiener process,  on $[s,\tau]$ with initial data $\xi^{\text{in}}$ at initial time $s$ if 
\begin{enumerate}[wide=0pt]
\item [](M1) $P (\{ \xi(t) = \xi^{\text{in}} \hspace{1mm} \forall \hspace{1mm} t \in [0,s] \}) = 1$ and  
\begin{equation}
P \left(  \{ \xi \in L_{\text{loc}}^{\infty} ( [0, \tau]; \dot{H}_{\sigma}^{\frac{1}{2}}) \cap L_{\text{loc}}^{2} ([0,\tau]; \dot{H}_{\sigma}^{\frac{1}{2} + \iota} ) \} \right) = 1, 
\end{equation} 
\item [](M2) for all $\psi^{k} \in C^{\infty} (\mathbb{T}^{2}) \cap \dot{H}_{\sigma}^{\frac{1}{2}}$ and $t\geq s$, the process 
\begin{align*}
M_{t\wedge \tau,s}^{k} \triangleq & \langle \xi(t\wedge \tau) - \xi^{\text{in}}, \psi^{k} \rangle\\
-& \int_{s}^{t\wedge \tau} \sum_{i,j=1}^{2} \langle \Lambda \xi_{i}, \partial_{i} \psi_{j}^{k} \xi_{j}\rangle_{\dot{H}_{x}^{-\frac{1}{2}} - \dot{H}_{x}^{\frac{1}{2}}} - \frac{1}{2} \langle \partial_{i} \xi_{j}, [ \Lambda, \psi_{i}^{k} ] \xi_{j} \rangle_{\dot{H}_{x}^{-\frac{1}{2}} -\dot{H}_{x}^{\frac{1}{2}}} - \langle \xi, \Lambda^{\gamma} \psi^{k} \rangle dr 
\end{align*} 
is a continuous, square-integrable $(\mathcal{B}_{t})_{t\geq s}$-martingale under $P$ such that $\langle \langle M_{t \wedge \tau ,s}^{k} \rangle \rangle =  \lVert G^{\ast} \psi^{k} \rVert_{U}^{2} (t \wedge \tau - s)$, 
\item [](M3) for any $q \in \mathbb{N}$, there exists $C_{q} \geq 1$ such that the process $\{ E^{q} (t \wedge \tau) - E^{q} (s) \}_{t \geq s}$ with $E^{q}$ defined by \eqref{est 124} is an a.s. $(\mathcal{B}_{t})_{t \geq s}$-supermartingale under $P$ and $s$ is a regular time of $E^{q}$. 
\end{enumerate} 
\end{define} 

The first result Proposition \ref{Proposition 4.1} is similar to \cite[Theorem 4.8]{HZZ21a} and \cite[Proposition 4.1]{Y23a} so that its proof is left in Section \ref{Section 7.8} for completeness.  
\begin{proposition}\label{Proposition 4.1} 
\begin{enumerate}[wide=0pt] 
\indent 
\item For every $(s, \xi^{\text{in}}) \in [0,\infty) \times \dot{H}_{\sigma}^{\frac{1}{2}}$, there exists a martingale solution $P \in \mathcal{P} (\Omega_{0})$ to \eqref{general msqg} where $G(v) dB = dB$, $B$ being a $GG^{\ast}$-Wiener process,  with initial data $\xi^{\text{in}}$ at initial time $s$ according to Definition \ref{Definition 6}. 
\item Additionally, if there exists a family $\{ (s_{l}, \xi_{l} ) \}_{l \in \mathbb{N}} \subset [0,\infty) \times \dot{H}_{\sigma}^{\frac{1}{2}}$ that satisfies $\lim_{l\to\infty} \lVert (s_{l}, \xi_{l}) - (s, \xi^{\text{in}}) \rVert_{\mathbb{R} \times \dot{H}_{x}^{\frac{1}{2}}} = 0$ and $P_{l} \in \mathcal{C} ( s_{l}, \xi_{l}, C_{q})$, then there exists a subsequence $\{P_{l_{k}} \}_{k \in \mathbb{N}}$ and $P \in \mathcal{C} ( s, \xi^{\text{in}}, C_{q})$ such that $P_{l_{k}}$ converges weakly to $P$. 
\end{enumerate} 
\end{proposition} 
The next two results are analogues of \cite[Lemmas 4.2-4.3]{Y23a} which are originally due to \cite[Propositions 3.2 and 3.4]{HZZ19}; the only difference is in the definition of a martingale solution, namely (M3) property.  
\begin{lemma}\label{Lemma 4.2}
Let $\tau$ be a bounded stopping time of $(\mathcal{B}_{t})_{t\geq 0}$. Then for every $\omega \in \Omega_{0}$, there exists $Q_{\omega} \triangleq \delta_{\omega} \otimes_{\tau(\omega)} R_{\tau(\omega), \xi(\tau(\omega), \omega)} \in \mathcal{P} (\Omega_{0})$, where $\delta_{\omega}$ is a point-mass at $\omega$, such that for $\omega \in \{\xi(\tau) \in \dot{H}_{\sigma}^{\frac{1}{2}} \}$, 
\begin{subequations}\label{est 224}
\begin{align} 
& Q_{\omega} ( \{ \omega' \in \Omega_{0}:\hspace{1mm}  \xi(t, \omega') = \omega(t) \hspace{1mm} \forall \hspace{1mm} t \in [0, \tau(\omega) ] \}) = 1,  \label{est 224a} \\
& Q_{\omega}(A) = R_{\tau (\omega), \xi(\tau(\omega), \omega)} (A) \hspace{1mm} \forall \hspace{1mm} A \in \mathcal{B}^{\tau(\omega)}, \label{est 224b}  
\end{align}
\end{subequations} 
where $R_{\tau(\omega), \xi(\tau(\omega), \omega)} \in \mathcal{P}(\Omega_{0})$ is a martingale solution to \eqref{general msqg} where $G(v) dB = dB$, $B$ being a $GG^{\ast}$-Wiener process,  with initial data $\xi(\tau(\omega), \omega)$ at initial time $\tau(\omega)$, and the mapping $\omega \mapsto Q_{\omega}(B)$ is $\mathcal{B}_{\tau}$-measurable for every $B \in \mathcal{B}$.
\end{lemma} 

\begin{proof}[Proof of Lemma \ref{Lemma 4.2}]
This proof is very similar to that of \cite[Proposition 4.10]{HZZ21a}, which in turn is similar to previous works such as \cite[Proposition 3.2]{HZZ19}, to which we refer interested readers for details (and I leave a detailed proof in the Appendix for Ourselves). 
\end{proof} 

\begin{lemma}\label{Lemma 4.3}
Let $\tau$ be a bounded stopping time of $(\mathcal{B}_{t})_{t\geq 0}$, $\xi^{\text{in}} \in \dot{H}_{\sigma}^{\frac{1}{2}}$, and $P$ be a martingale solution to \eqref{general msqg} where $G(v) dB = dB$, $B$ being a $GG^{\ast}$-Wiener process,  on $[0,\tau]$ with initial data $\xi^{\text{in}}$ at initial time 0 that satisfies Definition \ref{Definition 7}. Suppose that there exists a Borel set $\mathcal{N} \subset \Omega_{0,\tau}$ such that $P(\mathcal{N}) = 0$ and $Q_{\omega}$ from Lemma \ref{Lemma 4.2} satisfies for every $\omega \in \Omega_{0} \setminus \mathcal{N}$ 
\begin{equation}\label{est 226} 
Q_{\omega} (\{\omega' \in \Omega_{0}:\hspace{1mm}  \tau(\omega') = \tau(\omega) \}) = 1. 
\end{equation} 
Then the probability measure $P \otimes_{\tau}R \in \mathcal{P}(\Omega_{0})$ defined by 
\begin{equation}\label{est 225}  
P\otimes_{\tau} R (\cdot) \triangleq \int_{\Omega_{0}} Q_{\omega} (\cdot) P(d\omega) 
\end{equation} 
satisfies $P \otimes_{\tau}R  = P$ on the $\sigma$-algebra $\sigma \{\xi(t\wedge \tau): t \geq 0 \}$ and it is a martingale solution to \eqref{general msqg} where $G(v) dB = dB$, $B$ being a $GG^{\ast}$-Wiener process,  on $[0,\infty)$ with initial data $\xi^{\text{in}}$ at initial time 0. 
\end{lemma} 

\begin{proof}[Proof of Lemma \ref{Lemma 4.3}]
Due to the only difference from \cite[Lemma 4.3]{Y23a} being the (M3) property in the definition of the solution in Definitions \ref{Definition 6}-\ref{Definition 7}, it suffices to show that $P \otimes_{\tau} R$, as a martingale solution on $[0,\infty)$ with initial data $\xi^{\text{in}}$ at initial time 0, satisfies (M3). For any $q \in \mathbb{N}$, we first write using \eqref{est 226}-\eqref{est 225},
\begin{align}
\mathbb{E}^{P \otimes_{\tau} R} [ \lvert E^{q} (t) \rvert ] \leq&  2 \mathbb{E}^{P} [ \lvert E^{q} (t \wedge \tau) \rvert ]  \nonumber \\
&+\int_{\Omega_{0}} \mathbb{E}^{Q_{\omega}} [ \lvert E^{q} (t) - E^{q} (\tau(\omega)) \rvert 1_{\{\tau(\omega) \leq t\}} ] dP(\omega).\label{est 227}
\end{align} 
Within the second term of \eqref{est 227} we can estimate by using the (M3) property of $Q^{\omega}$ on $\{\tau(\omega) \leq t \}$ thanks to \eqref{est 224b} of Lemma \ref{Lemma 4.2}, 
\begin{align}
& \mathbb{E}^{Q_{\omega}} [ \lvert E^{q}(t) - E^{q} (\tau(\omega)) \rvert 1_{\{\tau(\omega) \leq t \}} ] \nonumber \\
\leq& 2 \mathbb{E}^{Q_{\omega}} \left[\left( \lVert \xi(\tau(\omega)) \rVert_{\dot{H}_{x}^{\frac{1}{2}}}^{2q} + C_{q} \lVert  G \rVert_{L_{2} (U, \dot{H}_{\sigma}^{\frac{1}{2}})}^{2} \int_{\tau(\omega)}^{t} \lVert \xi(r) \rVert_{\dot{H}_{x}^{\frac{1}{2}}}^{2q-2} dr \right) 1_{\{\tau(\omega) \leq t\}} \right] \label{est 127}.
\end{align}
Inductively, \eqref{est 127} leads to, for all $q \in \mathbb{N}$, 
\begin{align}
&\mathbb{E}^{Q_{\omega}} [ \lvert E^{q}(t) - E^{q} (\tau(\omega)) \rvert 1_{\{\tau(\omega) \leq t\}} ] \nonumber \\
 \leq& C \mathbb{E}^{Q_{\omega}} [ \lVert \xi(\tau(\omega)) \rVert_{\dot{H}_{x}^{\frac{1}{2}}}^{2q} 1_{\{\tau(\omega) \leq t\}} ] + C(t, \lVert G \rVert_{L_{2} (U, \dot{H}_{\sigma}^{\frac{1}{2}})}) \label{est 129}
\end{align} 
with $C$ and $C(t, \lVert G \rVert_{L_{2} (U, \dot{H}_{\sigma}^{\frac{1}{2}})}) $ independent of $\omega$. Hence, via \eqref{est 227}, \eqref{est 129}, and \eqref{est 124}, we obtain 
\begin{align}
& \mathbb{E}^{P \otimes_{\tau}R} [ \lvert E^{q} (t) \rvert ]  \leq C  \mathbb{E}^{P} \Big[ \lVert \xi(t \wedge \tau) \rVert_{\dot{H}_{x}^{\frac{1}{2}}}^{2q} + 2q \int_{0}^{t \wedge \tau} \lVert \xi(r) \rVert_{\dot{H}_{x}^{\frac{1}{2}}}^{2q-2} \lVert \xi(r) \rVert_{\dot{H}_{x}^{\frac{1}{2} + \iota}}^{2} dr  \label{est 130}\\
& \hspace{10mm} + C_{q} \lVert G \rVert_{L_{2} (U, \dot{H}_{\sigma}^{\frac{1}{2}})}^{2} \int_{0}^{t \wedge \tau} \lVert \xi(r) \rVert_{\dot{H}_{x}^{\frac{1}{2}}}^{2q-2} dr +  \lVert \xi(\tau) \rVert_{\dot{H}_{x}^{\frac{1}{2}}}^{2q} 1_{\{\tau \leq t\}}  \Big] + C(t, \lVert G \rVert_{L_{2}(U, \dot{H}_{\sigma}^{\frac{1}{2}})}), \nonumber
\end{align}
from which we can inductively conclude that  
\begin{equation}\label{est 131}
\mathbb{E}^{P \otimes_{\tau} R} [ \lvert E^{q} (t) \rvert ] \leq C \lVert \xi(0) \rVert_{\dot{H}_{x}^{\frac{1}{2}}}^{2q} + C(t, \lVert  G \rVert_{L_{2} (U, \dot{H}_{\sigma}^{\frac{1}{2}})}); 
\end{equation} 
i.e., $E^{q}$ is $P\otimes_{\tau} R$-integrable for all $q \in \mathbb{N}$. Next, let $t \geq s, A \in \mathcal{B}_{s}$ for $s \notin T_{Q_{\omega}}$ where $T_{Q_{\omega}}$ represents the exceptional set of $Q_{\omega}$ from Definition \ref{Definition 3}. First, if $t \geq \tau(\omega) \geq s$, then by \eqref{est 224b} and the (M3) property applied on $R_{\tau(\omega), \xi(\tau(\omega), \omega)}$, 
\begin{equation}\label{est 155} 
\mathbb{E}^{Q_{\omega}} [ \left(E^{q} (t) - E^{q} ( ( t \wedge \tau(\omega)) \vee s ) \right) 1_{A} ]  \leq 0. 
\end{equation} 
On the other hand, for $t \geq s \geq \tau(\omega)$, we also have due to \eqref{est 123}, 
\begin{equation}\label{est 132}
 \mathbb{E}^{Q_{\omega}} [ \left( E^{q} (t) - E^{q} ( ( t \wedge \tau(\omega)) \vee s) \right) 1_{A} ] \leq 0. 
\end{equation} 
Next, using \eqref{est 225} we write for any measurable set $D \subset [0, t]$, 
\begin{align}
& \int_{0}^{t} 1_{D} \mathbb{E}^{P\otimes_{\tau} R} [ (E^{q}(t) - E^{q}(s)) 1_{A} ] ds  \nonumber \\
=& \int_{0}^{t} 1_{D} \mathbb{E}^{P \otimes_{\tau} R} [ \left( E^{q}(t) - E^{q} (t\wedge \tau) + E^{q} ( t \wedge \tau) - E^{q} (s) \right) 1_{A \cap \{ \tau > s \}} ] ds  \nonumber \\
& +\int_{0}^{t} 1_{D} \mathbb{E}^{P} \left[ \mathbb{E}^{Q_{\omega}} [ \left( E^{q} (t) - E^{q} (s) \right) 1_{A \cap \{ \tau \leq s \}} ] \right] ds.  \label{est 133}
\end{align}
We will bound the right hand side of \eqref{est 133} from above by zero so that following the argument of \cite[p. 441]{FR08} we can conclude that for any fixed $t > 0$, there exists $T_{t} \subset (0,t)$ of null Lebesgue measure such that for all $s \notin T_{t}$, 
\begin{equation}\label{est 159} 
\mathbb{E}^{P \otimes_{\tau} R} [ E^{q} (t) - E^{q} (s) \lvert \mathcal{B}_{s} ] \leq 0 
\end{equation}
which implies that $\{E^{q}(t)\}_{t\geq 0}$ is an a.s. $(\mathcal{B}_{t})_{t\geq 0}$-supermartingale and hence an a.s. $(\mathcal{B}_{t}^{0})_{t\geq 0}$-supermartingale under under $P \otimes_{\tau}R$ as a consequence of Lemma \ref{Lemma 2.1} and hence concludes the proof of (M3).

Hence, we now focus on bounding \eqref{est 133} from above by 0, which follow from applying these inequalities into \eqref{est 133}: 
\begin{subequations}\label{est 276}
\begin{align}
&  \int_{0}^{t} 1_{D} \mathbb{E}^{P\otimes_{\tau} R} [ \left( E^{q}(t) - E^{q} (t \wedge \tau) \right) 1_{A \cap \{ \tau > s \}} ] ds \overset{\eqref{est 225} \eqref{est 226} \eqref{est 155} }{\leq} 0,   \label{est 276a}\\
& \mathbb{E}^{P \otimes_{\tau} R} [ \left(E^{q} ( t \wedge \tau) - E^{q} (s) \right) 1_{A \cap \{ \tau > s \}} ]  \nonumber \\
& \hspace{10mm} = \mathbb{E}^{P} [ \left(E^{q} (t \wedge \tau) - E^{q} (s) \right) 1_{A \cap \{ \tau > s \} } ] \leq 0, \label{est 276b}\\
&  \int_{0}^{t} 1_{D} \mathbb{E}^{P} \left[ \mathbb{E}^{Q_{\omega}} [ \left( E^{q} (t) - E^{q}(s) \right) 1_{A \cap \{ \tau \leq s \}} ] \right] ds \overset{\eqref{est 226} \eqref{est 132}}{\leq}0,  \label{est 276c}
\end{align}
\end{subequations}
where \eqref{est 276b} relied on Lemma \ref{Lemma 4.3}. This completes the proof of Lemma \ref{Lemma 4.3}. 
\end{proof} 

\subsection{Non-uniqueness in Law Globally-in-Time: Additive Case}\label{Section 4.2}
We combine the general results from Section \ref{Section 4.1} and Theorem \ref{Theorem 1.1} which constructed a solution up to a stopping time via convex integration. We fix a $GG^{\ast}$-Wiener process $B$ defined on $(\Omega, \mathcal{F}, \mathbf{P})$ and denote by $(\mathcal{F}_{t})_{t\geq 0}$ its normal filtration. For every $\omega \in \Omega_{0}$ we define processes (see \cite[Equation (71)]{Y23a}) 
\begin{subequations}\label{est 185}
\begin{align}
& M_{t,0}^{\omega} \triangleq \omega(t) - \omega(0) + \int_{0}^{t} \mathbb{P}  \left( (\Lambda \omega \cdot \nabla) \omega \right) (r) - \mathbb{P} ( ( \nabla \omega)^{T} \cdot \Lambda \omega ) (r) + \Lambda^{\gamma} \omega(r) dr,  \label{est 185a}\\
& Z^{\omega} (t) \triangleq M_{t,0}^{\omega} - \int_{0}^{t}  \Lambda^{\frac{3}{2} - 2 \sigma} e^{- (t-r) \Lambda^{\frac{3}{2} - 2\sigma}} M_{r,0}^{\omega} dr. \label{est 185b}
\end{align}
\end{subequations}
If $P$ is a martingale solution to \eqref{additive msqg}, then the mapping $\omega \mapsto M_{t,0}^{\omega}$ is a $GG^{\ast}$-Wiener process under $P$ and we can show that (cf. \cite[Equation (72)]{Y23a}) 
\begin{equation}\label{new 62}
Z(t) = \int_{0}^{t} e^{-\Lambda^{\frac{3}{2} - 2 \sigma}(t-r)}  dM_{r,0}. 
\end{equation}  
Based on Proposition \ref{prop solution z}, similarly to \eqref{additive stopping time}, we define for $\lambda \in \mathbb{N}$ 
\begin{align}
&\tau^{\lambda} (\omega) \triangleq 1 \wedge \inf \left\{ t \geq 0: C_{S} \lVert Z^{\omega} (t) \rVert_{\dot{H}_{x}^{\frac{5}{2} + 2 \sigma}} \geq 1 - \frac{1}{\lambda} \right\}  \label{est 146} \\
\wedge& \inf\left\{ t \geq 0: C_{S} \lVert Z^{\omega} \rVert_{C_{t}^{\frac{1}{2} - 2 \delta} \dot{H}_{x}^{\frac{7}{4} + 4 \sigma} } \geq 1 - \frac{1}{\lambda} \right\} \nonumber \\
\wedge& \inf\big\{ t \geq 0: C_{0} \left( C_{S} \lVert Z^{\omega}(t) \rVert_{\dot{H}_{x}^{\frac{3}{2}+ \frac{\sigma}{2}}}^{2} + 2 \lVert Z^{\omega}(t) \rVert_{\dot{H}_{x}^{\frac{3}{2} - \sigma}}^{2} \right) \geq \frac{ \varepsilon_\gamma}{C_{S} 32 (2\pi)^{2}} a^{(1- 2 \beta) b (b-1)}  \ushort{e} - \frac{1}{\lambda} \big\}.\nonumber
\end{align}   
It follows that $\{\tau^{\lambda}\}_{\lambda \in \mathbb{N}}$ is non-decreasing, and we can define 
\begin{equation}\label{est 141}
\tau \triangleq \lim_{\lambda \to \infty} \tau^{\lambda}. 
\end{equation} 
By \cite[Lemma 3.5]{HZZ19}, $\tau^{\lambda}$ is a $(\mathcal{B}_{t})_{t\geq 0}$-stopping time and consequently, so is $\tau$. We denote by 
\begin{equation}\label{est 142} 
P \triangleq \mathcal{L}(v) 
\end{equation}  
where $v$ is the solution to \eqref{additive msqg} constructed by Theorem \ref{Theorem 1.1} with a specific choice of 
\begin{equation}\label{est 134} 
e(t) \triangleq d_{0} + d_{1} t, \hspace{3mm} d_{1} > 0, \hspace{3mm} d_{0} > 2d_{1} + 4.
\end{equation}  

\begin{proposition}\label{Proposition 4.4} 
Define $\tau$ by \eqref{est 141}. Then the probability measure $P$ defined by \eqref{est 142} is a martingale solution to \eqref{additive msqg} on $[0,\tau]$ according to Definition \ref{Definition 7}. The corresponding constants $C_{q}$ for $q \in \mathbb{N}$ in (M3) of Definition \ref{Definition 7} depend only on $d_{0}$, $d_{1}$, and $\lVert G \rVert_{L_{2} (U, \dot{H}_{\sigma}^{\frac{1}{2}})}$. 
\end{proposition} 

\begin{proof}[Proof of Proposition \ref{Proposition 4.4}]
In accordance with Definition \ref{Definition 7}, we fix the constant $\iota \in (0, 1)$ from Theorem \ref{Theorem 1.1}. An argument that is identical to the proof of \cite[Proposition 4.5]{Y23a} shows that $\tau(v) = \mathfrak{t}$ $\mathbf{P}$-a.s. (see \cite[Equation (293)]{Y23a}) and that $P$ satisfies (M1) and (M2), where $\mathfrak{t}$ was defined in \eqref{additive stopping time}. Thus, we only need to verify (M3) to complete the proof of the first claim. We observe that there exists $d_{2} > 0$ such that for all $t \in [0,\tau]$, under $P$, 
\begin{equation}\label{est 137} 
\lVert \xi(t) \rVert_{\dot{H}_{x}^{\frac{1}{2}}}^{2} \overset{\eqref{energy} \eqref{est 134}}{=} d_{0} + d_{1} t  \hspace{2mm} \text{ and } \hspace{2mm}  \lVert \xi(t) \rVert_{\dot{H}_{x}^{\frac{1}{2} + \iota}}^{2} \overset{\eqref{est 135} \eqref{est 136} \eqref{est 41}}{\leq}d_{2}. 
\end{equation}  
Consequently, we can show that for 
\begin{equation}\label{est 138} 
C_{1} \geq \frac{d_{1} + 2d_{2}}{\lVert  G \rVert_{L_{2} (U, \dot{H}_{\sigma}^{\frac{1}{2}})}^{2}}, 
\end{equation} 
and all $t \geq s \geq 0$, $P$-a.s. 
\begin{align*}
E^{1} (t \wedge \tau) - E^{1} (s\wedge \tau)  \leq \left( d_{1} + 2d_{2} - C_{1} \lVert G \rVert_{L_{2} (U, \dot{H}_{\sigma}^{\frac{1}{2}})}^{2} \right) (t \wedge \tau - s \wedge \tau) \overset{\eqref{est 138}}{\leq} 0; 
\end{align*}
thus, $\{E^{1}(t \wedge \tau) - E^{1}(s \wedge \tau)\}_{t \geq s}$ is a supermartingale under $P$. For general $q \in \mathbb{N}$, if 
\begin{equation}\label{est 140}
C_{q} \geq \frac{ q(d_{1} + 2d_{2})}{ \lVert G \rVert_{L_{2} (U, \dot{H}_{\sigma}^{\frac{1}{2}})}^{2}}, 
\end{equation} 
then we can compute similarly 
 \begin{align*}
& E^{q} (t \wedge \tau) - E^{q} (s \wedge \tau) \\
\overset{\eqref{est 124} \eqref{est 137}}{\leq}& \left( [ d_{0} + d_{1} (t \wedge \tau) ]^{q} - [d_{0} + d_{1} (s \wedge \tau) ]^{q} \right) \left(1 + \frac{ 2qd_{2} - C_{q} \lVert G \rVert_{L_{2} (U, \dot{H}_{\sigma}^{\frac{1}{2}})}^{2}}{qd_{1}} \right) \overset{\eqref{est 140}}{\leq} 0, 
\end{align*}
which shows that $\{E^{q}(t \wedge \tau) - E^{q}(s\wedge \tau) \}_{t \geq s}$ is a supermartingale under $P$. This completes the proof of Proposition \ref{Proposition 4.4} as the second claim of Proposition \ref{Proposition 4.4} follows from \eqref{est 138}  and \eqref{est 140}. 
\end{proof} 

Finally, the following Proposition \ref{Proposition 4.5} proves non-uniqueness in law of the global-in-time solution to \eqref{additive msqg}. 
\begin{proposition}\label{Proposition 4.5} 
\noindent
\begin{enumerate}
\item The probability measure $P \otimes_{\tau} R$ defined by \eqref{est 225} with the $\tau$ defined by \eqref{est 141} is a martingale solution to \eqref{additive msqg} on $[0,\infty)$ according to Definition \ref{Definition 6}. 
\item There exist constant $C_{q} \geq 0$ for $q \in \mathbb{N}$ such that the martingale solutions to \eqref{additive msqg} associated to such $C_{q}$ in (M3) of Definition \ref{Definition 6} are not unique.  
\end{enumerate} 
\end{proposition} 

\begin{proof}[Proof of Proposition \ref{Proposition 4.5}] 
$ $\newline
(1) The fact that $P \otimes_{\tau} R$ is a martingale solution to  \eqref{additive msqg} on $[0,\infty)$ according to Definition \ref{Definition 6} can be proven identically to the proof of \cite[Proposition 4.6]{Y23a}; the difference in the definition of the solution in Definition \ref{Definition 6} and \cite[Definition 4.1]{Y23a}, namely (M3), will not directly play any role here as the proof consists of the verification of \eqref{est 226}. 

\noindent (2) Following the proof of \cite[Proposition 4.13]{HZZ21a} we are able to provide two proofs here. In fact, the first proof is in the spirit of the proof of the non-uniqueness statement of Theorem \ref{Theorem 1.1} and starts by defining two distinct energies 
\begin{align*}
e_{1}(t) \triangleq d_{0} + d_{1,1} t, \hspace{3mm} e_{2}(t) \triangleq d_{0 }+ d_{1,2}t, \hspace{3mm} \text{ such that } d_{1,1} \neq d_{1,2}, 
\end{align*}
where $d_{0}$ and both $d_{1,1}, d_{1,2}$ satisfy the respective roles of $d_{0}$ and $d_{1}$ in \eqref{est 134}. Then we construct the corresponding convex integration solutions $P_{1}$ and $P_{2}$ with the same constants $C_{q}$ for $q \in \mathbb{N}$ in (M3) of Definition \ref{Definition 7} via Theorem \ref{Theorem 1.1} and extend both to $[0,\infty)$ via $P_{1} \otimes_{\tau} R$ and $P_{2} \otimes_{\tau} R$ using Proposition \ref{Proposition 4.5} (1), respectively. It follows that they have same initial data but different energies (recall the last statement of Theorem \ref{Theorem 1.1}). 

In the second proof, we choose 
\begin{equation}\label{est 143} 
e(t) \triangleq d_{0} + d_{1} t \text{ such that } d_{1} > \lVert G \rVert_{L_{2} (U, \dot{H}_{\sigma}^{\frac{1}{2}})}^{2} 
\end{equation} 
where $d_{0}, d_{1}$ satisfy \eqref{est 134}, and construct a solution $P$ on $[0, \tau]$ such that 
\begin{equation}\label{est 144} 
\lVert \xi(t) \rVert_{\dot{H}_{x}^{\frac{1}{2}}}^{2} = e(t) \overset{\eqref{est 143}}{=} d_{0} + d_{1} t \hspace{2mm} P\text{-a.s. for all }t \in [0,\tau] 
\end{equation} 
via Theorem \ref{Theorem 1.1}. We notice that taking $t = 0$ in \eqref{est 144} implies $\lVert \xi(0) \rVert_{\dot{H}_{x}^{\frac{1}{2}}}^{2} =d_{0}$ where we recall that $\xi(0)$ is deterministic due to Theorem \ref{Theorem 1.1}. Then we rely on Proposition \ref{Proposition 4.5} (1) to obtain solution $P \otimes_{\tau} R$ on $[0,\infty)$ so that  
\begin{equation}\label{est 147}
\mathbb{E}^{P \otimes_{\tau} R} [ \lVert \xi(\tau) \rVert_{\dot{H}_{x}^{\frac{1}{2}}}^{2} ] \overset{\eqref{est 144}}{=} \mathbb{E}^{P\otimes_{\tau} R} [ d_{0} + d_{1} \tau ] = d_{0} + d_{1} \mathbb{E}^{P \otimes_{\tau} R} [ \tau], 
\end{equation}
where we used that $P \otimes_{\tau}R  = P$ on the $\sigma$-algebra $\sigma \{\xi(t\wedge \tau): t \geq 0 \}$ from Lemma \ref{Lemma 4.3}. On the other hand, if $Q$ is the law of the classical solution to \eqref{additive msqg} constructed via Galerkin approximation from $\xi(0)$, then we see that due to \eqref{est 143} and \eqref{est 144} 
\begin{equation}\label{est 148}
 \mathbb{E}^{Q} [ \lVert \xi(\tau) \rVert_{\dot{H}_{x}^{\frac{1}{2}}}^{2} ] + 2 \mathbb{E}^{Q} [ \int_{0}^{\tau} \lVert \xi \rVert_{\dot{H}_{x}^{\frac{1}{2} + \frac{\gamma}{2}}}^{2} dr]  < d_{0} + d_{1} \mathbb{E}^{Q} [ \tau].
\end{equation} 
Now $\tau = \lim_{\lambda \to \infty} \tau^{\lambda}$ due to \eqref{est 141} and $\tau^{\lambda}$ in \eqref{est 146} is defined as a function of time-dependent norms of $Z$. Due to \eqref{new 62}, \eqref{est 185a}, and \eqref{est 38}, we see that $Z$ solves a same linear equation as \eqref{est 42} under both $P \otimes_{\tau} R$ and $Q$ so that the laws of $\tau$ under $Q$ and $P \otimes_{\tau}R$ are same.  This allows us to conclude the claimed non-uniqueness in law due to 
\begin{equation}
\mathbb{E}^{P \otimes_{\tau} R} [ \lVert \xi(\tau) \rVert_{\dot{H}_{x}^{\frac{1}{2}}}^{2} ] \overset{\eqref{est 147}}{=} d_{0} + d_{1} \mathbb{E}^{P\otimes_{\tau} R} [\tau] \overset{\eqref{est 148}}{>} \mathbb{E}^{Q} [ \lVert \xi(\tau) \rVert_{\dot{H}_{x}^{\frac{1}{2}}}^{2} ]. 
\end{equation} 
Finally, we can find a common choice of $C_{q}, q \in \mathbb{N}$, namely the maximum, such that all probability measures $P_{1} \otimes_{\tau} R, P_{2} \otimes_{\tau} R, P \otimes_{\tau} R$, and $Q$ are all corresponding martingale solutions to \eqref{additive msqg}. 
This completes the proof of Proposition \ref{Proposition 4.5}. 
\end{proof}

\subsection{Non-uniqueness of a.s. Markov Selection: Additive Case} 

We fix the constants $C_{q}, q \in \mathbb{N}$, of the (M3) property. We gave the definition of an a.s. Markov family in Definition \ref{Definition 5}. We recall the definition of $\mathcal{C} (\xi^{\text{in}}, C_{q})$ from Definition \ref{Definition 6} and leave the definition of an a.s. pre-Markov family in Definition \ref{Definition 10}.  Due to Lemma \ref{Lemma 7.3}, the proof of the existence of an a.s. Markov selection of $\{\mathcal{C}(\xi^{\text{in}}, C_{q})\}_{\xi^{\text{in}} \in \dot{H}_{\sigma}^{\frac{1}{2}}}$ follows once we prove the Disintegration and Reconstruction properties of Definition \ref{Definition 10}, which is the content of next two propositions. 
 
 \begin{proposition}\label{Proposition 4.6}
There exists a Lebesgue null set $\mathcal{T} \subset (0,\infty)$ such that for all $T \not\in \mathcal{T}$, the family $\{ \mathcal{C} ( \xi^{\text{in}}, C_{q}) \}_{\xi^{\text{in}} \in \dot{H}_{\sigma}^{\frac{1}{2}}}$ satisfies the disintegration property in Definition \ref{Definition 10}. 
 \end{proposition} 
 
\begin{proof}[Proof of Proposition \ref{Proposition 4.6}]
This proof is similar to those of \cite[Lemma 4.20]{HZZ21a}, \cite[Lemma 6.10]{HZZ20}, and \cite[Lemma 4.4]{FR08}; we sketch it for completeness. 

We fix $\xi^{\text{in}} \in \dot{H}_{\sigma}^{\frac{1}{2}}, P \in \mathcal{C}(\xi^{\text{in}}, C_{q})$, and $T$ as a regular time of $P$. By Lemma \ref{Lemma 2.2} we may let $P(\cdot \lvert \mathcal{B}_{T}^{0})(\omega)$ be a r.c.p.d. of $P$ w.r.t. $\mathcal{B}_{T}^{0}$. According to Definition \ref{Definition 10} (1), we need to find a $P$-null set $N \in \mathcal{B}_{T}^{0}$ such that for all $\omega \notin N$, 
\begin{equation}
\xi(T, \omega) \in \dot{H}_{\sigma}^{\frac{1}{2}} \hspace{2mm} \text{ and } \hspace{2mm} P(\Phi_{T} (\cdot) \lvert \mathcal{B}_{T}^{0}) (\omega) \in \mathcal{C} ( \xi(T, \omega), C_{q}). 
\end{equation} 
The first claim follows from the fact that $T$ is a  regular time of $P$. To verify the second claim, by definition of the shift operator $\Phi_{T}$, it is equivalent to demonstrating that $P(\cdot \lvert \mathcal{B}_{T}^{0})(\omega)$ satisfies (M1)-(M3) with initial data $\xi(T, \omega)$ at initial time $T$.  For (M1), by Lemma \ref{Lemma 2.2} (3) there exists a $P(\cdot \lvert \mathcal{B}_{T}^{0})$-null set $N_{1,1}$ outside which $P(\cdot \lvert \mathcal{B}_{T}^{0})$ has the correct initial data at initial time $T$. Moreover, setting 
\begin{subequations}\label{est 178}
\begin{align}
& S_{T} \triangleq \{ \omega \in \Omega_{0}: \omega \rvert_{[0,T]} \in L^{\infty} ([0,T]; \dot{H}_{\sigma}^{\frac{1}{2}}) \cap L^{2} ([0,T]; \dot{H}_{x}^{\frac{1}{2} + \iota} ) \}, \\
& S^{T} \triangleq \{ \omega \in \Omega_{0}: \omega \rvert_{[T, \infty)} \in L_{\text{loc}}^{\infty} ([T, \infty); \dot{H}_{\sigma}^{\frac{1}{2}}) \cap L_{\text{loc}}^{2} ([T, \infty); \dot{H}_{x}^{\frac{1}{2} + \iota}) \},
\end{align}
\end{subequations} 
we can compute using Lemma \ref{Lemma 2.2} (4) that 
\begin{align*}
1 = \int_{S_{T}} P (S^{T}\lvert \mathcal{B}_{T}^{0}) (\tilde{\omega}) dP(\tilde{\omega})
\end{align*}
so that there exists $P(\cdot \lvert \mathcal{B}_{T}^{0})$-null set $N_{1,2}$ outside which $S^{T}$ occurs. We let $N_{1} \triangleq N_{1,1} \cup N_{1,2}$.  

The property (M2) can be proven by following the same argument on \cite[p. 424]{FR08}. We can rely on the hypothesis that $P \in \mathcal{C} ( \xi^{\text{in}}, C_{q})$ so that (M2) holds for $P$ and apply \cite[Proposition B.1]{FR08} (and also \cite[Lemma B.3]{GRZ09}) to find a $P$-null set $N_{2} \in \mathcal{B}_{T}^{0}$ outside which $P(\Phi_{T} (\cdot) \lvert \mathcal{B}_{T}^{0})(\omega)$ satisfies (M2). 

To verify the property (M3), we define for $q \in \mathbb{N}$
\begin{align*}
\alpha_{t}^{q} \triangleq \lVert \xi(t) \rVert_{\dot{H}_{x}^{\frac{1}{2}}}^{2q} + 2q \int_{0}^{t} \lVert \xi(r) \rVert_{\dot{H}_{x}^{\frac{1}{2}}}^{2q-2} \lVert \xi(r) \rVert_{\dot{H}_{x}^{\frac{1}{2} + \iota}}^{2} dr, \hspace{1mm} \beta_{t}^{q} \triangleq C_{q} \lVert  G \rVert_{L_{2} (U, \dot{H}_{\sigma}^{\frac{1}{2}})}^{2} \int_{0}^{t} \lVert \xi(r) \rVert_{\dot{H}_{x}^{\frac{1}{2}}}^{2q-2} dr
\end{align*} 
and apply \cite[Proposition B.4]{FR08} inductively to find a $P$-null set $N_{3}$ such that for all $\omega \notin N_{3}$, $\{E_{t}^{q}\}_{t \geq T}$ is an a.s. $(\mathcal{B}_{t}^{0})_{t\geq T}$-supermartingale under $P(\cdot\lvert \mathcal{B}_{T}^{0})(\omega)$. The proof is now complete by taking $N = \cup_{j=1}^{3} N_{j}$.  
\end{proof}

\begin{proposition}\label{Proposition 4.7} 
There exists a Lebesgue null set $\mathcal{T} \subset (0,\infty)$ such that for all $T \not\in \mathcal{T}$, the family $\{ \mathcal{C} ( \xi^{\text{in}}, C_{q}) \}_{\xi^{\text{in}} \in \dot{H}_{\sigma}^{\frac{1}{2}}}$ satisfies the reconstruction property of the Definition \ref{Definition 10}.  
\end{proposition} 

\begin{proof}[Proof of Proposition \ref{Proposition 4.7}]
We fix $\xi^{\text{in}} \in \dot{H}_{\sigma}^{\frac{1}{2}}$ and $P \in \mathcal{C} ( \xi^{\text{in}}, C_{q})$, and let $T$ be a regular time of $P$. We assume that a mapping $\mathcal{E}: \Omega_{0} \mapsto \mathcal{P}_{\dot{H}_{\sigma}^{\frac{1}{2}}} (\Omega_{0})$ defined by $\mathcal{E}(\omega) \triangleq Q_{\omega}$ satisfies the hypothesis of Lemma \ref{Lemma 2.3} and that there exists a $P$-null set $N \in \mathcal{B}_{T}^{0}$ such that for all $\omega \notin N$ we have $\xi(T, \omega) \in \dot{H}_{\sigma}^{\frac{1}{2}}$ and $Q_{\omega} \circ \Phi_{T} \in \mathcal{C} ( \xi(T, \omega), C_{q})$. According to Definition \ref{Definition 10} (2), we only need to show that $P \otimes_{T} Q \in \mathcal{C} (\xi^{\text{in}}, C_{q})$ and this can be shown by the same argument to the proof of Lemma \ref{Lemma 4.3} with ``$\tau$'' therein replaced by $T$.  
\end{proof} 

\begin{proof}[Proof of Theorem \ref{Theorem 1.2}]
$ $\newline
(1) This part is proven thanks to Proposition \ref{Proposition 4.5}. \\ 
\noindent 
(2) As a consequence of Proposition \ref{Proposition 4.5}, we know that there exists an initial data $\xi_{\ast}^{\text{in}} \in \dot{H}_{\sigma}^{\frac{1}{2}}$ that leads to at least two martingale solutions $P$ and $Q$ on $[0,\infty)$ and a functional $f: \dot{H}_{\sigma}^{\frac{1}{2}} \mapsto \mathbb{R}$ that is continuous and bounded such that 
\begin{equation}\label{est 180}
\mathbb{E}^{P} \left[ \int_{0}^{\infty} e^{-\lambda t} f(\xi(t)) dt \right] > \mathbb{E}^{Q} \left[ \int_{0}^{\infty} e^{-\lambda t} f(\xi(t)) dt \right], \hspace{3mm} \lambda > 0, 
\end{equation} 
similarly to the proof of \cite[Theorem 12.2.4]{SV97}. As a consequence of Propositions \ref{Proposition 4.6}-\ref{Proposition 4.7},  $\{ \mathcal{C} (\xi^{\text{in}}, C_{q}) \}_{\xi^{\text{in}} \in \dot{H}_{\sigma}^{\frac{1}{2}}}$ is an a.s. pre-Markov family so that we can apply Lemma \ref{Lemma 7.3}. We use Lemma \ref{Lemma 7.3} (1) with the $f$ from \eqref{est 180} to construct one a.s. Markov selection $\{ P_{\xi^{\text{in}}}^{+} \}_{\xi^{\text{in}} \in\dot{H}_{\sigma}^{\frac{1}{2}}}$ of $\{ \mathcal{C} (\xi^{\text{in}}, C_{q}) \}_{\xi^{\text{in}} \in \dot{H}_{\sigma}^{\frac{1}{2}}}$ and use $-f$ instead to construct another a.s. Markov selection $\{ P_{\xi^{\text{in}}}^{-} \}_{\xi^{\text{in}} \in\dot{H}_{\sigma}^{\frac{1}{2}}}$ of $\{ \mathcal{C} (\xi^{\text{in}}, C_{q}) \}_{\xi^{\text{in}} \in \dot{H}_{\sigma}^{\frac{1}{2}}}$ for which Lemma \ref{Lemma 7.3} (2) justifies the first and third inequalities below: 
\begin{align} 
&\mathbb{E}^{P_{\xi_{\ast}^{\text{in}}}^{-}}\left[ \int_{0}^{\infty} e^{-\lambda t} f(\xi(t)) dt \right] \leq \mathbb{E}^{Q} \left[ \int_{0}^{\infty} e^{-\lambda t} f(\xi(t))dt \right]  \nonumber \\
\overset{\eqref{est 180}}{<}& \mathbb{E}^{P} \left[ \int_{0}^{\infty} e^{-\lambda t} f(\xi(t)) dt \right] \leq \mathbb{E}^{P_{\xi_{\ast}^{\text{in}}}^{+}} \left[ \int_{0}^{\infty} e^{-\lambda t} f(\xi(t)) dt \right].  \label{est 237} 
\end{align}
Therefore, the two a.s. Markov selections $\{ P_{\xi^{\text{in}}}^{+} \}_{\xi^{\text{in}} \in\dot{H}_{\sigma}^{\frac{1}{2}}}$ and $\{ P_{\xi^{\text{in}}}^{-} \}_{\xi^{\text{in}} \in\dot{H}_{\sigma}^{\frac{1}{2}}}$ are different and this completes the proof of Theorem \ref{Theorem 1.2}. 
\end{proof} 

\section{Proof of Theorem \ref{Theorem 1.3}}\label{Section 5}

\subsection{Setup}\label{mult setup}
We begin by fixing the energy $e: \mathbb{R} \to [\ushort{e},\infty)$ that satisfies \eqref{est 118 mult}. 
\subsubsection{\except{toc}{Formation of }Random PDE and Stopping Time}
We fix an $\mathbb{R}$-valued Wiener process $B$ on $(\Omega, \mathcal{F}, \textbf{P})$ with $(\mathcal{F}_{t})_{t\geq 0}$ as its normal filtration. We apply It$\hat{\mathrm{o}}$'s formula to \eqref{mult msqg} with $y(t)$ defined by \eqref{define upsilon} to obtain
\begin{align}
&\partial_{t} y + \frac{1}{2} y + \Upsilon [(\Lambda y \cdot \nabla) y - (\nabla y)^{T} \cdot \Lambda y] \nonumber \\
& \hspace{25mm} + \nabla (\Upsilon^{-1} p) + \Lambda^{\gamma} y = 0, \hspace{3mm} \nabla\cdot y = 0, \hspace{3mm} \text{ for } t > 0. \label{est 265}
\end{align}
Then for all $q \in \mathbb{N}_{0}$ we add our error term $\divergence \mathring{R}_{q}$, where the Reynolds stress $\mathring{R}_{q}$ is again symmetric and trace-free, thus leading us to consider for all $q \in \mathbb{N}_{0}$, 
\begin{align}
\partial_{t} y_{q} + \frac{1}{2} y_{q} + \Upsilon [( \Lambda y_{q} \cdot \nabla) y_{q} - (\nabla y_{q})^{T} \cdot \Lambda y_{q} ]  + \nabla p_{q} + \Lambda^{\gamma} y_{q} = \text{div} \mathring{R}_{q},  \hspace{2mm} \nabla\cdot y_{q} = 0.  \label{mult convPDE}
\end{align}
We now fix $L > 1$ sufficiently large. We consider the same parameters from Section \ref{choice of parameters} while the choice of $b$ now depends on the fixed $L$: we take 
\begin{equation}\label{mult bound b}
b > \max\left\{\frac{6L}{\alpha - \frac{1}{2}}, \frac{4L}{3-2\alpha}, \frac{L}{\frac{3}{2} - \gamma} \right\}.  
\end{equation}
For example, the three lower bounds are necessarily used respectively in \eqref{est 191}, \eqref{est 190}, and \eqref{est 192}. By the same reasoning  of Section \ref{additive setup}, we can assume that \eqref{est 1} holds. For any $\delta\in (0,\frac{1}{4})$, we fix our stopping time 
\begin{equation}\label{mult stopping time}
T_{L} \triangleq L \wedge \inf\{t \geq 0: \hspace{0.5mm}  \lvert B(t) \rvert \geq L^{\frac{1}{4}} \} \wedge \inf\{t \geq 0: \hspace{0.5mm} \lVert B \rVert_{C_{t}^{\frac{1}{2} - 2\delta}} \geq L^{\frac{1}{2}} \}.
\end{equation}
\begin{remark}\label{Remark 5.1}
We point out that this stopping time is significantly different from an analog of the additive case \eqref{additive stopping time}. In the additive case, the nonlinear term at the step $q = 0$ would not vanish even if $y_{0} \equiv 0$ and thus the  need for the smallness of the norm of $z_{0}$. In contrast, the nonlinear term in the linear multiplicative case at step $q = 0$ vanishes when $y_{0} \equiv 0$. 
\end{remark}

In addition to the differences described in Remark \ref{Remark 5.1}, another difference between $\mathfrak{t}$ in \eqref{additive stopping time} and $T_{L}$ of \eqref{mult stopping time} is that 
\begin{equation}\label{limit as l to infinity}
\lim_{L\to\infty}  T_{L} = \infty \hspace{3mm} \mathbf{P}\text{-a.s.}
\end{equation}
while $\mathfrak{t}$ in the additive case was bounded by 1. Additionally, we define a new parameter that depends on $L$:  
\begin{equation}\label{est 73}
m_{L} \triangleq \sqrt{3} L^{\frac{1}{4}} e^{\frac{1}{2} L^{\frac{1}{4}}}.  
\end{equation}
Now, due to how we defined our stopping time, we can determine some bounds on $\Upsilon$ defined in \eqref{define upsilon}.
First, we extend $B$ to $[-4, 0)$ by the value at $t = 0$ so that 
\begin{equation*}
\lVert B \rVert_{C_{t,0}} \leq L^{\frac{1}{4}}  \text{ and } \lVert B \rVert_{C_{t,0}^{\frac{1}{2} - 2 \delta}} \leq L^{\frac{1}{2}} \hspace{3mm} \forall \hspace{1mm} t \in [-2, T_{L}]. 
\end{equation*}
Given this, the following bounds on $\Upsilon$ are available and useful:
\begin{subequations}\label{est 196}
\begin{align}
& \lvert\Upsilon(t)\rvert\leq e^{L^\frac{1}{4}}, \hspace{21mm}  \lvert\Upsilon^{-1}(t)\rvert\leq e^{L^{\frac{1}{4}}}, \hspace{12mm}   \lVert \Upsilon_{l}^{-\frac{1}{2}} \rVert_{C_{t,0}} \leq e^{\frac{1}{2} L^{\frac{1}{4}}} \label{est 196a}\\
&\lVert \Upsilon \rVert_{C_{t,0}^{\frac{1}{2} - 2 \delta}} \leq e^{L^{\frac{1}{4}}}L^{\frac{1}{2}} \leq m_{L}^{2},\hspace{2mm}   \lVert \Upsilon^{-1} \rVert_{C_{t,0}^{\frac{1}{2} - 2 \delta}} \leq e^{L^{\frac{1}{4}}}L^{\frac{1}{2}} \leq m_{L}^{2}.\label{est 196b}
\end{align}
\end{subequations} 

\subsubsection{Mollification}
We take the same mollifiers from the additive case (Section \ref{additive mollification}), and mollify $y_{q}$ and $\mathring{R}_{q}$ identically to \eqref{additive mollified} while additionally $\Upsilon$ from \eqref{define upsilon} in time to obtain 
\begin{equation}\label{mult mollify c}
\Upsilon_{l} \triangleq \Upsilon \ast_{t} \varphi_{l}. 
\end{equation} 
We redefine $D_{t,q}$ from \eqref{est 56} in the linear multiplicative case to instead be
\begin{equation}\label{est 57} 
D_{t,q} \triangleq \partial_{t} + \Upsilon_{l} \Lambda y_{l} \cdot \nabla.
\end{equation} 

\subsection{Inductive Hypothesis}
We require at each step $q$ that for all $t\in[t_q,T_{L}]$, the solution $(y_{q}, \mathring{R}_{q})$ to \eqref{mult convPDE} satisfies 
\begin{subequations}\label{mult inductive 1}
\begin{align}
&\supp \hat{y}_q\subset B(0,2\lambda_q), \label{mult inductive 1a}\\
&||y_q||_{C_{t,x,q}}\leq M_{0}^{\frac{1}{2}} \left(1+\sum_{1\leq j\leq q}\delta_j^{\frac{1}{2}} \right)m_{L}^{4} \bar{e}^{\frac{1}{2}}\leq 3M_{0}^{\frac{1}{2}}m_{L}^{4}\bar{e}^{\frac{1}{2}}, \label{mult inductive 1b}\\
&||y_q||_{C_{t,q}C^1_x}+||\Lambda y_q||_{C_{t,x,q}}\leq M_{0}^{\frac{1}{2}} m_{L}^{4} \lambda_q\delta_q^{\frac{1}{2}}\bar{e}^{\frac{1}{2}}, \label{mult inductive 1c}\\
&|| D_{t,q} y_q||_{C_{t,x,q}}\leq M_{0}\lambda_q^2\delta_qm_{L}^{8}e^{L^{\frac{1}{4}}} \bar{e}, \label{mult inductive 1d}
\end{align}
\end{subequations}
\begin{subequations}\label{mult inductive 2}
\begin{align}
&\supp{\widehat{\mathring{R}}_q}\subset B(0,4\lambda_q), \label{mult inductive 2a}\\
&||\mathring{R}_q||_{C_{t,x,q}}\leq \frac{\varepsilon_\gamma e^{-3 L^{\frac{1}{4}}}}{32(2\pi)^2}\lambda_{q+2}\delta_{q+2} e(t); \label{mult inductive 2b}
\end{align}
\end{subequations}
finally, to control the energy, we impose 
\begin{equation}\label{mult inductive 3}
\frac{3}{4}\lambda_{q+1}\delta_{q+1}e(t) \Upsilon^{-2}(t)\leq e(t) \Upsilon^{-2} (t) -||y_q (t)||^2_{\dot{H}_{x}^{\frac{1}{2}}}\leq \frac{5}{4}\lambda_{q+1}\delta_{q+1}e(t) \Upsilon^{-2}(t),
\end{equation}
where $\varepsilon_\gamma$ is that from the geometric lemma Lemma \ref{geometric lemma} and $D_{t,q}$ is defined in \eqref{est 57}.

\begin{remark}\label{Remark 5.2}
We point out that the energy hypothesis of \eqref{mult inductive 3} is different from \eqref{inductive 3} in terms of multiplication by $\Upsilon^{-2} = (\Upsilon^{-1})^{2}$ and this stems from the definition of $y = \Upsilon^{-1} v$ from \eqref{define upsilon} so that $\lVert y(t) \rVert_{\dot{H}_{x}^{\frac{1}{2}}}^{2} = \Upsilon^{-2} \lVert v(t) \rVert_{\dot{H}_{x}^{\frac{1}{2}}}^{2}$. This difference in \eqref{mult inductive 3} will now create a difference in \eqref{mult inductive 2b} of  ``$e^{-3L^{\frac{1}{4}}}$,'' which is just enough to counter $\Upsilon_{l}(t) \Upsilon^{-2} (t)$, in comparison to $||\mathring{R}_q||_{C_{t,x,q}}\leq \frac{\varepsilon_\gamma}{32(2\pi)^2}\lambda_{q+2}\delta_{q+2}e(t)$ from \eqref{inductive 2b} as follows. First, the function $\bar{\gamma}_{q}(t)$ in the linear multiplicative case must consist of the same proportion as ``$\Upsilon^{2}(t) \lVert y_{q+1}(t) \rVert_{\dot{H}_{x}^{\frac{1}{2}}}^{2}$,'' as can be seen in \eqref{est 98}. The perturbation $w_{q+1}$ will be constructed via its amplitude functions $\bar{a}_{k,j}$ and the amplitude function will be a multiple of the amplitude function from the additive case by $\Upsilon_{l}^{-\frac{1}{2}}$ (see \eqref{est 58b}). This implies that $\lVert y_{q}(t) \rVert_{\dot{H}_{x}^{\frac{1}{2}}}^{2}$ will create an extra factor of $(\Upsilon_{l}^{-\frac{1}{2}})^{2}$ from $\lVert w_{q+1}(t) \rVert_{\dot{H}_{x}^{\frac{1}{2}}}^{2}$ and hence from $\lVert y_{q+1}(t) \rVert_{\dot{H}_{x}^{\frac{1}{2}}}^{2}$. Therefore, $\bar{\gamma_{q}}(t)$, to be defined in \eqref{est 58c}, will consist of $\Upsilon_{l}(t) \Upsilon^{-2}(t)$; this will be elaborated more in detail in Remark \ref{Remark 5.3}. Then the definition of $\bar{\rho}_{j}$ in \eqref{est 58b} shows that the extra factor of $\mathring{R}_{q,j}$ should be same as that of $\bar{\gamma}_{l}$, suggesting  heuristically that the ``$e^{-3L^{\frac{1}{4}}}$'' in \eqref{mult inductive 2b} is the appropriate factor. This factor is also used crucially in \eqref{est 235}. 
\end{remark} 

\subsection{Base Step $q=0$}
\begin{proposition}\label{mult step q0}
Fix $L > 1$ and suppose that $b$ satisfies \eqref{mult bound b}.  Let $y_{0} \equiv 0$. Then together with $\mathring{R}_{0} \equiv 0$, the pair $(y_{0}, \mathring{R}_{0})$ solves \eqref{mult convPDE} and satisfies the inductive hypotheses \eqref{mult inductive 1}-\eqref{mult inductive 3} at level $q= 0$ on $[t_{0}, T_{L}]$. Moreover, $y_{0}(t,x)$ and $\mathring{R}_{0} (t,x)$ are both deterministic over $[t_{0}, 0]$. 
\end{proposition}

\begin{proof}[Proof of Proposition \ref{mult step q0}]
The hypotheses \eqref{mult inductive 1}-\eqref{mult inductive 2} are trivially satisfied while the hypothesis \eqref{mult inductive 3} follows due to $\lambda_{1} \delta_{1} \overset{\eqref{define lambda delta}}{=} 1$. 
\end{proof}

\subsection[Main Iteration]{\for{toc}{Main Iteration}\except{toc}{Main Iteration: Step $q\to q+1$}}

\begin{proposition}\label{mult q to qplus1}
Fix $L > 1$ and suppose that $b$ satisfies \eqref{mult bound b}. Then there exists a choice of $a_{0}$ in \eqref{bound a} and $\beta$ in \eqref{bound beta} such that the following holds: If $(y_q,\mathring{R}_q)$ is an $(\mathcal{F}_t)_{t\geq 0}$-adapted solution to \eqref{mult convPDE} on $[t_q,T_{L}]$ satisfying the inductive hypothesis \eqref{mult inductive 1}-\eqref{mult inductive 3} for some $q\in\mathbb{N}_{0}$, then there exists a solution $(y_{q+1},\mathring{R}_{q+1})$ to \eqref{mult convPDE} which is also $(\mathcal{F}_t)_{t\geq 0}$-adapted, satisfying the inductive hypotheses \eqref{mult inductive 1}-\eqref{mult inductive 3} and the Cauchy difference
\begin{equation}\label{lm cauchy}
\lVert y_{q+1}(t)-y_{q}(t)\rVert_{C_x}\leq \Upsilon_{l}^{-\frac{1}{2}}(t) e^{ \frac{3}{2}L^{\frac{1}{4}}}  M_{0}^{\frac{1}{2}}\bar{e}^\frac{1}{2}\delta_{q+1}^\frac{1}{2}
\end{equation}
on $[t_{q+1},T_{L}]$. Moreover, $w_{q+1} \triangleq y_{q+1}-y_l$ satisfies \eqref{support wqplus1}. Finally, if $(y_q,\mathring{R}_q)(t)$ is deterministic on $[t_q,0]$, then so is $(y_{q+1},\mathring{R}_{q+1})$ on $[t_{q+1},0]$.
\end{proposition}

\subsubsection{Construction of Solution at Step $q+1$}
First, similarly to \eqref{def Phi j} and \eqref{est 16} but now with $D_{t,q}$ defined in \eqref{est 57}, we define $\Phi_{j}(t,x)$ for all $j \in \mathbf{J}$ as the solution to 
\begin{equation}\label{est 75}
D_{t,q}\Phi_{j} = 0, \hspace{3mm} \Phi_{j} (\tau_{q+1} j, x) = x, 
\end{equation} 
and $\mathring{R}_{q,j}(t,x)$ as the solution to  
\begin{equation}\label{est 60}
D_{t,q} \mathring{R}_{q,j} = 0, \hspace{3mm} \mathring{R}_{q,j} (\tau_{q+1} j, x) = \mathring{R}_{l} (\tau_{q+1} j, x). 
\end{equation} 
Furthermore, we define
\begin{subequations}\label{est 58}
\begin{align}
&\bar{\gamma}_{q} (t) \triangleq \frac{\Upsilon^{-2}(t) \Upsilon_{l}(t)}{4 (2\pi)^{2}} [ e(t) (1-\lambda_{q+2} \delta_{q+2} ) - \Upsilon^{2}(t) \lVert y_{q} (t) \rVert_{\dot{H}_{x}^{\frac{1}{2}}}^{2}], \label{est 58c}\\
&\bar{\gamma}_{l} (t) \triangleq \bar{\gamma}_{q}\ast_{t} \varphi_{l} (t),\label{est 58d}\\
&\bar{\rho}_{j} \triangleq \varepsilon_\gamma^{-1} \sqrt{ l^{2} + \lVert \mathring{R}_{q,j} (\tau_{q+1} j) \rVert_{C_{x}}^{2}} + \bar{\gamma}_{l} (\tau_{q+1} j), \label{est 58b}
\end{align}
\end{subequations} 
so that we may define again our amplitude functions for all $k \in \Gamma_{j}$  
\begin{equation}\label{est 58a}
a_{k,j}(t,x) \triangleq  \lambda_{q+1}^{-\frac{1}{2}} \bar{\rho}_{j}^{\frac{1}{2}} \gamma_{k} \left( \Id - \frac{\mathring{R}_{q,j} (t,x)}{\bar{\rho}_{j}} \right)  \hspace{2mm} \text{ and } \hspace{2mm} \bar{a}_{k,j}(t,x) \triangleq \Upsilon_{l}^{-\frac{1}{2}}(t) a_{k,j}(t,x).
\end{equation} 
We emphasize that although $a_{k,j}$ in \eqref{est 58a} appears to be same as $a_{k,j}$ in \eqref{est 3}, $a_{k,j}$ in the current case is different  due to the difference of $\bar{\rho}_{j}$ and $\rho_{j}$ in \eqref{est 0a}. 

\begin{remark}\label{Remark 5.3}
\indent
\begin{enumerate}
\item We point out that the analogous equation on \cite[p. 14]{B23} in contrast has  
\begin{align*}
\text{``}\rho(t,x) : = \sqrt{ l^{2} + \lVert \mathring{R}_{l} (t,x) \rVert_{F}^{2}} + \Theta_{l} (t) \eta_{l} (t);\text{''} 
\end{align*} 
where the product ``$\Theta_{l}(t) \eta_{l}(t)$'' would be analogous to us multiplying $\bar{\gamma}_{l} (\tau_{q+1} j)$ in \eqref{est 58b} by $\Upsilon_{l} (t)$. 
\item As previewed in Remark \ref{Remark 5.2}, It is crucial here that we define $\bar{\gamma}_{q}(t)$ as in \eqref{est 58c} rather than the more natural
\begin{equation}\label{est 106} 
\bar{\gamma}_{q}(t) =  \frac{1}{4 (2\pi)^{2}} [ e(t) (1-\lambda_{q+2} \delta_{q+2} ) - \Upsilon^{2}(t) \lVert y_{q} (t) \rVert_{\dot{H}_{x}^{\frac{1}{2}}}^{2}]. 
\end{equation}
To be precise, if we define $\bar{\gamma}_{q}(t)$ by \eqref{est 106}, then subsequently, we will have to try to estimate 
\begin{align*}
\delta E(t) \overset{\eqref{est 93}}{=}& \left\lvert e(t) (1- \lambda_{q+2} \delta_{q+2} ) - \Upsilon^{2}(t) \lVert y_{q+1}(t) \rVert_{\dot{H}_{x}^{\frac{1}{2}}}^{2} \right\rvert \\
\overset{\eqref{est 106}}{=}& \left\lvert \bar{\gamma}_{q}(t) 4 (2\pi)^{2} + \Upsilon^{2}(t) \lVert y_{q}(t) \rVert_{\dot{H}_{x}^{\frac{1}{2}}}^{2} - \Upsilon^{2}(t) \lVert y_{q+1} (t) \rVert_{\dot{H}_{x}^{\frac{1}{2}}}^{2} \right\rvert 
\end{align*}
where decomposing the last term $- \Upsilon^{2}(t) \lVert y_{q+1}(t) \rVert_{\dot{H}_{x}^{\frac{1}{2}}}^{2}$ will produce various terms, one of them being 
\begin{align*}
-\Upsilon^{2}(t) 4 (2\pi)^{2} \sum_{j} \chi_{j}^{2}(t) \Upsilon_{l}^{-1} (t) \bar{\gamma}_{l} (\tau_{q+1} j)
\end{align*}
(see \eqref{est 99}). We were unable to find a good estimate for the difference 
\begin{align*}
\bar{\gamma}_{q}(t) 4 (2\pi)^{2} - \Upsilon^{2}(t) \Upsilon_{l}^{-1}(t) 4 (2\pi)^{2} \sum_{j} \chi_{j}^{2} (t) \bar{\gamma}_{l} (\tau_{q+1} j);
\end{align*}
thus, we chose to define $\bar{\gamma}_{q}$ as we did in \eqref{est 58c}. 
\end{enumerate} 
\end{remark} 
We observe that 
\begin{equation}\label{est 198}
\bar{\gamma}_{q}(t) \geq 0 
\end{equation}
due to \eqref{mult inductive 3} and \eqref{est 1} that led to \eqref{est 4}. Hence, we now see that 
\begin{align}\label{est 261}
\bar{\gamma}_{l} (t)  \geq 0 
\end{align}
so we have
\begin{align}
\bar{\rho}_{j} \overset{\eqref{est 58b}}{=} \varepsilon_\gamma^{-1} \sqrt{ l^{2} + \lVert \mathring{R}_{q,j} (\tau_{q+1} j) \rVert_{C_{x}}^{2}} + \bar{\gamma}_{l} (\tau_{q+1} j) \geq \varepsilon_\gamma^{-1} \sqrt{ l^{2} + \lVert \mathring{R}_{q,j} (\tau_{q+1} j) \rVert_{C_{x}}^{2}}; \label{est 59}
\end{align}
consequently, $\gamma_{k} \left( \Id - \frac{\mathring{R}_{q,j} (t,x)}{\bar{\rho}_{j}} \right)$ in \eqref{est 58a} can be shown to be well-defined using \eqref{est 2a}. This now allows us to define our perturbation: with $b_{k}$ from \eqref{est 66} and $\chi_{j}$ from \eqref{est 18},
\begin{subequations}\label{est 69}
\begin{align}
w_{q+1}(t,x) \triangleq& \sum_{j,k} \chi_{j}(t) \mathbb{P}_{q+1,k} \left( \bar{a}_{k,j} (t,x) b_{k} (\lambda_{q+1} \Phi_{j} (t,x)) \right), \label{est 69a}\\
\tilde{w}_{q+1, j, k} (t,x) \triangleq& \chi_{j} (t) \bar{a}_{k,j} (t,x) b_{k} (\lambda_{q+1} \Phi_{j} (t,x)).  \label{est 69b}
\end{align}
\end{subequations} 
We keep the same notations of $\psi_{q+1, j, k}$ in \eqref{est 14c}. Hence, we have identically to \eqref{est 67}:
\begin{equation}\label{est 68} 
w_{q+1} = \sum_{j,k} \mathbb{P}_{q+1,k} \tilde{w}_{q+1,j,k}, \hspace{3mm} b_{k} (\lambda_{q+1} \Phi_{j}(x)) = b_{k} (\lambda_{q+1} x) \psi_{q+1, j, k} (x). 
\end{equation}
With this perturbation we define $y_{q+1}\triangleq y_{l}+w_{q+1}$, identically to \eqref{est 44}. Now, similarly to \cite[Equations (16)-(17)]{Y23b}, we define the first commutator part of our error as
\begin{equation}\label{est 53}
R_{\text{Com1}} \triangleq \mathcal{B} \left[ \Upsilon_{l} [ \Lambda y_{l}^{\bot} (\nabla^{\bot} \cdot y_{l})] -  (\Upsilon [  \Lambda y_{q}^{\bot} (\nabla^{\bot} \cdot y_{q} ) ]) \ast_{x} \phi_{l} \ast_{t} \varphi_{l} \right],
\end{equation} 
mollify our pressure as $p_{l} \triangleq p_{q} \ast_{x} \phi_{l} \ast_{t} \varphi_{l}$ so that the mollified pair $(y_{l}, \mathring{R}_{l})$ satisfies 
\begin{equation}\label{est 52}  
\partial_{t} y_{l} + \frac{1}{2} y_{l} + \Upsilon_{l} [ ( \Lambda y_{l} \cdot\nabla) y_{l} - (\nabla y_{l} )^{T} \cdot \Lambda y_{l} ] + \nabla p_{l} + \Lambda^{\gamma} y_{l} = \divergence\left( \mathring{R}_{l} +  R_{\text{Com1}} \right).
\end{equation} 
Similarly to \cite[Equations (88)-(89)]{Y23b}, we then decompose our error into the following components: 
\begin{subequations}
\begin{align} 
\divergence R_{T} \triangleq& \partial_{t} w_{q+1} + \Upsilon_{l} (\Lambda y_{l} \cdot \nabla) w_{q+1},\label{est 55a} \\
\divergence R_{N} \triangleq& \Upsilon_{l} [ (\nabla \Lambda y_{l})^{T} \cdot w_{q+1} + (\Lambda w_{q+1} \cdot \nabla) y_{l} - (\nabla y_{l})^{T} \cdot \Lambda w_{q+1} ], \label{est 55b}\\
\divergence R_{L} \triangleq& \Lambda^{\gamma} w_{q+1} +  \frac{1}{2} w_{q+1}, \label{est 55c} \\
\divergence R_{O} \triangleq& \divergence \mathring{R}_{l} + \Upsilon_{l} [ (\Lambda w_{q+1} \cdot \nabla) w_{q+1} - (\nabla w_{q+1})^{T} \cdot \Lambda w_{q+1} ], \label{est 55d}
\end{align}
\end{subequations} 
the transport error, Nash error, linear error, and oscillation error, respectively. Together with the second commutator error component that was previewed in \eqref{est 55e}, we can define $\mathring{R}_{q+1}$ identically to \eqref{additive decomposition} as 
\begin{equation}\label{est 54} 
\divergence \mathring{R}_{q+1} = \divergence (R_{T} + R_{N} + R_{L} + R_{O} + R_{\text{Com1}} + R_{\text{Com2}}).
\end{equation} 

\subsubsection{Perturbative Estimates}
We estimate using \eqref{mult inductive 1b} and \eqref{mult inductive 1c}, 
\begin{equation}\label{est 47} 
\lVert \partial_{t} y_{q} (t) \rVert_{C_{x}} \leq \lVert D_{t,q} y_{q}(t) \rVert_{C_{x}} + \lVert \Upsilon_{l} \Lambda y_{l}(t) \cdot \nabla y_{q}(t) \rVert_{C_{x}}  \lesssim M_{0} \lambda_{q}^{2} \delta_{q} m_{L}^{8} e^{L^{\frac{1}{4}}} \bar{e}.
\end{equation} 
This estimate \eqref{est 47} leads to 
\begin{equation}\label{est 71}
\lVert y_{l} - y_{q} \rVert_{C_{t,x,q+1}}  \lesssim l \lVert y_{q} \rVert_{C_{t,q+1}^{1}C_{x}} +  l \lVert y_{q} \rVert_{C_{t,q+1}C_{x}^{1}}\overset{\eqref{est 47} \eqref{mult inductive 1c}}{\lesssim} l M_{0} m_{L}^{8} \lambda_{q}^{2} \delta_{q} \bar{e} e^{L^{\frac{1}{4}}}.
\end{equation} 
Next, we estimate directly from definitions in \eqref{est 58b} and \eqref{est 58c}  
\begin{align}
\lvert \bar{\rho}_{j} \rvert \leq&  \varepsilon_\gamma^{-1} ( l + \lVert \mathring{R}_{q,j} (\tau_{q+1} j) \rVert_{C_{x}})  \nonumber \\
& +  \left\lVert  \frac{\Upsilon_{l}}{4 (2\pi)^{2}} [ e \Upsilon^{-2}(1-\lambda_{q+2} \delta_{q+2} ) -  \lVert y_{q}  \rVert_{\dot{H}_{x}^{\frac{1}{2}}}^{2}] \right\rVert_{C_{\tau_{q+1} j,q+1}}. \label{est 266} 
\end{align}
Now by \eqref{mult inductive 3} at step $q=0$ and \eqref{est 4} we have for all $t \in [t_{q}, T_{L}]$,  
\begin{equation*}
\Upsilon_{l}(t) \left\lvert e(t) \Upsilon^{-2}(t) (1- \lambda_{q+2} \delta_{q+2}) - \lVert y_{q} (t) \rVert_{\dot{H}_{x}^{\frac{1}{2}}}^{2} \right\rvert \leq \frac{5}{4} \lambda_{q+1} \delta_{q+1} \bar{e}e^{3L^{\frac{1}{4}}};  
\end{equation*}
therefore, applying this estimate to \eqref{est 266} gives us 
\begin{equation}\label{est 64}
\lvert \bar{\rho}_{j} \rvert \leq \varepsilon_\gamma^{-1} (l + \lVert \mathring{R}_{q,j} (\tau_{q+1} j) \rVert_{C_{x}}) + \frac{1}{4 ( 2\pi)^{2}} \frac{5}{4} \lambda_{q+1} \delta_{q+1} \bar{e} e^{3L^{\frac{1}{4}}} \leq  \frac{1}{8\pi^{2}} \lambda_{q+1} \delta_{q+1} \bar{e}e^{3L^{\frac{1}{4}}}. 
\end{equation}
For subsequent convenience, we compute by using $\supp \hat{y}_{q} \subset B(0, 2 \lambda_{q})$, for $\beta > \frac{1}{2}$ sufficiently close to $\frac{1}{2}$ and $a_{0}$ sufficiently large, 
\begin{equation}\label{est 79} 
\tau_{q+1} \lVert D( \Upsilon_{l} \Lambda y_{l})  \rVert_{C_{t,x,q+1}} \overset{\eqref{tau}}{\lesssim} l^{\frac{1}{2}} \lambda_{q+2}^{-\frac{1}{2}} \delta_{q+2}^{-\frac{1}{2}} \lambda_{q+1}^{-\frac{1}{2}} \delta_{q+1}^{\frac{1}{4}} e^{L^{\frac{1}{4}}} \lambda_{q} \lVert \Lambda y_{q} \rVert_{C_{t,x,q+1}}  \overset{\eqref{mult inductive 1c}}{\ll} 1.
\end{equation} 

Now we may begin our key estimates with an analog of Proposition \ref{prop additive amplitudes}.
\begin{proposition}\label{prop mult amplitudes}
Define $a_{k,j}$ and $\bar{a}_{k,j}$ by \eqref{est 58a}. Then they satisfy the following bounds: for all $t \in \supp \chi_{j}$ and all $N \in \mathbb{N}$, 
\begin{subequations}\label{est 63}
\begin{align}
& \lVert a_{k,j}(t) \rVert_{C_{x}} \leq \frac{1}{\sqrt{8} \pi }  e^{ \frac{3}{2} L^{\frac{1}{4}}} \bar{e}^{\frac{1}{2}} \delta_{q+1}^{\frac{1}{2}} \sup_{k \in \Gamma_{1} \cup \Gamma_{2}} \lVert \gamma_{k} \rVert_{C(B(\Id, \varepsilon_\gamma ))},  \label{est 63a}  \\
& \lVert \bar{a}_{k,j}(t) \rVert_{C_{x}} \leq \frac{1}{\sqrt{8} \pi} \Upsilon_{l}^{-\frac{1}{2}}(t) e^{\frac{3}{2} L^{\frac{1}{4}}}  \bar{e}^{\frac{1}{2}} \delta_{q+1}^{\frac{1}{2}} \sup_{k \in \Gamma_{1} \cup \Gamma_{2}} \lVert \gamma_{k} \rVert_{C(B(\Id, \varepsilon_\gamma ))},  \label{est 63b}  \\
& \lVert D^{N} a_{k,j}(t) \rVert_{C_{x}}  \lesssim e^{\frac{3}{2} L^{\frac{1}{4}}}  \delta_{q+1}^{\frac{1}{2}} \bar{e}^{\frac{1}{2}} \lambda_{q}^{N},  \label{est 63c}\\
& \lVert D^{N} \bar{a}_{k,j}(t) \rVert_{C_{x}}  \lesssim \Upsilon_{l}^{-\frac{1}{2}}(t) e^{\frac{3}{2} L^{\frac{1}{4}}}  \delta_{q+1}^{\frac{1}{2}} \bar{e}^{\frac{1}{2}} \lambda_{q}^{N}.  \label{est 63d}
\end{align}
\end{subequations}
Furthermore, $y_{q+1}, w_{q+1}$, and $\mathring{R}_{q+1}$ from \eqref{est 44}, \eqref{est 69}, and \eqref{est 54} satisfy hypotheses \eqref{mult inductive 1a} and \eqref{mult inductive 2a} at level $q+1$, as well as the support condition \eqref{support wqplus1}.
\end{proposition}

\begin{proof}[Proof of Proposition \ref{prop mult amplitudes}]
We leave this proof in the Appendix Section \ref{Section 7.2} for completeness. 
\end{proof}

\begin{proposition}\label{prop mult perturbation}
Define $D_{t,q}$ by \eqref{est 57}. Then $w_{q+1}$ defined in \eqref{est 69a} satisfies the following bounds over $[t_{q+1}, T_{L}]$:  for $C_{1}$ from \eqref{est 8}, 
\begin{subequations}\label{est 70} 
\begin{align}
& \lVert w_{q+1}(t) \rVert_{C_{x}} \leq \frac{C_{1} \Upsilon_{l}^{-\frac{1}{2}} (t) e^{ \frac{3}{2} L^{\frac{1}{4}}} }{\sqrt{2} \pi} \bar{e}^{\frac{1}{2}} \delta_{q+1}^{\frac{1}{2}} \sup_{k \in \Gamma_{1} \cup \Gamma_{2}} \lVert \gamma_{k} \rVert_{C(B(\Id, \varepsilon_\gamma))},  \label{est 70a}\\
& \lVert D_{t,q} w_{q+1} (t) \rVert_{C_{x}} \lesssim \tau_{q+1}^{-1} \Upsilon_{l}^{-\frac{1}{2}} (t) e^{\frac{3}{2} L^{\frac{1}{4}}} \delta_{q+1}^{\frac{1}{2}} \bar{e}^{\frac{1}{2}}. 
\end{align}
\end{subequations}  
Consequently, the hypotheses \eqref{mult inductive 1b}, \eqref{mult inductive 1c}, and \eqref{mult inductive 1d} at level $q+1$, as well as the Cauchy difference \eqref{lm cauchy} all hold.
\end{proposition} 

\begin{proof}[Proof of Proposition \ref{prop mult perturbation}]
First, we estimate by relying on \eqref{est 69a}, \eqref{est 66}, and \eqref{est 8}
\begin{align*}
\lVert w_{q+1}(t) \rVert_{C_{x}}  \leq C_{1} \sum_{j,k} 1_{\supp \chi_{j}}(t) \lVert \bar{a}_{k,j} (t) \rVert_{C_{x}}, 
\end{align*}
to which applying \eqref{est 63b} gives \eqref{est 70a}. As $M_{0}^{\frac{1}{2}} \geq \frac{4C_{1}}{\pi} \sup_{k \in \Gamma_{1} \cup \Gamma_{2}} \lVert \gamma_{k} \rVert_{C(B(\Id, \varepsilon_\gamma))}$ due to \eqref{M0}, we can verify \eqref{lm cauchy} as
\begin{align}
&\lVert y_{q+1}(t) - y_{q}(t) \rVert_{C_{x}} \overset{\eqref{est 71} \eqref{est 70a} \eqref{M0}}{\leq} \frac{M_{0}^{\frac{1}{2}}}{4 \sqrt{2}} \Upsilon_{l}^{-\frac{1}{2}} (t) e^{\frac{3}{2} L^{\frac{1}{4}}} \bar{e}^{\frac{1}{2}} \delta_{q+1}^{\frac{1}{2}} + C l M_{0} m_{L}^{8} \lambda_{q}^{2} \delta_{q} \bar{e} e^{L^{\frac{1}{4}}}  \nonumber \\ 
&\overset{\eqref{define l} \eqref{est 73}\eqref{bound a}}{\leq} M_{0}^{\frac{1}{2}} \Upsilon_{l}^{-\frac{1}{2}} (t) e^{\frac{3}{2} L^{\frac{1}{4}}} \bar{e}^{\frac{1}{2}} \delta_{q+1}^{\frac{1}{2}} \left[ \frac{1}{4 \sqrt{2}} + C M_{0}^{\frac{1}{2}}  \bar{e}^{\frac{1}{2}} a^{L^{\frac{1}{4}}}a^{b( \beta - \frac{1}{2})} a^{b^{q} [ 2 - 2 \beta + b (\beta - \alpha) ]}\right]  \nonumber \\
& \hspace{30mm} \leq M_{0}^{\frac{1}{2}} \Upsilon_{l}^{-\frac{1}{2}}(t) e^{ \frac{3}{2} L^{\frac{1}{4}}} \bar{e}^{\frac{1}{2}} \delta_{q+1}^{\frac{1}{2}}, \label{est 191}
\end{align} 
where the last inequality used the fact that $b > \frac{ L^{\frac{1}{4}} + 1}{\alpha - \frac{1}{2}} < b$ due to \eqref{mult bound b}. Additionally, we are ready to prove \eqref{mult inductive 1b} at level $q+1$ as follows: 
\begin{align}
\lVert y_{q+1} \rVert_{C_{t,x,q+1}}\overset{ \eqref{est 44} \eqref{mult inductive 1b} \eqref{est 70a} \eqref{M0}}{\leq}& M_{0}^{\frac{1}{2}} \left(1+ \sum_{1 \leq j \leq q} \delta_{j}^{\frac{1}{2}} \right) m_{L}^{4} \bar{e}^{\frac{1}{2}} + \frac{M_{0}^{\frac{1}{2}}}{4 \sqrt{2}} \Upsilon_{l}^{-\frac{1}{2}}(t) e^{\frac{3}{2}L^{\frac{1}{4}}} \bar{e}^{\frac{1}{2}} \delta_{q+1}^{\frac{1}{2}} \nonumber  \\
\leq& M_{0}^{\frac{1}{2}} ( 1+ \sum_{1 \leq j \leq q} \delta_{j}^{\frac{1}{2}}) m_{L}^{4} \bar{e}^{\frac{1}{2}}. 
\end{align} 
Next, we verify \eqref{mult inductive 1c} as follows: 
\begin{align}
& \lVert y_{q+1} \rVert_{C_{t,q+1} C_{x}^{1}} + \lVert \Lambda y_{q+1} \rVert_{C_{t,x,q+1}} \\
\overset{\eqref{est 68}  \eqref{mult inductive 1c} \eqref{est 70a}\eqref{M0}}{\leq}& m_{L}^{4} \bar{e}^{\frac{1}{2}} M_{0}^{\frac{1}{2}}  \lambda_{q+1} \delta_{q+1}^{\frac{1}{2}} [ \frac{1}{\sqrt{2}} +  a^{b^{q} [ (1-\beta) + b(\beta -1)]} ]  \overset{\eqref{mult bound b}}{\leq} M_{0}^{\frac{1}{2}} m_{L}^{4} \bar{e}^{\frac{1}{2}} \lambda_{q+1} \delta_{q+1}^{\frac{1}{2}}. \nonumber 
\end{align}
Then due to \eqref{est 60} and \eqref{est 75}, we find
\begin{equation}\label{est 76}
D_{t,q} \bar{a}_{k,j}  (t) =  \partial_{t} \Upsilon_{l}^{-\frac{1}{2}} (t) \lambda_{q+1}^{-\frac{1}{2}} \bar{\rho}_{j}^{\frac{1}{2}} \gamma_{k} \left( \Id - \frac{\mathring{R}_{q,j} (t) }{\bar{\rho}_{j}} \right), \hspace{5mm} D_{t,q} b_{q} (\lambda_{q+1} \Phi_{j} (t)) = 0.
\end{equation} 
Thus, we can simplify
\begin{align}
D_{t,q} w_{q+1} &= \sum_{j,k} [D_{t,q}, \mathbb{P}_{q+1, k} ] \tilde{w}_{q+1, j, k} + \mathbb{P}_{q+1, k} \Big[ (\partial_{t} \chi_{j}) \bar{a}_{k,j} b_{k} (\lambda_{q+1} \Phi_{j} )  \nonumber \\
&+  \chi_{j} b_{k} (\lambda_{q+1} \Phi_{j}) \partial_{t} \Upsilon_{l}^{-\frac{1}{2}}  \lambda_{q+1}^{-\frac{1}{2}} \bar{\rho}_{j}^{\frac{1}{2}} \gamma_{k} \left( \Id - \frac{\mathring{R}_{q,j} (t,x)}{\bar{\rho}_{j}} \right) \Big]. \label{est 188}
\end{align}
Thus, we are ready to estimate, referencing \cite[Corollary A.8]{BSV19},
\begin{align}
&\lVert D_{t,q} w_{q+1}(t) \rVert_{C_{x}} \lesssim \sum_{j,k} \lVert \nabla (\Upsilon_{l} \Lambda y_{l} )(t) \rVert_{C_{x}} \lVert \tilde{w}_{q+1, j, k}(t) \rVert_{C_{x}} \nonumber  \\
&+ \lVert (\partial_{t} \chi_{j}) \bar{a}_{k,j} b_{k} (\lambda_{q+1} \Phi_{j} )(t) \rVert_{C_{x}} + \lVert \chi_{j}(t) \Upsilon_{l}^{-\frac{3}{2}} \partial_{t} \Upsilon_{l} \lambda_{q+1}^{-\frac{1}{2}} \bar{\rho}_{j}^{\frac{1}{2}} \gamma_{k} \left( \Id - \frac{ \mathring{R}_{q,j} (\tau_{q+1} j)}{\bar{\rho}_{j}} \right) \rVert_{C_{x}} \nonumber\\
&\overset{\eqref{mult inductive 1b} \eqref{est 64} }{\lesssim} \sum_{j,k} 1_{\supp \chi_{j}}(t) \left[ e^{L^{\frac{1}{4}}} M_{0}^{\frac{1}{2}} m_{L}^{4} \lambda_{q}^{2} \delta_{q}^{\frac{1}{2}} \bar{e}^{\frac{1}{2}} + \tau_{q+1}^{-1} + e^{2L^{\frac{1}{4}}} l^{-1} \right] \Upsilon_{l}^{-\frac{1}{2}} (t) e^{\frac{3}{2} L^{\frac{1}{4}}}  \bar{e}^{\frac{1}{2}} \delta_{q+1}^{\frac{1}{2}}  \nonumber \\
\lesssim& \tau_{q+1}^{-1} \Upsilon_{l}^{-\frac{1}{2}}(t) e^{\frac{3}{2} L^{\frac{1}{4}}}  \bar{e}^{\frac{1}{2}} \delta_{q+1}^{\frac{1}{2}}. \label{est 77} 
\end{align}
Similarly to \cite[Equation (106)]{Y23b}, we may now write 
\begin{align*}
D_{t, q+1} &y_{q+1} = D_{t,q} w_{q+1} + [\Upsilon_{\lambda_{q+2}^{-\alpha}} \Lambda y_{l} \ast_{x} \phi_{\lambda_{q+2}^{-\alpha}} \ast_{t} \varphi_{\lambda_{q+2}^{-\alpha}} - \Upsilon_{\lambda_{q+1}^{-\alpha}} \Lambda y_{l} ] \cdot \nabla w_{q+1}  \nonumber \\
&\hspace{5mm}+  (D_{t,q} y_{q})\ast_{x} \phi_{\lambda_{q+1}^{-\alpha}} \ast_{t} \varphi_{\lambda_{q+1}^{-\alpha}}  + \Upsilon_{\lambda_{q+2}^{-\alpha}} \Lambda y_{l} \ast_{x} \phi_{\lambda_{q+2}^{-\alpha}} \ast_{t} \varphi_{\lambda_{q+2}^{-\alpha}} \cdot \nabla y_{l}  \nonumber \\
&\hspace{5mm}- [\Upsilon_{\lambda_{q+1}^{-\alpha}} \Lambda y_{l}  \cdot \nabla y_{q} ] \ast_{x} \phi_{\lambda_{q+1}^{-\alpha}} \ast_{t} \varphi_{\lambda_{q+1}^{-\alpha}} + \Upsilon_{\lambda_{q+2}^{-\alpha}} \Lambda w_{q+1} \ast_{x} \phi_{\lambda_{q+2}^{-\alpha}} \ast_{t} \varphi_{\lambda_{q+2}^{-\alpha}} \cdot\nabla w_{q+1} \nonumber\\
&\hspace{5mm}+ \Upsilon_{\lambda_{q+2}^{-\alpha}} \Lambda w_{q+1}\ast_{x} \phi_{\lambda_{q+2}^{-\alpha}} \ast_{t} \varphi_{\lambda_{q+2}^{-\alpha}} \cdot \nabla y_{l},
\end{align*}
which we can estimate as
\begin{align}
& \lVert D_{t,q+1} y_{q+1} \rVert_{C_{t,x,q+1}} \overset{\eqref{est 68}}{\leq} \lVert D_{t,q} w_{q+1} \rVert_{C_{t,x,q+1}} + 6e^{L^{\frac{1}{4}}} \lVert \Lambda y_{q} \rVert_{C_{t,x,q+1}} \lambda_{q+1} \lVert w_{q+1} \rVert_{C_{t,x,q+1}}  \nonumber\\
&\hspace{10mm} + \lVert D_{t,q} y_{q} \rVert_{C_{t,x,q+1}} + 2e^{L^{\frac{1}{4}}} \lVert \Lambda y_{q} \rVert_{C_{t,x,q+1}}^{2} + 4e^{L^{\frac{1}{4}}} \lambda_{q+1}^{2} \lVert w_{q+1} \rVert_{C_{t,x,q+1}}^{2}  \nonumber\\
&\overset{\eqref{est 77} \eqref{est 70a}}{\leq} 4 e^{L^{\frac{1}{4}}} \lambda_{q+1}^{2} \left( \frac{C_{1} e^{2 L^{\frac{1}{4}}}}{\sqrt{2} \pi} \bar{e}^{\frac{1}{2}} \delta_{q+1}^{\frac{1}{2}} \sup_{k \in\Gamma_{1} \cup \Gamma_{2}} \lVert \gamma_{k} \rVert_{C(B( \Id, \varepsilon_{\gamma}))} \right)^{2}  \nonumber\\
& \hspace{4mm} + C \big( \tau_{q+1}^{-1} e^{2L^{\frac{1}{4}}} \bar{e}^{\frac{1}{2}} \delta_{q+1}^{\frac{1}{2}}   + e^{L^{\frac{1}{4}}} ( M_{0}^{\frac{1}{2}} m_{L}^{4}\lambda_{q} \delta_{q}^{\frac{1}{2}} \bar{e}^{\frac{1}{2}}) \lambda_{q+1} (e^{2 L^{\frac{1}{4}}} \bar{e}^{\frac{1}{2}} \delta_{q+1}^{\frac{1}{2}}) + M_{0} \lambda_{q}^{2} \delta_{q} \bar{e} m_{L}^{8} e^{L^{\frac{1}{4}}} \big)  \nonumber\\
& \hspace{70mm} \leq M_{0} m_{L}^{8} e^{L^{\frac{1}{4}}} \bar{e} \lambda_{q+1}^{2} \delta_{q+1}.\label{mult M0} 
\end{align}
Hence, \eqref{mult inductive 1d} is also satisfied at level $q+1$, concluding this proof.
\end{proof}

\subsubsection{Stress Decomposition Estimates}
First we show some useful bounds on the functions $\psi_{q+1, j, k}$.
\begin{proposition}\label{prop mult psi}
Define $\psi_{q+1, j, k}$ by \eqref{est 14c} with $\Phi_{j}$ therein by \eqref{est 75}. Then it satisfies the following bounds: for all $t \in \supp \chi_{j}$ and all $N\in\mathbb{N}$,
\begin{subequations}\label{est 82} 
\begin{align}
& \lVert \psi_{q+1, j, k} (t) \rVert_{C_{x}} = 1,  \hspace{3mm} \lVert D^{N} \psi_{q+1, j, k} (t) \rVert_{C_{x}} \lesssim \bar{e}^{\frac{1}{2}} M_{0}^{\frac{1}{2}} m_{L}^{4} e^{L^{\frac{1}{4}}} \lambda_{q+1} \tau_{q+1} \lambda_{q}^{N+1} \delta_{q}^{\frac{1}{2}}, \label{est 82a} \\
& \lVert D ( \bar{a}_{k,j} \psi_{q+1, j, k} ) (t) \rVert_{C_{x}} \lesssim \Upsilon_{l}^{-\frac{1}{2}} (t) e^{\frac{3}{2}L^{\frac{1}{4}}}  \delta_{q+1}^{\frac{1}{2}} \bar{e}^{\frac{1}{2}} \lambda_{q}, \label{est 82b} \\
& \lVert D^{2} ( \bar{a}_{k,j} \psi_{q+1, j, k} ) (t) \rVert_{C_{x}} \lesssim \Upsilon_{l}^{-\frac{1}{2}} (t) e^{\frac{3}{2}L^{\frac{1}{4}}} \delta_{q+1}^{\frac{1}{2}} \bar{e}^{\frac{1}{2}} \lambda_{q}^{2}, \label{est 82c} \\
& \lVert D ( a_{k,j} \psi_{q+1, j, k} ) (t) \rVert_{C_{x}} \lesssim  e^{\frac{3}{2}L^{\frac{1}{4}}} \delta_{q+1}^{\frac{1}{2}} \bar{e}^{\frac{1}{2}} \lambda_{q},\label{est 82d}  \\
& \lVert D^{2} (a_{k,j} \psi_{q+1, j, k} ) (t) \rVert_{C_{x}} \lesssim  e^{\frac{3}{2}L^{\frac{1}{4}}}  \delta_{q+1}^{\frac{1}{2}} \bar{e}^{\frac{1}{2}} \lambda_{q}^{2}.\label{est 82e} 
\end{align}
\end{subequations} 
\end{proposition}

\begin{proof}[Proof of Proposition \ref{prop mult psi}]
We leave this proof in the Appendix Section \ref{Section 7.2} for completeness. 
\end{proof}

We now begin our estimates of $\mathring{R}_{q+1}$, starting with the transport component:
\begin{proposition}\label{prop mRT}
$R_{T}$ from \eqref{est 55a} satisfies for $\beta > \frac{1}{2}$ sufficiently close to $\frac{1}{2}$ and $a_{0}$ sufficiently large 
\begin{equation*}
\lVert R_{T} \rVert_{C_{t,x,q+1}}  \ll e^{-3 L^{\frac{1}{4}}} \lambda_{q+3} \delta_{q+3}. 
\end{equation*}  
\end{proposition}

\begin{proof}[Proof of Proposition \ref{prop mRT}]
We can compute using \eqref{est 76} 
\begin{align}
& [\partial_{t} + \Upsilon_{l} \Lambda y_{l} \cdot \nabla ] ( \chi_{j} \bar{a}_{k,j}  b_{k} (\lambda_{q+1} \Phi_{j}) ) \nonumber \\
=& \partial_{t} \chi_{j} \bar{a}_{k,j}  b_{k} (\lambda_{q+1} \Phi_{j} ) + \chi_{j} \partial_{t} \Upsilon_{l}^{-\frac{1}{2}}  \lambda_{q+1}^{-\frac{1}{2}} \bar{\rho}_{j}^{\frac{1}{2}} \gamma_{k} \left( \Id - \frac{\mathring{R}_{q,j} }{\bar{\rho}_{j}} \right)   b_{k} (\lambda_{q+1} \Phi_{j}). \label{est 273}
\end{align}
Similarly to \eqref{new 12}, considering Fourier frequency support allows us write   
\begin{align*}
R_{T} =& \sum_{j,k} 1_{\supp \chi_{j}} \mathcal{B} \tilde{P}_{\approx \lambda_{q+1}} \Big[ [ \Upsilon_{l} \Lambda y_{l} \cdot \nabla, \mathbb{P}_{q+1, k} ] \tilde{w}_{q+1, j, k} + \mathbb{P}_{q+1, k} \left( \partial_{t} \chi_{j} \bar{a}_{k,j} b_{k} (\lambda_{q+1} \Phi_{j} ) \right) \\
& \hspace{20mm} + \mathbb{P}_{q+1, k} \left( \chi_{j}\partial_{t} \Upsilon_{l}^{-\frac{1}{2}} \left( \lambda_{q+1}^{-\frac{1}{2}} \bar{\rho}_{j}^{\frac{1}{2}} \gamma_{k} \left( \Id - \frac{\mathring{R}_{q,j}}{\bar{\rho}_{j}} \right) \right)  b_{k} (\lambda_{q+1} \Phi_{j} ) \right) \Big].
\end{align*}
We estimate by relying on \cite[Lemma A.6]{BSV19} for $\beta > \frac{1}{2}$ sufficiently close to $\frac{1}{2}$ and $a_{0}$ sufficiently large 
\begin{align}
\lVert R_{T} \rVert_{C_{t,x,q+1}} &\overset{\eqref{est 63b}}{\lesssim} \lambda_{q+1}^{-1} \big[ \sum_{j,k} \lVert \Upsilon_{l} \nabla \Lambda y_{l} \rVert_{C_{t,x,q+1}} \lVert 1_{\supp \chi_{j}} \tilde{w}_{q+1, j, k} \rVert_{C_{t,x,q+1}}  \label{est 190}\\
& \hspace{10mm} + \tau_{q+1}^{-1} ( e^{\frac{3}{2} L^{\frac{1}{4}}} \bar{e}^{\frac{1}{2}} \delta_{q+1}^{\frac{1}{2}}) + \lambda_{q+1}^{-\frac{1}{2}} \bar{\rho}_{j}^{\frac{1}{2}} e^{\frac{5}{2}L^{\frac{1}{4}}} l^{-1} \big] \nonumber \\
&\overset{\eqref{mult inductive 1c} \eqref{est 64}}{\lesssim}  \lambda_{q+1}^{-1} \Big[  e^{3 L^{\frac{1}{4}}} M_{0}^{\frac{1}{2}} m_{L}^{4} \bar{e} \lambda_{q}^{2} \delta_{q}^{\frac{1}{2}} \delta_{q+1}^{\frac{1}{2}} \nonumber \\
& \hspace{12mm} + \tau_{q+1}^{-1} e^{\frac{3}{2} L^{\frac{1}{4}}} \bar{e}^{\frac{1}{2}} \delta_{q+1}^{\frac{1}{2}} +   \lambda_{q+1}^{\alpha} \delta_{q+1}^{\frac{1}{2}} \bar{e}^{\frac{1}{2}} e^{4L^{\frac{1}{2}}}  \Big] \ll e^{-3 L^{\frac{1}{4}}}  \lambda_{q+3} \delta_{q+3}. \nonumber 
\end{align} 
\end{proof}

Next, we estimate the Nash error. 
\begin{proposition}\label{prop mRN}
$R_{N}$ from \eqref{est 55b} satisfies for $\beta > \frac{1}{2}$ sufficiently close to $\frac{1}{2}$ and $a_{0}$ sufficiently large 
\begin{equation}\label{est 252}
\lVert R_{N} \rVert_{C_{t,x,q+1}}  \ll e^{-3 L^{\frac{1}{4}}} \lambda_{q+3} \delta_{q+3}. 
\end{equation}  
\end{proposition}

\begin{proof}[Proof of Proposition \ref{prop mRN}]
Using the identity \eqref{est 120}  we can break apart
\begin{equation}\label{est 253}
R_{N} = N_{1} + N_{2} 
\end{equation} 
where 
\begin{equation}
N_{1} \triangleq \Upsilon_{l} \mathcal{B} [ (\nabla \Lambda y_{l})^{T} \cdot w_{q+1} ] \hspace{2mm} \text{ and } \hspace{2mm} N_{2} \triangleq \Upsilon_{l} \mathcal{B} [\Lambda w_{q+1}^{\bot} (\nabla^{\bot} \cdot y_{l})]. 
\end{equation} 
Then similarly to \eqref{additive frequency support}, considering the frequency support of $y_{l}$ and $w_{q+1}$ leads to   
\begin{equation}
N_{1} = \Upsilon_{l} \sum_{j,k} \mathcal{B} \tilde{P}_{\approx \lambda_{q+1}} \left[ ( \nabla \Lambda y_{l})^{T} \cdot \mathbb{P}_{q+1, k} \left( \chi_{j} \bar{a}_{k,j}  b_{k} (\lambda_{q+1} \Phi_{j} ) \right) \right].
\end{equation}
This lets us compute for $\beta > \frac{1}{2}$ sufficiently close to $\frac{1}{2}$ and $a_{0}$ sufficiently large 
\begin{align}
\lVert N_{1} \rVert_{C_{t,x,q+1}} \overset{\eqref{est 63b} \eqref{mult inductive 1c}}{\lesssim} M_{0}^{\frac{1}{2}} \bar{e} e^{5L^{\frac{1}{4}}} L  \lambda_{1}^{2\beta -1} \lambda_{q}^{2-\beta} \lambda_{q+1}^{-1-\beta} \overset{\eqref{mult bound b}}{\ll} e^{- 3L^{\frac{1}{4}}} \lambda_{q+3} \delta_{q+3}.\label{est 250}
\end{align}
Similarly to \cite[Equation (125)]{Y23b}, we can also rewrite 
\begin{align}
N_{2}(x) =&\Upsilon_{l} \lambda_{q+1}^{-1}  \mathcal{B} \Big( \nabla ( \nabla^{\bot} \cdot y_{l})(x)  \sum_{j,k} \Lambda \mathbb{P}_{q+1, k} \big(\chi_{j} \bar{a}_{k,j} (x) c_{k} (\lambda_{q+1} \Phi_{j} (x))\big)\Big)  \\
&+  \Upsilon_{l} \lambda_{q+1}^{-1} \mathcal{B} \Big(( \nabla^{\bot} \cdot y_{l})(x) \sum_{j,k} \Lambda \mathbb{P}_{q+1, k} \big(\chi_{j} \nabla (\bar{a}_{k,j} (x) \psi_{q+1, j, k} (x)) c_{k} (\lambda_{q+1} x)\big)\Big). \nonumber 
\end{align} 
Thus, we estimate, using \cite[Equation (A.11)]{BSV19}, for $\beta > \frac{1}{2}$ sufficiently close to $\frac{1}{2}$ and $a_{0}$ sufficiently large 
\begin{align}
&\lVert N_{2} \rVert_{C_{t,x,q+1}} \overset{\eqref{est 82b} }{\lesssim} e^{L^{\frac{1}{4}}} \lambda_{q+1}^{-1} \big[ \lambda_{q} (M_{0}^{\frac{1}{2}} m_{L}^{4} \lambda_{q} \delta_{q}^{\frac{1}{2}} \bar{e}^{\frac{1}{2}}) ( e^{2 L^{\frac{1}{4}}} \bar{e}^{\frac{1}{2}} \delta_{q+1}^{\frac{1}{2}})  \label{est 251} \\
& \hspace{10mm} + (M_{0}^{\frac{1}{2}} m_{L}^{4} \lambda_{q} \delta_{q}^{\frac{1}{2}} \bar{e}^{\frac{1}{2}}) ( e^{2 L^{\frac{1}{4}}} \bar{e}^{\frac{1}{2}} \delta_{q+1}^{\frac{1}{2}} \lambda_{q} ) \big] \lesssim a^{\frac{L}{2}} \lambda_{1}^{2\beta -1} \lambda_{q}^{2-\beta} \lambda_{q+1}^{-1-\beta}  \ll e^{-3 L^{\frac{1}{4}}} \lambda_{q+3} \delta_{q+3}. \nonumber 
\end{align}
Applying \eqref{est 250} and \eqref{est 251}  to \eqref{est 253} gives us \eqref{est 252}.  
\end{proof}

Next, we estimate the linear error. 
\begin{proposition}\label{prop mRL}
$R_{L}$ from \eqref{est 55c} satisfies for $\beta > \frac{1}{2}$ sufficiently close to $\frac{1}{2}$ and $a_{0}$ sufficiently large 
\begin{equation}\label{est 255}
\lVert R_{L} \rVert_{C_{t,x,q+1}}  \ll e^{-3 L^{\frac{1}{4}}} \lambda_{q+3} \delta_{q+3}. 
\end{equation}  
\end{proposition}

\begin{proof}[Proof of Proposition \ref{prop mRL}]
We can rewrite
\begin{equation}\label{est 256}
R_{L} = L_{1} + L_{2} \text{ where } L_{1} \triangleq \mathcal{B} \Lambda^{\gamma} w_{q+1} \text{ and } L_{2} \triangleq \frac{1}{2} \mathcal{B} w_{q+1}. 
\end{equation}
First, we estimate for $\beta > \frac{1}{2}$ sufficiently close to $\frac{1}{2}$ and $a_{0}$ sufficiently large, 
\begin{align}
\lVert L_{1} \rVert_{C_{t,x,q+1}} \overset{\eqref{est 70a}}{\lesssim} \lambda_{q+1}^{\gamma -1} \left( e^{2 L^{\frac{1}{4}}} \bar{e}^{\frac{1}{2}} \delta_{q+1}^{\frac{1}{2}} \right) \lesssim \lambda_{q+1}^{\gamma -1} a^{L}  \lambda_{1}^{\beta -\frac{1}{2}} \lambda_{q+1}^{-\beta}\ll e^{-3 L^{\frac{1}{4}}} \lambda_{q+3} \delta_{q+3}.\label{est 192}
\end{align}
Second, we similarly find for $\beta > \frac{1}{2}$ sufficiently close to $\frac{1}{2}$ and $a_{0}$ sufficiently large 
\begin{align}
\lVert L_{2} \rVert_{C_{t,x,q+1}} \overset{\eqref{est 70a}}{\lesssim} \lambda_{q+1}^{-1} e^{2L^{\frac{1}{4}}} \bar{e}^{\frac{1}{2}} \delta_{q+1}^{\frac{1}{2}} \lesssim a^{\frac{L}{2}} \lambda_{1}^{\beta - \frac{1}{2}} \lambda_{q+1}^{-1-\beta} \ll e^{-3 L^{\frac{1}{4}}} \lambda_{q+3} \delta_{q+3}. \label{est 254} 
\end{align}
Applying \eqref{est 192} and \eqref{est 254} to \eqref{est 256} gives us \eqref{est 255}. 
\end{proof}

Next, we treat the oscillation error.
\begin{proposition}\label{prop mRO}
$R_{O}$ from \eqref{est 55d} satisfies  for $\beta > \frac{1}{2}$ sufficiently close to $\frac{1}{2}$ and $a_{0}$ sufficiently large 
\begin{equation}\label{est 84} 
e^{2L^{\frac{1}{4}}}\lVert R_{O} \rVert_{C_{t,x,q+1}}  \ll e^{-3 L^{\frac{1}{4}}} \lambda_{q+3} \delta_{q+3}. 
\end{equation}  
\end{proposition}
\noindent We will see in \eqref{est 193} the benefit of the extra factor $e^{2L^{\frac{1}{4}}}$ on the left side of \eqref{est 84}.
\begin{proof}[Proof of Proposition \ref{prop mRO}]
From the definition in \eqref{est 55d}, along with \eqref{est 69a} and \eqref{est 58a}, we can rewrite
\begin{align*}
&\divergence R_{O} = \divergence \mathring{R}_{l} + \Upsilon_{l} [ (\Lambda w_{q+1} \cdot \nabla) w_{q+1} - (\nabla w_{q+1})^{T} \cdot \Lambda w_{q+1} ] \\ 
& \hspace{5mm} =\divergence \mathring{R}_{l} + \Big[ \big( \Lambda  \sum_{j,k} \chi_{j} \mathbb{P}_{q+1, k} ( a_{k,j} b_{k} (\lambda_{q+1} \Phi_{j} )) \cdot \nabla\big) \sum_{j',k'}\chi_{j'} \mathbb{P}_{q+1, k'} \big(a_{k', j'} b_{k'} (\lambda_{q+1} \Phi_{j'})\big) \\
& \hspace{5mm} - \big(\nabla  \sum_{j,k} \chi_{j} \mathbb{P}_{q+1, k} (a_{k,j} b_{k} (\lambda_{q+1} \Phi_{j} )))^{T} \cdot \Lambda \sum_{j',k'}\chi_{j'} \mathbb{P}_{q+1, k'} (a_{k', j'} b_{k'} (\lambda_{q+1} \Phi_{j'} )\big)\Big]. 
\end{align*}
This structure is identical to the oscillation error in the additive case -- the only difference is that we have $\bar{\rho}_{j}$ now rather than $\rho_{j}$. By utilizing the estimates \eqref{est 63a}, \eqref{est 82d}-\eqref{est 82e} for $a_{k,j}$, we can prove \eqref{est 84} in a similar manner to \eqref{est 195}; we leave the details in Section \ref{mult oscillation details} for completeness. 
\end{proof}

Finally, we bound the two commutator errors, beginning with $R_{\text{Com1}}$.
\begin{proposition}\label{prop mRcom1}
$R_{\text{Com1}}$ from \eqref{est 53} satisfies for $\beta > \frac{1}{2}$ sufficiently close to $\frac{1}{2}$ and $a_{0}$ sufficiently large 
\begin{equation*}
\lVert R_{\text{Com1}} \rVert_{C_{t,x,q+1}}  \ll e^{-3 L^{\frac{1}{4}}} \lambda_{q+3} \delta_{q+3}. 
\end{equation*}
\end{proposition}

\begin{proof}[Proof of Proposition \ref{prop mRcom1}]
Relying on \eqref{est 92} and \cite[Equation (5)]{CDS12b}, we compute
\begin{align*}
& \lVert R_{\text{Com1}} \rVert_{C_{t,x,q+1}} \\
&\lesssim \lVert ( \Lambda y_{q}^{\bot} \Upsilon^{\frac{1}{2}} \nabla^{\bot} \cdot y_{q} \Upsilon^{\frac{1}{2}} ) \ast_{x} \phi_{l} - [ \Lambda y_{q}^{\bot} \Upsilon^{\frac{1}{2}} \ast_{x} \phi_{l} \nabla^{\bot} \cdot y_{q} \Upsilon^{\frac{1}{2}} \ast_{x} \phi_{l} ] \rVert_{C_{t,x,q+1}} \\
&+ \lVert ( \Upsilon [ \Lambda y_{q}^{\bot} \ast_{x} \phi_{l} ( \nabla^{\bot} \cdot y_{q} \ast_{x} \phi_{l}) ] ) \ast_{t} \varphi_{l} - \Upsilon_{l} ( [\Lambda y_{q}^{\bot} \ast_{x} \phi_{l} ( \nabla^{\bot} \cdot y_{q} \ast_{x} \phi_{l} ) ] \ast_{t} \varphi_{l} ) \rVert_{C_{t,x,q+1}} \\
&+ \lVert \Upsilon \rVert_{C_{t,q+1}} \lVert [ \Lambda y_{q}^{\bot} \ast_{x} \phi_{l} ( \nabla^{\bot} \cdot y_{q} \ast_{x} \phi_{l} ) ] \ast_{t} \varphi_{l} - \Lambda y_{q}^{\bot} \ast_{x} \phi_{l} \ast_{t} \varphi_{l} ( \nabla^{\bot} \cdot y_{q} \ast_{x} \phi_{l} \ast_{t} \varphi_{l} ) \rVert_{C_{t,x,q+1}}  \\
&\overset{\eqref{est 196} \eqref{est 47}}{\lesssim} l^{2} m_{L}^{8} e^{L^{\frac{1}{4}}} \lambda_{q}^{4} \delta_{q} + l^{1- 4 \delta} m_{L}^{12 - 8 \delta} e^{L^{\frac{1}{4}} (\frac{1}{2} - 2 \delta)} \lambda_{q}^{3- 4 \delta} \delta_{q}^{\frac{5}{4} - \delta} + e^{3L^{\frac{1}{4}}} m_{L}^{16} l^{2} \lambda_{q}^{6} \delta_{q}^{2}
\end{align*}
for any $\delta \in (0, \frac{1}{18})$. The first term can be seen to be dominated by the third using that $\beta \in (\frac{1}{2}, \frac{3}{4})$ from \eqref{bound beta}. For those latter two terms, we have for $\beta > \frac{1}{2}$ sufficiently close to $\frac{1}{2}$ and $a_{0}$ sufficiently large,  
\begin{align*}
& l^{1- 4 \delta} m_{L}^{12 - 8 \delta} e^{L^{\frac{1}{4}} ( \frac{1}{2} - 2 \delta)} \lambda_{q}^{3- 4 \delta}\delta_{q}^{\frac{5}{4} - \delta} + e^{3L^{\frac{1}{4}}} m_{L}^{16} l^{2} \lambda_{q}^{6} \delta_{q}^{2}  \\
\lesssim& a^{\frac{L}{2}} a^{b^{q} [ - \frac{b\alpha}{2} + 3 - 2 \beta]} a^{b(2\beta -1)} + a^{\frac{L}{2}} a^{b^{q} [ -2\alpha b + 6 - 4 \beta]} a^{b2(2\beta -1)} \ll e^{-3L^{\frac{1}{4}}}  \lambda_{q+3} \delta_{q+3}.
\end{align*}
Hence, by these bounds it follows that $R_{\text{Com1}}\ll e^{-3L^{\frac{1}{4}}}  \lambda_{q+3} \delta_{q+3}$ for $\beta > \frac{1}{2}$ sufficiently close to $\frac{1}{2}$ and $a_{0}$ sufficiently large , as required.
\end{proof}

We conclude these estimates with the second commutator error. 
\begin{proposition}\label{prop mRcom2}
$R_{\text{Com2}}$ from \eqref{est 55e} satisfies for $\beta > \frac{1}{2}$ sufficiently close to $\frac{1}{2}$ and $a_{0}$ sufficiently large 
\begin{equation}\label{est 301}
\lVert R_{\text{Com2}} \rVert_{C_{t,x,q+1}}  \ll e^{-3 L^{\frac{1}{4}}} \lambda_{q+3} \delta_{q+3}. 
\end{equation}  
\end{proposition}

\begin{proof}[Proof of Proposition \ref{prop mRcom2}]
Similarly to \cite[Equations (41)-(42)]{Y23b}, we can write 
\begin{equation}\label{est 260} 
R_{\text{Com2}} = \sum_{k=1}^{3}  R_{\text{Com2,k}} 
\end{equation} 
where
\begin{subequations}\label{est 194}
\begin{align}
& R_{\text{Com2,1}} \triangleq - (\Upsilon - \Upsilon_{l}) \Upsilon_{l}^{-1}  \mathring{R}_{l}, \label{est 194a}\\
& R_{\text{Com2,2}} \triangleq (\Upsilon - \Upsilon_{l}) \Upsilon_{l}^{-1}  R_{O}, \label{est 194b}\\
& R_{\text{Com2,3}} \triangleq  (\Upsilon - \Upsilon_{l})  \mathcal{B}[\Lambda w_{q+1}^{\bot} \nabla^{\bot} \cdot y_{l} + \Lambda y_{l}^{\bot} \nabla^{\bot} \cdot w_{q+1} + \Lambda y_{l}^{\bot} \nabla^{\bot} \cdot y_{l} ].\label{est 194c}
\end{align}
\end{subequations} 
We first estimate for $\delta \in (0, \frac{1}{8})$, $\beta > \frac{1}{2}$ sufficiently close to $\frac{1}{2}$, and $a_{0}$ sufficiently large 
\begin{align}\label{est 258}
\lVert R_{\text{Com2,1}} \rVert_{C_{t,x,q+1}} \overset{\eqref{est 196}}{\lesssim} \lambda_{q+1}^{-\alpha ( \frac{1}{2} - 2 \delta)} m_{L}^{2} e^{L^{\frac{1}{4}}} e^{-3L^{\frac{1}{4}}} \lambda_{q+2} \delta_{q+2} \bar{e} \ll e^{-3L^{\frac{1}{4}}} \lambda_{q+3} \delta_{q+3}.
\end{align}

Second, by consequence of our recent bound on the oscillation $R_{O}$ from Proposition \ref{prop mRO}, we can bound for $\beta > \frac{1}{2}$ sufficiently close to $\frac{1}{2}$ and $a_{0}$ sufficiently large 
\begin{equation}\label{est 193}
\lVert R_{\text{Com2,2}} \rVert_{C_{t,x,q+1}} \overset{\eqref{est 194b}}{\leq} e^{2L^{\frac{1}{4}}} \lVert R_{O} \rVert_{C_{t,x,q+1}}\overset{\eqref{est 84}}{\ll} e^{-3L^{\frac{1}{4}}} \lambda_{q+3} \delta_{q+3}. 
\end{equation}

Third, we estimate for $\delta\in (0,\frac{1}{8})$, $\beta>\frac{1}{2}$ sufficiently small, and $a_{0}$ sufficiently large
\begin{align}
& \lVert R_{\text{Com2,3}} \rVert_{C_{t,x,q+1}}  \nonumber \\
&\lesssim \lVert \Upsilon - \Upsilon_{l} \rVert_{C_{t,x,q+1}} \big\lVert \mathcal{B} \tilde{P}_{\approx \lambda_{q+1}} [ \Lambda w_{q+1}^{\bot} \nabla^{\bot} \cdot y_{l} + \Lambda y_{l}^{\bot} \nabla^{\bot} \cdot w_{q+1} ] + \mathcal{B} ( \Lambda y_{l} \nabla^{\bot} \cdot y_{l}) \big\rVert_{C_{t,x,q+1}} \nonumber \\
&\overset{\eqref{mult inductive 1c} \eqref{est 70a}}{\lesssim}\lambda_{q+1}^{-\alpha(\frac{1}{2} - 2 \delta)} a^{\frac{1}{2} L} \left[ \lambda_{q}^{1-\beta} \lambda_{q+1}^{-\beta} \lambda_{1}^{2\beta -1} + \lambda_{1}^{2\beta -1} \lambda_{q}^{2-2\beta} \right] \ll e^{-3L^{\frac{1}{4}}} \lambda_{q+3} \delta_{q+3}. \label{est 259} 
\end{align}
Applying \eqref{est 258}, \eqref{est 193}, and \eqref{est 259} to \eqref{est 260} gives us \eqref{est 301}. 
\end{proof}

\subsubsection{Control of the Energy}
Finally, we verify the energy control hypothesis \eqref{mult inductive 3}.
\begin{proposition}\label{prop mult control energy}
Define 
\begin{equation}\label{est 93}
\delta E(t) \triangleq \lvert e(t) (1- \lambda_{q+2} \delta_{q+2}) - \Upsilon^{2}(t) \lVert y_{q+1}(t) \rVert_{\dot{H}_{x}^{\frac{1}{2}}}^{2} \rvert. 
\end{equation}
Then, for all $t \in [t_{q+1}, T_{L}]$, 
\begin{equation}\label{mult energy goal}
\delta E(t) \leq \frac{1}{4} \lambda_{q+2} \delta_{q+2} e(t) 
\end{equation}
so that \eqref{mult inductive 3} holds at level $q+1$.
\end{proposition}

\begin{proof}[Proof of Proposition \ref{prop mult control energy}]
We first write using \eqref{est 58c} 
\begin{align}
&\delta E(t) = \left\lvert \bar{\gamma}_{q}(t) 4 (2\pi)^{2} \Upsilon^{2}(t) \Upsilon_{l}^{-1}(t) + \Upsilon^{2}(t) \lVert y_{q}(t) \rVert_{\dot{H}_{x}^{\frac{1}{2}}}^{2} - \Upsilon^{2}(t) \lVert y_{q+1}(t) \rVert_{\dot{H}_{x}^{\frac{1}{2}}}^{2} \right\rvert \label{est 98} \\ 
&= \left\lvert \bar{\gamma}_{q}(t) 4 (2\pi)^{2}  \Upsilon^{2}(t) \Upsilon_{l}^{-1}(t) + \Upsilon^{2}(t)  \int_{\mathbb{T}^{2}} \lvert\Lambda^{\frac{1}{2}} y_{q}(t) \rvert^{2} -  \lvert\Lambda^{\frac{1}{2}} y_{l}(t) \rvert^{2}  - \lvert \Lambda^{\frac{1}{2}} w_{q+1}(t) \rvert^{2} dx \right\rvert \nonumber 
\end{align}
where the last equality used the fact that $\supp \hat{y}_{l} \cap \supp \hat{w}_{q+1} = \emptyset$. Similar reasoning to \eqref{est 95}-\eqref{est 96} gives us 
\begin{align}
&\int_{\mathbb{T}^{2}} \lvert \Lambda^{\frac{1}{2}} w_{q+1} \rvert^{2} dx \label{est 97} \\
=& \sum_{j,k} \int_{\mathbb{T}^{2}} \lambda_{q+1} \tilde{w}_{q+1,j, k} \cdot \tilde{w}_{q+1,j, -k} + \tilde{w}_{q+1, j, k} \chi_{j} \cdot [ \mathbb{P}_{q+1, -k} \Lambda, \bar{a}_{-k, j} \psi_{q+1, j, -k} ] b_{-k} (\lambda_{q+1} x) \nonumber \\
& \hspace{10mm} + \chi_{j} [ \mathbb{P}_{q+1, k}, \bar{a}_{k, j} \psi_{q+1, j, k} ] b_{k} (\lambda_{q+1} x) \cdot \Lambda w_{q+1}dx.   \nonumber 
\end{align}
We additionally can rewrite similarly to \eqref{est 197}
\begin{align*}
& \sum_{j,k} \int_{\mathbb{T}^{2}} \lambda_{q+1} \tilde{w}_{q+1, j, k} \cdot \tilde{w}_{q+1, j, -k} dx  \\
\overset{\eqref{est 69b} \eqref{est 66} \eqref{est 58a}}{=}& \sum_{j,k} \int_{\mathbb{T}^{2}} \lambda_{q+1}\chi_{j}^{2}(t) \left[ \Upsilon_{l}^{-\frac{1}{2}} \lambda_{q+1}^{-\frac{1}{2}} \bar{\rho}_{j}^{\frac{1}{2}} \gamma_{k} \left( \Id - \frac{ \mathring{R}_{q,j}(t,x)}{\bar{\rho}_{j}} \right) \right]^{2}  \Tr (k^{\bot} \otimes k^{\bot}) dx  \\
=& \sum_{j} \int_{\mathbb{T}^{2}} \chi_{j}^{2}(t) \Upsilon_{l}^{-1} \bar{\rho}_{j} 2 \Tr \left( \Id - \frac{\mathring{R}_{q,j}}{\bar{\rho}_{j}} \right) dx 
\end{align*}
and then conclude similarly to \eqref{est 43} that 
\begin{align}
& \sum_{j,k} \int_{\mathbb{T}^{2}} \lambda_{q+1} \tilde{w}_{q+1, j, k} \cdot \tilde{w}_{q+1, j, -k} dx \nonumber \\
\overset{\eqref{est 58b}}{=}&  (2\pi)^{2} 4 \varepsilon_\gamma^{-1} \Upsilon_{l}^{-1} \sum_{j} \chi_{j}^{2} \sqrt{l^{2} + \lVert \mathring{R}_{q,j} (\tau_{q+1} j) \rVert_{C_{x}}^{2}} + 4 (2\pi)^{2} \sum_{j} \chi_{j}^{2} \Upsilon_{l}^{-1} \bar{\gamma}_{l} (\tau_{q+1} j). \label{est 99}
\end{align}
Therefore, combining \eqref{est 98}, \eqref{est 97}, and \eqref{est 99}, we obtain 
\begin{align}
\delta E&(t)  =  \Big\lvert \bar{\gamma}_{q}(t) 4 (2\pi)^{2}\Upsilon^{2}(t) \Upsilon_{l}^{-1}(t) - \Upsilon^{2} (t) 4 (2\pi)^{2} \sum_{j} \chi_{j}^{2}(t) \Upsilon_{l}^{-1} (t) \bar{\gamma}_{l} (\tau_{q+1} j) \label{term F1} \tag{F-1} \\
& + \Upsilon^{2} (t) \int_{\mathbb{T}^{2}} \lvert \Lambda^{\frac{1}{2}} y_{q}(t) \rvert^{2} - \lvert \Lambda^{\frac{1}{2}} y_{l} (t) \rvert^{2} dx  \label{term F2} \tag{F-2} \\
& - \Upsilon^{2}(t) (2\pi)^{2} 4 \varepsilon_\gamma^{-1}  \Upsilon_{l}^{-1}(t)\sum_{j} \chi_{j}^{2} (t) \sqrt{ l^{2} + \lVert \mathring{R}_{q,j} (\tau_{q+1} j) \rVert_{C_{x}}^{2}} \label{term F3} \tag{F-3}\\
& - \Upsilon^{2}(t) \sum_{j,k} \int_{\mathbb{T}^{2}} \tilde{w}_{q+1, j, k}(t) \chi_{j}(t)  \cdot [ \mathbb{P}_{q+1, -k} \Lambda, \bar{a}_{-k,j}\psi_{q+1, j, -k} ] b_{-k} (\lambda_{q+1} x) dx \label{term F4} \tag{F-4}\\
& - \Upsilon^{2}(t) \sum_{j,k} \int_{\mathbb{T}^{2}} \chi_{j}(t) [ \mathbb{P}_{q+1, k}, \bar{a}_{k,j} \psi_{q+1, j, k} ] b_{k} (\lambda_{q+1} x) \cdot \Lambda w_{q+1}(t) dx \Big\rvert. \label{term F5} \tag{F-5}
\end{align}
We begin with the difficult term, which is \eqref{term F1}. We compute using \eqref{est 18} 
\begin{align}
& \lvert \bar{\gamma}_{q}(t) 4(2\pi)^{2} \Upsilon^{2}(t) \Upsilon_{l}^{-1}(t) - \Upsilon^{2}(t) 4(2\pi)^{2} \sum_{j} \chi_{j}^{2}(t) \Upsilon_{l}^{-1}(t) \bar{\gamma}_{l} (\tau_{q+1} j) \rvert \label{est 111} \\
=& 4 (2\pi)^{2} \Upsilon^{2}(t) \Upsilon_{l}^{-1}(t) \left\lvert \bar{\gamma}_{q} (t) - \bar{\gamma}_{l} (t) + \sum_{j} \chi_{j}^{2} (t) [ \bar{\gamma}_{l} (t) - \bar{\gamma}_{l} (\tau_{q+1} j) ] \right\rvert \leq 4(2\pi)^{2}  ( \RomanIII_{1} + \RomanIII_{2}), \nonumber 
\end{align}  
defining 
\begin{equation}\label{est 108}
\RomanIII_{1} \triangleq e^{3L^{\frac{1}{4}}}\lvert \bar{\gamma}_{q}(t) - \bar{\gamma}_{l} (t) \rvert \hspace{1mm} \text{ and } \hspace{1mm} \RomanIII_{2} \triangleq e^{3L^{\frac{1}{4}}}\sum_{j} \chi_{j}^{2}(t) \lvert \bar{\gamma}_{l}(t) -\bar{\gamma}_{l} (\tau_{q+1} j) \rvert. 
\end{equation} 
We will estimate the two $\RomanIII_{1}$ and $\RomanIII_{2}$ separately. We rewrite $\RomanIII_{1}$ using \eqref{est 58}: 
\begin{align}
\RomanIII_{1} =& \frac{e^{3L^{\frac{1}{4}}}}{4(2\pi)^{2}} \Big[ (1- \lambda_{q+2} \delta_{q+2}) [ ( \Upsilon^{-2} - \Upsilon^{-2} \ast_{t} \varphi_{l}) (t) \Upsilon_{l}(t) e(t) \nonumber \\
& \hspace{40mm} + (\Upsilon^{-2} \ast_{t} \varphi_{l} )(t) ( \Upsilon_{l} - (\Upsilon_{l} \ast_{t} \varphi_{l} )) (t) e(t) \nonumber \\
& \hspace{40mm} + (\Upsilon^{-2} \ast_{t} \varphi_{l} )(t) (\Upsilon_{l} \ast_{t} \varphi_{l})(t) ( e - e\ast_{t} \varphi_{l}) (t) \nonumber \\
& \hspace{40mm} + ( \Upsilon^{-2} \ast_{t}\varphi_{l}) (\Upsilon_{l} \ast_{t} \varphi_{l} ) (e \ast_{t} \varphi_{l}) (t) - (\Upsilon^{-2} \Upsilon_{l} e) \ast_{t} \varphi_{l} ] \nonumber \\  
& \hspace{5mm} - (\Upsilon_{l} - \Upsilon_{l} \ast_{t} \varphi_{l})(t) \lVert y_{q} (t) \rVert_{\dot{H}_{x}^{\frac{1}{2}}}^{2} \nonumber \\
& \hspace{5mm} - (\Upsilon_{l} \ast_{t} \varphi_{l}) \left( \int_{\mathbb{T}^{2}} \lvert \Lambda^{\frac{1}{2}} y_{q} (t) \rvert^{2} dx - \int_{\mathbb{T}^{2}} \lvert \Lambda^{\frac{1}{2}} y_{q} \rvert^{2} \ast_{t} \varphi_{l} (t) dx \right)  \nonumber \\
& \hspace{5mm} - \left( ( \Upsilon_{l} \ast_{t} \varphi_{l}) \left( \int_{\mathbb{T}^{2}} \lvert \Lambda^{\frac{1}{2}} y_{q} \rvert^{2} \ast_{t} \varphi_{l} dx \right) - \left( \Upsilon_{l} \lVert y_{q} \rVert_{\dot{H}_{x}^{\frac{1}{2}}}^{2} \right)\ast_{t} \varphi_{l} \right) \Big]. \label{est 107}
\end{align}
We can estimate each part of this separately. First, we simply bound $\lvert 1- \lambda_{q+2} \delta_{q+2}  \rvert \leq 1$. Second, we bound for any $\delta \in (0, \frac{1}{8})$, 
\begin{subequations}\label{est 269}
\begin{align}
&\lvert ( \Upsilon^{-2} - \Upsilon^{-2} \ast_{t} \varphi_{l}) (t) \Upsilon_{l}(t) e(t) \rvert \overset{\eqref{est 196}}{\lesssim} l^{\frac{1}{2} - 2 \delta} m_{L}^{2} e^{2L^{\frac{1}{4}}} \bar{e}, \label{est 269a}\\
& \lvert ( \Upsilon^{-2} \ast_{t} \varphi_{l})(t) ( \Upsilon_{l} - ( \Upsilon_{l} \ast_{t} \varphi_{l})) (t) e(t) \rvert \lesssim e^{2L^{\frac{1}{4}}} l^{\frac{1}{2} - 2 \delta} m_{L}^{2}\bar{e}, \label{est 269b}\\
& \lvert ( \Upsilon^{-2} \ast_{t} \varphi_{l}) (\Upsilon_{l} \ast_{t} \varphi_{l}) (t) (e - e \ast_{t} \varphi_{l} )(t) \rvert \lesssim e^{3L^{\frac{1}{4}}} l \lVert e' \rVert_{C_{x}} \overset{\eqref{est 118 mult}}{\lesssim} e^{3L^{\frac{1}{4}}} l \tilde{e}, \label{est 269c}
\end{align}
\end{subequations} 
where we used the fact that $t \in [t_{q+1}, T_{L}] \subset [-2, L]$ and \eqref{est 118 mult}. Third, using \cite[Equation (5) on p. 88]{CDS12b} we can estimate 
\begin{align*}
&\lvert ( \Upsilon^{-2} \ast_{t} \varphi_{l}) ( \Upsilon_{l} \ast_{t} \varphi_{l}) ( e \ast_{t} \varphi_{l}) (t) - ( \Upsilon^{-2} \Upsilon_{l} e) \ast_{t} \varphi_{l} (t) \rvert \\
\leq& \left\lvert [ ( \Upsilon^{-2} \ast_{t} \varphi_{l}) (\Upsilon_{l} \ast_{t} \varphi_{l}) - (\Upsilon^{-2} \Upsilon_{l}) \ast_{t} \varphi_{l}] (e \ast_{t} \varphi_{l})  \right\rvert  \\
&+ \left\lvert ( ( \Upsilon^{-2} \Upsilon_{l}) \ast_{t} \varphi_{l}) (e \ast_{t} \varphi_{l}) - ( \Upsilon^{-2} \Upsilon_{l} e) \ast_{t} \varphi_{l} \right\rvert \\
\lesssim& l^{2(\frac{1}{2} - 2 \delta)} \left(\lVert \Upsilon^{-2} \rVert_{C_{t}^{\frac{1}{2} - 2 \delta}} \lVert \Upsilon_{l} \rVert_{C_{t}^{\frac{1}{2} - 2 \delta}}  \bar{e}+  \lVert \Upsilon^{-2} \Upsilon_{l} \rVert_{C_{t}^{\frac{1}{2} - 2 \delta}} \lVert e \rVert_{C_{t}^{\frac{1}{2} - 2 \delta}} \right) \lesssim l^{1 - 4 \delta}  m_{L}^{4} e^{L^{\frac{1}{4}}} (\bar{e} + \tilde{e}).
\end{align*}
Fourth, 
\begin{equation}
\lvert (\Upsilon_{l} - \Upsilon_{l} \ast_{t} \varphi_{l}) \lVert y_{q} (t) \rVert_{\dot{H}_{x}^{\frac{1}{2}}}^{2} \lesssim l^{\frac{1}{2} - 2 \delta} \lVert \Upsilon_{l} \rVert_{C_{t}^{\frac{1}{2} - 2 \delta}} \lambda_{q} \lVert y_{q} \rVert_{C_{t,x,q+1}}^{2}\lesssim l^{\frac{1}{2} - 2 \delta} m_{L}^{10} \lambda_{q} \bar{e}. 
\end{equation} 
Fifth, similarly to \eqref{est 102}, 
\begin{align*}
& \left\lvert ( \Upsilon_{l} \ast_{t} \varphi_{l})(t) \left\lvert \int_{\mathbb{T}^{2}} \lvert \Lambda^{\frac{1}{2}} y_{q}(t) \rvert^{2} dx - \int_{\mathbb{T}^{2}} \lvert \Lambda^{\frac{1}{2}} y_{q} \rvert^{2} \ast_{t} \varphi_{l}(t) dx  \right\rvert  \right\rvert  \\
\lesssim& e^{L^{\frac{1}{4}}} l \lVert \Lambda y_{q} \rVert_{C_{t,x,q+1}} \lVert \partial_{t} y_{q} \rVert_{C_{t,x,q+1}} \overset{\eqref{mult inductive 1c} \eqref{est 47}}{\lesssim}l e^{2L^{\frac{1}{4}}} m_{L}^{12} \lambda_{q}^{3} \delta_{q}^{\frac{3}{2}} \bar{e}^{\frac{3}{2}}.
\end{align*}
Sixth, by \cite[Equation (5)]{CDS12b}, \eqref{mult inductive 1}, and \eqref{est 47} 
\begin{align}
&\left\lvert ( \Upsilon_{l} \ast_{t} \varphi_{l}) \left( \int_{\mathbb{T}^{2}} \lvert \Lambda^{\frac{1}{2}} y_{q} \rvert^{2} \ast_{t} \varphi_{l} dx \right) - \left( \Upsilon_{l} \lVert y_{q} \rVert_{\dot{H}_{x}^{\frac{1}{2}}}^{2} \right) \ast_{t} \varphi_{l} \right\rvert \\
\overset{\eqref{est 196}}{\lesssim}& l^{1 - 4 \delta} m_{L}^{2} \lambda_{q}\left( \lVert y_{q} \rVert_{C_{t,x,q+1}}^{\frac{1}{2} + 2 \delta} \lVert y_{q} \rVert_{C_{t,q+1}^{1}C_{x}}^{\frac{1}{2} - 2 \delta} \right)^{2} \lesssim l^{1- 4 \delta} m_{L}^{14 - 6 \delta} \lambda_{q}^{3- 8 \delta}\delta_{q}^{1 - 4\delta}  \bar{e}^{\frac{3}{2} - 2\delta} e^{L^{\frac{1}{4}} (1- 4 \delta)}. \nonumber 
\end{align}
Applying these estimates to \eqref{est 107} and making use of the fact that $e(t) \geq \ushort{e}$ lead us to, for $\delta \in (0, \frac{1}{8})$, $\beta > \frac{1}{2}$ sufficiently close to $\frac{1}{2}$ and $a_{0}$ sufficiently large, 
\begin{align}
\RomanIII_{1} \lesssim a^{L} \Big[ \lambda_{q+1}^{-\frac{\alpha}{4}} \lambda_{q} +  \lambda_{q+1}^{-\alpha}   \lambda_{q}^{3-3\beta} \lambda_{1}^{(2\beta -1) \frac{3}{2}} + \lambda_{q+1}^{-\frac{\alpha}{2}} \lambda_{q}^{3-\beta} \lambda_{1}^{(2\beta -1)\frac{1}{2}} \Big] \overset{\eqref{mult bound b}}{\ll}\lambda_{q+2} \delta_{q+2} e(t). \label{est 109}
\end{align}
Furthermore, we quickly see that just like \eqref{est 104} that we can estimate from \eqref{est 108} 
\begin{align}
\RomanIII_{2}  \overset{\eqref{est 198} }{\lesssim} e^{6L^{\frac{1}{4}}} \tau_{q+1} l^{-1} \bar{e} \overset{\eqref{tau} \eqref{bound a}}{\lesssim}  a^{L} \lambda_{q+1}^{\frac{\alpha}{2} - \frac{1}{2} - \frac{\beta}{2}} \lambda_{q+2}^{-\frac{1}{2} + \beta} \lambda_{1}^{(2\beta -1) (-\frac{1}{4})}  \ll \lambda_{q+2} \delta_{q+2} e(t). \label{est 110}
\end{align}
Applying \eqref{est 109} and \eqref{est 110} thus allows us to conclude for our estimate \eqref{est 111} of term \eqref{term F1} that for $\beta > \frac{1}{2}$ sufficiently close to $\frac{1}{2}$ and $a_{0}$ sufficiently large 
\begin{align}
&\lvert \bar{\gamma}_{q}(t) 4(2\pi)^{2} \Upsilon^{2}(t) \Upsilon_{l}^{-1}(t)   \nonumber \\
& \hspace{20mm} - \Upsilon^{2}(t) 4(2\pi)^{2} \sum_{j} \chi_{j}^{2}(t) \Upsilon_{l}^{-1}(t) \bar{\gamma}_{l} (\tau_{q+1} j) \rvert  \ll \lambda_{q+2}\delta_{q+2} e(t).  \label{est 233}
\end{align} 

Now, for the term \eqref{term F2}, we have for $\beta > \frac{1}{2}$ sufficiently close to $\frac{1}{2}$ and $a_{0}$ sufficiently large, making use of the fact that $e(t) \geq \ushort{e}$, 
\begin{align}
& \Upsilon^{2}(t) \left\lvert \int_{\mathbb{T}^{2}} \lvert \Lambda^{\frac{1}{2}} y_{q}(t) \rvert^{2} - \lvert \Lambda^{\frac{1}{2}} y_{l} (t) \rvert^{2} dx \right\rvert  \nonumber \\
\lesssim& e^{2L^{\frac{1}{4}}} \lVert \Lambda^{\frac{1}{2}} y_{q}(t) \rVert_{L_{x}^{2}} \big[ l \lVert \Lambda^{\frac{1}{2}} y_{q} \rVert_{C_{t,q+1} C_{x}^{1}} + l \lVert \Lambda^{\frac{1}{2}} y_{q} \rVert_{C_{t,q+1}^{1}C_{x}} \big] \nonumber \\
\overset{\eqref{mult inductive 1c}}{\lesssim}&  e^{2L^{\frac{1}{4}}}  \lambda_{q}^{\frac{1}{2}} m_{L}^{4} \bar{e}^{\frac{1}{2}} l \left[ \lambda_{q}^{2} \delta_{q} m_{L}^{8} e^{L^{\frac{1}{4}}} \bar{e} \right] \lesssim a^{L}  \lambda_{q+1}^{-\alpha} \lambda_{1}^{2\beta -1} \lambda_{q}^{3-2\beta} \overset{\eqref{mult bound b}}{\ll} \lambda_{q+2}\delta_{q+2} e(t). \label{est 234}
\end{align}

Next, we estimate \eqref{term F3} for $\beta > \frac{1}{2}$ sufficiently close to $\frac{1}{2}$ and $a_{0}$ sufficiently large, 
\begin{align}
& \Upsilon^{2} (t) (2\pi)^{2} 4 \varepsilon_\gamma^{-1} \Upsilon_{l}^{-1}(t)  \sum_{j} \chi_{j}^{2} (t) \sqrt{ l^{2} + \lVert \mathring{R}_{q,j} (\tau_{q+1} j) \rVert_{C_{x}}^{2}}  \nonumber \\
\overset{\eqref{est 60}}{\leq} & \Upsilon^{2}(t) (2\pi)^{2} 4 \varepsilon_\gamma^{-1}\Upsilon_{l}^{-1}(t) \sum_{j} \chi_{j}^{2}(t) \sqrt{l^{2} + \lVert  \mathring{R}_{q} \ast_{t} \varphi_{l} ( \tau_{q+1} j)\rVert_{C_{x}}^{2}}  \nonumber \\ 
\overset{\eqref{est 18}  \eqref{mult inductive 2b}}{\leq}& e^{3L^{\frac{1}{4}}} (2\pi)^{2} 4 \varepsilon_\gamma^{-1} \lambda_{q+1}^{-\alpha} + \frac{1}{8} \lambda_{q+2} \delta_{q+2} e(t)  \overset{\eqref{mult bound b}}{\leq} \frac{1}{6} \lambda_{q+2}\delta_{q+2} e(t). \label{est 235}
\end{align}

Finally, for the last terms \eqref{term F4} and \eqref{term F5} we estimate by relying on \cite[Equation (A.17)]{BSV19} for $\beta > \frac{1}{2}$ sufficiently close to $\frac{1}{2}$ and $a_{0}$ sufficiently large,  
\begin{align}
& \Upsilon^{2} (t) \sum_{j,k} \Big[\big\lvert \int_{\mathbb{T}^{2}} \tilde{w}_{q+1, j,k} \chi_{j} (t) \cdot [\mathbb{P}_{q+1, -k} \Lambda, \bar{a}_{-k,j} \psi_{q+1, j, -k} ] b_{-k} (\lambda_{q+1} x) dx \big\rvert  \label{est 236}\\
& \hspace{20mm} + \big\lvert \int_{\mathbb{T}^{2}} \chi_{j}(t) [ \mathbb{P}_{q+1, k} \bar{a}_{k,j} \psi_{q+1, j, k} ] b_{k} (\lambda_{q+1} x) \cdot \Lambda w_{q+1} (t) dx \big\rvert \Big] \nonumber \\
&\overset{\eqref{est 69b} \eqref{est 70a} \eqref{est 82d}}{\lesssim} \Upsilon^{2}(t) \sum_{j,k} 1_{\supp \chi_{j}}(t) \Big[ \lVert \bar{a}_{k,j} \rVert_{C_{t,x,q+1}} \big( e^{2L^{\frac{1}{4}}} \delta_{q+1}^{\frac{1}{2}} \bar{e}^{\frac{1}{2}} \lambda_{q} \big)  \nonumber \\
& \hspace{20mm} + (e^{2L^{\frac{1}{4}}} \delta_{q+1}^{\frac{1}{2}} \bar{e}^{\frac{1}{2}} \lambda_{q}) ( e^{2 L^{\frac{1}{4}}} \bar{e}^{\frac{1}{2}} \delta_{q+1}^{\frac{1}{2}} ) \Big] \overset{\eqref{est 63b}}{\lesssim} a^{L} \lambda_{q} \delta_{q+1} \bar{e} \overset{\eqref{mult bound b}}{\ll}   \lambda_{q+2} \delta_{q+2} e(t). \nonumber 
\end{align}
Applying \eqref{est 233}, \eqref{est 234}, \eqref{est 235}, and \eqref{est 236} to \eqref{term F1}-\eqref{term F5} shows that for any $\iota > 0$, 
\begin{equation*} 
\delta E(t) \leq \frac{1}{6} \lambda_{q+2} \delta_{q+2} e(t) + \iota \lambda_{q+2} \delta_{q+2} e(t).
\end{equation*} 
Taking $\iota \leq \frac{1}{12}$ verifies \eqref{mult energy goal} and completes the proof of Proposition \ref{prop mult control energy}.
\end{proof}

\subsubsection{Proof of Proposition \ref{mult q to qplus1}}\label{proof mult q to qplus1}
\begin{proof}
All the necessary inductive hypotheses \eqref{mult inductive 1}, \eqref{mult inductive 2}, and \eqref{mult inductive 3} are satisfied from the statements of Propositions \ref{prop mult amplitudes}-\ref{prop mult control energy}. The details are identical to Section \ref{proof additive q to qplus1} in the additive case and previous works (e.g. \cite{HZZ19, HZZ21a, Y23a}) and hence omitted. 
\end{proof}

We are now ready to conclude the proof of Theorem \ref{Theorem 1.3} 
\begin{proof}[Proof of Theorem \ref{Theorem 1.3}]
For any fixed $T > 0$ and $\kappa \in (0,1)$, we take $L > 1$ sufficiently large and define $\mathfrak{t} \triangleq T_{L}$ to deduce \eqref{est 50} thanks to \eqref{limit as l to infinity}.  With such an $L > 1$ fixed, we take an arbitrary $e\in C_b^1(\mathbb{R};[\ushort{e},\infty))$ that satisfies \eqref{est 118 mult}. Then using such $e$, starting from $(y_{0}, \mathring{R}_{0}) = (0,0)$ as in Proposition \ref{mult step q0}, by Proposition \ref{mult q to qplus1}, we can obtain a family of solutions $\{ y_{q}, \mathring{R}_{q}\}_{q \in\mathbb{N}_{0}}$ to \eqref{mult convPDE} that satisfy the inductive hypotheses \eqref{mult inductive 1}-\eqref{mult inductive 3} and the Cauchy difference \eqref{lm cauchy}. Therefore, for all $\beta' \in (\frac{1}{2}, \beta)$, 
\begin{align*}
\sum_{q \geq 0} \lVert y_{q+1} (t) - y_{q} (t) \rVert_{C_{x}^{\beta'}} \lesssim& \sum_{q\geq 0} \lVert y_{q+1} (t) - y_{q}(t) \rVert_{C_{x}}^{1-\beta'} \lVert y_{q+1}(t) - y_{q}(t) \rVert_{C_{x}^{1}}^{\beta'} \nonumber \\
\overset{\eqref{lm cauchy} \eqref{mult inductive 1c} }{\lesssim}& e^{2L^{\frac{1}{4}} (1-\beta')} M_{0}^{\frac{1}{2}} \bar{e}^{\frac{1}{2}} m_{L}^{4\beta'} a^{b(\beta - \frac{1}{2})} \sum_{q\geq 0} a^{b(q+1) (\beta' - \beta)} < \infty.
\end{align*}
Similarly, we compute the temporal regularity: for all $\eta \in [0, \frac{\beta}{2-\beta})$ and all $q \in \mathbb{N}$, 
\begin{align*}
\lVert y_{q} - y_{q-1} \rVert_{C_{t}^{\eta} C_{x}} \lesssim& \lVert y_{q} - y_{q-1} \rVert_{C_{t,x}}^{1-\eta} \lVert y_{q} - y_{q-1} \rVert_{C_{t}^{1}C_{x}^{0}}^{\eta} \\
\overset{\eqref{lm cauchy} \eqref{est 47}}{\lesssim}& M_{0}^{\frac{1+ \eta}{2}} \bar{e}^{\frac{1+\eta}{2}} \lambda_{1}^{(2\beta -1) (\frac{1+\eta}{2})} \lambda_{q}^{-\beta (1+ \eta) + 2\eta}  e^{(2-\eta) L^{\frac{1}{4}}} m_{L}^{8\eta} 
\end{align*}
which vanishes as $q\nearrow \infty$; therefore, $\{y_{q}\}_{q \in \mathbb{N}_{0}}$ is Cauchy in $C( [ 0, \mathfrak{t}]; C_{x}^{\beta'}) \cap C^{\eta} ([0, \mathfrak{t}]; C_{x})$. It follows that we get a weak solution 
\begin{equation*}
y \triangleq \lim_{q\to\infty} y_{q} \in C([0,\mathfrak{t}]; C_{x}^{\beta'}) \cap C^{\eta} ( [0, \mathfrak{t}]; C_{x}) 
\end{equation*} 
to \eqref{est 265} for which there exists a deterministic constant $C_{L} < \infty$ such that 
\begin{equation}\label{est 163}
\lVert y \rVert_{C_{T_{L}} C_{x}^{\beta'} } + \lVert y \rVert_{C_{T_{L}}^{\eta} C_{x}} \leq C_{L} \hspace{3mm} \forall  \hspace{1mm} \beta' \in \left(\frac{1}{2}, \beta \right), \hspace{1mm} \forall \hspace{1mm} \eta \in \left[0, \frac{\beta}{2-\beta} \right).
\end{equation} 
Taking $q \nearrow \infty$ in \eqref{mult inductive 2b} shows that $\lVert \mathring{R}_{q} \rVert_{C_{t,x,q}} \to 0$ so that $v=\Upsilon y$ is a weak solution to \eqref{mult msqg} that satisfies \eqref{est 119}. Additionally, taking $q \nearrow \infty$ in \eqref{mult inductive 3} gives us the energy identity \eqref{energy}. Finally, the proof of non-uniqueness is identical to that of Theorem \ref{Theorem 1.1}. This completes the proof of Theorem \ref{Theorem 1.3}. 
\end{proof}

\section{Proof of Theorem \ref{Theorem 1.4}}\label{Section 6}
\subsection{Probabilistically Weak Solution}\label{Section 6.1}
The proof of Theorem \ref{Theorem 1.4} has some similarities to the proof of Theorem \ref{Theorem 1.2}; we sketch the proof elaborating on important differences. We start with a definition of a probabilistically weak solution to \eqref{general msqg} with linear multiplicative noise so that $G(v) dB = v dB$ that relies on the definitions of an $(\bar{\mathcal{B}}_{t}^{0})_{t\geq s}$-supermartingale, regular and exceptional time from Definition \ref{Definition 3}.   
\begin{define}\label{Definition 8}  
Fix $\iota \in (0,1)$ from Theorem \ref{Theorem 1.3}. Let $s \geq 0$, $\xi^{\text{in}} \in \dot{H}_{\sigma}^{\frac{1}{2}}$, and $\theta^{\text{in}} \in U_{1}$. Then $P \in \mathcal{P} (\bar{\Omega}_{0})$ is a probabilistically weak solution to \eqref{general msqg} where $G(v) dB = v dB$ with initial data $(\xi^{\text{in}}, \theta^{\text{in}})$ at initial time $s$ if 
\begin{enumerate}[wide=0pt] 
\item [](M1) $P (\{ \xi(t) = \xi^{\text{in}}, \theta(t) = \theta^{\text{in}} \hspace{1mm} \forall \hspace{1mm} t \in [0,s] \}) = 1$ and 
\begin{equation}
P \left( \{ (\xi,\theta) \in L_{\text{loc}}^{\infty} ([0,\infty); \dot{H}_{\sigma}^{\frac{1}{2}} \times U_{1} ) \cap L_{\text{loc}}^{2} ([0,\infty); \dot{H}_{\sigma}^{\frac{1}{2} + \iota} \times U_{1} ) \} \right) = 1, 
\end{equation} 
\item [](M2) under $P$, $\theta$ is a cylindrical $(\bar{\mathcal{B}}_{t})_{t\geq s}$-Wiener process on $U$ starting from initial data $\theta^{\text{in}}$ at initial time $s$ and for all $\psi^{k} \in C^{\infty} (\mathbb{T}^{2}) \cap \dot{H}_{\sigma}^{\frac{1}{2}}$ and $t\geq s$,  
\begin{align}
&\langle \xi(t) - \xi(s), \psi^{k} \rangle - \int_{s}^{t} \sum_{i,j=1}^{2} \langle \Lambda \xi_{i}, \partial_{i} \psi_{j}^{k} \xi_{j} \rangle_{\dot{H}_{x}^{-\frac{1}{2}} -\dot{H}_{x}^{\frac{1}{2}}} \nonumber \\
& \hspace{10mm} - \frac{1}{2} \langle \partial_{i} \xi_{j}, [\Lambda,  \psi_{i}^{k}] \xi_{j} \rangle_{\dot{H}_{x}^{-\frac{1}{2}}-\dot{H}_{x}^{\frac{1}{2}}} - \langle \xi, \Lambda^{\gamma} \psi^{k} \rangle dr = \int_{s}^{t} \langle \psi^{k}, G(\xi(r)) d\theta(r) \rangle, 
\end{align}
\item [](M3) for any $q \in \mathbb{N}$, there exists $C_{q} \geq 1$ such that the process $\{E^{q} (t) - E^{q} (s) \}_{t \geq s}$ with $E^{q}$ defined by    
\begin{align}\label{est 125} 
E^{q} (t) \triangleq \lVert \xi(t) \rVert_{\dot{H}_{x}^{\frac{1}{2}}}^{2q} e^{-C_{q} t} + \int_{0}^{t} 2q e^{-C_{q} r} \lVert \xi(r) \rVert_{\dot{H}_{x}^{\frac{1}{2}}}^{2q-2} \lVert \xi(r) \rVert_{\dot{H}_{x}^{\frac{1}{2} + \iota}}^{2} dr 
\end{align} 
is an a.s. $(\bar{\mathcal{B}}_{t}^{0})_{t\geq s}$-supermartingale under $P$ and $s$ is a regular time of $E^{q}$. 
\end{enumerate} 
The set of all such martingale solutions with the same constant $C_{q}$ in \eqref{est 125} will be denoted by $\mathcal{W} ( s, \xi^{\text{in}}, \theta^{\text{in}}, C_{q} )$; for brevity we write $\mathcal{W} (\xi^{\text{in}}, \theta^{\text{in}}, C_{q} ) \triangleq \mathcal{W} (\xi^{\text{in}}, \theta^{\text{in}}, C_{q} )\rvert_{s=0}$.  Given a probabilistically weak solution $P$, we define the exceptional times of $P$ as $T_{P}  \triangleq \cup_{q \in \mathbb{N}} \{\text{exceptional times of } E^{q}  \}$. 
\end{define} 

Next, we define a probabilistically weak solution up to a stopping time. 
\begin{define}\label{Definition 9} 
Fix $\iota \in (0,1)$ from Theorem \ref{Theorem 1.3}. Let $s \geq 0$, $\xi^{\text{in}} \in \dot{H}_{\sigma}^{\frac{1}{2}}$, and $\theta^{\text{in}} \in U_{1}$. Let $\tau: \bar{\Omega}_{0} \mapsto [0,\infty]$ be a $(\bar{\mathcal{B}}_{t})_{t\geq s}$-stopping time and set 
\begin{equation*}
\bar{\Omega}_{0,\tau} \triangleq \{\omega( \cdot \wedge \tau(\omega)): \omega \in \bar{\Omega}_{0} \} = \{ \omega \in \bar{\Omega}_{0}: (\xi, \theta) (t,\omega) = (\xi, \theta) (t\wedge \tau(\omega), \omega)  \hspace{1mm} \forall \hspace{1mm} t \geq 0\}, 
\end{equation*} 
which is a Borel subset of $\bar{\Omega}_{0}$ due to Borel measurability of $\tau$ so that $\mathcal{P} (\bar{\Omega}_{0,\tau}) \subset \mathcal{P} ( \bar{\Omega}_{0})$. Then $P \in \mathcal{P} (\bar{\Omega}_{0,\tau})$ is a probabilistically weak solution to \eqref{general msqg} where $G(v) dB = v dB$ on $[s,\tau]$ with initial data $(\xi^{\text{in}}, \theta^{\text{in}})$ at initial time $s$ if 
\begin{enumerate}[wide=0pt]
\item [](M1) $P (\{ \xi(t) = \xi^{\text{in}}, \theta(t) = \theta^{\text{in}} \hspace{1mm} \forall \hspace{1mm} t \in [0,s] \}) = 1$ and 
\begin{equation*}
P \left( \{  (\xi, \theta) \in L_{\text{loc}}^{\infty} ([0, \tau]; \dot{H}_{\sigma}^{\frac{1}{2}} \times U_{1}) \cap L_{\text{loc}}^{2} ([0,\tau]; \dot{H}_{\sigma}^{\frac{1}{2} + \iota} \times U_{1} ) \} \right) = 1, 
\end{equation*} 
\item [](M2) under $P$, $\langle \theta (\cdot \wedge \tau), l^{i} \rangle_{U}$ where $\{l^{i}\}_{i\in\mathbb{N}}$ is an orthonormal basis of $U$, is a continuous, square-integrable $(\bar{\mathcal{B}}_{t})_{t\geq s}$-martingale with initial data $\langle \theta^{\text{in}}, l^{i}  \rangle$ at initial time $s$ with a quadratic variation process given by $(t\wedge \tau - s) \lVert l^{i} \rVert_{U}^{2}$ and for every $\psi^{k} \in C^{\infty} (\mathbb{T}^{2}) \cap \dot{H}_{\sigma}^{\frac{1}{2}}$ and $t \geq s$,   
\begin{align*}
&\langle \xi(t \wedge \tau) - \xi(s), \psi^{k} \rangle - \int_{s}^{t\wedge \tau} \sum_{i,j=1}^{2} \langle \Lambda \xi_{i}, \partial_{i} \psi_{j}^{k} \xi_{j} \rangle_{\dot{H}_{x}^{-\frac{1}{2}} -\dot{H}_{x}^{\frac{1}{2}}}  \\
& \hspace{12mm} - \frac{1}{2} \langle \partial_{i} \xi_{j}, [\Lambda,  \psi_{i}^{k}] \xi_{j} \rangle_{\dot{H}_{x}^{-\frac{1}{2}}-\dot{H}_{x}^{\frac{1}{2}}} - \langle \xi, \Lambda^{\gamma} \psi^{k} \rangle dr  = \int_{s}^{t \wedge \tau} \langle \psi^{k}, \xi(r) \rangle d\theta(r), \nonumber
\end{align*} 
\item [](M3) for any $q \in \mathbb{N}$, there exists $C_{q} \geq 1$ such that the process $\{E^{q} ( t\wedge \tau) - E^{q}(s) \}_{t\geq s}$ with $E^{q}$ defined by \eqref{est 125} is an $(\bar{\mathcal{B}}_{t})_{t \geq s}$-supermartingale under $P$ and $s$ is a regular time of $E^{q}$.  
\end{enumerate} 
\end{define} 

The first result Proposition \ref{Proposition 6.1} can be proven by adapting Proposition \ref{Proposition 4.1} in the additive case; we leave details in Section \ref{Section 7.8} for completeness. 

\begin{proposition}\label{Proposition 6.1}  
\begin{enumerate}[wide=0pt] 
\indent 
\item For every $(s, \xi^{\text{in}}, \theta^{\text{in}}) \in [0,\infty) \times \dot{H}_{\sigma}^{\frac{1}{2}} \times U_{1}$, there exists a probabilistically weak solution $P \in \mathcal{P} (\bar{\Omega}_{0})$ to \eqref{general msqg} where $G(v) dB = v dB$ with initial data $(\xi^{\text{in}}, \theta^{\text{in}})$ at initial time $s$ according to Definition \ref{Definition 8}. 
\item Additionally, if there exists a family $\{ (s_{l}, \xi_{l}, \theta_{l} ) \}_{l \in \mathbb{N}} \subset [0,\infty) \times \dot{H}_{\sigma}^{\frac{1}{2}} \times U_{1}$ such that $\lim_{l\to\infty} \lVert (s_{l}, \xi_{l}, \theta_{l}) - (s, \xi^{\text{in}}, \theta^{\text{in}}) \rVert_{\mathbb{R} \times \dot{H}_{x}^{\frac{1}{2}} \times U_{1}} = 0$ and $P_{l} \in \mathcal{W} ( s_{l}, \xi_{l}, \theta_{l}, C_{q})$, then there exists a subsequence $\{P_{l_{k}} \}_{k \in \mathbb{N}}$ and $P \in \mathcal{W} ( s, \xi^{\text{in}}, \theta^{\text{in}}, C_{q})$ such that $P_{l_{k}}$ converges weakly to $P$. 
\end{enumerate} 
\end{proposition} 

The next two results are analogues of \cite[Lemmas 4.2-4.3]{Y23b} which are originally due to \cite[Propositions 5.2-5.3]{HZZ19}. 

\begin{lemma}\label{Lemma 6.2}
Let $\tau$ be a bounded $(\bar{\mathcal{B}}_{t})_{t \geq 0}$-stopping time. Then for every $\omega \in \bar{\Omega}_{0}$, there exists $Q_{\omega} \in \mathcal{P}(\bar{\Omega}_{0})$ such that for $\omega \in \{\xi(\tau) \in \dot{H}_{\sigma}^{\frac{1}{2}} \}$ 
\begin{subequations}\label{est 156}
\begin{align}
& Q_{\omega} ( \{ \omega' \in \bar{\Omega}_{0}: \hspace{0.5mm} \hspace{1mm}  ( \xi, \theta) (t, \omega') = (\xi, \theta) (t,\omega) \hspace{1mm} \forall \hspace{1mm} t \in [0, \tau(\omega)] \}) = 1, \label{est 156a} \\
& Q_{\omega} (A) = R_{\tau(\omega), \xi(\tau(\omega), \omega), \theta(\tau(\omega), \omega)} (A) \hspace{1mm} \forall \hspace{1mm} A \in \bar{\mathcal{B}}^{\tau(\omega)},  \label{est 156b}
\end{align} 
\end{subequations}
where $R_{\tau(\omega), \xi(\tau(\omega), \omega), \theta(\tau(\omega), \omega)} \in \mathcal{P} (\bar{\Omega}_{0})$ is a probabilistically weak solution to \eqref{general msqg} where $G(v) dB = v dB$ with initial data $(\xi(\tau(\omega), \omega), \theta(\tau(\omega), \omega))$ at initial time $\tau(\omega)$. Moreover, for every $A \in \bar{\mathcal{B}}$, the mapping $\omega \mapsto Q_{\omega}(A)$ is $\bar{\mathcal{B}}_{\tau}$-measurable.  
\end{lemma} 

\begin{proof}[Proof of Lemma \ref{Lemma 6.2}]
This is essentially identical to the proof of Lemma \ref{Lemma 4.2} applied to $\bar{\Omega}_{0}$ instead of $\Omega_{0}$ and relying on Proposition \ref{Proposition 6.1} instead of Proposition \ref{Proposition 4.1}. 
\end{proof}

\begin{lemma}\label{Lemma 6.3} 
Let $\tau$ be a bounded $(\bar{\mathcal{B}}_{t})_{t\geq 0}$-stopping time, $\xi^{\text{in}}  \in \dot{H}_{\sigma}^{\frac{1}{2}}$, and $P\in \mathcal{P} (\bar{\Omega}_{0})$ be a probabilistically weak solution to \eqref{general msqg} where $G(v) dB = v dB$ on $[0,\tau]$ with initial data $(\xi^{\text{in}}, 0)$ at initial time $0$ according to Definition \ref{Definition 9}. Suppose that there exists a Borel set $\mathcal{N} \subset \bar{\Omega}_{\tau}$ such that $P(\mathcal{N}) = 0$ and that $Q_{\omega}$ from Lemma \ref{Lemma 6.2} satisfies for every $\omega \in \bar{\Omega}_{\tau} \setminus \mathcal{N}$ 
\begin{equation}\label{est 151} 
Q_{\omega} (\{ \omega' \in \bar{\Omega}: \hspace{1mm}  \tau(\omega') = \tau(\omega) \}) = 1. 
\end{equation} 
Then the probability measure $P\otimes_{\tau}R \in \mathcal{P} (\bar{\Omega}_{0})$ defined by 
\begin{equation}\label{est 150}   
P \otimes_{\tau} R (\cdot) \triangleq \int_{\bar{\Omega}_{0}} Q_{\omega} (\cdot) P(d \omega)
\end{equation} 
satisfies $P\otimes_{\tau}R = P$ on the $\sigma$-algebra $\sigma \{ (\xi,\theta)(t\wedge \tau), t \geq 0 \}$ and it is a probabilistically weak solution to \eqref{general msqg} where $G(v) dB = v dB$ on $[0,\infty)$ with initial data $(\xi^{\text{in}}, 0)$ at initial time $0$. 
\end{lemma} 

\begin{proof}[Proof of Lemma \ref{Lemma 6.3}]
This proof is similar to the proof of Lemma \ref{Lemma 4.3}; we sketch it for completeness.  Again, the only difference in the definition of a solution from \cite[Definitions 4.1-4.2]{Y23b} and Definitions \ref{Definition 8}-\ref{Definition 9} is (M3). Thus, it suffices to show that $P \otimes_{\tau}R$, as a probabilistically weak solution on $[0,\infty)$ with initial data $(\xi^{\text{in}}, 0)$ at initial time 0, satisfies (M3). Analogously to \eqref{est 227}, \eqref{est 151}-\eqref{est 150} give us for any $q \in \mathbb{N}$
\begin{align}
& \mathbb{E}^{P \otimes_{\tau} R} [ \lvert E^{q} (t) \rvert ]   \nonumber \\
\leq& 2 \mathbb{E}^{P} [  \lvert E^{q}(t \wedge \tau) \rvert ] + \int_{\bar{\Omega}_{0}} \mathbb{E}^{Q_{\omega}} [ \lvert E^{q} (t) - E^{q} (\tau(\omega)) \rvert 1_{\{\tau(\omega) \leq t}\} ] dP(\omega). \label{est 275}
\end{align}
Analogously to \eqref{est 127}, within the second term of \eqref{est 275} we can estimate by the (M3) property of $Q^{\omega}$ on $\{\tau(\omega) \leq t \}$ due to Lemma \ref{Lemma 6.2},  
\begin{equation}\label{est 152} 
\mathbb{E}^{Q_{\omega}} [ \lvert E^{q} (t) - E^{q} (\tau(\omega)) \rvert 1_{\{\tau(\omega) \leq t\}} ] \leq 2 \mathbb{E}^{Q_{\omega}} [ \lVert \xi(\tau(\omega)) \rVert_{\dot{H}_{x}^{\frac{1}{2}}}^{2q} e^{-C_{q} \tau(\omega)} 1_{\{\tau(\omega) \leq t\}} ]. 
\end{equation} 
Following further analogous steps that led to \eqref{est 131} shows that 
\begin{align}\label{est 154}
\mathbb{E}^{P\otimes_{\tau} R} [  \lvert E^{q} (t) \rvert ] \leq C \lVert\xi(0) \rVert_{\dot{H}_{x}^{\frac{1}{2}}}^{2q} < \infty
\end{align}
so that $E^{q}$ is $P \otimes_{\tau} R$-integrable for all $q \in \mathbb{N}$. Next, analogous arguments that led to \eqref{est 155} and \eqref{est 132} lead us to, for $A \in \bar{\mathcal{B}}_{s}$, 
\begin{equation}\label{est 157}
\mathbb{E}^{Q_{\omega}} [ \left( E^{q} (t) - E^{q} \left( ( t \wedge \tau(\omega) ) \vee s \right) \right) 1_{A} ] \leq 0 \text{ if } t \geq \tau(\omega), t \geq s,  \text{ and } s \notin T_{Q_{\omega}}, 
\end{equation} 
where $T_{Q_{\omega}}$ is the exceptional set of $Q_{\omega}$ from Definition \ref{Definition 3}. Using \eqref{est 150}, for any measurable set $D \subset [0, t]$ and $A \in \bar{\mathcal{B}}_{s}$, we can obtain the same identity as \eqref{est 133}. Then, we can bound the right hand side of \eqref{est 133} identically to \eqref{est 276} utilizing \eqref{est 151}, \eqref{est 150}, \eqref{est 157}, and Lemma \ref{Lemma 6.3} and thereby deduce 
\begin{equation}\label{est 242} 
\int_{0}^{t} 1_{D} \mathbb{E}^{P\otimes_{\tau} R} [ (E^{q}(t) - E^{q}(s)) 1_{A} ] ds \leq 0 
\end{equation} 
which allows us to conclude that $E^{q}$ is an a.s. $(\bar{\mathcal{B}}_{t}^{0})$-supermartingale, completing the proof of (M3) and hence the proof of Lemma \ref{Lemma 6.3}. 
\end{proof} 

\subsection{Non-uniqueness in Law Globally-in-Time: Linear Multiplicative Case}
Similarly to Section \ref{Section 4.2}, here we combine the general results from Section \ref{Section 6.1} and Theorem \ref{Theorem 1.3}. We restrict ourselves to \eqref{mult msqg} where the driving nose $B$ is a $\mathbb{R}$-valued Wiener process on $(\Omega, \mathcal{F}, \mathbf{P})$ and consequently $U = U_{1} = \mathbb{R}$. We assume that an arbitrary $T > 0$ and $\kappa \in (0,1)$ have been fixed so we can find $L > 1$ that leads to \eqref{est 50} thanks to \eqref{limit as l to infinity}. Then for $\lambda \in \mathbb{N} \setminus \{1\}, \delta \in (0, \frac{1}{4})$ we define 
\begin{align}
\tau_{L}^{\lambda} (\omega) \triangleq& L \wedge \inf\left\{ t \geq 0: \lvert \theta(t) \vert \geq \left( L - \frac{1}{\lambda} \right)^{\frac{1}{4}} \right\} \nonumber \\
& \wedge \inf\left\{ t \geq 0: \lVert \theta \rVert_{C_{t}^{\frac{1}{2} - 2 \delta}} \geq \left( L - \frac{1}{\lambda} \right)^{\frac{1}{2}} \right\}. \label{est 160} 
\end{align}
It follows that $\{\tau_{L}^{\lambda}\}_{\lambda \in \mathbb{N}}$ is non-decreasing. By \cite[Lemma 3.5]{HZZ19} we see that $\tau_{L}^{\lambda}$ is a $(\bar{\mathcal{B}}_{t})_{t\geq 0}$-stopping time and consequently so is 
\begin{equation}\label{est 167}
\tau_{L} \triangleq \lim_{\lambda \to \infty} \tau_{L}^{\lambda}. 
\end{equation} 
We assume that $v$ is the solution constructed by Theorem \ref{Theorem 1.3} with the prescribed energy 
\begin{equation}\label{est 162} 
e(t) \triangleq d_{0} e^{d_{1} t}, \hspace{3mm} d_{1} > 0, d_{0} > 4 e^{2d_{1}}
\end{equation}  
(cf. \eqref{est 134} in the additive case) and define 
\begin{equation}\label{est 161}
P \triangleq \mathcal{L} ( (v,B)). 
\end{equation} 

\begin{proposition}\label{Proposition 6.4} 
Define $\tau_{L}$ by \eqref{est 167}. Then the probability measure $P$ defined by \eqref{est 161} is a probabilistically weak solution to \eqref{mult msqg} on $[0,\tau_{L}]$ according to Definition \ref{Definition 9}. The corresponding constants $C_{q}$ for $q \in \mathbb{N}$ in (M3) of Definition \ref{Definition 9} depend only on $d_{0}$ and $d_{1}$. 
\end{proposition} 

\begin{proof}[Proof of Proposition \ref{Proposition 4.4}] 
Similarly to previous works (e.g. \cite[Propositions 5.4]{HZZ19} and \cite[Proposition 4.4]{Y23b}), we can readily verify that $P$ satisfies (M1) and (M2), where $T_{L}$ was defined in \eqref{mult stopping time}. To verify (M3) we make two observations: for all $t \in [0,\tau_{L}]$, under $P$, for some $d_{2} > 0$, 
\begin{equation}\label{est 164}
\lVert \xi(t) \rVert_{\dot{H}_{x}^{\frac{1}{2}}}^{2}  \overset{\eqref{energy} \eqref{est 162}}{=} d_{0} e^{d_{1} t} \hspace{2mm} \text{ and } \hspace{2mm} \lVert \xi(t) \rVert_{\dot{H}_{x}^{\frac{1}{2} + \iota}}^{2} \overset{\eqref{define upsilon} \eqref{est 163} \eqref{mult stopping time}}{\leq} d_{2}.
\end{equation}   
It follows that as long as 
\begin{equation}\label{est 165}
C_{1} \geq d_{1} + \frac{2d_{2}}{d_{0}}, 
\end{equation} 
for all $t \geq s \geq 0$ we have 
\begin{align*}
& E^{1} (t \wedge \tau_{L}) - E^{1} (s \wedge \tau_{L})   \nonumber \\
\overset{\eqref{est 125} \eqref{est 164}}{\leq}& d_{0} e^{(d_{1} - C_{1}) (t \wedge \tau_{L})} - d_{0} e^{(d_{1} - C_{1}) (s \wedge \tau_{L})} + 2 \int_{s \wedge \tau_{L}}^{t \wedge \tau_{L}} e^{-C_{1} r} d_{2} dr    \nonumber \\
\leq& \left( d_{0} + \frac{ 2d_{2}}{d_{1} - C_{1}} \right) [ e^{(d_{1} - C_{1}) (t \wedge \tau_{L}) } - e^{(d_{1} - C_{1}) (s \wedge \tau_{L} )} ]  \leq 0 
\end{align*}
and hence $E^{1} (t \wedge \tau_{L})$ is a supermartingale under $P$. For general $q \in \mathbb{N}$, it follows similarly that as long as 
\begin{equation}\label{est 166}
C_{q} \geq d_{1} q + \frac{2q d_{2}}{d_{0}}, 
\end{equation} 
we have for all $t \geq s \geq 0$ 
\begin{align*}
& E^{q} (t \wedge \tau_{L}) - E^{q} ( s \wedge \tau_{L})  \nonumber \\
\overset{\eqref{est 125} \eqref{est 164}}{\leq}& d_{0}^{q} [ e^{d_{1} (t \wedge \tau_{L}) q - C_{q} (t \wedge \tau_{L} )} - e^{d_{1} ( s \wedge \tau_{L}) q - C_{q}(s \wedge \tau_{L})} ]+ 2q d_{2} d_{0}^{q-1} \int_{s \wedge \tau_{L}}^{t \wedge \tau_{L}} e^{r [d_{1} (q-1) - C_{q}]} dr    \nonumber \\
\leq& \left( d_{0}^{q} + \frac{2q d_{2} d_{0}^{q-1}}{d_{1} q - C_{q}} \right) [ e^{(d_{1} q - C_{q}) (t \wedge \tau_{L})} - e^{(d_{1} q - C_{q}) (s\wedge \tau_{L} )} ] \leq 0,  
\end{align*}
and hence $E^{q} (t \wedge \tau_{L})$ is a supermartingale under $P$ for all $q \in \mathbb{N}$. This completes the proof of Proposition \ref{Proposition 6.4}. 
\end{proof} 

Finally, the following Proposition \ref{Proposition 6.5} proves non-uniqueness in law of the global-in-time solution to \eqref{mult msqg}. 
\begin{proposition}\label{Proposition 6.5} 
\noindent
\begin{enumerate}
\item The probability measure $P \otimes_{\tau_{L}} R$ defined by \eqref{est 150}  with the $\tau_{L}$ defined by \eqref{est 167} is a probabilistically weak solution to \eqref{mult msqg} on $[0,\infty)$ according to Definition \ref{Definition 8}. 
\item There exist constant $C_{q} \geq 0$ for $q \in \mathbb{N}$ such that the probabilistically weak solutions to \eqref{mult msqg} associated to such $C_{q}$ in (M3) of Definition \ref{Definition 8} are not unique.  
\end{enumerate} 
\end{proposition} 

\begin{proof}[Proof of Proposition \ref{Proposition 6.5}]
$ $\newline 
(1) This part can be proven by verifying \eqref{est 151} and relying on Lemmas \ref{Lemma 6.2}-\ref{Lemma 6.3} identically to previous works such as  \cite[Proposition 4.5]{Y23b} and \cite[Proposition 5.5]{HZZ19}. 

\noindent (2) Once again, we are able to provide two proofs. For the first proof, we choose two distinct energies 
\begin{equation*}
e_{1}(t) \triangleq d_{0} e^{d_{1,1} t}, \hspace{3mm} e_{2} (t) \triangleq d_{0} e^{d_{1,2} t}, \hspace{3mm} d_{1,1} \neq d_{1,2}, 
\end{equation*} 
where $d_{0}, d_{1,1}$, and $d_{1,2}$ satisfy \eqref{est 162}, construct the corresponding convex integration solutions $P_{1}, P_{2}$ with the same constants $C_{q}$ for $q \in \mathbb{N}$ in (M3) of Definition \ref{Definition 9}, and extend both to $[0,\infty)$ via $P_{1} \otimes_{\tau_{L}} R$ and $P_{2}\otimes_{\tau_{L}} R$ by Proposition \ref{Proposition 6.5} (1). They have same initial data but different energies. 

For the second proof, we actually deviate from the proof in the additive case in which the construction o the solution in Theorem \ref{Theorem 1.1} was only up to $\tau \leq 1$. We fix $d_{1} > 2$ for simplicity. Then we fix an arbitrary $T > 0$ and $\kappa \in (0,1)$; without loss of generality, we add a lower bound on $\kappa$ as follows: 
\begin{equation}\label{est 176}
\kappa \in (e^{-T},1) \text{ so that } \kappa e^{d_{1} T} \geq e^{T}. 
\end{equation} 
Then for such fixed $T > 0$ and $\kappa \in (e^{-T}, 1)$, we find $L > 1$ sufficiently large so that \eqref{est 50} holds. Having fixed such $L > 1$,  we then  choose 
\begin{equation}\label{est 168}
e(t) \triangleq d_{0} e^{d_{1} t}, \hspace{3mm} d_{0} > 4e^{2d_{1}},
\end{equation} 
so that $d_{0}$ and $d_{1}$ satisfy \eqref{est 162}, and construct a solution $P$ on $[0, \tau_{L}]$ such that 
\begin{equation}\label{est 169} 
\lVert \xi(t) \rVert_{\dot{H}_{x}^{\frac{1}{2}}}^{2} = e(t)  \overset{\eqref{est 168}}{=} d_{0} e^{d_{1} t} \hspace{3mm} P\text{-a.s. for all } t \in [0, \tau_{L} ] 
\end{equation} 
via Theorem \ref{Theorem 1.3}. Taking $t = 0$ in \eqref{est 169} implies 
\begin{equation}\label{est 277} 
d_{0} = e(0) = \lVert \xi(0) \rVert_{\dot{H}_{x}^{\frac{1}{2}}}^{2}. 
\end{equation} 
Then we rely on Proposition \ref{Proposition 6.5} (1) to obtain a probabilistically weak solution $P \otimes_{\tau_{L}} R$ to \eqref{mult msqg} on $[0,\infty)$. Identically to the \cite[Proof of Theorem 2.2]{Y23b} (see also \cite[Proof of Theorem 1.4]{HZZ19}), because $P \otimes_{\tau_{L}} R ( \{ \tau_{L} \geq T \}) = \mathbf{P} ( \{ T_{L} \geq T \})$, we can obtain 
\begin{equation}\label{est 175}
P \otimes_{\tau_{L}} R ( \{ \tau_{L} \geq T \}) > \kappa
\end{equation} 
using \eqref{est 50} that we already secured. It follows that 
\begin{equation}\label{est 177}
\mathbb{E}^{P \otimes_{\tau_{L}} R} [ \lVert \xi(T) \rVert_{\dot{H}_{x}^{\frac{1}{2}}}^{2} ]  \overset{\eqref{est 175} \eqref{est 169}}{>} \kappa d_{0} e^{d_{1} T} \overset{\eqref{est 176}}{\geq} d_{0} e^{T}. 
\end{equation} 
If $Q$ is the law of the classical solution constructed via Galerkin approximation from $\xi(0)$, then \eqref{est 277} gives us 
\begin{equation}\label{est 174}
\mathbb{E}^{Q} [ \lVert \xi(T) \rVert_{\dot{H}_{x}^{\frac{1}{2}}}^{2} ] \leq d_{0} e^{T} 
\end{equation} 
so that comparing \eqref{est 177} and \eqref{est 174} reveals that the two solutions are different. This implies a lack of joint uniqueness in law and consequently non-uniqueness in law due to Cherny's theorem from \cite{C03} and \cite[Lemma C.1]{HZZ19}. 

Finally, identical argument to the additive case shows that there exist constants $C_{q} \geq 0, q \in \mathbb{N}$ such that the probabilistically weak solutions to \eqref{mult msqg} associated with $C_{q} \geq 0, q \in \mathbb{N}$, are not unique. 
\end{proof}

\subsection{Non-uniqueness of a.s. Markov Selection: Linear Multiplicative Case} 
We fix the constants $C_{q}, q \in \mathbb{N}$. We recall the definitions of an a.s. Markov family from Definition \ref{Definition 5} and an a.s. pre-Markov family in Definition \ref{Definition 9}. 

\begin{proposition}\label{Proposition 6.6}
There exists a Lebesgue null set $\mathcal{T} \subset (0,\infty)$ such that for all $T \not\in \mathcal{T}$, the family $\{ \mathcal{W} ( \xi^{\text{in}}, \theta^{\text{in}}, C_{q}) \}_{( \xi^{\text{in}}, \theta^{\text{in}}) \in \dot{H}_{\sigma}^{\frac{1}{2}} \times \mathbb{R}}$ satisfies the disintegration property in Definition \ref{Definition 10}. 
\end{proposition} 

\begin{proof}[Proof of Proposition \ref{Proposition 6.6}]
We fix $(\xi^{\text{in}}, \theta^{\text{in}}) \in \dot{H}_{\sigma}^{\frac{1}{2}} \times \mathbb{R}, P \in \mathcal{W} ( \xi^{\text{in}}, \theta^{\text{in}}, C_{q})$, and $T$ a regular time of $P$. Let $P ( \cdot \lvert \bar{\mathcal{B}}_{T}^{0}) (\omega)$ be a r.c.p.d. of $P$ w.r.t. $\bar{\mathcal{B}}_{T}^{0}$. According to Definition \ref{Definition 10} (1), we need to find a $P$-null set $N \in \bar{\mathcal{B}}_{T}^{0}$ such that for all $\omega \notin N$, 
\begin{align*}
( \xi, \theta) (T, \omega) \in \dot{H}_{\sigma}^{\frac{1}{2}} \times \mathbb{R}, \hspace{3mm} P (\Phi_{T} (\cdot) \lvert \bar{\mathcal{B}}_{T}^{0})(\omega) \in \mathcal{W} ( \xi(T,\omega), \theta(T, \omega), C_{q}).
\end{align*}
The first claim follows by the assumption that $T$ is a regular time of $P$ and the latter claim can also be proven by relying on Lemma \ref{Lemma 2.2} and \cite[Propositions B.1 and B.4]{FR08} similarly to the proof of Proposition \ref{Proposition 4.6}. 
\end{proof} 

\begin{proposition}\label{Proposition 6.7} 
There exists a Lebesgue null set $\mathcal{T} \subset (0,\infty)$ such that for all $T \not\in \mathcal{T}$, the family $\{ \mathcal{W} (\xi^{\text{in}}, \theta^{\text{in}}, C_{q}) \}_{(\xi^{\text{in}},\theta^{\text{in}}) \in \dot{H}_{\sigma}^{\frac{1}{2}} \times \mathbb{R}}$ satisfies the reconstruction property of Definition \ref{Definition 10}. 
\end{proposition} 

\begin{proof}[Proof of Proposition \ref{Proposition 6.7}]
We fix $(\xi^{\text{in}}, \theta^{\text{in}}) \in \dot{H}_{\sigma}^{\frac{1}{2}} \times \mathbb{R}, P \in \mathcal{W} ( \xi^{\text{in}},\theta^{\text{in}}, C_{q})$, and let $T$ be a regular time of $P$. We assume that a mapping $\mathcal{E}: \bar{\Omega}_{0} \mapsto \mathcal{P}_{\dot{H}_{\sigma}^{\frac{1}{2}} \times \mathbb{R}} ( \bar{\Omega}_{0})$ defined by $\mathcal{E}(\omega) \triangleq Q_{\omega}$ satisfies the hypothesis of Lemma \ref{Lemma 2.3} (specifically, for any $A \in \bar{\mathcal{B}}$, $\omega \mapsto Q_{\omega}(A)$ is $(\bar{\mathcal{B}}_{T}^{0})$-measurable and for any $\omega \in \bar{\Omega}_{0}, Q_{\omega} ( \{ \tilde{\omega} \in \bar{\Omega}_{0}: (\xi,\theta) (T, \tilde{\omega}) = (\xi,\theta) (T,\omega) \}) = 1$) and that there exists a $P$-null set $N \in \bar{\mathcal{B}}_{T}^{0}$ such that for all $\omega \notin N$, $(\xi,\theta) (T, \omega) \in \dot{H}_{\sigma}^{\frac{1}{2}} \times U_{1}, Q_{\omega} \circ \Phi_{T} \in \mathcal{W} ( \xi(T,\omega), \theta(T,\omega), C_{q})$. According to Definition \ref{Definition 10} we only need to show that $P \otimes_{T} Q \in \mathcal{W}(\xi^{\text{in}},\theta^{\text{in}}, C_{q})$ and this can be shown identically to the proof of Lemma \ref{Lemma 6.3} with ``$\tau$'' therein replaced by $T$. 
\end{proof} 

\begin{proof}[Proof of Theorem \ref{Theorem 1.4}] 
$ $\newline
(1) This part was already proven thanks to Proposition \ref{Proposition 6.5}. \\
\noindent (2) As a consequence of  Proposition \ref{Proposition 6.5} we know that there exists an initial data $\xi^{\text{in}} \in \dot{H}_{\sigma}^{\frac{1}{2}}$ that leads to at least two probabilistically weak solutions $P, Q$ to \eqref{mult msqg} on $[0,\infty)$ and a functional $f: \dot{H}_{\sigma}^{\frac{1}{2}} \mapsto \mathbb{R}$ that is continuous and bounded that satisfies \eqref{est 180}. With Propositions \ref{Proposition 6.6}-\ref{Proposition 6.7} in hand, we can apply Lemma \ref{Lemma 7.3} to prove non-uniqueness of a.s. Markov selections of $\{ \mathcal{W} ( \xi^{\text{in}}, \theta^{\text{in}}, C_{q})\}_{(\xi^{\text{in}}, \theta^{\text{in}}) \in \dot{H}_{\sigma}^{\frac{1}{2}} \times \mathbb{R}}$ identically to \eqref{est 237} in proof of the additive case. 
\end{proof}

\section{Appendix}\label{Appendix}
\subsection{Further Preliminaries}\label{Section 7.1}
First is an important geometric lemma.
\begin{lemma}\rm{(\hspace{1sp}\cite[Lemma 6.6]{BV19b})}\label{geometric lemma}
Let $B(\Id,\varepsilon)$ denote the ball of symmetric 2x2 matrices centered on the identity of radius $\varepsilon>0$. There exists $\varepsilon_{\gamma}>0$ for which there exist disjoint finite subsets $\Gamma_j\subset \mathbb{S}^{1}$ for $j\in\{1,2\}$ and smooth positive functions $\gamma_{k}\in C^{\infty}(B(\Id,\varepsilon_{\gamma})),k\in\Gamma_j$ such that
\begin{align*}
&(1)~5\Gamma_j\subset \mathbb{Z}^2,\\
&(2)~\mathrm{If~} k\in\Gamma_j,\mathrm{~then~} -k\in\Gamma_j \mathrm{~and~} \gamma_k=\gamma_{-k},\\
&(3)~\frac{1}{2}\sum_{k\in\Gamma_j}(\gamma_k(R))^2(k^{\perp}\otimes k^{\perp})=R~ \hspace{1mm}\forall \hspace{1mm} R\in B(\Id,\varepsilon_{\gamma}),\\
\text{and }&(4)~|k+k'|\geq\frac{1}{2}~\hspace{1mm}\forall \hspace{1mm}k,k'\in\Gamma_j\mathrm{~such~ that~}k+k'\neq 0.
\end{align*}
\end{lemma}

\begin{lemma}\label{transport lemma}\rm{(\hspace{1sp}\cite[Lemma A.2]{BSV19}, \cite[Proposition D.1]{BDIS15})}
Within this lemma, we denote $D\triangleq \partial_{t} + u \cdot \nabla$. Given a smooth vector field $u(t,x)$, as well as $f_{0}$ and $g$, consider a smooth solution $f$ to
\begin{align*}
\partial_{t} f + (u\cdot\nabla) f = g, \hspace{5mm}  f(t_{0}, x) = f_{0}(x), 
\end{align*}
where $u(t,x)$ is a given smooth vector field. Let $\Phi$ be the inverse of the flux $X$ of $u$ starting at time $t_{0}$ as the identity (i.e., $\partial_{t} X = u(X, t), X(t_{0}, x) = x$) so that $f(t,x) = f_{0} (\Phi(t,x))$ in case $g \equiv 0$. It follows that for $t > t_{0}$,
\begin{subequations}\label{est 2} 
\begin{align}
& \lVert f (t) \rVert_{C_{x}} \leq \lVert f_{0} \rVert_{C_{x}} + \int_{t_{0}}^{t} \lVert g(\tau) \rVert_{C_{x}} d\tau, \label{est 2a} \\
& \lVert Df (t) \rVert_{C_{x}} \leq \lVert D f_{0} \rVert_{C_{x}} e^{( t - t_{0}) \lVert Du \rVert_{C_{t, x}}} + \int_{t_{0}}^{t} e^{(t- \tau)  \lVert D u \rVert_{C_{t, x}}} \lVert Dg(\tau) \rVert_{C_{x}} d\tau. \label{est 2b} 
\end{align}
\end{subequations}  
More generally, for any $N \in \mathbb{N} \setminus \{1\}$, there exists a constant $C = C(N)$ such that 
\begin{align}
\lVert D^{N} f(t) &\rVert_{C_{x}}  \leq ( \lVert D^{N} f_{0} \rVert_{C_{x}} + C(t-t_{0})  \lVert D^{N} u \rVert_{C_{t,x}} \lVert Df_{0} \rVert_{C_{x}} ) e^{C(t-t_{0})  \lVert D u \rVert_{C_{t,x}}} \label{est 7}\\
&+ \int_{t_{0}}^{t} e^{C(t-\tau)  \lVert Du \rVert_{C_{t,x}}} (  \lVert D^{N} g(\tau) \rVert_{C_{x}} + C(t-\tau)  \lVert D^{N} u \rVert_{C_{t,x}} \lVert Dg(\tau) \rVert_{C_{x}} )d\tau.    \nonumber 
\end{align} 
Additionally, 
\begin{subequations}\label{est 15}
\begin{align}
&\lVert D\Phi(t) - \Id \rVert_{C_{x}} \leq e^{(t-t_{0}) \lVert Du \rVert_{C_{t,x}}} - 1 \leq (t- t_{0}) \lVert Du \rVert_{C_{t,x}} e^{(t- t_{0}) \lVert Du \rVert_{C_{t,x}}}, \label{est 15a} \\
& \lVert D^{N} \Phi(t) \rVert_{C_{x}} \leq C(t-t_{0})  \lVert D^{N} u \rVert_{C_{t,x}} e^{C(t-t_{0})  \lVert Du \rVert_{C_{t,x}}} \hspace{3mm} \forall \hspace{1mm} N \in \mathbb{N}\setminus \{1\}. \label{est 15b} 
\end{align}
\end{subequations}
\end{lemma}

\begin{define}\rm{(\hspace{1sp}\cite[Definition 2.5]{FR08}, \cite[Definition 2.6]{GRZ09})}\label{Definition 10}
Let the mapping $\mathcal{D}: \dot{H}_{\sigma}^{\frac{1}{2}} \mapsto \text{Comp} ( \mathcal{P}_{\dot{H}_{\sigma}^{\frac{1}{2}}} (\Omega_{0}))$ defined by $\mathcal{D} (\xi^{\text{in}}) = \mathcal{C} (\xi^{\text{in}}, C_{q})$ (resp. $\mathcal{D}: \dot{H}_{\sigma}^{\frac{1}{2}} \times U_{1} \mapsto \text{Comp} ( \mathcal{P}_{\dot{H}_{\sigma}^{\frac{1}{2}} \times U_{1}} (\bar{\Omega}_{0}))$ defined by $\mathcal{D} (\xi^{\text{in}}, \theta^{\text{in}}) = \mathcal{W} (\xi^{\text{in}}, \theta^{\text{in}}, C_{q})$) be Borel measurable.  We say that a family $\{ \mathcal{C} ( \xi^{\text{in}}, C_{q}) \}_{\xi^{\text{in}} \in \dot{H}_{\sigma}^{\frac{1}{2}}}$ (resp. $\{ \mathcal{W} ( \xi^{\text{in}}, \theta^{\text{in}}, C_{q}) \}_{(\xi^{\text{in}}, \theta^{\text{in}}) \in \dot{H}_{\sigma}^{\frac{1}{2} \times U_{1}}}$) is a.s. pre-Markov if for all $\xi^{\text{in}} \in \dot{H}_{\sigma}^{\frac{1}{2}}$ and all $P \in \mathcal{C} (\xi^{\text{in}}, C_{q})$ (resp. $(\xi^{\text{in}}, \theta^{\text{in}}) \in \dot{H}_{\sigma}^{\frac{1}{2}} \times U_{1}$ and all $P \in \mathcal{W} (\xi^{\text{in}}, \theta^{\text{in}}, C_{q})$), there exists a Lebesgue null set $\mathcal{T} \subset (0,\infty)$ such that the following holds for all $T \notin \mathcal{T}$. 
\begin{enumerate}
\item (Disintegration) There exists a $P$-null set $N \in \mathcal{B}_{T}^{0}$ such that for all $\omega \notin N$, 
\begin{align*}
&\xi(T, \omega) \in \dot{H}_{\sigma}^{\frac{1}{2}}, \hspace{3mm} P(\Phi_{T} (\cdot) \lvert \mathcal{B}_{T}^{0}) (\omega) \in \mathcal{C} ( \xi(T, \omega), C_{q})\\
& (\text{resp. } (\xi, \theta)(T, \omega) \in \dot{H}_{\sigma}^{\frac{1}{2}} \times U_{1}, \hspace{3mm} P(\Phi_{T} (\cdot) \lvert \bar{\mathcal{B}}_{T}^{0}) (\omega) \in \mathcal{W} ( \xi(T, \omega), \theta(T, \omega), C_{q})). 
\end{align*} 
\item (Reconstruction) If a mapping $\mathcal{E}: \Omega_{0} \mapsto \mathcal{P}_{\dot{H}_{\sigma}^{\frac{1}{2}}} (\Omega_{0})$ defined by $\mathcal{E} (\omega) = Q_{\omega}$ (resp. $\mathcal{E}: \bar{\Omega}_{0} \mapsto \mathcal{P}_{\dot{H}_{\sigma}^{\frac{1}{2}} \times U_{1}} (\bar{\Omega}_{0})$ defined by $\mathcal{E} (\omega) = Q_{\omega}$) satisfies the hypothesis of Lemma \ref{Lemma 2.3} and there exists a $P$-null set $N \in \mathcal{B}_{T}^{0}$ (resp. $N \in \bar{\mathcal{B}}_{T}^{0}$) such that for all $\omega \notin N$, 
\begin{align*}
&\xi(T, \omega) \in \dot{H}_{\sigma}^{\frac{1}{2}}, \hspace{3mm} Q_{\omega} \circ \Phi_{T} \in \mathcal{C} (\xi(T, \omega), C_{q}) \\
&  (\text{resp. } (\xi,\theta) (T, \omega) \in \dot{H}_{\sigma}^{\frac{1}{2}} \times U_{1}, \hspace{3mm} Q_{\omega} \circ \Phi_{T} \in \mathcal{W} (\xi(T, \omega), \theta(T, \omega), C_{q})), 
\end{align*} 
then $P \otimes_{T} Q \in \mathcal{C} (\xi^{\text{in}}, C_{q})$ (resp. $P \otimes_{T} Q \in \mathcal{W} (\xi^{\text{in}}, \theta^{\text{in}}, C_{q})$). 
\end{enumerate} 
\end{define} 

The following is the abstract Markov selection theorem (cf. \cite[Theorem 2.8]{FR08}, \cite[Theorem 2.7]{GRZ09}, and \cite[Theorem 6.7]{HZZ20}) with the addition of the maximization of a given functional that can be achieved by choosing this functional in the selection procedure as the first functional to be maximized, as explained in \cite{HZZ21a, HZZ20}. 

\begin{lemma}\rm{(\hspace{1sp}\cite[Theorem 4.19]{HZZ21a})}\label{Lemma 7.3}
\indent
Suppose that $\{ \mathcal{C} (\xi^{\text{in}}, C_{q}) \}_{\xi^{\text{in}} \in \dot{H}_{\sigma}^{\frac{1}{2}}}$ (resp.  \\$\{ \mathcal{W} (\xi^{\text{in}}, \theta^{\text{in}}, C_{q}) \}_{(\xi^{\text{in}}, \theta^{\text{in}}) \in \dot{H}_{\sigma}^{\frac{1}{2}} \times U_{1}}$) be an a.s. pre-Markov family and that $\mathcal{C} (\xi^{\text{in}}, C_{q})$ for each $\xi^{\text{in}}  \in \dot{H}_{\sigma}^{\frac{1}{2}}$ (resp. $\mathcal{W} (\xi^{\text{in}}, \theta^{\text{in}}, C_{q})$ for each $(\xi^{\text{in}}, \theta^{\text{in}}) \in  \dot{H}_{\sigma}^{\frac{1}{2}} \times U_{1}$) is non-empty and convex. 
\begin{enumerate}
\item Then there exist a measurable selection 
\begin{align*}
\dot{H}_{\sigma}^{\frac{1}{2}} \mapsto \mathcal{P}_{\dot{H}_{\sigma}^{\frac{1}{2}}} (\Omega_{0}), \xi^{\text{in}} \mapsto P_{\xi^{\text{in}}} \hspace{3mm} (\text{resp. } \dot{H}_{\sigma}^{\frac{1}{2}} \times U_{1} \mapsto \mathcal{P}_{\dot{H}_{\sigma}^{\frac{1}{2}} \times U_{1}} (\bar{\Omega}_{0}), (\xi^{\text{in}}, \theta^{\text{in}}) \mapsto P_{\xi^{\text{in}}, \theta^{\text{in}}} )
\end{align*} 
such that $P_{\xi^{\text{in}}} \in \mathcal{C}(\xi^{\text{in}}, C_{q})$ for all $\xi^{\text{in}} \in \dot{H}_{\sigma}^{\frac{1}{2}}$ (resp. $P_{\xi^{\text{in}}, \theta^{\text{in}}} \in \mathcal{W}(\xi^{\text{in}}, \theta^{\text{in}}, C_{q})$ for all $(\xi^{\text{in}}, \theta^{\text{in}}) \in \dot{H}_{\sigma}^{\frac{1}{2}} \times U_{1}$) and $\{ P_{\xi^{\text{in}}} \}_{\xi^{\text{in}} \in \dot{H}_{\sigma}^{\frac{1}{2}}}$ (resp. $\{ P_{\xi^{\text{in}}, \theta^{\text{in}}} \}_{(\xi^{\text{in}},\theta^{\text{in}}) \in \dot{H}_{\sigma}^{\frac{1}{2}} \times U_{1} }$) is an a.s. Markov family. Such $\{ P_{\xi^{\text{in}}} \}_{\xi^{\text{in}} \in\dot{H}_{\sigma}^{\frac{1}{2}}}$ (resp. $\{ P_{\xi^{\text{in}}, \theta^{\text{in}}} \}_{(\xi^{\text{in}},\theta^{\text{in}}) \in \dot{H}_{\sigma}^{\frac{1}{2}} \times U_{1} }$) is called an a.s. Markov selection of $\{ \mathcal{C} (\xi^{\text{in}}, C_{q}) \}_{\xi^{\text{in}} \in \dot{H}_{\sigma}^{\frac{1}{2}}}$ (resp. a.s. Markov selection of of $\{ \mathcal{W} (\xi^{\text{in}}, \theta^{\text{in}}, C_{q}) \}_{(\xi^{\text{in}}, \theta^{\text{in}}) \in \dot{H}_{\sigma}^{\frac{1}{2}} \times U_{1}}$). 
\item Finally, if $\delta: \dot{H}_{\sigma}^{\frac{1}{2}} \mapsto \mathbb{R}$  (resp. $\delta: \dot{H}_{\sigma}^{\frac{1}{2}} \times U_{1} \mapsto \mathbb{R}$) is a bounded continuous functional and $\lambda > 0$, then the selection can be chosen for every $\xi^{\text{in}} \in \dot{H}_{\sigma}^{\frac{1}{2}}$ (resp. $(\xi^{\text{in}}, \theta^{\text{in}}) \in \dot{H}_{\sigma}^{\frac{1}{2}} \times U_{1}$) to maximize 
\begin{align*}
\mathbb{E}^{P} \left[ \int_{0}^{\infty} e^{-\lambda s} \delta (\xi(s)) ds \right] \hspace{3mm} \left(\text{resp. } \mathbb{E}^{P} \left[ \int_{0}^{\infty} e^{-\lambda s} \delta (\xi(s), \theta(s)) ds \right] \right)
\end{align*}
among all the martingale solutions $P$ with the initial data $\xi^{\text{in}}$ (resp. probabilistically weak solutions $P$ with the initial data $(\xi^{\text{in}}, \theta^{\text{in}})$). 
\end{enumerate} 
\end{lemma} 

\subsection{Proofs of Propositions \ref{prop solution z}, \ref{prop additive amplitudes}, \ref{prop additive psi}, \ref{prop mult amplitudes}, and \ref{prop mult psi}}\label{Section 7.2}
\begin{proof}[Proof of Proposition  \ref{prop solution z}]
We can directly solve \eqref{est 42} to obtain
\begin{align}\label{est 38} 
z(t) = \int_{0}^{t} e^{- \Lambda^{\frac{3}{2} - 2 \sigma} (t-s)} \mathbb{P} dB(s). 
\end{align}
We define 
\begin{equation}\label{est 37} 
Y(s) \triangleq \frac{ \sin(\pi \alpha)}{\pi} \int_{0}^{s} e^{- \Lambda^{\frac{3}{2} - 2 \sigma} (s-r)} (s-r)^{-\alpha} \mathbb{P} dB(r), 
\end{equation}  
where $\alpha \in (0,1)$ (see \cite[p. 131]{DZ14}). For our first result, we choose 
\begin{equation}\label{est 39}
\epsilon \triangleq  \frac{\sigma}{(\frac{3}{2} - 2 \sigma) } \text{ and } \alpha \in (0, \epsilon). 
\end{equation} 
It follows by \cite[p. 131]{DZ14}, \eqref{est 38}-\eqref{est 37} that 
\begin{equation}\label{new3}
\int_{0}^{t} (t-s)^{\alpha -1} e^{-\Lambda^{\frac{3}{2} - 2 \sigma} (t-s)} Y(s) ds = z(t). 
\end{equation} 
We estimate using Gaussian hypercontractivity theorem (e.g. \cite[Theorem 3.50]{J97}), \eqref{est 37} and \eqref{est 39}, for all $l \in \mathbb{N}$, 
\begin{align}
\mathbb{E}^{\mathbf{P}} [ \lVert (-\Delta)^{\frac{5}{4} + \sigma} Y(s) \rVert_{L_{x}^{2}}^{2l} ]   \overset{\eqref{est 37} \eqref{regularity of noise}}{\lesssim}_{l}  \left( \int_{0}^{s} (s-r)^{-2\alpha-1 + 2\epsilon} dr \right)^{l} \overset{\eqref{est 39}}{\lesssim}_{l} 1. \label{est 40}
\end{align}
It follows from \eqref{est 40} that for all $T > 0$, 
\begin{equation}\label{est 262} 
\mathbb{E}^{\mathbf{P}} [ \int_{0}^{T} \lVert (-\Delta)^{\frac{5}{4} + \sigma} Y(s)  \rVert_{L_{x}^{2}}^{2l}ds ] \lesssim_{l,T} 1. 
\end{equation} 
Thus, for all $l > \frac{1}{2\alpha}$, 
\begin{equation*}
\mathbb{E}^{\mathbf{P}} [ \sup_{t \in [0,T]} \lVert (-\Delta)^{\frac{5}{4} + \sigma} z(t) \rVert_{L_{x}^{2}}^{2l} ] \overset{\eqref{new3}}{\lesssim_{l}} \mathbb{E}^{\mathbf{P}} \left[ \lVert (-\Delta)^{\frac{5}{4} + \sigma} Y \rVert_{L_{T}^{2l} L_{x}^{2}}^{2l} \right]   \overset{\eqref{est 262}}{\lesssim}_{l,T} 1. 
\end{equation*} 
This implies the first result in \eqref{est 41}. Next, differently from $\epsilon$ and $\alpha$ in \eqref{est 39}, we work with an arbitrary $\alpha \in (0, \frac{1}{2})$ and estimate by Gaussian hypercontractivity theorem, It$\hat{\mathrm{o}}$'s isometry, and \eqref{regularity of noise} for any $l \in \mathbb{N}$, 
\begin{equation*}
\mathbb{E}^{\mathbf{P}} [ \lVert (-\Delta)^{\frac{7}{8} + 2 \sigma} Y(s) \rVert_{L_{x}^{2}}^{2l} ] \overset{\eqref{est 37}}{\lesssim} \left(\int_{0}^{s} (s-r)^{-2\alpha}  \Tr \left( ( -\Delta)^{\frac{7}{4} + 4 \sigma} GG^{\ast} \right) dr \right)^{l}    \lesssim_{l,s} 1.
\end{equation*} 
Integrating over $[0,T]$ and applying Fubini's theorem show again that 
\begin{align}\label{est 263}
\mathbb{E}^{\mathbf{P}}\left[ \int_{0}^{T} \lVert Y(s) \rVert_{\dot{H}_{x}^{\frac{7}{4} + 4 \sigma} }^{2l} ds\right] \lesssim_{l,T} 1. 
\end{align}
Now applications of \cite[Proposition A.1.1]{DZ92} and Kolmogorov's test (e.g. \cite[Theorem 3.3]{DZ14}), using \eqref{new3}, complete the proof of Proposition \ref{prop solution z}. 
\end{proof}

\begin{proof}[Proof of Proposition \ref{prop additive amplitudes}]
We begin with the proof of \eqref{est 10} as follows: 
\begin{align*}
\lVert a_{k,j}(t) \rVert_{C_{x}} \overset{\eqref{est 3}}{\leq}& \lambda_{q+1}^{-\frac{1}{2}} \lvert \rho_{j} \rvert^{\frac{1}{2}} \left\lVert \gamma_{k} \left( \Id - \frac{ \mathring{R}_{q,j} (t,x)}{\rho_{j}} \right) \right\rVert_{C_{x}} \nonumber \\
\overset{\eqref{est 28} \eqref{est 9} }{\leq}&  \frac{1}{\sqrt{8} \pi} \bar{e}^{\frac{1}{2}} \delta_{q+1}^{\frac{1}{2}} \sup_{k \in \Gamma_{1} \cup \Gamma_{2}} \lVert \gamma_{k} \rVert_{C(B(\Id, \varepsilon_\gamma ))}.
\end{align*}
Next, for any $N \in \mathbb{N}$, we compute from \eqref{est 3} using \eqref{est 28} and \cite[Equation (130)]{BDIS15} 
\begin{align}\label{new5}
\lVert D^{N} a_{k,j}(t) \rVert_{C_{x}} \lesssim \delta_{q+1}^{\frac{1}{2}} \bar{e}^{\frac{1}{2}} \left[ \left\lVert \frac{ D^{N} \mathring{R}_{q,j} (t)}{\rho_{j}} \right\rVert_{C_{x}} + \left\lVert \frac{ D \mathring{R}_{q,j} (t)}{\rho_{j}} \right\rVert_{C_{x}}^{N} \right].
\end{align}
We can now rely on \eqref{est 7} to compute 
\begin{align}
& \left\lVert \frac{ D^{N} \mathring{R}_{q,j} (t) }{\rho_{j}} \right\rVert_{C_{x}}  \nonumber \\
&\overset{\eqref{inductive 2a} \eqref{inductive 1a} \eqref{est 17}  }{\lesssim} \frac{ ( \lambda_{q}^{N}  \lVert \mathring{R}_{l} (\tau_{q+1} j) \rVert_{C_{x}} + \tau_{q+1} \lambda_{q}^{N} \lVert \Lambda y_{l} + \Lambda z_{l} \rVert_{C_{t,x,q+1}} \lambda_{q} \lVert \mathring{R}_{l} (\tau_{q+1} j) \rVert_{C_{x}})}{\varepsilon_\gamma^{-1} \sqrt{ l^{2} + \lVert \mathring{R}_{q,j} (\tau_{q+1} j) \rVert_{C_{x}}^{2}} + \gamma_{l} (\tau_{q+1} j)}  \nonumber \\
&\overset{\eqref{est 27} \eqref{est 16} \eqref{inductive 1c} \eqref{additive stopping time} }{\lesssim}  \lambda_{q}^{N} +  \tau_{q+1} \lambda_{q}^{N+2} \delta_{q}^{\frac{1}{2}} \bar{e}^{\frac{1}{2}}.  \label{new 6}
\end{align}
Second, we compute using \eqref{est 2b} and \eqref{est 17} 
\begin{align}\label{new 7}
\left\lVert \frac{ D \mathring{R}_{q,j} (t)}{\rho_{j}} \right\rVert_{C_{x}}^{N} \overset{\eqref{inductive 2a}}{\lesssim} \left( \frac{ \lambda_{q} \lVert \mathring{R}_{l} (\tau_{q+1} j) \rVert_{C_{x}}}{ \varepsilon_\gamma^{-1} \sqrt{ l^{2} + \lVert \mathring{R}_{q,j} (\tau_{q+1} j) \rVert_{C_{x}}^{2}} + \gamma_{l} (\tau_{q+1} j)} \right)^{N}\overset{\eqref{est 16} \eqref{est 27} }{\lesssim} \lambda_{q}^{N}.  \
\end{align}
Therefore, we conclude now by applying \eqref{new 6} and \eqref{new 7} to \eqref{new5} that 
\begin{align*}
\lVert D^{N} a_{k,j} \rVert_{C_{t, x, q+1}}  \lesssim \delta_{q+1}^{\frac{1}{2}} \bar{e}^{\frac{1}{2}} [ \lambda_{q}^{N} + \tau_{q+1} \lambda_{q}^{N+2} \delta_{q}^{\frac{1}{2}} \bar{e}^{\frac{1}{2}} ] \overset{\eqref{est 17}}{\lesssim} \delta_{q+1}^{\frac{1}{2}} \bar{e}^{\frac{1}{2}} \lambda_{q}^{N}. 
\end{align*}
Finally, thanks to $\mathbb{P}_{q+1, k}$ in \eqref{est 67a} that restricts the frequency support to $\{ \xi: \frac{7}{8} \lambda_{q+1} \leq \lvert \xi \rvert \leq \frac{9}{8} \lambda_{q+1}\}$, we see that \eqref{support wqplus1} holds. As $y_{q+1} = y_{l} + w_{q+1}$ from \eqref{est 44} and $\supp \hat{y}_{l} \subset B(0, 2 \lambda_{q})$ due to \eqref{inductive 1a} at level $q$, we see that \eqref{inductive 1a} at level $q+1$ follows. It also follows from this and \eqref{additive decomposition} that \eqref{inductive 2a} at level $q+1$ holds. This completes the proof of Proposition \ref{prop additive amplitudes}.
\end{proof}

\begin{proof}[Proof of Proposition \ref{prop additive psi}]
The first claim in \eqref{psi 1} is immediate by definition from \eqref{est 14c}. Next, by \cite[Equation (130)]{BDIS15}, for the case $N=1$ we have
\begin{align}
&\lVert D \psi_{q+1, j, k} (t) \rVert_{C_{x}}  \overset{\eqref{est 14c}}{\lesssim} \lambda_{q+1} \lVert D \Phi_{j} (t) - \Id \rVert_{C_{x}}   \label{est 78}\\
&\overset{\eqref{est 15a}\eqref{inductive 1a} \eqref{est 17}}{\lesssim} \lambda_{q+1} \tau_{q+1} \left( \lambda_{q} \lVert \Lambda y_{l} \rVert_{C_{t,x,q+1}} + f(q) \lVert  \Lambda z_{l} \rVert_{C_{t,x,q+1}} \right)  \overset{\eqref{inductive 1c} }{\lesssim} \lambda_{q+1} \tau_{q+1} \lambda_{q}^{2} \delta_{q}^{\frac{1}{2}} \bar{e}^{\frac{1}{2}} M_{0}^{\frac{1}{2}}.\nonumber 
\end{align}
Next, for $N \in \mathbb{N} \setminus \{1\}$, we rely on \cite[Equation (130)]{BDIS15} to compute 
\begin{align} 
&\lVert D^{N} \psi_{q+1, j,k} (t) \rVert_{C_{x}} \overset{\eqref{est 14c}}{\lesssim} \lambda_{q+1} \lVert \Phi_{j} (t) \rVert_{C_{x}^{N}} + \lambda_{q+1}^{N} \lVert \nabla \Phi_{j}(t) - \Id \rVert_{C_{x}}^{N}  \label{est 80}  \\
\overset{\eqref{est 15b}  \eqref{est 17}}{\lesssim}&  \lambda_{q+1} \tau_{q+1} \lambda_{q}^{N+1} M_{0}^{\frac{1}{2}} \delta_{q}^{\frac{1}{2}} \bar{e}^{\frac{1}{2}} + ( \lambda_{q+1} \tau_{q+1} \lambda_{q}^{2} M_{0}^{\frac{1}{2}} \delta_{q}^{\frac{1}{2}} \bar{e}^{\frac{1}{2}})^{N}  \lesssim \lambda_{q+1} \tau_{q+1} \lambda_{q}^{N+1} \delta_{q}^{\frac{1}{2}} \bar{e}^{\frac{1}{2}} M_{0}^{\frac{1}{2}},   \nonumber 
\end{align}
giving us the second claim in \eqref{psi 1} as well; here, the last inequality of \eqref{est 80} used the fact that for $\beta > \frac{1}{2}$ sufficiently close to $\frac{1}{2}$ and $a_{0}$ sufficiently large 
\begin{equation}\label{est 81} 
\lambda_{q+1} \tau_{q+1} \lambda_{q} \delta_{q}^{\frac{1}{2}} \overset{\eqref{tau} \eqref{define lambda delta} \eqref{define l}}{=} a^{b ( \frac{\beta}{2} - \frac{1}{4})} a^{b^{q} [ 1- \beta + b( \frac{1}{2} - \frac{\alpha}{2} - \frac{\beta}{2}) + b^{2} (-\frac{1}{2} + \beta) ]} \overset{\eqref{additive bound b}}{\ll} 1.
\end{equation}
Next, we can compute 
\begin{align*}
&\lVert D (a_{k,j} \psi_{q+1, j,k} ) (t) \rVert_{C_{x}} \overset{\eqref{est 19} \eqref{psi 1}}{\lesssim} \delta_{q+1}^{\frac{1}{2}} \bar{e}^{\frac{1}{2}} \lambda_{q} + (\delta_{q+1}^{\frac{1}{2}} \bar{e}^{\frac{1}{2}}) ( \lambda_{q+1} \tau_{q+1} \lambda_{q}^{2} \delta_{q}^{\frac{1}{2}} \bar{e}^{\frac{1}{2}}) \overset{\eqref{est 81}}{\lesssim} \lambda_{q} \delta_{q+1}^{\frac{1}{2}} \bar{e}^{\frac{1}{2}},  \\
&\lVert D^{2} (a_{k,j} \psi_{q+1, j, k} ) (t) \rVert_{C_{x}} \overset{\eqref{est 19} \eqref{psi 1}}{\lesssim}  \delta_{q+1}^{\frac{1}{2}} \bar{e}^{\frac{1}{2}} \lambda_{q}^{2}  + ( \delta_{q+1}^{\frac{1}{2}} \bar{e}^{\frac{1}{2}}) (\lambda_{q+1} \tau_{q+1} \lambda_{q}^{3} \delta_{q}^{\frac{1}{2}} \bar{e}^{\frac{1}{2}}) \overset{\eqref{est 81}}{\lesssim} \lambda_{q}^{2} \delta_{q+1}^{\frac{1}{2}} \bar{e}^{\frac{1}{2}},
\end{align*}
so we have \eqref{psi 2} and the proof of Proposition \ref{prop additive psi} is complete.
\end{proof}

\begin{proof}[Proof of Proposition \ref{prop mult amplitudes}]
First, \eqref{est 63b} follows from definition in \eqref{est 58a} and an application of \eqref{est 64}; the inequality in \eqref{est 63a} follows from the same computation. Next, for any $N \in \mathbb{N}$, using \cite[Equation (130)]{BDIS15} we compute 
\begin{align}
\lVert D^{N} \bar{a}_{k,j} (t) \rVert_{C_{x}} \lesssim&  \Upsilon_{l}^{-\frac{1}{2}} (t) e^{\frac{3}{2}L^{\frac{1}{4}}} \delta_{q+1}^{\frac{1}{2}} \bar{e}^{\frac{1}{2}} \nonumber \\
& \times \left[ \lVert D^{N} \left( \frac{ \mathring{R}_{q,j} (t,x)}{\bar{\rho}_{j}} \right) \rVert_{C_{x}} + \lVert \frac{ D \mathring{R}_{q,j} (t,x)}{\bar{\rho}_{j}} \rVert_{C_{x}}^{N} \right]. \label{est 65}
\end{align}
Next, we estimate 
\begin{align}
& \lVert \frac{D^{N} \mathring{R}_{q,j} (t,x)}{\bar{\rho}_{j}} \rVert_{C_{x}}  \nonumber \\
\overset{\eqref{est 7} \eqref{est 58b} \eqref{mult inductive 2a} \eqref{est 79} }{\lesssim}& \frac{ \lambda_{q}^{N} \lVert \mathring{R}_{l} (\tau_{q+1} j) \rVert_{C_{x}} + \tau_{q+1} \lambda_{q}^{N} e^{L^{\frac{1}{4}}} \lVert \Lambda y_{l} (t) \rVert_{C_{x}} \lambda_{q} \lVert \mathring{R}_{l} (\tau_{q+1} j) \rVert_{C_{x}}}{ \varepsilon_\gamma^{-1} \sqrt{ l^{2} + \lVert \mathring{R}_{q,j} (\tau_{q+1} j) \rVert_{C_{x}}^{2}} + \bar{\gamma}_{l} (\tau_{q+1} j)}  \nonumber  \\
\overset{\eqref{est 261} \eqref{est 60} \eqref{mult inductive 1c}}{\lesssim}&\lambda_{q}^{N} [ 1 +  \tau_{q+1} \lambda_{q}^{2} \delta_{q}^{\frac{1}{2}} e^{L^{\frac{1}{4}}} M_{0}^{\frac{1}{2}} m_{L}^{4} \bar{e}^{\frac{1}{2}}]. \label{est 186} 
\end{align}
Second, relying on \eqref{est 58b}, \eqref{est 2b}, \eqref{est 261} and \eqref{est 79} gives us 
\begin{equation}
\lVert \frac{ D \mathring{R}_{q,j}(t,x)}{\bar{\rho}_{j}} \rVert_{C_{x}}^{N}  \lesssim \left(  \frac{  \lambda_{q} \lVert \mathring{R}_{l} (\tau_{q+1} j) \rVert_{C_{x}} }{\varepsilon_\gamma^{-1} \sqrt{ l^{2} + \lVert \mathring{R}_{q,j} (\tau_{q+1} j) \rVert_{C_{x}}^{2}} + \bar{\gamma}_{l} (\tau_{q+1} j)}\right)^{N} \lesssim \lambda_{q}^{N}.  \label{est 187}
\end{equation} 
Therefore, we conclude by applying \eqref{est 186} and \eqref{est 187} to \eqref{est 65} that 
\begin{align*}
\lVert D^{N} \bar{a}_{k,j} (t) \rVert_{C_{x}}  \lesssim \Upsilon_{l}^{-\frac{1}{2}} (t) e^{\frac{3}{2} L^{\frac{1}{4}}} \delta_{q+1}^{\frac{1}{2}} \bar{e}^{\frac{1}{2}} \lambda_{q}^{N} [ 1 + \tau_{q+1} \lambda_{q}^{2} \delta_{q}^{\frac{1}{2}} e^{L^{\frac{1}{4}}} M_{0}^{\frac{1}{2}} m_{L}^{4} \bar{e}^{\frac{1}{2}}],
\end{align*} 
so that as $\tau_{q+1} \lambda_{q}^{2} \delta_{q}^{\frac{1}{2} } e^{L^{\frac{1}{4}}}m_{L}^{4} \ll 1$ can be verified immediately from \eqref{est 17}, we may deduce the desired result \eqref{est 63d}. Additionally, \eqref{est 63c} follows from \eqref{est 63d}. Identically to the additive case, due to the operator $\mathbb{P}_{q+1,k}$ in \eqref{est 68} we see that \eqref{support wqplus1} is satisfied. This also leads to $\supp \hat{y}_{q+1} \subset B(0,2\lambda_{q+1})$ and $\supp \hat{\mathring{R}}_{q+1} \subset B(0, 4 \lambda_{q+1})$ due to \eqref{est 44} and \eqref{est 54} as claimed in \eqref{mult inductive 1a} and \eqref{mult inductive 2a} respectively. This completes the proof of Proposition \ref{prop mult amplitudes}. 
\end{proof}

\begin{proof}[Proof of Proposition \ref{prop mult psi}]
The first inequality of  \eqref{est 82a} is immediate by definition from \eqref{est 14c}. Next, by \cite[Equation (130)]{BDIS15}, we can proceed similarly to \eqref{est 78} as 
\begin{align*}
\lVert D \psi_{q+1, j, k} (t) \rVert_{C_{x}} \lesssim&  \lambda_{q+1} \lVert D \Phi_{j} (t) - \Id \rVert_{C_{x}} \overset{\eqref{est 15a}\eqref{est 79}\eqref{mult inductive 1c}}{\lesssim}  \lambda_{q+1} \tau_{q+1} \lambda_{q}^{2} \delta_{q}^{\frac{1}{2}} \bar{e}^{\frac{1}{2}} M_{0}^{\frac{1}{2}} m_{L}^{4} e^{L^{\frac{1}{4}}}. 
\end{align*}
For $N \in \mathbb{N}\setminus \{1\}$, we then rely on \cite[Equation (130)]{BDIS15} to deduce similarly to \eqref{est 80}, 
\begin{align*}
& \lVert D^{N} \psi_{q+1, j, k} (t) \rVert_{C_{x}} \\  
\overset{\eqref{est 15} \eqref{est 79} \eqref{mult inductive 1c}}{\lesssim}& \lambda_{q+1} \tau_{q+1} \lambda_{q}^{N} e^{L^{\frac{1}{4}}} ( M_{0}^{\frac{1}{2}} m_{L}^{4} \lambda_{q} \delta_{q}^{\frac{1}{2}} \bar{e}^{\frac{1}{2}})  \\
&+ \lambda_{q+1}^{N} \left( \tau_{q+1} \lambda_{q} e^{L^{\frac{1}{4}}} ( M_{0}^{\frac{1}{2}} m_{L}^{4} \lambda_{q} \delta_{q}^{\frac{1}{2}} \bar{e}^{\frac{1}{2}}) \right)^{N} \lesssim  \lambda_{q+1} \tau_{q+1} \lambda_{q}^{N+1} \delta_{q}^{\frac{1}{2}} \bar{e}^{\frac{1}{2}} M_{0}^{\frac{1}{2}} m_{L}^{4} e^{L^{\frac{1}{4}}},
\end{align*}
where the last inequality relied on the fact that for $\beta > \frac{1}{2}$ sufficiently close to $\frac{1}{2}$ and $a_{0}$ sufficiently large, 
\begin{equation}\label{est 189} 
\lambda_{q+1} \tau_{q+1} \lambda_{q} e^{L^{\frac{1}{4}}} m_{L}^{4} \delta_{q}^{\frac{1}{2}} \bar{e}^{\frac{1}{2}} \lesssim a^{b^{q} [b^{2} (-\frac{1}{2} + \beta) + b ( \frac{1- \alpha- \beta}{2})+ 1- \beta]}a^{b(2\beta -1) \frac{1}{4}} a^{L^{\frac{1}{4}}} \overset{\eqref{mult bound b}}{\ll} 1. 
\end{equation}
Next, we are able to compute
\begin{align*}
&\lVert D ( \bar{a}_{k,j} \psi_{q+1, j, k} )(t) \rVert_{C_{x}} \overset{\eqref{est 63b} \eqref{est 63d} \eqref{est 82a}}{\lesssim}  \Upsilon_{l}^{-\frac{1}{2}} (t) e^{\frac{3}{2} L^{\frac{1}{4}}}  \delta_{q+1}^{\frac{1}{2}} \bar{e}^{\frac{1}{2}} \lambda_{q} \\
& \hspace{28mm} \times [ 1 + \bar{e}^{\frac{1}{2}} M_{0}^{\frac{1}{2}} m_{L}^{4} e^{L^{\frac{1}{4}}} \lambda_{q+1} \tau_{q+1} \lambda_{q} \delta_{q}^{\frac{1}{2}} ]  \overset{\eqref{est 189}}{\lesssim}  \Upsilon_{l}^{-\frac{1}{2}} (t) e^{\frac{3}{2} L^{\frac{1}{4}}}  \delta_{q+1}^{\frac{1}{2}} \bar{e}^{\frac{1}{2}} \lambda_{q},
\end{align*}
and similarly,
\begin{align*}
&\lVert D^{2} ( \bar{a}_{k,j} \psi_{q+1, j, k} ) (t) \rVert_{C_{x}}\overset{\eqref{est 63b} \eqref{est 63d} \eqref{est 82a}}{\lesssim}  \Upsilon_{l}^{-\frac{1}{2}} (t) e^{\frac{3}{2} L^{\frac{1}{4}}}  \delta_{q+1}^{\frac{1}{2}} \bar{e}^{\frac{1}{2}} \lambda_{q}^{2}  \\
& \hspace{28mm} \times [ 1 + \bar{e}^{\frac{1}{2}} M_{0}^{\frac{1}{2}} m_{L}^{4} e^{L^{\frac{1}{4}}} \lambda_{q+1} \tau_{q+1} \lambda_{q} \delta_{q}^{\frac{1}{2}} ] \overset{\eqref{est 189}}{\lesssim} \Upsilon_{l}^{-\frac{1}{2}} (t) e^{\frac{3}{2} L^{\frac{1}{4}}}  \delta_{q+1}^{\frac{1}{2}} \bar{e}^{\frac{1}{2}} \lambda_{q}^{2}.
\end{align*}
Thus, we verified \eqref{est 82b} and \eqref{est 82c}. The analogous results \eqref{est 82d} and \eqref{est 82e} for $a_{k,j}$ follow from definition in \eqref{est 58a}, concluding this proof of Proposition \ref{prop mult psi}.
\end{proof}

\subsection{Details of the Proof of \eqref{est 84} \except{toc}{to Bound $R_O$}}\label{mult oscillation details}
Being mindful of the difference in the definition of $a_{k,j}$ from the additive case, we can define $\vartheta_{j,k}, s^{m}, \mathcal{T}^{m}$, and $\mathcal{L}_{j,k}^{ml}$ identically to \eqref{est 85} from the additive case. We can then break apart
\begin{equation}\label{est 230}
R_{O} = R_{O, \text{approx}} + R_{O, \text{low}} + R_{O, \text{high}}
\end{equation}
where 
\begin{subequations}\label{est 90} 
\begin{align}
R_{O, \text{approx}} \triangleq& \sum_{j} \chi_{j}^{2} (\mathring{R}_{l} - \mathring{R}_{q,j}), \label{est 90a} \\
R_{O, \text{low}} \triangleq& \sum_{j} \chi_{j}^{2} \mathring{R}_{q,j} + \sum_{j,k} \mathring{\mathcal{L}}_{j,k},  \label{est 90b}\\
R_{O, \text{high}} \triangleq& \mathcal{B} \tilde{P}_{\approx \lambda_{q+1}} \sum_{j, j', k, k': k + k' \neq 0}  \label{est 90c}  \\
& \times \Big[\left( \Lambda \mathbb{P}_{q+1, k} ( \chi_{j} a_{k,j}  b_{k} ( \lambda_{q+1} \Phi_{j})) \right)  \cdot \nabla \mathbb{P}_{q+1, k'} ( \chi_{j'} a_{k',j'} b_{k'} ( \lambda_{q+1} \Phi_{j})) \nonumber\\
& \hspace{5mm} - \left(\nabla \mathbb{P}_{q+1, k} ( \chi_{j} a_{k,j} b_{k} ( \lambda_{q+1} \Phi_{j})) \right)^{T} \cdot \Lambda \mathbb{P}_{q+1, k'} \big( \chi_{j'} a_{k',j'} b_{k'} ( \lambda_{q+1} \Phi_{j'})\big) \Big]; \nonumber 
\end{align}
\end{subequations} 
we carefully point out that \eqref{est 90c} is not identical to \eqref{new 16c} because e.g. we replaced ``$\tilde{w}_{q+1, j, k}$'' by $\chi_{j} a_{k,j} b_{k} (\lambda_{q+1} \Phi_{j})$ and they are not the same (recall \eqref{est 69b}).   

All the computations from \eqref{est 86}-\eqref{new 19} go through ($\rho_{j}$ never appears therein). We define $O_{1}$ and $O_{2}$ identically to \eqref{est 87} to retain the identity \eqref{new 18}: $R_{O, \text{low}} = \mathring{O}_{1} + \mathring{O}_{2}$. We also retain the key computation \eqref{est 31} with $\rho_{j}$ replaced by $\bar{\rho}_{j}$; i.e., 
\begin{equation}\label{est 88}
O_{1} = \sum_{j} \chi_{j}^{2} \lambda_{q+1} \Id \left[ \frac{\bar{\rho}_{j}}{\lambda_{q+1}} - \frac{1}{2} \sum_{k} a_{k,j}^{2} \right]. 
\end{equation} 
Therefore, we retain 
\begin{equation}\label{est 274}
R_{O,\text{low}} = \mathring{O}_{1} + \mathring{O}_{2} = \mathring{O}_{2} = \mathring{O}_{21} + \mathring{O}_{22}
\end{equation} 
from \eqref{new 18} where $O_{21}$ and $O_{22}$ were defined identically to \eqref{est 87}. Due to $D_{t,q} (\mathring{R}_{l} -\mathring{R}_{q,j}) = D_{t,q} \mathring{R}_{l}$ and $(\mathring{R}_{l} - \mathring{R}_{q,j}) (\tau_{q+1} j) = 0$ from \eqref{est 60}, we compute for $\beta > \frac{1}{2}$ sufficiently close to $\frac{1}{2}$ and $a_{0}$ sufficiently large 
\begin{align*}
&\lVert ( \mathring{R}_{l} - \mathring{R}_{q,j} ) (t) \rVert_{C_{x}} \overset{\eqref{est 2a}}{\lesssim}  \tau_{q+1}\sup_{s \in [\tau_{q+1} j, t]} \ \lVert ( \partial_{s} + \Upsilon_{l} \Lambda y_{l} \cdot \nabla) \mathring{R}_{l} (s)  \rVert_{C_{x}}  \\
\overset{\eqref{mult inductive 1} \eqref{mult inductive 2}}{\lesssim}& \tau_{q+1} l^{-1} \lambda_{q+2} \delta_{q+2} e^{-3L^{\frac{1}{4}}} \bar{e} \left( 1+ a^{L}  a^{b( \beta - \frac{1}{2})} a^{b^{q} [ 2 - \beta - b \alpha]} \right) \lesssim \tau_{q+1} l^{-1} \lambda_{q+2} \delta_{q+2} e^{-3L^{\frac{1}{4}}} \bar{e}.
\end{align*}
Applying this estimate to \eqref{est 90a} gives us now 
\begin{align}\label{est 231}
\lVert R_{O, \text{approx}} \rVert_{C_{t,x,q+1}} \leq \sum_{j} \lVert 1_{\supp \chi_{j}} (\mathring{R}_{l} - \mathring{R}_{q,j}) \rVert_{C_{t,x,q+1}} \lesssim \tau_{q+1} l^{-1} \lambda_{q+2} \delta_{q+2} e^{-3L^{\frac{1}{4}}} \bar{e}.  
\end{align}
Then we estimate $O_{21}$ defined in \eqref{est 87b}, being mindful that $a_{k,j}$ therein is defined by \eqref{est 58a} that has $\bar{\rho}_{j}$ rather than $\rho_{j}$, 
\begin{align*}
\lVert O_{21} &\rVert_{C_{t,x,q+1}} \lesssim \sup_{s \in [t_{q+1}, T_{L}]} \sum_{j,k} \chi_{j}^{2}(s) \lVert \nabla (a_{k,j} \psi_{q+1, j, k} ) \rVert_{C_{t,x,q+1}} \\
& \times  \lVert a_{k,j} \psi_{q+1, j, -k} \rVert_{C_{t,x,q+1}} \sup_{r, \bar{r} \in [0,1]} \lVert \mathcal{K}_{k, r, \bar{r}}^{(1)} \rVert_{L^{1} (\mathbb{R}^{2} \times \mathbb{R}^{2})} \overset{ \eqref{new 19} \eqref{est 82d} \eqref{est 63a}}{\lesssim} e^{3L^{\frac{1}{4}}} \lambda_{q} \delta_{q+1} \bar{e}. 
\end{align*}
As an identical computation also holds for $O_{22}$ in \eqref{est 87b}, we conclude 
\begin{equation}\label{est 232}
\lVert R_{O, \text{low}} \rVert_{C_{t,x,q+1}} \overset{\eqref{est 274}}{\leq} \sum_{l=1}^{2} \lVert O_{2l} \rVert_{C_{t,x,q+1}}  \lesssim e^{3L^{\frac{1}{4}}} \lambda_{q} \delta_{q+1} \bar{e}. 
\end{equation} 
Now, similarly to \cite[Equations (184)-(185)]{Y23b} we can write 
\begin{equation}\label{est 228} 
R_{O, \text{high}} = O_{3} - O_{4}
\end{equation} 
where
\begin{align*}
O_{3} \triangleq& \mathcal{B} \tilde{P}_{\approx \lambda_{q+1}} \times  \Big( \sum_{j, j', k, k': k + k' \neq 0} (\Lambda \mathbb{P}_{q+1, k} [\chi_{j} a_{k,j} b_{k} (\lambda_{q+1} \Phi_{j})]) \nonumber\\
&\hspace{35mm} \cdot \nabla \mathbb{P}_{q+1, k'} [\chi_{j'} a_{k',j'} b_{k'} (\lambda_{q+1} \Phi_{j'})] \Big),  \\
O_{4} \triangleq& \mathcal{B} \tilde{P}_{\approx \lambda_{q+1}} \times  \Big( \sum_{j, j', k, k': k + k' \neq 0}  (\nabla \mathbb{P}_{q+1, k} [\chi_{j} a_{k,j} b_{k} (\lambda_{q+1} \Phi_{j})])^{T}\nonumber \\
&\hspace{35mm} \cdot \Lambda \mathbb{P}_{q+1, k'} [\chi_{j'} a_{k',j'} b_{k'} (\lambda_{q+1} \Phi_{j'})] \Big).  
\end{align*}
We can further split 
\begin{equation}\label{est 199}
O_{3} = \sum_{k=1}^{3} O_{3k} \text{ and } O_{31} = O_{311} + O_{312} 
\end{equation} 
where
\begin{subequations}\label{est 200} 
\begin{align}
&O_{311}(x)  \triangleq \frac{1}{2} \mathcal{B} \tilde{P}_{\approx \lambda_{q+1}} \lambda_{q+1} \sum_{j, j', k, k': k + k' \neq 0} \chi_{j} \chi_{j'} \nabla (a_{k,j} (x) a_{k', j'} (x) \psi_{q+1, j', k'}(x) \psi_{q+1, j, k}(x))  \nonumber \\
& \hspace{35mm} \times (b_{k}(\lambda_{q+1} x) \cdot b_{k'} (\lambda_{q+1} x) - e^{i(k+k') \cdot \lambda_{q+1} x}),    \\
& O_{312}(x) \triangleq  \mathcal{B} \tilde{P}_{\approx \lambda_{q+1}} \lambda_{q+1} \sum_{j,j',k,k': k + k' \neq 0} \chi_{j} \chi_{j'} \times b_{k'}(\lambda_{q+1} x) \nonumber\\
& \hspace{25mm}  \otimes b_{k} (\lambda_{q+1} x) \nabla (a_{k,j} (x) a_{k', j'} (x) \psi_{q+1, j', k'} (x) \psi_{q+1, j, k} (x)), 
\end{align}
\end{subequations}
\begin{subequations}\label{est 201}
\begin{align}
O_{32}(x) \triangleq& \mathcal{B} \tilde{P}_{\approx \lambda_{q+1}} \lambda_{q+1} \sum_{j, j',k, k': k + k'\neq 0} \divergence ( [\chi_{j} a_{k,j} b_{k} (\lambda_{q+1} \Phi_{j})](x) \nonumber \\
&\hspace{25mm} \otimes \chi_{j'} [\mathbb{P}_{q+1, k'}, a_{k', j'}(x) \psi_{q+1, j', k'}(x)] b_{k'} (\lambda_{q+1} x) ), \\
O_{33}(x) \triangleq& \mathcal{B} \tilde{P}_{\approx \lambda_{q+1}}  \sum_{j, j', k, k': k + k'\neq 0} \times \divergence (\chi_{j} [\mathbb{P}_{q+1} \Lambda, a_{k, j}(x) \psi_{q+1, j, k}(x)] b_{k} (\lambda_{q+1} x) \nonumber \\
&\hspace{35mm} \otimes \mathbb{P}_{q+1, k'} [\chi_{j'} a_{k',j'} b_{k'} (\lambda_{q+1} \Phi_{j'})](x)). 
\end{align}
\end{subequations} 
We estimate using \eqref{est 82d} and \eqref{est 63a} as follows:
\begin{align*}
& \lVert O_{31} \rVert_{C_{t,x,q+1}} \\
\lesssim&\sum_{j, j', k, k': k + k' \neq 0} \lVert 1_{\supp \chi_{j}}  \nabla (a_{k,j} \psi_{q+1, j, k}) \rVert_{C_{t,x,q+1}} \lVert 1_{\supp \chi_{j'}} a_{k', j'} \psi_{q+1, j', k'} \rVert_{C_{t,x,q+1}} \\
&+ \lVert 1_{\supp \chi_{j'}}  \nabla (a_{k', j'} \psi_{q+1, j', k'}) \rVert_{C_{t,x,q+1}} \lVert 1_{\supp \chi_{j}} a_{k,j} \psi_{q+1, j, k} \rVert_{C_{t,x,q+1}} \lesssim  e^{3L^{\frac{1}{4}}} \lambda_{q} \delta_{q+1} \bar{e}; 
\end{align*}
additionally relying on \cite[Equation (A.17)]{BSV19} allows us to deduce 
\begin{align*}
& \lVert O_{32} \rVert_{C_{t,x,q+1}} + \lVert O_{33} \rVert_{C_{t,x,q+1}} \\
\lesssim& \sum_{j, j', k, k': k + k' \neq 0}  \lVert 1_{\supp \chi_{j}} a_{k,j} \rVert_{C_{t,x,q+1}}   \lVert 1_{\supp \chi_{j'}} \nabla (a_{k', j'} \psi_{q+1, j', k'}) \rVert_{C_{t,x,q+1}}  \\
&+ \lVert 1_{\supp \chi_{j}} \nabla (a_{k,j} \psi_{q+1, j, k}) \rVert_{C_{t,x,q+1}}  \lVert 1_{\supp \chi_{j'}} a_{k', j'} b_{k'} (\lambda_{q+1} \Phi_{j'}) \rVert_{C_{t,x,q+1}} \lesssim  e^{3L^{\frac{1}{4}}} \lambda_{q} \delta_{q+1} \bar{e}. 
\end{align*}
Thus, we conclude from \eqref{est 199} that 
\begin{equation}\label{est 229}
\lVert O_{3} \rVert_{C_{t,x,q+1}} \lesssim e^{3L^{\frac{1}{4}}} \lambda_{q} \delta_{q+1} \bar{e}. 
\end{equation} 
Next, similarly to \cite[Equations (191)-(192)]{Y23b}
we can write 
\begin{equation}\label{est 203}
O_{4} = O_{41} + O_{42}
\end{equation} 
where
\begin{subequations}\label{est 202} 
\begin{align}
&O_{41}(x) \triangleq \sum_{j, j', k, k': k + k' \neq 0}  \mathcal{B} \tilde{P}_{\approx \lambda_{q+1}} \times  \Big(\chi_{j} \nabla ( [ \mathbb{P}_{q+1, k}, a_{k,j} (x) \psi_{q+1, j, k} (x) ] b_{k} (\lambda_{q+1} x) )^{T} \nonumber \\
&\hspace{35mm}   \cdot \lambda_{q+1} [\chi_{j'} a_{k',j'} b_{k} (\lambda_{q+1} \Phi_{j'})](x)\Big), \label{est 202a}  \\
&O_{42}(x) \triangleq \sum_{j, j', k, k': k + k' \neq 0}  \mathcal{B} \tilde{P}_{\approx \lambda_{q+1}} \times \Big((\nabla \mathbb{P}_{q+1, k} [\chi_{j} a_{k,j} b_{k} (\lambda_{q+1} \Phi_{j})] (x))^{T} \nonumber \\
& \hspace{35mm}  \cdot \chi_{j'} [\mathbb{P}_{q+1, k'} \Lambda, a_{k', j'}(x) \psi_{q+1, j', k'} (x)] b_{k'} (\lambda_{q+1} x) \Big).  \label{est 202b} 
\end{align}
\end{subequations} 
After writing 
\begin{align*}
& \nabla ( [ \mathbb{P}_{q+1, k}, a_{k,j} (x) \psi_{q+1, j, k} (x) ] b_{k} (\lambda_{q+1} x) )  \\
& \hspace{10mm} = [ \mathbb{P}_{q+1, k}, \nabla (a_{k,j} \psi_{q+1, j, k}) ] b_{k} (\lambda_{q+1} x) + [\mathbb{P}_{q+1, k}, a_{k,j} \psi_{q+1, j, k} ] \nabla b_{k} (\lambda_{q+1} x), 
\end{align*}
we can estimate using \cite[Equation (A.17)]{BSV19} 
\begin{align*}
& \lVert O_{41} \rVert_{C_{t,x,q+1}} \lesssim  \sum_{j, j', k, k': k + k' \neq 0} [ \lambda_{q+1}^{-1} \lVert 1_{\supp \chi_{j}} \nabla^{2} (a_{k,j} \psi_{q+1, j, k}) \rVert_{C_{t,x,q+1}} \lVert b_{k} (\lambda_{q+1} x) \rVert_{C_{t,x,q+1}} \\
& \hspace{3mm} + \lambda_{q+1}^{-1} \lVert 1_{\supp \chi_{j}} \nabla (a_{k,j} \psi_{q+1, j, k}) \rVert_{C_{t,x,q+1}} \lVert \nabla b_{k} (\lambda_{q+1} x) \rVert_{C_{t,x,q+1}} ] \lVert 1_{\supp \chi_{j'}} a_{k', j'} \rVert_{C_{t,x,q+1}} \\
&\hspace{3mm}\overset{\eqref{est 82d} \eqref{est 82e}}{\lesssim} e^{3L^{\frac{1}{4}}} \delta_{q+1} \lambda_{q} \bar{e},
\end{align*}
and similarly,
\begin{align*} 
& \lVert O_{42} \rVert_{C_{t,x,q+1}} \lesssim   \sum_{j, j', k, k': k + k' \neq 0}  \lVert 1_{\supp \chi_{j}} [  a_{k,j} b_{k} (\lambda_{q+1} \Phi_{j} ) ] \rVert_{C_{t,x,q+1}} \\
&\hspace{5mm} \times \lVert 1_{\supp \chi_{j'}} \nabla (a_{k', j'} \psi_{q+1, j', k'}) \rVert_{C_{t,x,q+1}} \lVert b_{k'} (\lambda_{q+1} x) \rVert_{C_{t,x,q+1}} \overset{\eqref{est 63a} \eqref{est 82d}}{\lesssim} e^{3L^{\frac{1}{4}}} \bar{e} \lambda_{q} \delta_{q+1}. 
\end{align*}
Hence, we have
\begin{align}\label{est 204a}
\lVert O_{4} \rVert_{C_{t,x,q+1}} \overset{\eqref{est 203}}{\leq} \sum_{k=1}^{2} \lVert O_{4k} \rVert_{C_{t,x,q+1}} \lesssim e^{3L^{\frac{1}{4}}} \lambda_{q} \delta_{q+1} \bar{e}. 
\end{align}
Therefore, 
\begin{align}\label{est 205} 
\lVert R_{O, \text{high}} \rVert_{C_{t,x,q+1}} \overset{\eqref{est 228}}{\leq} \lVert O_{3} \rVert_{C_{t,x,q+1}} + \lVert O_{4} \rVert_{C_{t,x,q+1}} \overset{\eqref{est 229} \eqref{est 204a}}{\lesssim}  e^{3L^{\frac{1}{4}}} \lambda_{q} \delta_{q+1} \bar{e};
\end{align}
this allows us to conclude for $\beta > \frac{1}{2}$ sufficiently close to $\frac{1}{2}$ and $a_{0}$ sufficiently large 
\begin{align*}
e^{2L^{\frac{1}{4}}} \lVert R_{O}& \rVert_{C_{t,x,q+1}} \overset{\eqref{est 230}}{\leq} e^{2L^{\frac{1}{4}}} [ \lVert R_{O,\text{approx}}  +   R_{O,\text{low}}   +  R_{O,\text{high}} \rVert_{C_{t,x,q+1}} ] \\
\overset{\eqref{est 231} \eqref{est 232} \eqref{est 205}}{\lesssim}& e^{2L^{\frac{1}{4}}} [ e^{-3L^{\frac{1}{4}}} \tau_{q+1} l^{-1} \lambda_{q+2} \delta_{q+2}\bar{e} + e^{3L^{\frac{1}{4}}} \lambda_{q} \delta_{q+1} \bar{e} ]  \overset{\eqref{mult bound b}}{\ll} e^{-3L^{\frac{1}{4}}} \lambda_{q+3} \delta_{q+3}. 
\end{align*}

\subsection{Proofs of Propositions \ref{Proposition 4.1} and \ref{Proposition 6.1}}\label{Section 7.8} 

\begin{proof}[Proposition \ref{Proposition 4.1}]
We prove part (2) first. By the hypothesis that $P_{l} \in \mathcal{C} ( s_{l}, \xi_{l}, C_{q})$ and Definition \ref{Definition 6}, an inductive argument partially following the approach of \cite[p. 441]{FR08} (cf. \cite[Equation (A.3)]{FR08} and \cite[Equations (75)--(76)]{Y19}), and an application of \cite[Corollary B.3]{FR08}, we obtain for any $T> 0$, 
\begin{equation}\label{est 210} 
\mathbb{E}^{P_{l}} [ \sup_{t \in [0,T]} \lVert \xi(t) \rVert_{\dot{H}_{x}^{\frac{1}{2}}}^{2} + \int_{s_{l}}^{T} \lVert \xi(r) \rVert_{\dot{H}_{x}^{\frac{1}{2} + \iota}}^{2} dr] \leq C, 
\end{equation}  
where the bound is uniform in $l$. This, together with the same reasonings in \cite[Equations (266)-(271)]{Y23a} hat lead to \cite[Equation (257)]{Y23a} and an application of Kolmogorov's test (see the argument in the rest of the proof of \cite[Proposition 4.1]{Y23a} in \cite[Section 5.1]{Y23a}) deduces that the family $\{P_{l}\}_{l \in \mathbb{N}}$ is tight in 
\begin{equation}\label{new 51}
\mathbb{S} \triangleq C( [0,\infty); (H_{\sigma}^{4})^{\ast}) \cap L_{\text{loc}}^{2} ([0,\infty); \dot{H}_{\sigma}^{\frac{1}{2}})
\end{equation} 
(see \cite[Equation (258)]{Y23a}). By Prokhorov's and Skorokhod's theorems, relabeling subsequence if necessary, we deduce that $P_{l}$ converges weakly to some probability measure $P \in \mathcal{P} (\Omega_{0})$ and that there exists a probability space $(\tilde{\Omega}, \tilde{\mathcal{F}}, \tilde{P})$ and $\mathbb{S}$-valued random variables $\tilde{\xi}_{l}$ and $\tilde{\xi}$ such that 
\begin{equation}\label{est 213} 
\mathcal{L} (\tilde{\xi}_{l}) = P_{l} \hspace{1mm} \forall \hspace{1mm} l \in \mathbb{N}, \hspace{1mm} \tilde{\xi}_{l} \to \tilde{\xi} \text{ in } \mathbb{S} \hspace{1mm} \tilde{P}\text{-a.s., and } \mathcal{L} ( \tilde{\xi}) = P. 
\end{equation} 
The remaining task is to show that $P \in \mathcal{C} ( s, \xi^{\text{in}}, C_{q})$; i.e., the properties (M1)-(M3).  Concerning (M1) of Definition \ref{Definition 6}, the first part $P (\{ \xi(t) = \xi^{\text{in}} \hspace{1mm} \forall \hspace{1mm} t \in [0,s] \}) = 1$ is same as \cite[Definition 4.1 (M1)]{Y23a} and the proof of \eqref{est 211} follows from application of \cite[Corollary B.3]{FR08} identically to \cite[p. 441]{FR08}. The verification of (M2) also follows from \cite{Y23a} because it is same as \cite[Definition 4.1 (M2)]{Y23a}. Thus, we now verify (M3) mostly following \cite[p. 441]{FR08}. We focus on verifying the case $q = 1$ as the other cases follow similarly. First, the integrability of $E^{1}$ w.r.t. $P$ follows from \eqref{est 213} and \eqref{est 210} taking advantage of the lower semi-continuity property of the norm therein. Thus, by Definition \ref{Definition 3}, to conclude that $E^{1}$ is an a.s. $(\mathcal{B}_{t}^{0})_{t\geq s}$-supermartingale under $P$, we only need to show \eqref{est 123}. Following the same approach of approximation on \cite[p. 441]{FR08}, we can consider a positive $\mathbb{R}$-valued $\mathcal{B}_{r}^{0}$-measurable continuous function $g$ on $\mathbb{S}$; thus, given any $t > 0$, it suffices to find a Lebesgue null set $T_{t} \subset (0,t)$ such that for all $r \notin T_{t}$, 
\begin{align}\label{est 214} 
\mathbb{E}^{P} [ g(\xi) ( E^{1} (t) - E^{1} (s)) ] \leq \mathbb{E}^{P} [ g(\xi) (E^{1} (r) - E^{1} (s)) ]
\end{align}
so that $\mathbb{E}^{P} [ E^{1}(t) - E^{1}(r) \lvert \mathcal{B}_{r}^{0} ] \leq 0$, after which we can define $T \triangleq \cup_{t \in D} T_{t}$ for any countable dense set $D \subset [0,\infty)$ as the set of exception times. Now, by hypothesis we have $s \in [0,\infty)$. If $s = 0$ so that we only have to consider $r > s$, then it is a typical situation and the proof in \cite[Lemma A.3]{FR08} applies to verify \eqref{est 214}. If $s > 0$ and $r < s$, then we only have to verify 
\begin{equation}\label{est 215}
\mathbb{E}^{P} [ g(\xi) \left(E^{1}(t) - E^{1} (s) \right) ] \leq 0.   
\end{equation} 
By hypothesis we have $P_{l} \in \mathcal{C} (s_{l}, \xi_{l}, C_{q})$ which implies that $E^{1}(t) - E^{1}(s_{l}) $ is an a.s. $(\mathcal{B}_{t}^{0})_{t \geq s_{l}}$-supermartingale under $P_{l}$. Therefore, the desired \eqref{est 215} can be verified using \eqref{est 213} once we show 
\begin{equation}\label{est 216} 
\int_{s}^{t} \lVert \tilde{\xi}(r) \rVert_{\dot{H}_{x}^{\frac{1}{2} + \iota}}^{2} dr \leq \liminf_{l\to\infty} \int_{s_{l}}^{t} \lVert \tilde{\xi}_{l} (r) \rVert_{\dot{H}_{x}^{\frac{1}{2} + \iota}}^{2} dr \hspace{3mm} \tilde{P}\text{-a.s.}
\end{equation} 
To verify \eqref{est 216}, we observe that due to \eqref{new 51} and \eqref{est 213}
\begin{align*}
& \lVert \tilde{\xi}_{l} 1_{[s_{l}, t]} - \tilde{\xi} 1_{[s, t]} \rVert_{L^{2} (0, T; \dot{H}_{\sigma}^{\frac{1}{2}})}^{2} \\
\leq& \int_{0}^{T} \lVert ( \tilde{\xi}_{l} - \tilde{\xi}) 1_{[s_{l}, t]} (r) \rVert_{\dot{H}_{x}^{\frac{1}{2}}}^{2} dr + \int_{0}^{T} \lVert \tilde{\xi} \left( 1_{[s_{l}, t]} - 1_{[s, t]} \right) (r) \rVert_{\dot{H}_{x}^{\frac{1}{2}}}^{2} dr  \to 0 
\end{align*}
as $l\nearrow + \infty$ $\tilde{P}$-a.s. From this, \eqref{est 216} follows by lower semi-continuity. Therefore, we conclude that $E^{1}$ is an a.s. $(\mathcal{B}_{t}^{0})$-supermartingale under $P$ and thus the verification of (M3). 

We now prove part (1) of Proposition \ref{Proposition 4.1}. This follows from Galerkin approximation as in \cite{FR08, GRZ09}; the only difference is (M3) that we verify. We consider a Galerkin approximation of 
\begin{subequations}\label{est 218} 
\begin{align}
& d v_{l} + \Pi_{l} \mathbb{P} [ u_{l}^{\bot} ( \nabla^{\bot} \cdot v_{l}) ] dt + \Lambda^{\gamma} v_{l} dt = \Pi_{l} \mathbb{P} dB, \hspace{3mm} u_{l} = \Lambda v_{l},   \hspace{3mm} \nabla \cdot v_{l} = 0,\\
& v_{l}(t) = \Pi_{l} \xi^{\text{in}} \text{ for all } t \in [0, s], 
\end{align}
\end{subequations} 
where $\Pi_{l}$ is the standard Galerkin projection operator of Fourier modes up to$l$. It follows that the solution exists in some probability space $(\Omega, \mathcal{F}, \mathbf{P})$, along with a $GG^{\ast}$-Wiener process $B$ (e.g. \cite[Section 5.1]{Y23a} and \cite[Theorem 4.3.2]{Z12b}). Following the proof of \cite[Proposition 4.8]{HZZ21a}, for notational simplicity and w.l.o.g., we assume that the probability space and the Wiener processes do not depend on $l$ in the following computations as our goal here is to show that the law of $v_{l}$, denoted by $P_{l}$, satisfies (M3). First, we apply It$\hat{\mathrm{o}}$'s formula on \eqref{est 218} to deduce that 
\begin{align} 
M_{l, v_{l}}^{E} (t) \triangleq& \lVert v_{l} (t) \rVert_{\dot{H}_{x}^{\frac{1}{2}}}^{2} - \lVert v_{l} (s) \rVert_{\dot{H}_{x}^{\frac{1}{2}}}^{2} + 2 \int_{s} ^{t} \lVert v_{l} (r) \rVert_{\dot{H}_{x}^{\frac{1}{2} + \frac{\gamma}{2}}}^{2} dr \nonumber \\
& \hspace{10mm}  - \lVert \Pi_{l} G \rVert_{L_{2} (U, \dot{H}_{\sigma}^{\frac{1}{2}})}^{2} (t-s)  \label{est 221}
\end{align} 
is a martingale, where we especially relied on the following identity: as $\Lambda v_{l} = u_{l}$, 
\begin{align*}
\int_{\mathbb{T}^{2}} \Lambda^{\frac{1}{2}} v_{l} \cdot \Lambda^{\frac{1}{2}} \Pi_{l} \mathbb{P} [ u_{l}^{\bot} ( \nabla^{\bot} \cdot v_{l}) ] dx = \int_{\mathbb{T}^{2}} \Lambda v_{l} \cdot u_{l}^{\bot} ( \nabla^{\bot} \cdot v_{l} ) dx = 0.
\end{align*}
Now for all $t \geq r \geq s$, as $\iota \leq \frac{\gamma}{2}$, $P_{l}$-a.s., 
\begin{align}
& \lVert v_{l} (t) \rVert_{\dot{H}_{x}^{\frac{1}{2}}}^{2} - \lVert v_{l} (r) \rVert_{\dot{H}_{x}^{\frac{1}{2}}}^{2} + 2 \int_{r}^{t} \lVert v_{l} (m) \rVert_{\dot{H}_{x}^{\frac{1}{2} + \iota}}^{2} dm - \lVert  \Pi_{l} G \rVert_{L_{2} (U, \dot{H}_{\sigma}^{\frac{1}{2}})}^{2} (t-r)   \nonumber \\
\leq&M_{l,v_{l}}^{E} (t) - M_{l, v_{l}}^{E} (r)  \label{est 220} 
\end{align}
so that because $C_{1} \geq 1$ by Definition \ref{Definition 6}, 
\begin{equation}\label{est 219}
E^{1} (t) - E^{1 }(r) \overset{\eqref{est 220}}{\leq} M_{l, v_{l}}^{E} (t) -  M_{l, v_{l}}^{E} (r).
\end{equation} 
It follows that $\{E^{1}(t) - E^{1}(s) \}_{t \geq s}$ is a $(\mathcal{B}_{t}^{0})_{t \geq s}$-supermartingale and the proof of Proposition \ref{Proposition 4.1} is complete.  
\end{proof}

\begin{proof}[Proof of Proposition \ref{Proposition 6.1}]
We prove part (2) first similarly to the proof of Proposition \ref{Proposition 4.1}. By hypothesis, we have $P_{l} \in \mathcal{W} (s_{l}, \xi_{l}, \theta_{l}, C_{q})$ and thus  $E^{q}(t)$ defined by \eqref{est 125} is an $(\bar{\mathcal{B}}_{t}^{0})_{t\geq s_{l}}$-supermartingale under $P_{l}$. Thus, a similar approach to the additive case leads us to the identical inequality in \eqref{est 210}. It follows from the same arguments in \cite[Section 6.2]{Y23b} that $\{P_{l}\}_{l \in \mathbb{N}}$ is tight in 
\begin{equation*}
\bar{\mathbb{S}} \triangleq C ( [0,\infty); (H_{\sigma}^{4})^{\ast} \times U_{1}) \cap L_{\text{loc}}^{2} ([0,\infty); \dot{H}_{\sigma}^{\frac{1}{2}} \times U_{1})
\end{equation*} 
(see \cite[Equation (150)]{Y23b}). By Prokhorov's and Skorokhod's theorems, relabeling subsequence if necessary, we deduce that $P_{l}$ converges weakly to some probability measure $P \in \mathcal{P} (\bar{\Omega}_{0})$ and that there exists some probability space $(\tilde{\Omega}, \tilde{\mathcal{F}}, \tilde{P})$ and $\bar{\mathbb{S}}$-valued random variable ($\tilde{\xi}_{l},\tilde{\theta}_{l})$ and $(\tilde{\xi}, \tilde{\theta})$ such that 
\begin{equation}\label{est 238} 
\mathcal{L}(\tilde{\xi}_{l}, \tilde{\theta}_{l}) = P_{l} \hspace{1mm} \forall \hspace{1mm} l \in \mathbb{N}, \hspace{2mm} (\tilde{\xi}_{l}, \tilde{\theta}_{l}) \to (\tilde{\xi},\tilde{\theta}) \text{ in } \bar{\mathbb{S}} \hspace{1mm} \tilde{P}\text{-a.s.}, \hspace{2mm} \text{ and } \mathcal{L} ( \tilde{\xi},\tilde{\theta}) = P. 
\end{equation}  
We denote by $(\tilde{\mathcal{F}}_{t})_{t\geq 0}$ the $\tilde{P}$-augmented canonical filtration of $(\tilde{\xi},\tilde{\theta})$, and consequently, $\tilde{\theta}$ is a cylindrical Wiener process on $U$ w.r.t. $(\tilde{\mathcal{F}}_{t})_{t\geq 0}$.  The remaining tasks are to verify the properties (M1)-(M3). The verifications of (M1) and (M3) is similar to the additive case while (M3) can be demonstrated similarly to the proof of \cite[Proposition 4.1]{Y23b}.  

Concerning Proposition \ref{Proposition 6.1} (1), as described in the proof of \cite[Theorem 5.1]{HZZ19} and \cite[Proposition 4.1]{Y23b}, the existence of a probabilistically weak solution follows from the existence of a martingale solution via martingale representation theorem. This completes the proof of Proposition \ref{Proposition 6.1}. 
\end{proof}

\end{document}